\documentclass[12pt,a4paper, oneside]{book}
\pdfoutput=1
\usepackage{blindtext} 
\usepackage{amsmath}
\usepackage{amscd}
\usepackage{amssymb}
\usepackage{mathtools}
\usepackage{amsthm}
\usepackage{amsfonts}
\usepackage{dsfont}
\usepackage{graphicx}
\usepackage{enumerate}
\usepackage{mathrsfs}
\usepackage{bbm}
\usepackage{float}
\usepackage{longtable}
\usepackage[all,cmtip]{xy}
\usepackage{stmaryrd}
\usepackage{multicol}
\usepackage{url}
\usepackage{fullpage}
\usepackage{color}
\usepackage[nottoc,numbib]{tocbibind}
\usepackage[symbol]{footmisc}
\usepackage{hyperref}
\usepackage{cleveref}
\usepackage{tikz}
\usetikzlibrary{matrix,arrows,decorations.pathmorphing, positioning, fit, calc, shapes}
\tikzset{
	subseteq/.style={
		draw=none,
		edge node={node [sloped, allow upside down, auto=false]{$\subseteq$}}},
	Subseteq/.style={
		draw=none,
		every to/.append style={
			edge node={node [sloped, allow upside down, auto=false]{$\subseteq$}}}
	},
}
\tikzset{
	symbol/.style={
		draw=none,
		every to/.append style={
			edge node={node [sloped, allow upside down, auto=false]{$#1$}}}
	}
}
\usepackage[symbol]{footmisc}
\usepackage{hyperref}
\usepackage{cleveref}
\newtheorem{thm}{Theorem}[section]
\newtheorem{theorem}[thm]{Theorem} 
\newtheorem{cor}[thm]{Corollary}
\newtheorem{lem}[thm]{Lemma}

\newtheorem{prop}[thm]{Proposition}

\theoremstyle{definition}
\newtheorem{exm}[thm]{Example}
\newtheorem{exms}[thm]{Examples}
\newtheorem{ex}[thm]{Example}
\newtheorem{defn}[thm]{Definition}
\newtheorem{dfn}[thm]{Definition}
\newtheorem{rmk}[thm]{Remark}
\newtheorem{rmks}[thm]{Remarks}
\newtheorem{remark}[thm]{Remark}
\newtheorem{exr}[thm]{Exercise}
\newtheorem{exercise}[thm]{Exercise}

\newtheorem{quest}[thm]{Question}
\newtheorem{question}[thm]{Question}
\newtheorem{claim}[thm]{Claim}

\newtheorem{conv}[thm]{Convention}

\newtheoremstyle{nonum}{}{}{\itshape}{}{\bfseries}{.}{ }{\thmnote{#3}}
\theoremstyle{nonum}

\DeclareMathOperator{\rk}{rank} \DeclareMathOperator{\tr}{tr}
 \DeclareMathOperator{\spn}{span}

 \DeclareMathOperator{\ind}{ind}
\DeclareMathOperator{\dist}{dist}

 \DeclareMathOperator{\im}{im}
 
\DeclareMathOperator{\grph}{graph}

\DeclareMathOperator{\length}{length}

 \DeclareMathOperator{\diam}{diam}

\DeclareMathOperator{\proj}{proj}
\DeclareMathOperator{\spec}{spec}
\DeclareMathOperator{\Spec}{Spec}
\DeclareMathOperator{\dis}{dis}
\DeclareMathOperator{\Totaldim}{Totaldim}
\DeclareMathOperator{\coim}{coim}

\newcommand{\mrm}{\mathrm}

\newcommand{\calA}{{\mathcal{A}}}
\newcommand{\calB}{{\mathcal{B}}}
\newcommand{\calC}{{\mathcal{C}}}
\newcommand{\calD}{{\mathcal{D}}}

\newcommand{\calF}{{\mathcal{F}}}

\newcommand{\calH}{{\mathcal{H}}}

\newcommand{\calL}{{\mathcal{L}}}

\newcommand{\calT}{{\mathcal{T}}}
\newcommand{\calU}{{\mathcal{U}}}

\newcommand{\epsi}{\varepsilon}

\newcommand{\vphi}{\varphi}

\newcommand{\T}{\mathbb T}

\newcommand{\setm}{\setminus}

\newcommand{\id}{\mathbbm 1}

\newcommand{\de}{\partial}
 
\renewcommand{\d}{d}

\newcommand{\restr}{\big|}  

\newcommand{\Symp}{\mrm{Symp}}

\newcommand{\ol}{\overline}
\newcommand{\til}{\tilde}

\newcommand{\into}{\hookrightarrow}

\newcommand{\RN}[1]{%
	\textup{\uppercase\expandafter{\romannumeral#1}}%
}

\def\ep{\epsilon}

\def\d{d_{\rm Hofer}}
\def\f{{\mathfrak{f}}}
\def\g{{\mathfrak{g}}}
\def\R{\mathbb{R}}
\def\Z{\mathbb{Z}}
\def\N{\mathbb{N}}
\def\C{\mathbb{C}}

\def\F {\mathbb{F}}

\def\CF{{\rm CF}}
\def\HF{{\rm HF}}
\def\SH{{\rm SH}}

\def\Ham{{\rm Ham}}
\def\SBM{{\rm SBM}}

\def\Diff{{\rm Diff}}

\DeclarePairedDelimiter\floor{\lfloor}{\rfloor}
\DeclarePairedDelimiter\ceil{\lceil}{\rceil}
\newcommand{\Osc}{{\rm Osc}}

\begin{document}

\frontmatter

\title{Topological Persistence in Geometry and Analysis}
\date{}
\author{Leonid Polterovich, Daniel Rosen, Karina Samvelyan, Jun Zhang}

\maketitle

\section*{An erratum to this arXiv version}

Since this preprint was posted to arXiv in April 2019, the
text has been updated, see the web-page
\url{https://sites.google.com/site/polterov/miscellaneoustexts/topological-persistence-in-geometry-and-analysis},
and eventually polished and published by the American Mathematical Society, University Lecture Series, Volume 74 (2020). Its print ISBN is 978-1-4704-5495-1. For more details, please go to the website: \url{https://bookstore.ams.org/ulect-74}. \\

For reader's convenience, in this erratum we point out and correct several errors {\it of the current arXiv version}. Most of them are corrected in later versions of the text. We are also aware that in this arXiv version, several references in the bibliography are either missing or appear
not in their most updated format. We refer the reader to the published version for updated references and acknowledgments, as well as for
an improved exposition of various topics discussed in this preprint.  

\bigskip

\noindent $\bullet$ The paragraph above Exercise \ref{exr:end_interval_module} that addresses the second proof of the uniqueness of the Normal Form is inaccurate. In fact, the proof given on Page 19 follows the proof of Theorem 3.6 in Nathan Jacobson's {\it Basic Algebra {\rm II}} (1980). \\

\noindent $\bullet$ The proper persistence modules that are defined and discussed in Section \ref{sec-prop-pm} are re-named as {\it persistence modules of locally finite type} in the updated and also published versions. This seems to be a more commonly-used name that circulates around applied algebraic topologists. \\

\noindent $\bullet$ In the item (2) in subsection \ref{ssec-proof-3.2.2}, the second diagram should be 
\begin{center}
		\begin{tikzpicture}
		\matrix (m) [matrix of math nodes,row sep=2em,column sep=2em,minimum width=2em]
		{
			\im \Phi_W^{2\delta} & \im f[\delta] & W[2\delta] \\
		};
		\path[-stealth, decoration={snake,segment length=4,amplitude=3,
			post=lineto,post length=10pt}]
		(m-1-1) edge node [above] {$j$} (m-1-2)
		(m-1-2) edge node [above] {$i$} (m-1-3)
		(m-1-1) edge[bend left=-20] node [below] {$k$} (m-1-3);
		\end{tikzpicture} \;.
	\end{center}
for some morphism $k$ (instead of $\Phi_W^{2\delta}$). Similarly, in its following diagram, it should be $\mu_{inj}(k)$ instead of $\mu_{inj}(\Phi_W^{2\delta})$, as well as $\mu_{inj} (k) (b,d-2\delta] = (b-2\delta, d - 2\delta]$. Finally, in the item (3), $\mathcal B(W)$ should be $\mathcal B(W[2\delta])$, and  $\mu_{inj} (k) (a+2\delta, d] = \mu_{inj} (i) (b,d] = (a, d]$. \\

\noindent $\bullet$ In Figure \ref{fig: multiplicity_func}, $\mu_1$ should be equal to $\frac{a_4-a_1}{4}$. \\

\noindent $\bullet$ In Example \ref{ex-4.4.9}, both $((V', \pi), \rho^{V})$ and $((W', \theta), \rho^{W})$ are subrepresentations of the type that is discussed in Example \ref{ex-power}. \\

\noindent $\bullet$ In the line above the equation (\ref{eq-baran-graded}), it should be $\phi_k: I_k \to \Omega_{m_{k-1}} \backslash I_{k-1}$. Similarly, in its following paragraph, the index $j$ is chosen from $\Omega_{m_k} \backslash (I_k \sqcup \phi_{k+1}(I_{k+1}))$. \\

\noindent $\bullet$ After equation \eqref{eq-spectral-norm}: $C^0$-continuity of the spectral norm was proved by Seyfaddini \cite{Sey13} in the case of surfaces, and by Buhovsky, Humili\'{e}re and Seyfaddini \cite{BHS18} for general closed
symplectically aspherical manifolds. \\

\noindent $\bullet$ The equation (\ref{c0-sn}) from a recent work by A.~Kislev and E.~Shelukhin in \cite{KS18} holds only with {\it shifted} barcodes. More precisely, the correct version should be 
\[ d_{bot}(\mathcal B(\phi), \mathcal B(\psi)[a]) \leq \frac{1}{2} \gamma(\psi^{-1} \circ \phi)\;,\]
where $a = -\frac{1}{2}(c(\psi^{-1} \circ \phi, [M]) + c(\psi^{-1} \circ \phi, [pt]))$. Here, $\mathcal B(\psi)[a]$ is the image of $\mathcal B(\psi)$ under the shift $t \to t-a$. \\

\noindent $\bullet$ The equation (\ref{c0-b}) should be
\[  \mathcal B: (\overline{{\rm Ham}}(M, \omega), d_{C^0}) \to (\overline{\rm Barcodes}, d_{bot})\;.\]
Here $\overline{{\rm Ham}}(M, \omega)$ stands for the
Hamiltonian homeomorphism group of $(M, \omega)$.  \\

\noindent $\bullet$ Section \ref{sec-barhomeo}, end: Roughly speaking, weakly conjugate elements cannot be distinguished by
any {\bf conjugation invariant} continuous functional on the group
{\it (this omission pertains to the published version as well)}.\\

\noindent $\bullet$ The definition of symplectic Banach-Mazur distance in Section \ref{sec-sbm} involving Definition \ref{dfn-SBM} and Definition \ref{dfn-SBM-2} missed a key condition (called the {\it unknottedness condition} by Gutt-Usher in \cite{GU17}). The correct Definition \ref{dfn-SBM} should be given as follows. 

\begin{dfn} \label{dfn-SBM-new} Let $U, V \in \mathcal S^{2n}$. A real number $C>1$ is called {\it (U,V)-admissible} if there exists a pair of symplectomorphisms $\phi, \psi \in {\rm Symp}_{ex}(M)$ such that $\frac{1}{C}U \xhookrightarrow{\phi} V \xhookrightarrow{\psi} CU$ and there exists an isotopy of Liouville morphisms from $\frac{1}{C}U$ to $CU$, denoted by $\{\Phi_s\}_{s \in [0,1]}$, such that $\Phi_0 = \mathds{1}$ and $\Phi_1 = \psi \circ \phi$. Note that by the definition of a Liouville morphism, for every $s \in [0,1]$, $\Phi_s(\overline{\frac{1}{C} U}) \subset CU$.  \end{dfn}

This extra second condition (unknottedness condition) in Definition \ref{dfn-SBM-new} enables the applications of symplectic persistence module theory that is elaborated in Section \ref{sec-spm}. However, it makes the definition of being $(U,V)$-admissible as above less symmetric (in fact, it is in general not symmetry due to Theorem 1.4 in \cite{Ush18}). Therefore, Definition \ref{dfn-SBM-2} should be corrected as follows. 

\begin{dfn} (Ostrover, Polterovich, Gutt, Usher) Define the {\it symplectic Banach-Mazur distance} between $U$ and $V$ by
\[ d_{\rm SBM}(U, V) = \inf\{\ln C>0 \,|\, \mbox{$C$ is both $(U,V)$-admissible and $(V,U)$-admissible}\}. \]
\end{dfn}

\tableofcontents

\chapter*{Preface}
\addcontentsline{toc}{chapter}{\protect\numberline{}Preface}

The theory of persistence modules is an emerging field of algebraic topology which originated in topological data analysis and which lies on the crossroads of several disciplines including metric geometry
and the calculus of variations. Persistence modules were introduced by G.~Carlsson and A.~Zamorodian \cite{ZC05} in 2005 as an abstract algebraic language for dealing with persistent homology, a version of homology theory
invented by H.~Edelsbrunner, D.~Letscher and A.~Zomorodian \cite{edelsbrunner2000topological} at the beginning of the millennium aimed, in particular, at extracting robust information from noisy topological patterns.
We refer to the articles  by H.~Edelsbrunner and J.~Harer \cite{edelsbrunner_harer_08}, R.~Ghrist \cite{ghrist_barcodes_2008}, G.~Carlsson \cite{carlsson_09_top&data}, S.~Weinberger \cite{weinberger_11_whatis}, U.~Bauer and M.~Lesnick \cite{Bauer_Lesnick_13-16} and the monographs by H.~Edelsbrunner \cite{edelsbrunner_14short}, S.~Oudot \cite{oudot_persistence_book_15}, F.~Chazal, V.~de Silva, M.~Glisse and S.~Oudot \cite{Chazal_DeSila_Glisee_Oudot} for various aspects of this rapidly developing subject. In the past few years, the theory of persistence modules expanded its ``sphere of influence" within pure mathematics exhibiting a fruitful interaction with function theory and symplectic geometry.  The purpose of these notes is to provide a concise introduction into this field and to give an account on some of the recent advances emphasizing applications to geometry and analysis. The material should be accessible to readers with a basic background in algebraic and differential topology. More advanced preliminaries in geometry and function theory will be reviewed.

\medskip

Topological data analysis deals with data clouds modeled by finite metric spaces. Its main motto is
$$\text{geometry} + \text{scale}= \text{topology}\;.$$ In case when a finite metric space appears as
a discretization of a Riemannian manifold $M$, the above equation enables one to infer the topology of
$M$ provided one knows the mesh. In general,  given a scale $t >0$, one can associate a topological
space $R_t$ called {\it the Vietoris-Rips complex} associated to any abstract metric space $(X,d)$. By definition,
$R_t$ is a subcomplex of the full simplex $\Sigma$ formed by the points of $X$, where $\sigma \subset X$ is a simplex of $R_t$
whenever the diameter of $\sigma$ is $< t$. For instance, the Rips complex for the vertices of the unit square in
the plane is presented in \Cref{fig: square_rips_barcode} a). Thus we get a filtered topological space, a.k.a., a collection
of topological spaces $R_t$, $t \in \R$ with $R_s \subset R_t$ for $s < t$. Let us mention that $R_t$ is empty
for $t \leq 0$ and $R_t = \Sigma$ for $t > \text{diam}(X,d)$. Some authors call this structure {\it a topological signature}
of the data cloud $(X,d)$. Rips complexes, which were originated in geometric group theory \cite{Bridson_Haefliger_2011},  play also an important role in detecting low-dimensional topological patterns in big data, nowadays an active area of applied mathematics (see e.g. \cite{Niyogi_Smale_Weinberger_08}.

\begin{figure}[!ht]
	\centering
	\includegraphics[scale=1]{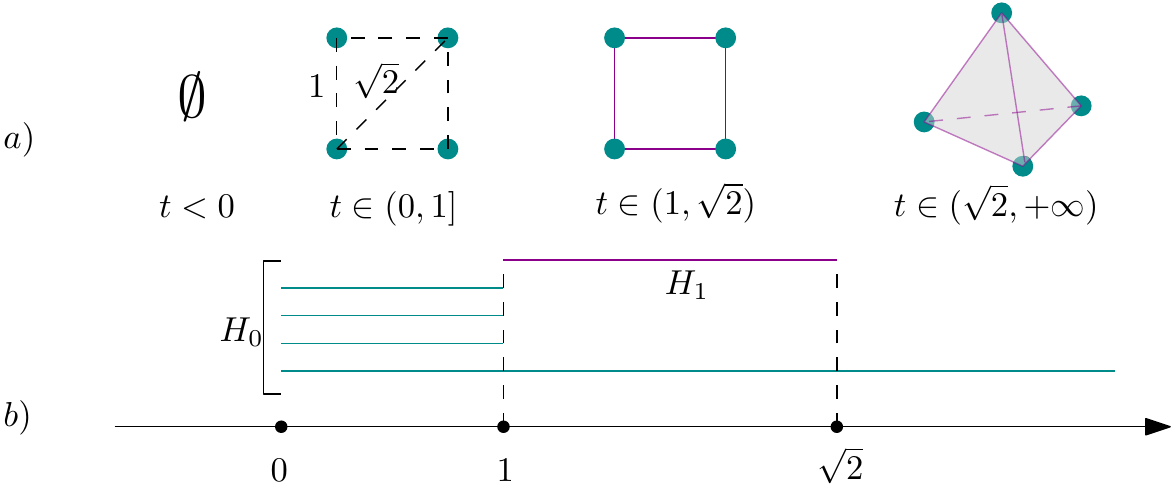}
	\caption{The Rips complex of a square and the corresponding barcode.}
	\label{fig: square_rips_barcode}
\end{figure}

The calculus of variations studies critical points and critical values of functionals, the simplest case
being smooth functions on manifolds. Sublevel
sets $R_t:= \{f < t\}$ of a function $f$ on a closed manifold $M$ induce a structure of a filtered topological space.
According to Morse theory, the topology of $R_t$, $t \in \R$  changes exactly when the parameter $t$
hits a critical value of $f$. Note that $R_t = \emptyset$ when $ t \leq \min f$ and $R_t = M$ when $t > \max f$. See \Cref{fig: sphere_homol_barcode} a) illustrating
sublevel sets of a function on the two dimensional sphere with two local maxima and one local minima.

\begin{figure}[!ht]
	\centering
	\includegraphics[scale=1]{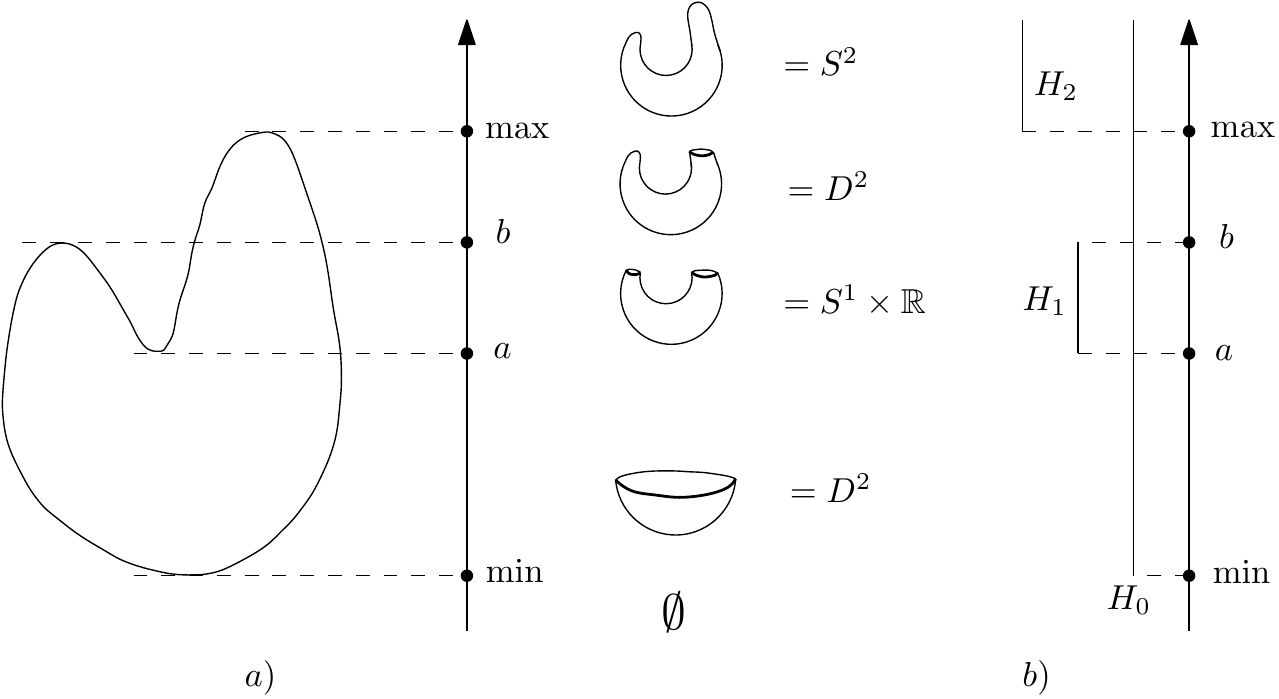}
	\caption{The height function on the (topological) sphere and the corresponding barcode.}
	\label{fig: sphere_homol_barcode}
\end{figure}

We are going to study filtered topological spaces by using algebraic tools.
Fix a field $\F$ and look at the homology $V_t:= H(R_t,\F)$ of spaces $R_t$ as above with coefficients in $\F$. The family of vector spaces
$V_t$, $t \in \R$ together with the morphisms $V_s \to V_t$, $s < t$ induced by the inclusions, form an algebraic object called {\it a persistence module},  which plays a central role in the present notes.

\medskip

Let us discuss the contents of the book in more detail. Part \ref{part-I} lays foundations of the theory of persistence modules and introduces basic examples. It turns out that persistence modules (which are defined in Chapter
\ref{chp1_definition_first_examples}) are classified by simple combinatorial
objects, the barcodes, which are defined as collections of intervals and rays in $\R$ with multiplicities. While the real meaning of barcodes will be clarified later on, some intuition can be gained by looking at Figures \ref{fig: square_rips_barcode} b) and \ref{fig: sphere_homol_barcode} b). In this figures, for illustrative purposes,
the bars are equipped with an additional decoration corresponding to the degree of homology they represent. The number of bars in degree $k$ over a point $t \in \R$ equals the $k$-th Betti number of the space $R_t$. For instance, for the bar in degree $1$  manifests that the spaces
$R_t$ possess non-trivial first homology for $t \in (1,\sqrt{2}]$ on \Cref{fig: square_rips_barcode} b) and for $t \in (a,b]$ on \Cref{fig: sphere_homol_barcode} b). Look also at the bars in
degree $0$ on \Cref{fig: sphere_homol_barcode} b), that is $(0,1)$ taken with multiplicity $3$ and $(0,+\infty)$. This carries the following information:
$H_0(R_s) = \F^4$ for $s \in (0,1]$, $H_0(R_t) = \F$ for $t >1$, and the map $H_0(R_s) \to H_0(R_t)$ does not vanish. Very roughly speaking,
this means that one (and only one) of the four generators of $H_0(R_s)$ persists when $s$ increases and hits the value $1$.
In Chapter \ref{chp2_barcodes} we will make this intuitive picture rigorous.

A highlight of the theory of persistence modules is an isometry between the space of persistence modules equipped with a certain algebraic distance, which naturally appears in applications but is hard to calculate, and the space of barcodes equipped with a user-friendly {\it bottleneck distance} of a combinatorial nature. This is a difficult fact
discovered by F.~Chazal, D.~Cohen-Steiner, M.~Glisse, L.~Guibas and S.~Oudot in \cite{chazal2009proximity}.
It will be proved below (see Chapters \ref{chp2_barcodes} and \ref{chp: isometry_thm_proof}) following the approach
by U.~Bauer and M.~Lesnick \cite{Bauer_Lesnick_13-16}.

Thus one can associate to a Morse function on a closed manifold or to a finite metric space
a barcode. Remarkably, this correspondence is stable, or, more precisely, Lipschitz with respect to the the uniform norm on functions and the Gromov-Hausdorff distance on metric spaces. This fundamental phenomenon was discovered by D.~Cohen-Steiner, H.~Edelsbrunner, and J.~Harer
\cite{cohen2007stability} for functions, and by F.~Chazal, V.~de Silva and S.~Oudot \cite{CDO14} for metric spaces.
In particular,  metric spaces whose barcodes are remote in the bottleneck distance are far from being isometric, and a small $C^0$-perturbation of a function cannot significantly change its barcode. The stability of barcodes with respect to $C^0$-perturbations of functions paves way to applications of persistence modules to topological function theory, a theme which we
develop in Chapter \ref{chap5a-functheory}.

In Chapter \ref{chp3_read_from_a_barcode} we discuss some natural Lipschitz functionals on the space of barcodes
which yield interesting numerical invariants of functions and metric spaces. They include, for instance, the end-points of infinite rays, which in the case of functions correspond to the homological min-max.
Another example is given by the length of the longest finite bar in the barcode which is called {\it the boundary depth}, an invariant introduced by M.~Usher in \cite{Ush11}. The boundary depth gives rise to a non-negative functional on smooth functions on a manifold which is Lipschitz in the uniform norm, invariant under the action of diffeomorphisms on functions, and sends each function to the difference between a pair of its critical values. The very existence of such a functional different from $f \mapsto \max f - \min f$ is not at all obvious.  We conclude
with the multiplicity function, an invariant which appears in the study of representations of finite groups on persistence modules and which will be useful for applications to symplectic geometry in
Chapter \ref{chp7_hamiltonian_persistence_modules}.

\medskip

In Part \ref{part-II} of the book we elaborate applications of persistence modules
to metric geometry and function theory. Chapter  \ref{chap5-rips} focuses on Rips complexes. After reviewing their origins in geometric group theory (here our exposition closely follows a book \cite{Bridson_Haefliger_2011}
by M.~Bridson and A.~Haefliger), we discuss appearance of Rips complexes in data analysis.
We present a toy version of manifold learning motivated by a seminal paper \cite{Niyogi_Smale_Weinberger_08}
by P.~Nyogi, S.~Smale and S.~Weinberger.

Chapter \ref{chap5a-functheory} deals with topological function theory which
studies features of smooth functions on a manifold that are invariant under
the action of the diffeomorphism group. The theory of persistence modules provides a wealth of invariants coming
from the homology of the sublevel sets of a function. We shall focus, roughly speaking, on the ``size" of the barcode which can be considered as a useful measure of oscillation of a function. We prove bounds on this size in terms of norms of a function and its derivatives and discuss links to approximation theory. This chapter is mostly based on papers \cite{CSEHM} by D.~Cohen-Steiner, H.~Edelsbrunner, J.~Harer and Y.~Mileyko,
\cite{Polterovich-Sodin} by L.~Polterovich and M.~Sodin and \cite{Polterovich2-Stojisavljevic} by
I.~Polterovich, L.~Polterovich and V.~Stojisavljevi\'c. In the course of exposition we present also
an algorithm for finding a canonical normal form of filtered complexes with a preferred basis due
to S.~Barannikov \cite{barannikov1994framed}.

\medskip

In Part \ref{part: applications_symp}, after a crash-course on symplectic geometry and Hamiltonian dynamics
(see Chapter \ref{chp-intro-sg}), we discuss their interactions with the theory of persistence modules. Here instead of functions on a finite-dimensional manifolds the object of interest is the classical action functional on the loop space of a symplectic manifold. It was a great insight due to A.~Floer \cite{Flo89} that by using the theory of
elliptic PDEs and Gromov's theory of pseudo-holomorphic curves in symplectic manifolds \cite{Gro85}
one can properly define a Morse-type homology theory for sublevel sets of the action functional.
L.~Polterovich and E.~Shelukhin \cite{PS16}
and M.~Usher and J.~Zhang \cite{UZ16} showed that  filtered Floer homology gives rise to persistence modules and barcodes. We shall elaborate this construction in two different contexts: Hamiltonian diffeomorphisms of symplectic manifolds (\Cref{chp7_hamiltonian_persistence_modules})
and starshaped domains of Liouville manifolds (\Cref{chp8_symplectic_persistence_modules}).
The group of Hamiltonian diffeomorphisms is equipped
with Hofer's bi-invariant metric introduced by H.~Hofer in 1990 \cite{Hof90}, which is playing a central role in symplectic topology for almost 3 decades, while
the space of starshaped domains also has a natural structure of a metric space with respect to a
non-linear analogue of the Banach-Mazur classical distance on convex bodies (\Cref{chp8_symplectic_persistence_modules}).
The exploration of the non-linear Banach-Mazur distance, which has been introduced following  unpublished ideas of Y.~Ostrover and L.~Polterovich circa 2015 with an important modification by M.~Usher and J.~Gutt \cite{GU17}, nowadays is making its very first steps, see papers \cite{SZ18} by V.~Stojisavljevi\'c and J.~Zhang and \cite{Ush18} by M.~Usher. We shall outline the proof of symplectic stabilities theorems stating that the correspondence sending a Hamiltonian diffeomorphism (resp., a starshaped domain) to its barcode is Lipschitz with respect to Hofer's (resp., non-linear Banach-Mazur) distance.

Barcodes of Hamiltonian diffeomorphisms carry some interesting information.
For instance, one can read from them spectral invariants introduced by C.~Viterbo \cite{Vit92}, M.~Schwarz \cite{Sch00}  and Y.-G.~Oh \cite{Oh05} , as
well as the above-mentioned boundary depth \cite{Ush11}. Furthermore, the natural action by  conjugation of a diffeomorphism on the Floer homology of its power gives rise to a basic representation theory of the cyclic group $\Z_p$ on Floer's barcodes, yielding in turn applications to geometry and dynamics. In Chapter  \ref{chp7_hamiltonian_persistence_modules} we discuss some of these advances due to
L.~Polterovich and E.~Shelukhin \cite{PS16}, J.~Zhang \cite{Zha16}, and  L.~Polterovich, E.~Shelukhin and V.~Stojisavljevi\'c \cite{PSS17}.

Persistence modules associated to starshaped domains have meaningful applications to embedding problems in symplectic topology. We illustrate this by presenting a proof of M.~Gromov's famous non-squeezing theorem \cite{Gro85}
in Chapter \ref{chp8_symplectic_persistence_modules}.

\medskip

The notes are based on various mini-courses given by L.P. at Tel Aviv University, University of Chicago,
Kazhdan's Sunday seminar in the Hebrew University, CIRM at Luminy and MSRI,
as well as on several seminar talks by L.P. and J.Z. We thank the speakers of the guided reading courses at Tel Aviv University, Arnon Chor, Yaniv Ganor, Pazit Haim-Kislev, Asaf Kislev and Shira Tanny for their input.
In particular, our exposition of the Bauer-Lesnick proof of the isometry theorem used unpublished notes
due to Asaf Kislev. The authors cordially thank Lev Buhovsky, David Kazhdan, Yaron Ostrover, Iosif Polterovich, Egor Shelukhin, Vuka\v sin Stojisavljevi\'c, Michael Usher  and Shmuel Weinberger for numerous useful discussions on persistent homology. Special thanks go to Peter Albers for very useful comments on an early draft of this book.
L.P., D.R. and J.Z. were partially supported by the European Research Council Advanced grant 338809.
K.S. was partially supported by the Israel Science Foundation grant 178/13.

\mainmatter

\part{A primer of persistence modules}\label{part-I}

\chapter{Definition and first examples} \label{chp1_definition_first_examples}


\section{Persistence modules}

Initially developed in the realm of topological data analysis, persistence homology has been found useful in keeping information coming from various homology theories that appear in symplectic topology.
We introduce the category of persistence modules and discuss several examples to get started.

Let us fix a field $\F$.

\begin{defn} \label{defn: pm}
	A \emph{persistence module} is a pair $(V, \pi)$, where
	$V$ is a collection $\{ V_t \}$, ${t\in \R}$, of finite dimensional vector spaces over $\F$,
	and
	$\pi$ is a collection $\{ \pi_{s,t} \}$ of linear maps $\pi_{s,t} : V_s \to V_t$ for all $s\leq t$ in $\R$,
	so that
	\begin{enumerate} [(1)]
		\item (\emph{Persistence})
			For any $s \leq t \leq r$ one has $\pi_{s,r} = \pi_{t,r} \circ \pi_{s,t}$, i.e.\ the following diagram commutes:
			\begin{center}
				\begin{tikzpicture}[scale=1.5, baseline=0]		
				\node (V_s) at (-1,0) {$V_s$};
				\node (V_t) at (0,0) {$V_t$};
				\node (V_r) at (1,0) {$V_r$};
				\path[->,font=\scriptsize,>=angle 90]
				%
				(V_s) edge node[below]{$\pi_{s,t}$} (V_t)
				(V_t) edge node[below]{$\pi_{t,r}$} (V_r)
				(V_s) edge[bend left=30] node[above]{$\pi_{s,r}$} (V_r);
				\end{tikzpicture}
				\;.
			\end{center}
		\item
			For all but a finite number of points $t\in \R$ there exists a neighborhood $U$ of $t$, such that $\pi_{s,r}$ is an isomorphism for any $s<r$ in $U$.
		\item (\emph{Semicontinuity})
			For any $t\in \R$ and any $s \leq t$ sufficiently close to $t$, the map $\pi_{s,t}$ is an isomorphism.
		\item
			There exists some $s_{-} \in \R$, such that $V_s = 0$ for any $s\leq s_{-}$.
	\end{enumerate}
	
\end{defn}

	\noindent
	Let us elaborate on the various conditions in Definition \ref{defn: pm}. The \emph{persistence} condition (1) is the heart of the definition, and some authors take it as the sole condition in the definition of a persistence module. Conditions (2) and (4) are sometimes called ``finite-type" assumptions, and they greatly simplify the presentation. As we will see, adopting these restrictions still allows for interesting examples of persistence modules, although they are sometimes omitted in favor of a more general definition (see \Cref{chp8_symplectic_persistence_modules}). Finally, condition (3) is superficial, and is included simply to allow uniqueness of decomposition of persistence modules into basic ``blocks" (see the Normal Form Theorem \ref{thm: normal_form_thm}).


\begin{rmks} \label{rmks: on_defn_of_pm_giving_def_of_V_infty}
	\begin{enumerate}
		\item
			Note that by conditions (1) and (3), for any $t\in \R$, $\pi_{t,t} = id_{V_t}$.
		\item
			One may check that by condition (2) there is $s_{+} \in \R$, such that for any $t>s\geq s_{+}$,
			$\pi_{s,t} : V_s \to V_t$ is an isomorphism, i.e.\ the collection $\{V_t\}$ stabilizes starting at some $s_+$.	We will use the notation $V_\infty$ when referring to this ``terminal" vector space, i.e.\ $V_\infty = V_t$ for $t$ large enough. Note also that $V_\infty$ is the direct limit of the system $\{V_t, \pi_{s,t} \}$.
	\end{enumerate}
\end{rmks}

Let us present now two fundamental examples that will reappear in the exposition.
\begin{exm} [Morse theory] \label{exm: morse_theory_pm}
	Let $X$ be a closed manifold (i.e.\, a smooth compact manifold without boundary) and let $f:X \to \R$ be a Morse function. Fix $0 \leq k\in \Z$ and put $V_t = H_k (\{ f<t \})$ (taking homology with coefficients in $\F$ throughout the text, where $\F$ is an arbitrary fixed field, unless stated otherwise).
	Consider the natural inclusion $\xymatrix{ \{ f<s \} \ \ar@{^(->}[r]_{i_{s,t}} & \{ f<t \} }$
	 for $s \leq t$. It induces the map
	 $\pi_{s,t} := (i_{s,t})_{*} : V_s \to V_t$ in homology, and one can verify that we get a persistence module.
\end{exm}

\begin{rmk}
	Later on, we will also write $V_t = H_* (\{f<t\})$, referring to homology of some arbitrary degree $*$.
\end{rmk}

\begin{exm} [Finite metric spaces, Rips complex] \label{exm: finite_metric_sp_Rips_pm}
	Let $(X,d)$ be a finite metric space.
	For $0< \alpha \in \R$ define the simplicial complex $R_\alpha(X)$, called the \emph{Rips complex}, as follows: the vertices of $R_\alpha (X)$ are the points of $X$, and $k+1$ points in $X$ determine a $k$-simplex $\sigma = [x_0, \ldots, x_k]$ if and only if $d(x_i, x_j)<\alpha$ for all $i,j$.
	This construction is illustrated in \Cref{fig: rips_exm} (see also a discussion in the preface).
	Note that in fact the Rips complex is completely determined by its $1$-skeleton, it is in fact a \emph{flag complex}. Due to this feature, the Rips complex is relatively easy to compute, which on the other hand might result in loss of information regarding the original space (as opposed to other complexes that might be attached to $(X,d)$, see \Cref{subsec: Cech_VS_Rips}).
	
	\begin{figure}[!ht]
		\centering
		\includegraphics[scale=1]{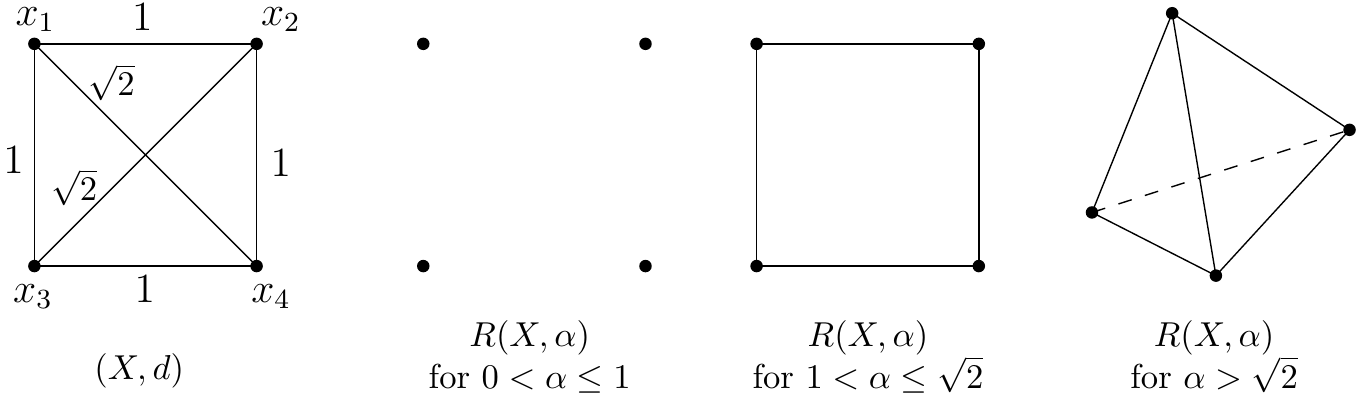}
		\caption{An example of Rips complex of a given metric space consisting of four points.}
		\label{fig: rips_exm}
	\end{figure}

	Note that
	for $0 < \alpha \leq \min_{x,y\in X, \ x\neq y} d(x,y) $ the complex
		$R_\alpha (X)$ is a finite collection of points,
	while for
		$\alpha > \diam (X)$, $R_\alpha (X)$ is a simplex of dimension $|X| - 1$.
		
	\noindent	
	For $\alpha \leq \beta$, there is a natural simplicial map
		$i_{\alpha, \beta} : R_\alpha (X) \to R_\beta (X)$.
	Thus, taking $V_\alpha (X) = H_{*} (R_\alpha (X))$ and $\pi_{\alpha,\beta} = (i_{\alpha,\beta})_{*}$, we get a persistence module, which will be referred to as the \emph{Rips module}.
	
	Let us mention that Rips complexes were first introduced by Vietoris in \cite{vietoris1927}. Rips reintroduced them in order to study hyperbolic groups.
	We will follow Gromov \cite{Gromov1987hyperbolic} and stick to the name Rips, although they are sometimes called \emph{Vietoris} or \emph{Vietoris-Rips} complexes, see \cite{Hausmann95}.
\end{exm}

\begin{defn}\label{def: persistence_homology}
	Let $(V, \pi)$ be a persistence module. The collection of spaces $P_{s,t} = \im (\pi_{s,t})$ will be called the \emph{persistent homology} of $V$.
	Note that in fact, by condition (2) in \Cref{defn: pm}, it would be enough to record only a finite number of such spaces $P_{s,t}$, since there is a finite number of ``jump" points when $\pi_{s,t} \neq \id$.
\end{defn}

\section{Morphisms}

Let $(V,\pi)$ and $(V',\pi')$ be two persistence modules.
\begin{defn}
	A \emph{morphism}
	$A: (V,\pi) \to (V',\pi')$ is a family of linear maps $A_t : V_t \to V_t'$, such that the following diagram commutes for all $s \leq t$:
\begin{center} 
	\begin{tikzpicture}[scale=1.5]
	\node (s) at (0,0) {$V_s$};
	\node (s') at (0,-1) {$V_s'$};
	\node (t) at (1,0) {$V_t$};
	\node (t') at (1,-1) {$V_t'$};
	\path[->,font=\scriptsize,>=angle 90]
	%
	(s) edge[->] node[left]{$A_s$} (s')
	(t) edge[->] node[right]{$A_t$} (t')
	%
	(s) edge node[above]{$\pi_{s,t}$} (t)
	(s') edge node[below] {$\pi'_{s,t}$} (t');
	\end{tikzpicture}
\end{center}
\end{defn}

\noindent
Thus, one can now speak of the category of persistence modules.

In particular, we have the notion of an \emph{isomorphism}: two persistence modules $(V,\pi)$ and $(V', \pi')$ are \emph{isomorphic} if there exists two morphisms $A: V \to V'$ and $B: V' \to V$ so that both compositions $A\circ B$ and $B \circ A$ are the identity morphisms on the corresponding persistence module. (The \emph{identity morphism} on $V$ is the identity on $V_t$ for all $t$.)

\begin{exm}[Shift] \label{exm: pm_shift_def}
	For a persistence module $(V,\pi)$ and $\delta \in \R$, define a persistence module
	$\left( V[\delta], \pi[\delta] \right)$
	by taking
	$\left( V [\delta] \right)_t = V_{t+\delta}$ and $\left( \pi [\delta] \right)_{s,t} = \pi_{s+\delta, t+\delta}$.
	This new persistence module is called a \emph{$\delta$-shift} of $V$.
	For $\delta > 0$, the map $\Phi^\delta : (V,\pi) \to (V[\delta], \pi[\delta])$ defined by
	$\Phi_t^\delta = \pi_{t,t+\delta}$ is a morphism of persistence modules (it will be referred to as \emph{$\delta$-shift morphism}).
	Also, if we have a morphism $F: V \to W$ between two persistence modules, let us denote by
	$F[\delta] : V[\delta] \to W[\delta]$ the corresponding morphism between their $\delta$-shifts.
\end{exm}

\begin{exr} Prove that $\Phi^\delta$ is a morphism indeed.
\end{exr}

\begin{defn} \label{defn: persistence_submodule}
	Let $(V, \pi)$ be a persistence module. A \emph{persistence submodule} $(W, \til \pi)$ of $V$ is a collection of subspaces $W_s \subseteq V_s$ for all $s\in \R$, such that the maps
	$\til \pi_{s,t} := \pi_{s,t} \restr_{W_s}: W_s \to W_t$ are well-defined for all $s\leq t$, and yield a persistence module $(W, \til \pi)$.
\end{defn}

\begin{exr} \label{exr: def_of_ker_im_are_sub_pm}
	Let $\Phi: V \to V'$ be a morphism between two persistence modules $(V, \pi)$ and $(V', \pi')$. We can define the kernel and the image of $\Phi$ as follows.
	The kernel $(\ker \Phi, \pi^{\ker \Phi})$ is a collection of the vector spaces $\{ \ker \Phi_t \}_t$ for all $t\in \R$, equipped with a collection of linear maps $\pi_{s,t} \restr_{\ker \Phi_s}$ for all $s\leq t$.
	Similarly, the image $(\im \Phi, \pi^{\im \Phi})$ of $\Phi$ is a collection $\{ \im \Phi_t \}_t$ of vector spaces, $t\in \R$, equipped with a collection of linear maps $\pi_{s,t} \restr_{\im \Phi_s}$ for all $s\leq t$ in $\R$.
	Prove that $\ker \Phi$ and $\im \Phi$ are persistence submodules of $V$ and $V'$ respectively.
\end{exr}

\begin{conv}
	We will use the notation $(a,b]$ with $-\infty < a <b \leq +\infty$, meaning either a bounded interval when $b< \infty$, or a ray of the form $(a,+\infty)$, when $b=+\infty$.
\end{conv}

\begin{exm}[Interval modules] \label{exm: pm_intervals_pw_const}
	For an interval $(a,b]$ (with $b \leq +\infty$), define a persistence module $\F(a,b]$ as follows:
	$$
		\F(a,b]_t = \left\{
		\begin{array}{ll}
			\F  & \text{if } t\in(a,b] \\
			0 	& \text{otherwise}
		\end{array}
		\right. \;,
		\pi_{s,t} = \left\{
		\begin{array}{ll}
		\mathds{1}  & \text{if } s,t\in(a,b] \\
		0			& \text{otherwise}
		\end{array}
		\right. \;.
	$$

\noindent
Such persistence modules will be called \emph{interval modules}.

Consider the natural inclusions $\F(1,2] \xrightarrow{} \F(1,3]$ and $\F(2,3] \xrightarrow{} \F(1,3]$.
Are they morphisms? As one can check, the first one is not a morphism, while the second one is.
(See \Cref{fig: pm_segments_morphism_or_not}.)
\begin{figure}[!ht]
	\centering
	\includegraphics[scale=1]{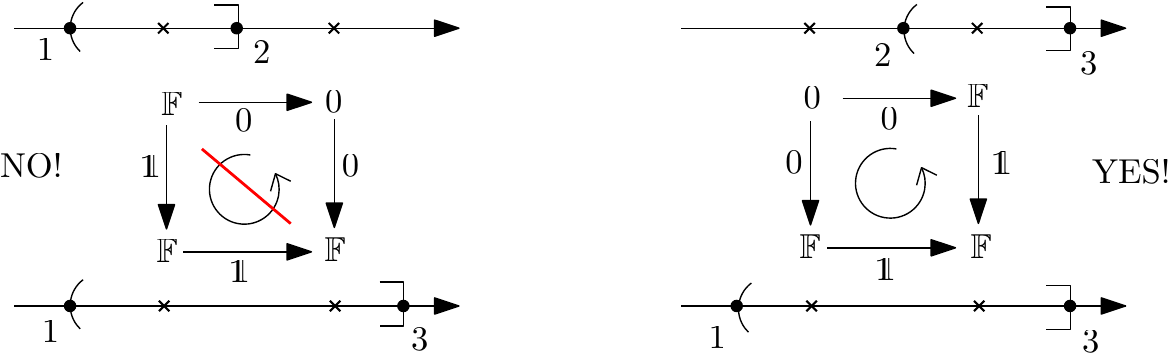}
	\caption{Comparison of two situations.}
	\label{fig: pm_segments_morphism_or_not}
\end{figure}

\begin{exr} \label{exr: morphism_between_Q(segment)s}
	More generally, check that for two intersecting intervals $(a,b]$ and $(c,d]$, there is a non-zero morphism $\F(a,b] \to \F(c,d]$ if and only if
	$c\leq a$ and $a < d\leq b$.
	(Moreover, any morphism between $\F(a,b]$ and $\F(c,d]$ is given by multiplication by some element $\lambda \in \F$.)
\end{exr}

\begin{defn}\label{def: direct_sum_of_pm}
	Let $(V,\pi)$ and $(V',\pi')$ be two persistence modules. Their direct sum $(W,\theta)$ is a persistence module whose underlying vector spaces are $W_t = V_t \oplus V'_t$ (direct sum of vector spaces) and accordingly, $\theta_{s,t} = \pi_{s,t} \oplus \pi'_{s,t}$.
\end{defn}

Following the example illustrated in \Cref{fig: pm_segments_morphism_or_not}, let us note that in general $\F(a,b] \approx \F(a,c] / \F(b,c]$ (as vector spaces for each $t$), and we have an exact sequence of persistence modules
	$$
		0 \to \F(b,c] \to \F(a,c] \to \F(a,b] \to 0  \;.
	$$
However, one can check that $\F(a,c] \neq \F(a,b] \oplus \F(b,c] $!
(Follow \Cref{def: direct_sum_of_pm}, see \Cref{fig: pm_cannot_split_interval}.)
This will also follow later from \Cref{thm: normal_form_thm}.
In other words, an exact sequence in the category of persistence modules does not necessarily split.
In fact, $\F (a,b]$ is not a submodule of $\F (a,c]$. (In other terms, it means that a direct summand, e.g. $\F (a,b]$, is not necessarily a submodule, where a summand of $\F (a,b]$ is a subset $S$ that can be completed to the whole space by a submodule, i.e. there is a submodule $T \subseteq \F (a,b]$ s.t. $S\oplus T = \F(a,b]$.)

\begin{figure}[!ht]
	\centering
	\includegraphics[scale=1]{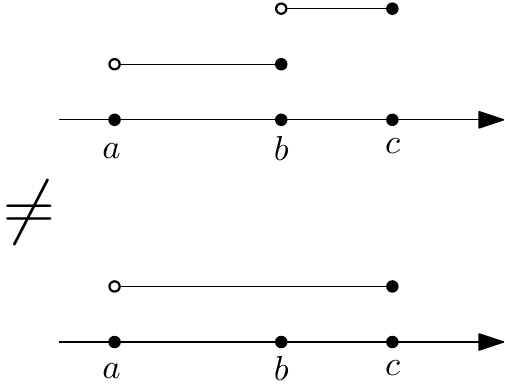}
	\caption{$\F (a,b] \oplus \F (b,c] \neq \F(a,c]$.}
	\label{fig: pm_cannot_split_interval}
\end{figure}

\end{exm}

\begin{exm} \label{exm: shift_morphism_image}
	Let us give a concrete example of a $\delta$-shift persistence module and a $\delta$-shift morphism.
	Consider $V = \F (0,1]$ and $\delta = \frac{1}{3}$.
	Then $V[\delta] = \F (-\frac{1}{3}, \frac{2}{3}]$, but $\im \Phi^\delta = \F (0, \frac{2}{3}]$.
	(See \Cref{fig: shift_morphism}, and definition of an image of a morphism in \Cref{exr: def_of_ker_im_are_sub_pm}.).
	So $\Phi^\delta$ in fact ``chops" $V$ by $\delta$ from the right.
\end{exm}

\begin{figure}[!ht]
	\centering
	\includegraphics[scale=1]{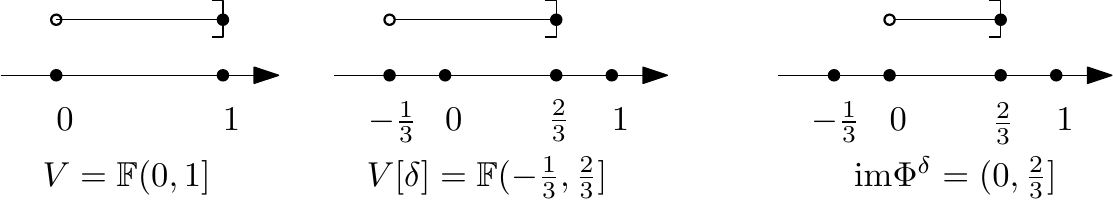}
	\caption{$\Phi^\delta$ ``chops" $V$ from the right.}
	\label{fig: shift_morphism}
\end{figure}

\section{Interleaving distance}
We would like to have a metric, or at least a pseudo-metric, on the space of persistence modules.

\begin{defn} \label{defn: interleaved_pm_interleaving_morph}
	Given a $\delta > 0$, we say that two persistence modules $(V,\pi)$ and $(W,\theta)$ are \emph{$\delta$-interleaved} if
	there exist two morphisms $F: V \to W[\delta]$ and $G:W\to V[\delta]$,
	such that the following diagrams commute:

	\begin{center} %
		\begin{tikzpicture}
		\matrix (m) [matrix of math nodes,row sep=2em,column sep=2em,minimum width=2em]
		{
			V & W[\delta] & V[2\delta] \\
		};
		\path[-stealth, decoration={snake,segment length=4,amplitude=3,
			post=lineto,post length=10pt}]
		(m-1-1) edge node [above] {$F$} (m-1-2)
		(m-1-2) edge node [above] {$G[\delta]$} (m-1-3)
		(m-1-1) edge[bend left=-20] node [below] {$\Phi^{2\delta}_V$} (m-1-3);
		\end{tikzpicture}
\hspace{1cm}
		\begin{tikzpicture}
		\matrix (m) [matrix of math nodes,row sep=2em,column sep=2em,minimum width=2em]
		{
			W & V[\delta] & W[2\delta] \\
		};
		\path[-stealth, decoration={snake,segment length=4,amplitude=3,
			post=lineto,post length=10pt}]
		(m-1-1) edge node [above] {$G$} (m-1-2)
		(m-1-2) edge node [above] {$F[\delta]$} (m-1-3)
		(m-1-1) edge[bend left=-20] node [below] {$\Phi^{2\delta}_W$} (m-1-3);
		\end{tikzpicture} \;,
	\end{center}
	
	\noindent
	where $\Phi^{2\delta}_V$ and $\Phi^{2\delta}_W$ are the shift morphisms (see \Cref{exm: pm_shift_def}).
	We will also refer to such a pair of morphisms $F$ and $G$ as \emph{$\delta$-interleaving morphisms}.
\end{defn}

\begin{exr} \label{exr: basic_properties_of_interleaved_pm}
	\begin{enumerate}
		\item
			Show that two persistence modules $(V,\pi)$ and $(W,\theta)$ are $\delta$-interleaved with finite $\delta$ if and only if
			$\dim V_\infty = \dim W_\infty$ (see definition in \Cref{rmks: on_defn_of_pm_giving_def_of_V_infty}.).
			
		\item
			Prove that if $V,W$ are $\delta$-interleaved, then they are $\delta'$-interleaved for any $\delta' > \delta$.
		\item
			Prove that if $V,W$ are $\delta_1$-interleaved and $W,Z$ are $\delta_2$-interleaved,
			then $V,Z$ are $(\delta_1+\delta_2)$-interleaved.
	\end{enumerate}
\end{exr}

\begin{defn} \label{defn: interleaving_dist}
	For two persistence modules $(V,\pi)$ and $(W, \theta)$, define the interleaving distance between them to be
	\begin{equation*}
	d_{int} (V,W) = \inf\ \{ \delta > 0\ |\  (V,\pi) \text{ and } (W,\theta) \text{ are } \delta \text{-interleaved} \} \;.
	\end{equation*}
	(For brevity, we use the notation $d_{int}(V,W)$, writing just $V$ instead of $(V,\pi)$ and similarly for $(W,\theta)$, unless there could be a confusion.)
\end{defn}

\noindent
Note that in this way we get a pseudo-metric on isomorphism classes of persistence modules with the same $V_\infty$. A priori, it might happen that $d_{int}(V,W)$ vanishes for non-isomorphic $V$ and $W$.
%
However, because of the semicontinuity condition we pose on persistence modules, we will be able to show that $d_{int}$ is a genuine metric, i.e.\ that it is non-degenerate (see \Cref{thm: isometry_thm} and \Cref{exr: d_bot_is_non_degenerate_hence_d_int_also}).\\


\subsection{First example: interval modules}

\begin{claim} \label{claim: _int_distance_two_intervals}
	Fix $a,b,c,d < \infty$, with $a<b,\ c<d$, and consider $d_{int} (\F(a,b], \F(c,d] )$, between the persistence modules $\F(a,b]$ and $\F (c,d]$ are as defined in \Cref{exm: pm_intervals_pw_const}. Then
	\begin{equation} \label{eq: upper_bound_d_int}
	d_{int} \big( \F(a,b], \F(c,d] \big) \leq
	\min \bigg( \max \Big( \frac{b-a}{2}, \frac{d-c}{2} \Big), \max \big( |a-c|, |b-d| \big) \bigg) \;.
	\end{equation}
	We will see later that in fact an equality holds.
	For now, let us prove this inequality by exploring two strategies of interleaving $\F(a,b]$ and $\F(c,d]$.
	
	\begin{enumerate}[I.]
		\item
		Take $\delta = \max \left( |a-c|, |b-d| \right)$.
		We want to show that $\F(a,b]$ and $\F(c,d]$ are $\delta$-interleaved.
		By definition,
		$a - 2\delta \leq c-\delta \leq a$ and $b - 2\delta \leq d-\delta \leq b$.
		In view of \Cref{exr: morphism_between_Q(segment)s}, one can take the morphisms
		$F: \F(a,b] \to \F(c-\delta, d - \delta]$ and
		$G: \F(c,d] \to \F(a-\delta, b - \delta]$.
		They might be zero, e.g. if $d-\delta<a$ then $F=0$.
		(see \Cref{fig: intervals_interleaving_I}.)
		\begin{figure}[!ht]
			\centering
			\includegraphics[scale=1]{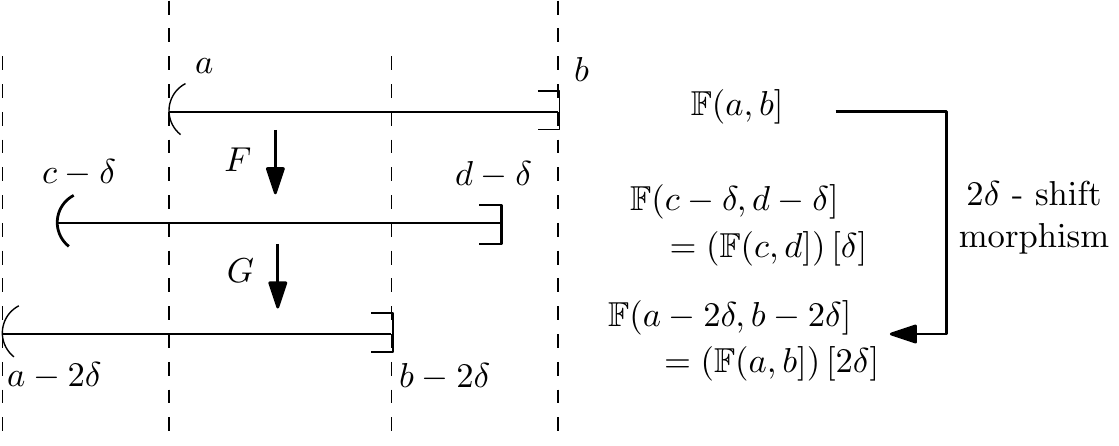}
			\caption{First method of interleaving $\F(a,b]$ and $\F(b,d]$: by $\delta = \max\left(|a-c|,|b-d|\right)$.}
			\label{fig: intervals_interleaving_I}
		\end{figure}

		\item
		Put this time $\delta = \max\left(\frac{b-a}{2}, \frac{d-c}{2}\right)$.
		Note that the shift morphism by $2\delta$ vanishes for both modules, see \Cref{fig: intervals_interleaving_II} (e.g., the shift between $\F(a,b]$ and $\F(a-2\delta, b - 2\delta]$ vanishes, as $b - 2\delta \leq a$, i.e.\ $(a,b] \cap (a-2\delta, b-2\delta] = \emptyset$).
		Taking the interleaving morphisms to be $0$ we get the desired result.
		
		\begin{figure}[!ht]
			\centering
			\includegraphics[scale=1]{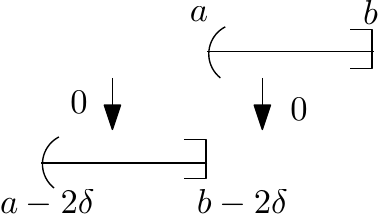}
			\caption{Second method of interleaving. The shift morphism vanishes.}
			\label{fig: intervals_interleaving_II}
		\end{figure}
		
	\end{enumerate}
\end{claim}

\begin{exr}
	For two infinite intervals, $d_{int} \big( \F(a,\infty), \F(c,\infty) \big) = |a-c|$.
\end{exr}

\begin{exm}
	In order to get the flavor of this bound let us list some concrete examples (we write $\delta_{\RN{1}}$ and $\delta_{\RN{2}}$ for $\delta$ taken as in the cases $\RN{1}$ and $\RN{2}$ of \Cref{claim: _int_distance_two_intervals} respectively):
	\begin{enumerate}
		\item
		For $\F(1,2]$ and $\F(1,3]$, $\delta_{\RN{1}} = \delta_{\RN{2}} = 1$, so $d_{int} \leq 1$.
		\item
		For $\F(1,2]$ and $\F(2,3]$, $\delta_{\RN{1}} = 1$, $\delta_{\RN{2}} = \frac{1}{2}$, so $d_{int} \leq \frac{1}{2}$.
		\item
		For $\F(1,4]$ and $\F(2,5]$, $\delta_{\RN{1}} = 1$, $\delta_{\RN{2}} = \frac{3}{2}$, so $d_{int} \leq 1$.
	\end{enumerate}
	
	\noindent
	As we remarked above regarding (\ref{eq: upper_bound_d_int}), these bounds are in fact the exact values of $d_{int}$.
\end{exm}

\section{Morse persistence modules and approximation}
	\label{sec: morse_approx_question}
	Take a closed manifold $M$ and a Morse function $f:M \to \R$.
	Put $\|f\| = \max|f|$ (the uniform norm of $f$).
	
	As before, we define a persistence module $V(f)$ by setting $V_t (f) = H_{*} (\{ f<t \})$.
	\noindent
	Note that in these notations $V (f-\delta) = V (f)[\delta]$.
	Also, if $g: M \to \R$ is another Morse function and $f\leq g$, then $\{ g<t \} \subset \{ f<t \}$, and we get a natural morphism
	$F : V(g) \to V (f)$.

	For any $f,g:M\to \R$ we have $f - \| f-g \| \leq g$. Denote $\delta = \|f-g\|$. By the above considerations, since $\{ g<t \} \subseteq \{ f-\delta < t \}$, there is a natural morphism
		$F: V(g) \to V (f)[\delta]$.
	Similarly, $g - \delta \leq f$, hence we have another morphism $G : V(f) \to V (g) [\delta]$.
	Combining these two inequalities, we obtain $f - 2\delta \leq g - \delta \leq f$, that is, we actually have three natural morphisms, yielding the following commutative diagram:
	
	\begin{center}	
		\begin{tikzpicture}
		\matrix (m) [matrix of math nodes,row sep=2em,column sep=2em,minimum width=2em]
		{
			V(f) & V(g)[\delta] & V(f) [2\delta] \\
		};
		\path[-stealth, decoration={snake,segment length=4,amplitude=3,
			post=lineto,post length=10pt}]
		(m-1-1) edge node [above] {$G$} (m-1-2)
		(m-1-2) edge node [above] {$F[\delta]$} (m-1-3)
		(m-1-1) edge[bend left=-20] node [below] {$\Phi^{2\delta}_{V(f)}$} (m-1-3);
		\end{tikzpicture} \;,
	\end{center}

\noindent
	where $\Phi_{V(f)}^{2\delta}$ stands for the $2\delta$-shift morphism of $V(f)$.	
	By a symmetric argument, we get the second diagram required by \Cref{defn: interleaved_pm_interleaving_morph}, hence
	$d_{int} \big( V(f), V(g) \big) \leq \delta = \| f-g \|$.
	
	Note that for any $\vphi \in \Diff (M)$, the persistence modules $V(f)$ and $V(\vphi^* f)$ are isomorphic, hence
	\begin{equation}\label{eq: interleaving_dist_smaller_than_functions_norm}
		d_{int} \big( V(f), V(g) \big) \leq
		\inf_{\vphi \in \Diff(M)} \| f - \vphi^* g \| \;.
	\end{equation}
	
	Let us have a closer look at this inequality by considering a sub-example. Take a Morse function $f:S^2 \to \R$.
	How well can it be $C^0$-approximated by a Morse function with exactly two critical points? (See \Cref{fig: approx_morse_sphere}.)
	
	\begin{figure}[!ht]
		\centering
		\includegraphics[scale=0.75]{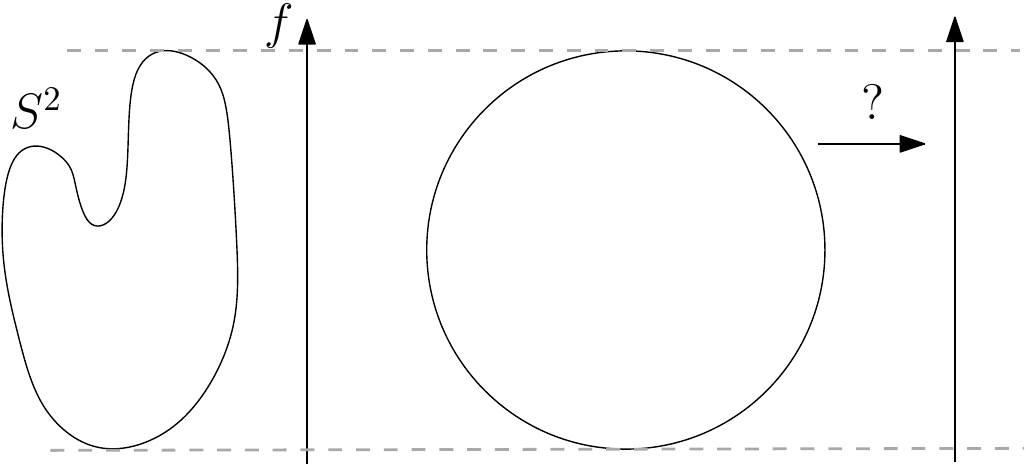}
		\caption{Approximation question}
		\label{fig: approx_morse_sphere}
	\end{figure}
	
	\noindent
	For such functions illustrated in \Cref{fig: approx_morse_sphere}, we shall calculate the lower bound given in (\ref{eq: interleaving_dist_smaller_than_functions_norm}) in \Cref{exm: approx_function_heart_sphere} below.

\medskip

In Chapter \ref{chap5a-functheory} we discuss further applications of  persistence modules to function theory and
to approximation.

\section{Rips modules and the Gromov-Hausdorff distance}

Let $X,Y$ be finite sets. A \emph{surjective correspondence} $C : X \rightrightarrows Y$ between $X$ and $Y$ is a subset
$C \subset X \times Y$, such that $\proj _X (C) = X$ and $\proj_Y (C) = Y$. The \emph{inverse correspondence} $C^T : Y \rightrightarrows X$ is defined by
$C^T = \{ (y,x):\ (x,y)\in C \}$. Let us note that $C$ is a surjective correspondence if and only if there exist $f: X\to Y$ and $g: Y\to X$, such that $\grph (f) \subset C$ and $\grph (g) \subset C^T$.

\begin{defn}\label{defn: distortion_of_surj_correspondence}
Assume now that $(X,\rho),\ (Y,r)$ are finite metric spaces. The \emph{distortion} of a surjective correspondence $C : X \rightrightarrows Y$ is
\begin{equation}
	\dis (C) = \max_{(x,y), (x',y')\in C} | \rho (x,x') - r (y,y') | \;.
\end{equation}
\end{defn}

\noindent
For instance, if we take $C$ to be a graph of a function $f:X\to Y$, then $(x,y)\in C$ means $y = f(x)$, and so
$$
	\dis (C) = \max_{x, x' \in X} | \rho (x,x') - r(f(x), f(x')) | \;.
$$
\noindent
In particular, $\dis(C) = 0$ if and only if $f$ is an isometry.

Let us adopt the following notion of distance between metric spaces:
\begin{defn}
	The Gromov-Hausdorff distance between two finite metric spaces $(X,\rho)$ and $(Y,r)$ is
	$$
		d_{GH} \big( (X,\rho), (Y, r) \big) = \frac{1}{2} \min_C \dis(C) \;,
	$$
	where the minimum is taken over all surjective correspondences $C : X \rightrightarrows Y$.
\end{defn}

\begin{exr}
	Prove that $d_{GH}$ is a distance between isometry classes of finite metric spaces.
\end{exr}

For the finite metric space $(X,\rho)$, consider its Rips complex $R_t (X) $  and accordingly the Rips persistence module $V_t (X) = H_* \big( R_t (X) \big)$.
(See \Cref{exm: finite_metric_sp_Rips_pm}.)

\begin{thm}[See \cite{CDO14}]
	\label{thm: GH_dist_bounded_below_by_interleaving_distance}
	$$
		d_{GH} \big( (X,\rho), (Y,r) \big) \geq
		\frac{1}{2} d_{int} \big( V(X,\rho), V(Y,r) \big) \;.
	$$
\end{thm}

\begin{proof}
	Take a surjective correspondence $C: X \rightrightarrows Y$, and any $\delta > \dis (C)$.
	We need to show that $V(X)$ and $V(Y)$ are $\delta$-interleaved.
	
	Pick any $f: X \to Y$ with $\grph (f) \subset C$.
	Note that since $\delta > \dis (C)$, we have $r \big( f(x), f(x') \big) < \rho (x,x') + \delta$, so
	$f$ induces a simplicial map $F: R_t (X) \to R_{t+\delta} (Y)$.
	%
	Let $F_{*} : V_t (X) \to V_{t+\delta} (Y) = \big( V (Y)[\delta] \big) _t$ be the induced map on homology. Similarly, taking $g: Y\to X$ for which $\grph (g) \subset C^T$, we get a map
	$G: R_t (Y) \to R_{t+\delta} (X)$, which induces a map
	$G_{*} : V_t(Y) \to \big( V(X)[\delta ] \big) _t$ in homology.
	
	We claim that the maps $F_{*}$ and $G_{*}$ are $\delta$-interleaving morphisms.
	To prove it, we have to show that the following diagram and a similar diagram for the converse composition both commute:
	\begin{center}
		\begin{tikzpicture}
		\matrix (m) [matrix of math nodes,row sep=2em,column sep=2em,minimum width=2em]
		{
			V(X) & V(Y)[\delta] & V(X) [2\delta] \\
		};
		\path[-stealth, decoration={snake,segment length=4,amplitude=3,
			post=lineto,post length=10pt}]
		(m-1-1) edge node [above] {$F_*$} (m-1-2)
		(m-1-2) edge node [above] {$G_*[\delta]$} (m-1-3)
		(m-1-1) edge[bend left=-20] node [below] {$i_*$} (m-1-3);
		\end{tikzpicture} \;.
	\end{center}
	(Here $i: R_t(X) \to R_{t+2\delta} (X)$ is the natural inclusion.)

	We recall that two simplicial maps $H, H' : K \to L$ (between simplicial complexes $K, L$) are called \emph{contiguous} if for any simplex $\sigma \in K$, $H(\sigma) \cup H'(\sigma)$ is a simplex in $L$.
	For contiguous maps $H$ and $H'$, one gets that $H_* = H'_*$ (see \cite[Theorem 12.5]{Munkres_alg_top_84}).

	Let us show that $G\circ F$ and $i$ are contiguous as maps $R_t(X) \to R_{t+2\delta}(X)$.
	Choose any simplex $[x_0, \ldots, x_k] \in R_t(X)$. Note that $i(x_j) = x_j$.
	Thus, we have to check that $[gf(x_0), \ldots, gf(x_k), x_0, \ldots, x_k]$ is a simplex in $R_{t+2\delta}(X)$.

	By definition of the distortion of $C$, we know that for any $x,x'\in X,\ y,y'\in Y$ that satisfy $(x,y), (x',y')\in C$, we have $|\rho (x,x') - r(y,y')|\leq \dis (C) < \delta$.
	So for all $0\leq i, j \leq k$,
	\begin{align*}
		\rho \big( gf(x_i), x_j \big)
		< r \big( f(x_i), f(x_j) \big) + \delta
			%
		< \rho (x_i, x_j) + 2\delta
			%
		< t + 2\delta \;.
	\end{align*}
	Here the first inequality holds since $\big( gf(x_i), f(x_i) \big), \big( x_j, f(x_j) \big) \in C\ $ for all $i,j$, the second one follows, as again $(x_i, f(x_i)), (x_j, f(x_j)) \in C$, and the last one is by the definition of $R_t(X)$.
	\noindent
	Similarly, we get that
	$$
		\rho \big( gf(x_i), gf(x_j) \big) < r \big( f(x_i), f(x_j) \big) + \delta < t+ 2\delta \;.
	$$
	\noindent
	So $G \circ F$ and $i$ are contiguous, hence the result follows.
\end{proof}

\medskip

See Chapter \ref{chap5-rips} for further discussion on persistence modules associated to Rips complexes.

\chapter{Barcodes} \label{chp2_barcodes}

\begin{defn}\label{defn_barcode}
	A \emph{barcode} $\calB$ is a finite multiset of intervals, i.e. it is a finite collection $\{ (I_i, m_i) \}$ of intervals $I_i$ with given multiplicities $m_i\in \N$. For us, the intervals $I_i$ are all either finite of the form $(a,b]$ or infinite of the form $(a,+\infty)$.
	The intervals in a barcode will be sometimes called \emph{bars}.
\end{defn}

\section{The Normal Form Theorem} \label{sec-normal-form}

The main result of this section is that any persistence module can be expressed as a direct sum of ``simple" persistence modules of the form $\F(I)$ (as were defined in \Cref{exm: pm_intervals_pw_const}), with $I$ being either a left-opened right-closed interval, or an infinite ray open on the left.

\begin{thm}[Normal Form Theorem] \label{thm: normal_form_thm}
	Let $(V,\pi)$ be a persistence module. Then there exists a finite collection $\{ (I_i, m_i) \}_{i=1}^N$ of intervals $I_i$ with their multiplicities $m_i$, where $I_i=(a_i,b_i]$ or $I_i = (a_i, \infty)$, $m_i \in \N$, $I_i \neq I_j$ for $i\neq j$,
	such that
	$$
		V = \bigoplus_{i=1}^N \F(I_i)^{m_i} \;.
	$$
	\noindent
	By equality here we mean that they are isomorphic as persistence modules.
	
	\noindent
	Moreover, this data is unique up to permutations, i.e., to any persistence module there corresponds a unique barcode $\calB (V)$, which consists of the intervals $I_i$ with multiplicity $m_i$. This barcode will be called \emph{the barcode of $V$}.
\end{thm}

Let us note here that this statement holds also under weaker assumptions (with a more general definition of a barcode), namely, assuming that the persistence module is point-wise finite dimensional, i.e. $V_t$ are finite dimensional for all $t$ (see \cite{crawley2015_normalform}). Let us mention that the normal form theorem for homology of filtered complexes
was proved by S.~Barannikov  \cite{barannikov1994framed} in 1994. The ``birth-death" diagrams introduced in \cite{barannikov1994framed} encoding the canonical form of filtered complexes are equivalent to what later was called barcodes.

We start with some preparations towards proving \Cref{thm: normal_form_thm}.

\begin{defn}\label{def-spectral}
	A point $t\in \R$ is called \emph{spectral} for a persistence module $(V, \pi)$ if for any neighborhood $U\ni t$ there exist $s < r$ in $U$, such that $\pi_{s,r}: V_s \to V_r$ is not an isomorphism.
\end{defn}

\noindent
Denote by $\Spec V =\Spec(V,\pi)$ the collection of spectral points of $(V,\pi)$ together with $+\infty$ (artificially added), this set will be called the \emph{spectrum} of $V$. We will omit $\pi$ unless there is an ambiguity.
By condition (2) in \Cref{defn: pm}, $\Spec V$ is a finite set.

\begin{exr}
	Assume that $s,t$ belong to the same connected component of $\R \setm \Spec V$.
	Prove that $\pi_{s,t} : V_s \to V_t$ is an isomorphism.	
\end{exr}

\begin{exr} \label{exr: spec_pm_iso_invariant}
	Show that $\Spec V$ is an isomorphism invariant of persistence modules.
\end{exr}

\begin{exr} \label{exr: spec_of_intervals_directsum}
	Find the spectrum of the direct sum $\bigoplus_{i=1}^N \F(I_i)^{m_i}$, where $(I_i, m_i)$ are as in \Cref{thm: normal_form_thm}.
\end{exr}

Let $(V,\pi)$ be a persistence module and let $\Spec V = \{ a_1, \ldots, a_N \}\cup\{+\infty\}$ be its spectrum, where $a_1<\ldots<a_N < a_{N+1} = +\infty$ (see \Cref{fig: spec_Vi}).
\noindent
We also set $a_0 = -\infty$ in order to have more pleasant notations,
but we warn the reader that it is not considered to be a spectral point.\\
\noindent
Denote by $Q_i = (a_{i-1},a_{i}]$ for $1\leq i \leq N$ and $Q_{N+1} = (a_N, \infty)$ the intervals defined by adjacent $a_i$-s.
Note that these $Q_i$ are not the intervals $I_i$ we search for in \Cref{thm: normal_form_thm}, as $I_i$ would not necessarily be defined by adjacent points of $\Spec V$.\\

For any $i\in \{1,\ldots, N+1 \} $, define the limit vector space $V^i$ by considering the direct limit of the direct system $\{V_s\}$ for $s\in Q_i$:
%
\begin{equation}\label{dfn-vi}
		V^i = \coprod_{s \in Q_i} V_s \Big/ \sim \;,
\end{equation}
where $V_s \ni v_s \sim v_t \in V_t$
	for $s<t$ if $\pi_{s,t} (v_s) = v_t$.
	
Let us observe that $V^i$ is naturally isomorphic to $V_{a_i}$ since $\pi_{s,t}$ are all isomorphisms for any $s,t\in Q_i$.
We equip this collection $\{V^i\}$ with the natural morphisms $p_{i,j} : V^i \to V^j$ (for $i\leq j$) induced by $\pi_{s,t}$.
Denote
$
	\Totaldim(V) = \sum_i \dim V^i \;.
$\\

\begin{figure}[!ht]
	\centering
	\includegraphics[scale=1]{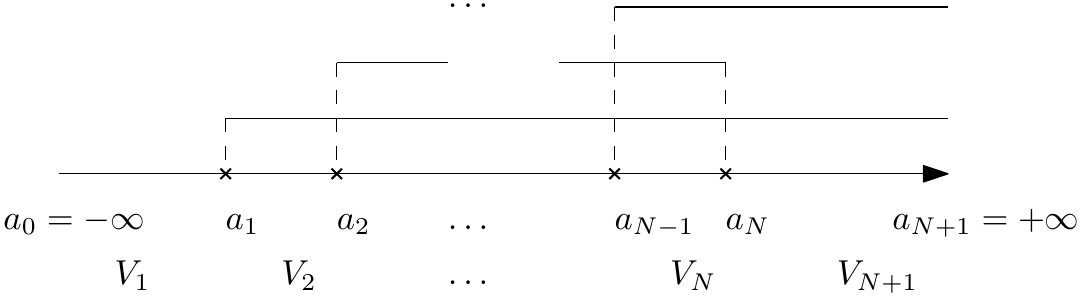}
	\caption{Spectral points and limit spaces $V^i$.}
	\label{fig: spec_Vi}
\end{figure}

Let $W \subset V$ be a persistence submodule (recall \Cref{defn: persistence_submodule}).

\begin{defn}\label{def: semi_surj_submodule}
	We will say that a submodule $W$ of $V$ is \emph{semi-surjective} if there exists  $r\in \R$, such that:
	\begin{enumerate}[(a)]
		\item
			$W_t = V_t$ for all $t\leq r$,
		\item
			$\pi_{s,t} : W_s \to W_t$ is onto if $r<s<t$.
	\end{enumerate}
\end{defn}

\begin{exm}
	$\F(0,\infty)$ is a semi-surjective submodule of $\F(0,\infty) \oplus \F (1,2]$.
\end{exm}

\begin{exr} \label{exr: r_semi_surj_witness_supremal}
	Let $W \subset V$ be a semi-surjective submodule. Show that
	\begin{enumerate}
		\item
			$\Spec W \subset \Spec V$ and $\Totaldim W \leq \Totaldim V$,
		\item
			$r:=\sup \{ t:\ W_s = V_s\ \forall s\leq t \} \in \Spec V$. This $r$ satisfies the conditions in \Cref{def: semi_surj_submodule}.
			(As an illustration, in \Cref{fig: spec_V_W} the smallest $i$ for which $W^i \subsetneq V^i$ is $i=5$, and $r = a_4$.)
	\end{enumerate}
\end{exr}

We shall encode semi-surjective submodules $W$ of $V$ by the data $W^i \subseteq V^i$, still taking the indices $i = 1, \ldots, N+1$ according to the intervals $Q_i$ (that were associated to the spectrum of $V$).
Note that as $a_i$ needs not be a spectral point of $W$, $p_{i,i+1} : W^i \to W^{i+1}$ may be an isomorphism.  See \Cref{fig: spec_V_W}.

\begin{figure}[!ht]
	\centering
	\includegraphics[scale=1]{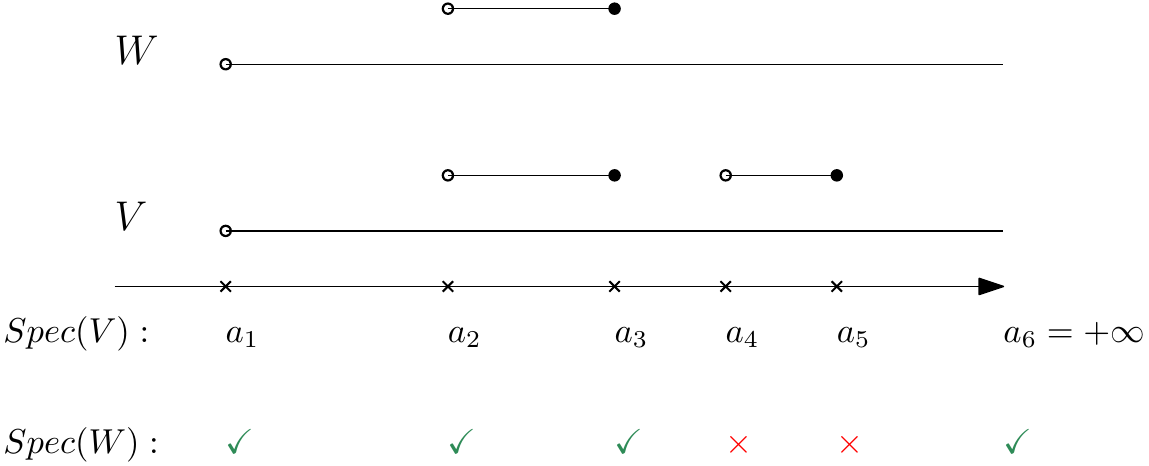}
	\caption{A point in $\Spec V $ might be not in $\Spec W$.}
	\label{fig: spec_V_W}
\end{figure}

\begin{lem}\label{lem: semi_surj_submodule}
	Let $W \subsetneq V$ be a semi-surjective submodule.
	Then there exists a semi-surjective submodule $W_{\sharp} \subset V$, such that
	$W_\sharp \cong W\oplus \F(I)$, where $I = (a,b]$ with $a,b\in \Spec V$.
\end{lem}

\begin{proof}[Proof of \Cref{lem: semi_surj_submodule}]
	Since $W \subseteq V$ is a semi-surjective submodule, we have $W_t=V_t$ for $t\leq r$ up to some $r$, and hence also $W^i = V^i$ up to some index.
	Take the minimal $i$ for which $W^i \subsetneq V^i$ and choose $z^i \in V^i \setm W^i$.
(Note that looking back at the representatives in the persistence modules, this means that the smallest value of $t$ for which $W_t \subsetneq V_t$ is $a_{i-1}$.)\\

\noindent
Set $z^k = p_{i,k} (z^i) \in V^k$ for $k>i$. Two cases are possible:
\begin{enumerate}
	\item
		For all $k>i$, $z^k\notin W^k$. (This case corresponds to an infinite interval $I$ that starts at $a_{i-1}$.)
	\item
		Otherwise, there exists some $k>i$ for which $z^k$ falls into $W^k$.
		(This case corresponds to adding a finite interval $I$.)
\end{enumerate}

 We will describe the rest of the construction for the second case, and then comment on the first case.\\

Choose the minimal $j > i$ for which $z^j \in W^j$.
Since $p_{i,j} : W^i \to W^{j}$ is onto, there is an $x^i \in W^i$ such that
$p_{i,j} (z^i) = z^j = p_{i,j} (x^i)$. Put $y^i = z^i - x^i$
From now on, we shall work with $y^i$ instead of $z^i$. See a diagram below.

\begin{center}
	\begin{tikzpicture}[xscale = 2]
	\node (V1) at 	(0,1) {$V^1$};
	\node (dts11) at (1,1) {$\ldots$};
	\node (Vi) at 	(2,1) {$V^i$};
	\node (dts12) at (3,1) {$\ldots$};
	\node (Vj) at 	(4,1) {$V^j$};
	\node (Vj+1) at (5,1) {$V^{j+1}$};
	\node (dts13) at (6,1) {$\ldots$};
	\node (zi) at 	(2,0) {$z^i$};
	\node (dts22) at (3,0) {$\ldots$};
	\node (zj) at 	(4,0) {$z^j$};
	\node (W1) at 	(0,-1) {$W^1$};
	\node (dts31) at (1,-1) {$\ldots$};
	\node (Wi) at 	(2,-1) {$W^i$};
	\node (dts32) at (3,-1) {$\ldots$};
	\node (Wj) at 	(4,-1) {$W^j$};
	\node (Wj+1) at (5,-1) {$W^{j+1}$};
	\node (dts33) at (6,-1) {$\ldots$};
	\node (yi) at 	(2,-2) {\small{$0\neq y^i = z^i - x^i$}};
	\node (dts42) at (3,-2) {$\ldots$};
	\node (zero) at (4,-2) {$y^j = 0$};
	\node (zero2) at (5,-2) {$0$};
	\node (zero3) at (6,-2) {$\ldots$};
	\path[->,font=\scriptsize,>=angle 90]
	(V1) edge node[above]{} (dts11)
	(dts11) edge (Vi)
	(Vi) edge (dts12)
	(dts12) edge (Vj)
	(Vj) edge (Vj+1)
	(Vj+1) edge (dts13)
	(zi) edge[draw=none] node [sloped] {$\mapsto$} (dts22)
	(dts22) edge[draw=none] node [sloped] {$\mapsto$} (zj)
	(W1) edge node[above]{} (dts31)
	(dts31) edge (Wi)
	(Wi) edge (dts32)
	(dts32) edge (Wj)
	(Wj) edge (Wj+1)
	(Wj+1) edge (dts33)
	(yi) edge[draw=none] node [sloped] {$\mapsto$} (dts42)
	(dts42) edge[draw=none] node [sloped] {$\mapsto$} (zero)
	(zero) edge[draw=none] node [sloped] {$\mapsto$} (zero2)
	(zero2) edge[draw=none] node [sloped] {$\mapsto$} (zero3)
	%
	(zi) edge[draw=none] node [sloped, allow upside down] {$\in$} (Vi)
	(zi) edge[draw=none] node [sloped, allow upside down] {$\notin$} (Wi)
	(zj) edge[draw=none] node [sloped, allow upside down] {$\in$} (Vj)
	(zj) edge[draw=none] node [sloped, allow upside down] {$\in$} (Wj);
	\end{tikzpicture}
\end{center}
Note that $p_{i,k}(y^i) \notin W^k$ for all $i<k<j$ (since $j$ is the minimal index after $i$ for which $z^j$ lands at $W^j$).
Also, $p_{i,j} (y^i) = 0$, by linearity of $p_{i,j}$, and hence $p_{i,k} (y^i) = (p_{j,k} \circ p_{i,j}) (y^i) = 0$ for all $k\geq j$). That is, $y^j$ is where $p_{i,j} (y^i)$ vanishes for the first time (and after which it stays zero).

\noindent
Denote $y^k = p_{i,k} (y^i)$.
We shall build a submodule $P$ of $V$ using the following data: for the element $y^k \in V^k$, which is an equivalence class, we take its representatives $(y^k)_s \in V_s$ for $s\in  (a_{k-1}, a_{k}]$, and construct:
$$
P_s = \left\{
	\begin{array}{ll}
		0  					& s \notin (a_{i-1}, a_{j-1}]  \\
		\spn_{\F}((y^k)_s) 	& s \in (a_{k-1}, a_{k}] \subseteq (a_{i-1},a_{j-1}],\ k=i, \ldots, j-1
	\end{array}
		\right. \;,
$$
where the persistence morphisms are induced from the morphisms $\pi^V_{s,t}$ of $V$, i.e.
$$
\pi^P_{s,t} = \left\{
\begin{array}{ll}
	\pi^V_{s,t}  	& s, t \in (a_{i-1}, a_{j-1}]  \\
	0 				& \text{otherwise}
\end{array}
\right. \;.
$$

\noindent
Clearly, $P = \{P_s\}$ it is a submodule of $V$ isomorphic to $\F (a_{i-1}, a_{j-1}]$.
\begin{claim} \label{claim: W_sharp_adding_submodule}
	Take $W_{\sharp} = W + P$. Then:
\begin{enumerate}
	\item
		$W_{\sharp} = W \oplus P$,
	\item
		$W_{\sharp}$ is a semi-surjective submodule of $V$.
\end{enumerate}
\end{claim}

\noindent
This finishes the proof of \Cref{lem: semi_surj_submodule} for the second case.\\

For the first case, i.e. if $z^j \notin W^j$ for all $j > i$, we shall build a suitable submodule $P$ using $z^i$ instead. We take a submodule $P$ which corresponds to $I = (a_{i-1}, +\infty)$ in a manner similar to the previous case, by setting
$$
P_s = \left\{
\begin{array}{ll}
0  					& s \leq a_{i-1}  \\
\spn_{\F}((z^k)_s) 	& s \in (a_{k-1}, a_{k}] \subseteq (a_{i-1}, +\infty),\ k= i, i+1, \ldots
\end{array}
\right. \;.
$$
Then $W_{\sharp} = W \oplus P$ is again a semi-surjective submodule of $V$ and $P = \{P_s\}$ is isomorphic to $\F (a_{i-1}, +\infty)$, thus finishing the proof for the first case of \Cref{lem: semi_surj_submodule}.
\end{proof}

\begin{proof}[Proof of \Cref{claim: W_sharp_adding_submodule}]
\begin{enumerate}
	\item
We need to show that for every $s\in \R$, $W_s \cap P_s = \{0\}$.
Note that if $s \notin (a_{i-1}, a_{j-1}]$, then $P_s = \{0\}$, and hence clearly $W_s \cap P_s = \{0\}$.
Let $s\in (a_{k-1}, a_{k}] \subset (a_{i-1}, a_{j-1}]$. We only have to show that $V_s \ni (y^k)_s \notin W_s$.\\
Assume on the contrary that $(y^k)_s \in W_s$.
Take $r = \sup \{ t\ :\ W_s=V_s \ \forall s\leq t \}$, which satisfies the definition of semi-surjectivity of $W$ (in fact, $r=a_{i-1}$ in the notations of the proof of \Cref{lem: semi_surj_submodule}).
Then for every $r \leq a_{k-1} < t < s$ there is an element $w_t \in W_t$ which satisfies $\pi_{t,s} (w_t)=(y^k)_s$.
Consider the element $\til w \in W^k$ whose representatives in each $W_t$ are:
$$
	(\til w)_t = \left\{
	\begin{array}{ll}
	w_t  					& a_{k-1} < t < s  \\
	(y^k)_s 	& t=s \\
	\pi_{st} ((y^k)_s) & s < t \leq a_{k}
	\end{array}
	\right. \;.
$$
Note that $\tilde{w}$ is well-defined, in the sense that its definition is consistent with the persistence morphisms of $W$. In fact, $\til w = y^k$, hence $y^k \in W^k$.
But this conclusion contradicts the minimality of $j$. Hence $(y^k)_s \notin W_s$ for all $s\in (a_{i-1}, a_{j-1}]$.

	\item
	First of all, let us note that $W_\sharp$ is a submodule of $V$, as it is a direct sum of two submodules of $V$.
	Denote by $\pi_{s,t}^{P}$ the persistence morphisms of the persistence module $P = \F(a_i, a_j]$.
Let $r$ be as in the proof of the first part. Then by definition and \Cref{exr: r_semi_surj_witness_supremal}, for any $t\leq r$ we have $W_t = V_t$. Since for $t<r=a_{i-1}$ by construction $P_t=0$, we have also $(W_\sharp)_t=V_t$ for all $t\leq r$.
Next, note that the persistence morphisms of $W_{\sharp}$ are obtained by taking direct sum of the morphisms of $W$ and of $P$: $\pi_{s,t}^{W} \oplus \pi_{s,t}^{P}$.
Both of these morphisms are onto for any $t>s>r$, hence also their direct sum is onto.
\end{enumerate}

\end{proof}

\begin{proof}[Proof of the Normal Form Theorem]
	
	First, the existence of the described decomposition follows from \Cref{lem: semi_surj_submodule}.
	Indeed, take $W(0)=\{0\}$ and inductively build a sequence $W(i)$ of semi-surjective submodules by taking $W(i+1)=W(i)_{\sharp}$ from \Cref{lem: semi_surj_submodule}. At each step, we increase the total dimension of $W(i)$ (at least by $1$), hence this process will terminate when we reach $\Totaldim V$.\\

	Let us now show uniqueness of the decomposition. (See another proof at the end of this section.)
	
	\noindent
	Recall from \Cref{exr: spec_pm_iso_invariant} that the spectrum of a persistence module is isomorphism invariant, hence given a persistence module $V$, the set $\Spec(V)$ determines the end-points of the intervals $I$ that should appear in its Normal Form decomposition (see also \Cref{exr: spec_of_intervals_directsum}). Hence we only have to show that given $V$ it is possible to reconstruct the multiplicities of the intervals in such a decomposition uniquely.
	
	Let $\calB = \{ (I_i, m_i) \}$ be a barcode satisfying \Cref{thm: normal_form_thm} for $V$, that is, $V = \bigoplus_i \F(I_i)^{m_i}$.
	Consider all of their end-points $a_1 < a_2 <\ldots < a_N < a_{N+1}=+\infty$
	(noting that $a_1, \ldots, a_{N+1}$ form the spectrum of $V$).
	
	Denote by $\hat \calB$ the collection of all intervals of the form $I_{ij}=(a_i, a_j]$ for $1\leq i<j\leq N+1$, with multiplicities $\hat m_{ij}$, where $\hat m_{ij} = m_{ij}$ if $I_{ij}$ is present in $\calB$ and $0$ otherwise.
	
	In order to prove uniqueness of the decomposition, we shall recover the multiplicities $m_{ij}$ that correspond to $V$. Let us consider the limit persistence module ${V^i}$ with the natural morphisms $p_{i,j}: V^i \to V^j$.
	Denote $b_{ij} = \rk p_{i,j}$,
	setting also $p_{i,j} = 0$ if $i\leq0$ or $j>N+1$.

	Every interval in $I_{\alpha \beta}$ in $\hat \calB$ that begins before $a_i$ and ends after or at $a_j$ contributes $m_{\alpha \beta}$ to $b_{i,j}$, hence we have
	\begin{equation} \label{eq: b_ij_formula}
		b_{ij} = \sum_{\alpha < i,\ \beta \geq j} m_{\alpha \beta} = \sum_{\alpha \leq i-1,\ \beta \geq j} m_{\alpha \beta} \;.
	\end{equation}
	From this expression, one obtains the following formula for $m_{ij}$,
	\begin{equation} \label{eq: m_ij_formula_from_ranks}
		m_{ij} = b_{i+1,j} + b_{i,j+1} - b_{i,j} - b_{i+1,j+1} \;,
	\end{equation}
	
	\noindent	
	thus reconstructing the multiplicities from the data encapsulated in the collection $\{V^i\}$ that corresponds to $V$.
\end{proof}

\begin{exr}
	Verify that (\ref{eq: m_ij_formula_from_ranks}) follows from (\ref{eq: b_ij_formula}).
	
\end{exr}

\begin{exm}
	As an illustration of (\ref{eq: m_ij_formula_from_ranks}), one can consider the persistence module $\F(a_1, +\infty)$. Then $\Spec \left( \F(a_1, +\infty) \right) = \{a_1, a_2=+\infty\} $, $V^1 = {0}$, $V^2 = \F$, and indeed we have
	$$
		m_{12} = b_{22} + b_{13} - b_{23} - b_{12} = 1 \;.
	$$
	(Only $b_{22} = 1$ is non-zero in the expression for $m_{12}$. See \Cref{fig: one_infinite_bar_spec}.)
	\begin{figure}[!ht]
		\centering
		\includegraphics[scale=1]{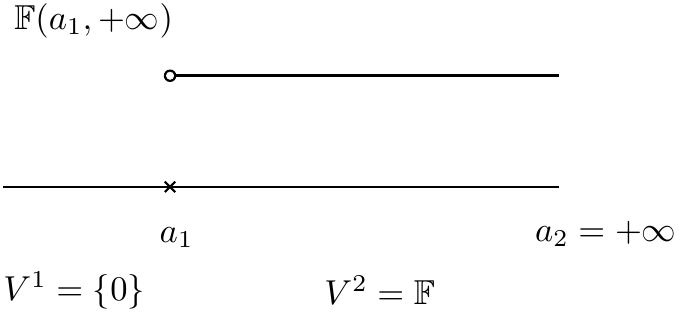}
		\caption{The persistence module $\F (a_1, +\infty)$.}
		\label{fig: one_infinite_bar_spec}
	\end{figure}
\end{exm}

For the sake of completeness, let us present a second proof of the uniqueness of the Normal Form. The argument below is taken from \cite{Chazal_DeSila_Glisee_Oudot}, and is a baby version of the Krull--Remak--Schmidt--Azumaya Theorem \cite{azumaya}. It relies on the following property of interval modules.

\begin{exr}\label{exr:end_interval_module}
Let $I$ be an interval, and consider the persistence module $\F(I)$. Prove that its endomorphism ring is isomorphic to $\F$.
\end{exr}	
	
\begin{proof}[An alternative proof of the uniqueness in the Normal Form theorem]
%
%
%

	Suppose that we have two isomorphic persistence modules: $V = \bigoplus_{i=1}^N \F(I_i)$ and $W = \bigoplus_{j=1}^L \F(J_j)$. We want to show that $N=L$, and that the two collections of intervals are the same up to permutation.
	Suppose that $f : V \to W$ is an isomorphism and $g : W \to V$ is its inverse.
	The proof proceeds by induction on $N$, the base case $N=0$ being trivial.
	Suppose that the claim holds for $N-1$. We shall prove that for $I_1$ there exists $1 \leq j\leq L$ so that $I_1 = J_j$.
	Consider the following compositions for each $1 \leq j \leq L$:
	\begin{align*}
	f_j : \F(I_1) \into V \xrightarrow{\,\,\, f \,\,\, } W \twoheadrightarrow \F(J_j) \,\,\text{ and }\,\,
	g_j : \F(J_j) \into W \xrightarrow{\,\,\, g \,\,\, } V \twoheadrightarrow \F(I_1) \;.
	\end{align*}
	Here the first arrow in each composition is a natural inclusion and the last arrow in each composition is the natural projection ($\twoheadrightarrow$ denotes a surjection). By definition,
	\begin{equation}\label{eq: sum_of_comp_directsum_is_identity}
		\sum_j g_j \circ f_j = \id_{\F(I_1)} \;.
	\end{equation}
	
	In particular, at least one of the summands, which by reordering we may assume to be $g_1 f_1$, is non-zero. By \Cref{exr:end_interval_module}, the only non-invertible element in the endomorphism ring of $\F(I_1)$ is the zero endomorphism. Therefore $g_1 f_1$ is invertible, and it easily follows that $\F(J_1) \simeq \F(I_1)$, and hence clearly $J_1 = I_1$. We also get $\oplus_{i=2}^N \F(I_i) \simeq \oplus_{j=2}^L \F(J_j)$, and so by the induction hypothesis, $L=N$ and, up to reordering, $J_i= I_i$ for $2 \leq i \leq N$. This completes the proof.
\end{proof}

	Let us explain (\ref{eq: sum_of_comp_directsum_is_identity}).
	Denote by $e_1 : \F (I_1) \into V$ and $\iota_j : \F (J_j) \into W$ the natural embeddings, and by $p_j : W \to  \F (J_j)$ and $\pi_1 : V \to \F (I_1)$ the natural projections.
	Let $v \in \F (I_1)$ be a vector and denote $y = \left( f \circ e_1 \right) (v) \in W$. Then
	\begin{align*}
		\left( \sum_j g_j \circ f_j \right)(v) & =
			\sum_j g_j \circ p_j \circ f \circ e_1 \ (v) = \sum_j g_j \circ p_j \ (y) =\\
			& = \sum_j \pi_1 \circ g \circ \iota_j \circ p_j \ (y) =
				\pi_1 \circ g \circ \left( \sum_j \iota_j \circ p_j \right) (y) = \\
			& = \left( \pi_1 \circ g \right) (y) = \left( \pi_1 \circ g \circ f \circ e_1 \right) (v) = \left( \pi_1 \circ e_1 \right) (v) = v \;.
	\end{align*}

\section{Bottleneck distance and the Isometry theorem}
Let us introduce a distance on the space of barcodes.
Given an interval $I = (a,b]$, denote by $I^{-\delta} = (a-\delta, b+\delta]$ the interval obtained from $I$ by expanding by $\delta$ on both sides.
Let $\calB$ be a barcode. For $\epsi > 0$, denote by ${\calB}_\epsi$ the set of all bars from $\calB$ of length greater than $\epsi$. (That is, by considering ${\calB}_\epsi$ we neglect ``short bars".)

A \emph{matching} between two finite multi-sets $X,Y$ is a bijection $\mu : X' \to Y'$, where $X' \subset X,\ Y' \subset Y$. In this case, $X' = \coim \mu,\ Y' = \im \mu$, and we say that elements of $X'$ and $Y'$ are \emph{matched}. If an element appears in the multi-set several times, we treat its different copies separately, e.g. it could happen that only part of its copies are matched.

\begin{defn} \label{defn: delta_matching}
	A \emph{$\delta$-matching} between two barcodes $\calB$ and $\calC$ is a matching $\mu : {\calB}\to {\calC}$, such that:
	\begin{enumerate}
		\item
			${\calB}_{2\delta} \subset \coim \mu$,
		\item
			${\calC}_{2\delta} \subset \im \mu$ ,
		\item
			If $\mu (I) = J$, then $I \subset J^{-\delta},\ J \subset I^{-\delta}$.
	\end{enumerate}
\end{defn}

\begin{exr}
	Show that if $\calB, \calC$ are \emph{$\delta$-matched} (i.e., there is a $\delta$-matching between them) and $\calC, \calD$ are $\gamma$-matched, then $\calB, \calD$ are $(\delta+\gamma)$-matched.
\end{exr}

\begin{defn} \label{defn: bottleneck_distance}
	The \emph{bottleneck distance}, $d_{bot} (\calB, \calC)$, between two barcodes $\calB, \calC$ is defined to be the infimum over all $\delta$ for which there is a $\delta$-matching between $\calB$ and $\calC$.
\end{defn}

\begin{exr}
	Two barcodes $\calB$ and $\calC$ are $\delta$-matched with a finite $\delta$ if and only if they have the same number of infinite rays.
\end{exr}

\begin{cor}
	$d_{bot}$ is a distance on the space of barcodes with the same amount of infinite rays.
\end{cor}

\begin{exm}
	Consider the persistence modules $\F(a,b] \text{ and } \F(c,d]$ of two intervals ($a,b,c,d\in \R$)
	and the corresponding barcodes $\calB =\{ (a,b] \}$ and $\calC = \{ (c,d] \}$.
	Then there is either an empty $\delta$-matching between them for $\delta = \max \big( \frac{b-a}{2}, \frac{d-c}{2} \big)$ (as then the lengths of both intervals do not exceed $2\delta$), or a matching $(a,b] \to (c,d]$ for $\delta = \max (|a-c|, |b-d|)$.
	Therefore  $d_{bot} (\calB, \calC ) \leq \min \Big( \max \big( \frac{b-a}{2}, \frac{d-c}{2} \big), \max (|a-c|, |b-d|) \Big)$,
	(cf. \Cref{claim: _int_distance_two_intervals}).
\end{exm}

\begin{exr} \label{exr: bot_to_interl_for_two_intervals}
	Let $I$, $J$ be two $\delta$-matched intervals. Show that the corresponding interval modules $\F (I)$ and $\F (J)$ are $\delta$-interleaved.
\end{exr}

Recall that we denote by $\calB (V)$ the barcode corresponding to a persistence module $V$, as given by \Cref{thm: normal_form_thm}.

\begin{thm}[Isometry Theorem] \label{thm: isometry_thm}
	The map $V \mapsto \calB (V)$ is an isometry, i.e.\ for any two persistence modules $V, W$, we have
	$
		d_{int} (V,W) = d_{bot} (\calB (V), \calB(W)) \;.
	$
\end{thm}

A proof will be given in \Cref{chp: isometry_thm_proof}.

\begin{exr} \label{exr: d_bot_is_non_degenerate_hence_d_int_also}
	Prove that for any two barcodes $\calB$ and $\calC$ we have $d_{bot} (\calB, \calC) = 0$ if and only if $\calB = \calC$.\\
	Deduce that $d_{int} (V,W) = 0$ if and only if $V = W$.
\end{exr}

\section{Proper persistence modules}\label{sec-prop-pm}
For applications to manifold learning in Section \ref{subsec: Cech_VS_Rips} and to symplectic topology in \Cref{chp8_symplectic_persistence_modules} we shall need a slightly
more sophisticated version of persistence modules than the one we discussed so far. A family of finite-dimensional vector spaces and morphisms is called a {\it proper persistence module} if it satisfies items (1) (persistence) and (3) (semicontinuity) of Definition \ref{defn: pm} while item (2) is modified as follows: the set of exceptional points (i.e., spectral points, see Definition \ref{def-spectral}) is assumed to be a closed discrete bounded from below subset of $\R$ (but not necessarily finite, as in the original definition). Let us emphasize that we do not assume anymore item (4) of
Definition \ref{defn: pm}, that is, the spaces $V_{-t}$ may not vanish for $t$ sufficiently large.
However, since the space of exceptional points is bounded from below, these spaces are pair-wise isomorphic.

We also have to modify accordingly Definition \ref{defn_barcode} of a barcode. A {\it proper barcode} is
a countable collection of bars of the form $(a,b]$, $-\infty \leq a < b \leq +\infty$ with multiplicities
such that
\begin{itemize}
\item for every $c \in \R $ the number of bars (with multiplicities)  containing $c$ is finite;
\item the real endpoints of the bars form a closed discrete bounded from below subset of $\R$.
\end{itemize}
Let us emphasize that in contrast to the original definition we allow (a finite number of) bars of the form $(-\infty,+\infty)$ and $(-\infty, b]$. The theory developed in this chapter (the normal form theorem and the isometry theorem) easily extends  to proper persistence modules and proper barcodes. \footnote{In fact, it extends even further to so called point-wise finite dimensional persistence modules which we do not discuss in this book, see e.g. \cite{Bauer_Lesnick_13-16}.}

\medskip
\noindent\begin{exr}\label{exr-normal-form-proper}
{\rm Prove the analogue of the normal form theorem (Theorem \ref{thm: normal_form_thm}) for proper persistence
modules along the following lines. Let $(V,\pi)$ be a proper persistence module. Write $a_i$, $i \geq 0$
for the points of its spectrum and define vector spaces $V^i$ associated to the interval $(a_{i-1},a_i]$ as (\ref{dfn-vi}) in Section \ref{sec-normal-form}.
Note that for proper persistence modules the spectrum could be infinite, in which case
the total dimension  $\Totaldim (V)$ is not defined anymore. We shall go round this difficulty as follows.
Put
$$\Totaldim_k(V) = \sum_{a_i \leq k}V^i, \;\; k \in \N\;.$$  Define a submodule $W^0 \subset V$  by
$W^0_t = {\rm{im}}(\pi_{-\infty,t})$,  where $\pi_{-\infty,t}$ stands for $\pi_{-s,t}$ with  $s$ sufficiently large.
It is easy to see that the normal form theorem holds for $W^0$. Its barcode $\calB^0$ consists of rays
of the type $(-\infty, b)$ for some $-\infty < b \leq +\infty$.

Next, starting with $W^0$, apply the algorithm whose step is described in the proof of Lemma \ref{lem: semi_surj_submodule}.
At the $j$-th step we get a semi-surjective submodule
$$W^{j}=W^{j-1} \oplus \mathbb F(c_j,d_j]$$ with $c_j > -\infty$.
This eventually yields an increasing sequence of semi-surjective persistence submodules $W^0 \subset W^1 \subset \dots$. Our algorithm guarantees that for given $k \in \N$, at each step of this process $\Totaldim_k(W^j)$
increases with $j$ until it reaches $\Totaldim_k(V)$. Roughly speaking, this means that for every $k \in \N$, the sequence of persistence modules $W^j$ stabilizes on $(-\infty, k]$ for sufficiently large $j$. In particular, this procedure
yields a proper barcode $\calB = \calB^0 \oplus \mathbb F(c_j,d_j]$. It follows that $V = \oplus_{I \in \calB} \mathbb F(I)$,
which completes the proof.
}
\end{exr}

\medskip

Consider now the space of proper barcodes equipped with the bottleneck distance which is defined
exactly as before. We say that two barcodes are equivalent if the bottleneck distance between them is
finite. We do not know a transparent description of the space of equivalence classes.

\medskip
\noindent\begin{exm}\label{exm-geodesics}{\rm Let $(M,g)$ be a closed Riemannian manifold. For $a \in \R$, denote by
$\Lambda^a M$ the space of smooth loops $\gamma: S^1 \mapsto M$ with $\text{length}_g(\gamma) < e^a$.
For a generic metric $g$, the homology $H_*(\Lambda^a M,\mathbb{F})$ with coefficients in a field $\mathbb{F}$ form a
proper persistence module denoted by $V(M,g)$. Note that since for any other metric $g'$ on $M$ there exists
a constant $C>0$ such that $C^{-1}g \leq g' \leq Cg$, the interleaving distance between
the persistence modules $V(M,g)$ and $V(M,g')$ is $\leq \frac{1}{2}\log(C)$. It follows that the equivalence class
of the barcode of $V(M,g)$ is a topological invariant of the manifold (see \cite{Weinberger-alternance}).
We refer to \cite{Weinberger-alternance} for unexpected applications of $V(M,g)$ to variational theory of geodesics.
}
\end{exm}

\chapter{Proof of the Isometry theorem} \label{chp: isometry_thm_proof}
In this chapter we give a proof of \Cref{thm: isometry_thm}. We closely follow \cite{Bauer_Lesnick_13-16}, see also a historical exposition therein.

Note that one of the directions admits a quick proof using the Normal Form theorem (\Cref{thm: normal_form_thm}):
\begin{thm}
	\label{thm: converse_algebraic_stability_aka_easier_half_of_isometry_thm}
	Let $V$ and $W$ be persistence modules. If there is a $\delta$-matching between their barcodes, then $V$ and $W$ are $\delta$-interleaved. In particular,
	$d_{int} (V,W) \leq d_{bot} \big( \calB (V), \calB (W) \big)$.
\end{thm}

\begin{proof}
	[Proof of \Cref{thm: converse_algebraic_stability_aka_easier_half_of_isometry_thm}]
	(Following \cite{Bauer_Lesnick_13-16} and \cite{Lesnick_intdist_multidim_pm_2011})
	By the Normal Form theorem, there are two finite collections of intervals, such that
	$$
		V = \bigoplus_{I \in \calB (V)} \F (I) \ \text{ and }  W = \bigoplus_{J \in \calB (W)} \F (J) \;.
	$$
	Assume that $\mu: \calB (V) \to \calB (W)$ is a $\delta$-matching.
	In order to construct a $\delta$-interleaving between $V$ and $W$, we shall use the matched intervals and neglect the unmatched, which are relatively small.
	Denote:
	\begin{align*}
		V_Y & = \bigoplus_{I \in \coim \mu} \F (I),\ &
		W_Y & = \bigoplus_{J \in \im \mu} \F (J), \\
		V_N & = \bigoplus_{I \in \calB (V) \setm \coim \mu} \F (I),\ &
		W_N & = \bigoplus_{J \in \calB (W) \setm \im \mu} \F (J)	\;. \\
	\end{align*}
		
	Clearly, $V = V_Y \oplus V_N$ and $W = W_Y \oplus W_N$.
	Now, for any matched pair $I,J$, we know that $I \subseteq J^{-\delta}$ and $J \subseteq I^{-\delta}$, so we can choose a pair of $\delta$-interleaving morphisms
	$f_I : \F (I) \to \F(J) [\delta]$ and $g_J : \F (J) \to  \F(I) [\delta]$ (see \Cref{exr: bot_to_interl_for_two_intervals}).
	(Recall the notations $(b,c]^{-\delta} = (b-\delta, d+\delta]$ and $ \F(I)[\delta] = \F (I-\delta)$ for a $\delta$-shift of a persistence module.)
	
	\noindent
	These pairs induce a pair of $\delta$-interleaving morphisms
		$$f_Y : V_Y \to W_Y[\delta] \  \text{ and } g_Y : W_Y \to V_Y [\delta] \;. $$
	Since the intervals that are not matched by $\mu$ are of length $< 2 \delta$, $V_N$ is $\delta$-interleaved with the empty set, and so is $W_N$.
	Overall, using $f_Y$ and $g_Y$ we can produce $\delta$-interleaving morphisms between $V$ and $W$ as follows: take $f: V\to W$ to be $f \restr_{V_Y} = f_Y$, $f \restr_{V_N} = 0$ and similarly for $g: W\to V$.
\end{proof}

Let us state separately the second direction of \Cref{thm: isometry_thm}, also called the \emph{Algebraic Stability Theorem}.
\begin{thm} \label{thm: thm_algebraic_stability}
	Let $V$ and $W$ be persistence modules and $\calB (V)$, $\calB (W)$ be their barcodes.
	Then $d_{int} (V,W) \geq d_{bot} \big( \calB (V), \calB (W) \big)$.
\end{thm}

\begin{proof}[Proof of \Cref{thm: isometry_thm}]
	The Isometry theorem follows from \Cref{thm: converse_algebraic_stability_aka_easier_half_of_isometry_thm} and \Cref{thm: thm_algebraic_stability}.
\end{proof}

The proof of \Cref{thm: thm_algebraic_stability} occupies the rest of this chapter.

\section{Preliminary claims}

\subsection{Monotonicity with respect to injections and surjections}
Let $(V, \pi)$ and $(W, \theta)$ be two persistence modules with barcodes $\calB$ and $\calC$ respectively.
For an interval $I= (b,d]$, with $d\in \R \cup \{+\infty\}$, denote by $\calB_I^- \subseteq \calB$ the collection of all bars $(a,d]\in \calB$ with $a \leq b$, i.e.\ bars that begin no later than $b$ and end exactly at $d$ (taken with multiplicity). See \Cref{fig: barcode_minus}.

\begin{figure}[!ht]
	\centering
	\includegraphics[scale=1]{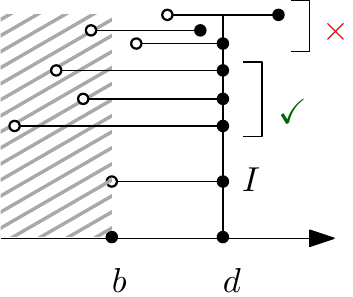}
	\caption{
		Bars $(a,d]$ to be included (or not) in $\calB_I^{-}$ for $I = (b,d]$.
	}
	\label{fig: barcode_minus}
\end{figure}

\begin{prop} \label{prop: inj_between_pm_monotonicity_of_B-}
	Let $I = (b,d]$ be an interval. Assume that there exists an injective morphism $\iota: (V, \pi) \to (W, \theta)$. Then
		$\# \calB_I^- \leq \# \calC_I^-$.
\end{prop}

\begin{exm}
	For $V= \F(b,d]$ and $W= \F (a,d]$ with $b\geq a$ we have a natural injection $\iota : V \to W$, and indeed for any interval $I$, $\# \calB_I^- \leq \# \calC_I^-$, see also \Cref{fig: interval_search_for_longer_on_left}.
	\begin{figure}[!ht]
		\centering
		\includegraphics[scale=1]{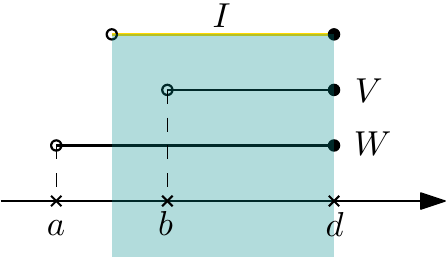}
		\caption{Monotonicity with respect to injections.}
		\label{fig: interval_search_for_longer_on_left}
	\end{figure}
\end{exm}

\begin{proof}[Proof of \Cref{prop: inj_between_pm_monotonicity_of_B-}]
	Put $\displaystyle{ E_I^- (V) = \bigcap_{b<s<d} \im \pi_{s,d} \cap \bigcap_{r>d} \ker \pi_{d,r} \subseteq V_d }$.
	The set $E_I^-(V)$ consists of the elements in $V_d$ which come from all $V_s$, $s\in (b,d)$ and disappear in all $V_r$, $r>d$.
	So $\dim E_I^- (V) = \# \calB_I^-$.
	Note that for every morphism $p: (V,\pi) \to (W, \theta)$ the diagram
	\noindent
	\begin{center} %
		\begin{tikzpicture}[scale=1.5]
		\node (s) at (0,0) {$V_s$};
		\node (r) at (1,0) {$V_r$};
		\node (sw) at (0,-1) {$W_s$};
		\node (rw) at (1,-1) {$W_r$};
		\path[->,font=\scriptsize,>=angle 90]
		%
		(s) edge[->] node[above]{$\pi_{s,r}$} (r)
		(sw) edge[->] node[above]{$\theta_{s,r}$} (rw)
		%
		(s) edge node[left]{$p_s$} (sw)
		(r) edge node[right] {$p_r$} (rw);
		\end{tikzpicture}
	\end{center}
	\noindent
	commutes, so $p_r(\im \pi_{s,r}) \subseteq \im \theta_{s,r}$ and $p_s (\ker \pi_{s,r}) \subseteq \ker \theta_{s,r}$. Using the first inclusion for $r = d$ and every $b<s<d$ and the second for $s=d$ and every $r>d$, we get $p_d (E_I^- (V)) \subseteq E_I^- (W)$.
	Applying this result for an injection we obtain $\dim E_I^- (V) \leq \dim E_I^- (W)$.
\end{proof}

%

Analogously, for $I = (b,d]$, denote by $\calB_I^+ \subseteq \calB$ the collection of all bars of the form $(b,c]$ in $\calB$ with $c\geq d$ (counting with multiplicity).
Imitating the proof above, one can prove the following claim, which we leave as an exercise:

\begin{prop} \label{prop: surj_between_pm_monotonicity_of_B+}
	Using the notations above,
	if there exists a surjection from $V$ to $W$, then $\# \calB_I^+ \geq \# \calC_I^+$.
\end{prop}

\subsection{Induced matchings construction}
Given a morphism between persistence modules, we need a procedure that will associate a matching to it, that will be called an \emph{induced matching}, following \cite{Bauer_Lesnick_13-16}. We start with the case of such a morphism being either an injection or a surjection, using which we later explain how to associate a matching to a general morphism.

As before, let $(V, \pi)$ and $(W, \theta)$ be two persistence modules and denote by $\calB$ and $\calC$ the corresponding barcodes.
\begin{defn} \label{defn: induced_matching_inj}
	Suppose that there exists an \textbf{injection} $\iota: V \to W$.
	Let us define the \emph{induced matching} $\mu_{inj} : \calB \to \calC$ as follows. For every $d\in \R \cup \{\infty\}$, sort the bars of $\calB$ of the form $(\cdot, d]$ in ``longest-first" order:
	$$
	(b_1, d] \supset (b_2, d] \supset \ldots \supset (b_k, d] \ \text{ in $\calB$, with } b_1\leq b_2 \leq \ldots \leq b_k \;,
	$$
	and similarly for $\calC$:
	$$
	(c_1, d] \supset (c_2, d] \supset \ldots \supset (c_l, d] \ \text{ in $\calC$, with } c_1\leq c_2 \leq \ldots \leq c_l \;.
	$$	
	\noindent
	Note that by \Cref{prop: inj_between_pm_monotonicity_of_B-}, $k\leq l$.
	Now, match the bars according to the ``longest-first" order, i.e., at each step, take the longest interval from the first list and match it with the longest interval of the second list.
	Do the same for all $d \in \R \cup \{\infty\}$ to obtain a matching $\mu_{inj} : \calB \to \calC$.
\end{defn}

\begin{prop} \label{prop: coim_mu_inj}
	If there is an injection from $(V, \pi)$ to $(W, \theta)$, then the induced matching $\mu_{inj}: \calB \to \calC$ satisfies:
	\begin{enumerate}[(1)]
		\item
			$\coim \mu_{inj} = \calB$,
		\item
			For all $(b,d] \in \calB$, $\mu_{inj} (b,d] = (c,d]$ with $c\leq b$.
	\end{enumerate}	
\end{prop}

\begin{proof}
	As mentioned, the first part follows from \Cref{prop: inj_between_pm_monotonicity_of_B-} applying it to the interval $(b_k,d]$, i.e. $k\leq l$, which implies that all the bars from $\cal B$ are matched. Since we match ``longest-first", inductively we get that $\mu (b_i,d] = (c_i, d]$. For the second part, by the same proposition applied to the intervals $(b_i, d]$, yields inductively that $b_i \leq c_i$ for each $1\leq i\leq k$.
\end{proof}

\begin{rmk} \label{rmk: matching_independ_of_particular_inj}
	Note that the induced matching does not depend on the injection $\iota$, but only on the assumption that there exists an injection.
\end{rmk}

Now, assume instead that there exists a \textbf{surjection} $\sigma: V \to W$ between the two persistence modules.
\begin{defn} \label{defn: induced_matching_sur}
	Define the \emph{induced matching} $\mu_{sur}: \calB \to \calC$ as follows.
	For every $b\in \R$, sort the intervals $(b, \cdot] \in \calB$ in decreasing order:
	$$
	(b, d_1] \supset (b, d_2] \supset \ldots \supset (b, d_k] \ \text{ in $\calB$, with } d_1\geq d_2 \geq \ldots \geq d_k \;,
	$$
	and similarly for $\calC$:
	$$
	(b, e_1] \supset (b, e_2] \supset \ldots \supset (b, e_l] \ \text{ in $\calC$, with } e_1\geq e_2 \geq \ldots \geq e_l \;.
	$$
\noindent
Then match them according to the ``longest-first" principle, and assemble these matchings for all $b$.
\end{defn}

This construction again is independent of the particular surjection $\sigma$ (see \Cref{rmk: matching_independ_of_particular_inj}).
We have the following analogue of \Cref{prop: coim_mu_inj}, which we leave as an exercise.
\begin{prop} \label{prop: im_mu_sur}
	If there exists a surjection from $(V, \pi)$ to $(W, \theta)$, then the induced matching $\mu_{sur}: \calB \to \calC$ satisfies:
	\begin{enumerate} [(1)]
		\item
			$\im \mu_{sur} = \calC$,
		\item
			$\mu_{sur} (b,d] = (b,e]$ with $d \geq e$. \\
	\end{enumerate}	
\end{prop}

Let us now present the strategy of the proof of \Cref{thm: thm_algebraic_stability}, the details would be carried out in the next sections.
For any morphism $f: V \to W$, we can write the following decomposition: 

\noindent
\begin{center} %
	\begin{tikzpicture}[column sep=large] 
	\node (V) at (0,0) {$V$};
	\node (imf) at (2,0) {$\im f$};
	\node (W) at (4,0) {$W$};

	\path[->,font=\scriptsize,>=angle 90]
	%
	(V) edge[->] node[above]{surjection} (imf)
	(imf) edge[->] node[above]{injection} (W);
	%
	%
	\end{tikzpicture} \;.
\end{center}

\noindent
According to \Cref{prop: im_mu_sur} and \Cref{prop: coim_mu_inj}, having these two maps, we can build the induced matchings
	$\mu_{sur}: \calB(V) \to \calB(\im f)$ and
	$\mu_{inj}: \calB(\im f)\to \calB(W)$.
	
\begin{defn} \label{defn: induced_matching_general}
	For a general morphism $f$ we define the \emph{induced matching} $\mu(f) : \calB (V) \to \calB (W)$ to be
	the composition $\mu(f) = \mu_{inj} \circ \mu_{sur}$, which is defined $\im \mu_{sur} = \calB (\im f) = \coim (\mu_{inj})$.
\end{defn}
Note that in fact $\mu(f)$ depends only on $\im f$, but not on $f$ itself. (See \Cref{rmk: matching_independ_of_particular_inj}.)

\begin{rmk}
	Via this construction we in fact associate a matching to \emph{any} mapping between persistence modules, not only injections or surjections.
	In case $f: V \to W$ is either an injection or a surjection, the induced matching $\mu(f)$ coincides either with $\mu_{inj}$ or with $\mu_{sur}$ respectively.
\end{rmk}

\begin{exm}
	\footnote{Taken from \cite{Bauer_Lesnick_13-16}.}
	Take $V = \F (1,3] \oplus \F (1,2]$ and $W = \F (3,4] \oplus \F (0,2]$, and a morphism $f: V \to W$ defined by
	$f\restr_{\F(1,3]} = 0$ and
	$f \restr_{\F(1,2]} : \F(1,2] \to \F(0,2]$ corresponds to multiplication by $1$ (recall \Cref{exr: morphism_between_Q(segment)s}).\\
	\noindent
	Then $\im f = 0 \oplus \F(1,2] \subset \F(3,4] \oplus \F(0,2]$. (See \Cref{fig: sur_image_inj}.)
	
	\begin{figure}[!ht]
		\centering
		\includegraphics[scale=1]{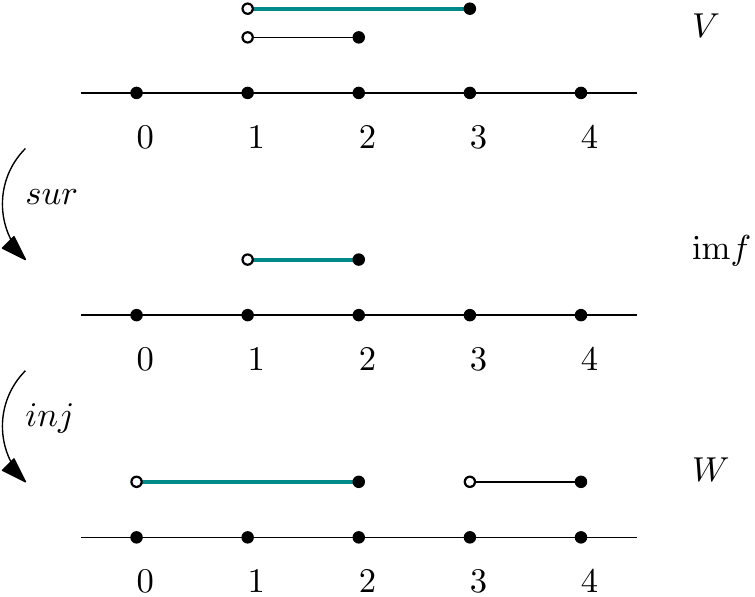}
		\caption{Composing two matchings.}
		\label{fig: sur_image_inj}
	\end{figure}
	
	\noindent
	So $\mu_{sur} : (1,3] \mapsto (1,2]$, $\mu_{inj} (1,2] \mapsto (0,2]$.
	Thus, the map $\mu(f): \calB(V) \to \calB(W)$ takes $\mu (f) : (1,3] \mapsto (0,2]$, despite the fact that $f \restr_{\F(1,3]} = 0$.
\end{exm}


	Let us consider the category of barcodes with morphisms being matchings. We have established a correspondence between the objects of the category of persistence modules and these of the category of barcodes, namely, a persistence module corresponds to its barcode, $V \mapsto \calB (V)$.
	Moreover, having a morphism $f: V\to W$ between two persistence modules, we have  defined a matching $\mu (f)$ between $\calB (V)$ and $\calB (W)$. But does this mapping give a functor between the two categories?

	It turns out that in general this is not a functorial correspondence.
	
	\begin{exm} \label{exm: non-commutative_if_general_morphism}
		\footnote{
		This example is a modification of an example in \cite{Bauer_Lesnick_13-16}.
		}
		Let $I$ be any interval and consider the following persistence modules:
		$$
		U = V  = \F(I) \oplus \F (I),\
		W  = \F(I)  \;,
		$$
		and two morphisms $f:U \to V $ and $g:V\to W$ given by:
		$$
		f(s,t) = (s,0) \text{ and } g(s,t)= t \;.
		$$
		Thus, $\mu (f)$ matches one copy of $I$ to a copy of $I$ in $\calB (V)$, and the second copy remains unmatched. Then, $\mu (g)$  matches again one copy of $I$ to $I$. So overall, $\mu (g) \circ \mu (f)$ matches one copy of $I$ to the bar $I$ of $\calB (W)$ and the second one stays unmatched.
		On the other hand, $g \circ f = 0$, the reader can check that the corresponding matching is empty.
	\end{exm}
	
	However, if we restrict the morphisms between persistence modules to be either only injections or only surjections, the mapping that takes a persistence module $V$ to its barcode $\calB (V)$ and a morphism $f: V \to W$ to the induced matching (either $\mu_{inj}$ or $\mu_{sur}$) is a functor, as stated in \Cref{claim: diagram_commutes_on_level_of_barcodes_surj_inj}.

\begin{claim} \label{claim: diagram_commutes_on_level_of_barcodes_surj_inj}
	Consider the following commutative diagram in the category of persistence modules with either injections only or surjections only:
	\noindent
	\begin{center} %
		\begin{tikzpicture}
		\matrix (m) [matrix of math nodes,row sep=2em,column sep=4em,minimum width=2em]
		{
			U & V & W \\
		};
		\path[-stealth, decoration={snake,segment length=4,amplitude=3,
			post=lineto,post length=10pt}]
		(m-1-1) edge node [above] {$f$} (m-1-2)
		(m-1-2) edge node [above] {$g$} (m-1-3)
		(m-1-1) edge[bend left=-20] node [below] {$h$} (m-1-3);
		\end{tikzpicture} \;.
	\end{center}
	Then the corresponding diagram on the level of barcodes commutes as well:
	\noindent
	\begin{center} %
		\begin{tikzpicture}
		\matrix (m) [matrix of math nodes,row sep=2em,column sep=4em,minimum width=2em]
		{
			\calB(U) & \calB (V) & \calB (W) \\
		};
		\path[-stealth, decoration={snake,segment length=4,amplitude=3,
			post=lineto,post length=10pt}]
		(m-1-1) edge node [above] {$\mu_{\natural} (f)$} (m-1-2)
		(m-1-2) edge node [above] {$\mu_{\natural} (g)$} (m-1-3)
		(m-1-1) edge[bend left=-20] node [below] {$\mu_{\natural} (h)$} (m-1-3);
		\end{tikzpicture} \;,
	\end{center}
	where $\mu_\natural$ denotes either $\mu_{inj}$ or $\mu_{sur}$ respectively.
\end{claim}

We prove functoriality in the case of injections, leaving the second case to the reader.
\begin{proof}
	By \Cref{defn: induced_matching_inj} and \Cref{prop: inj_between_pm_monotonicity_of_B-}, for any $d\in \R \cup \{+\infty\}$, the barcodes corresponding to $U, V, W$ consist of the following bars that end at $d$:
	\begin{align*}
		\calB (U) &:\ (a_1, d] \supseteq \ldots \supseteq (a_k, d]\\
		\calB (V) &:\ (b_1, d] \supseteq \ldots \supseteq (b_k, d] \supseteq \ldots \supseteq (b_l, d]\\
		\calB (W) &:\ (c_1, d] \supseteq \ldots \supseteq (c_k, d] \supseteq \ldots \supseteq (c_l, d] \supseteq \ldots \supseteq (c_q, d] \;,
	\end{align*}
	where $k\leq l\leq q$. Moreover, $\mu_{inj}(f) (a_i,d] = (b_i,d]$, $\mu_{inj} (g) (b_i, d] = (c_i, d]$ and $\mu_{inj} (h) (a_i,d] = (c_i, d]$ for any $1\leq i \leq k$. This holds for any $d$, so the required diagram on the level of barcodes commutes.
\end{proof}

%


\section{Main lemmas and proof of the theorem} \label{section: main_lemmas_and_the_proof}

	Assume that $(V, \pi^V)$ and $(W, \pi^W)$ are $\delta$-interleaved, i.e.\ there exist two morphisms
	$f: V \to W[\delta]$ and $g: W \to V[\delta]$, such that
	$g[\delta] \circ f = \Phi^{2\delta}_V$ and $f[\delta] \circ g = \Phi^{2\delta}_W$, where
	$\Phi^{2\delta}_V = \pi^V_{t,t+2\delta}$ and similarly for $\Phi^{2\delta}_W$.
	Our aim is to build a $\delta$-matching between $\calB(V)$ and $\calB(W)$.
	
	Recall the notation $\calB_{\epsi}$ for the collection bars of length $>\epsi$ in a barcode $\calB$.
	
\begin{lem} \label{lem: pm_surj}
	Assume that we have two $\delta$-interleaved persistence modules $(V, \pi^V)$ and $(W, \pi^W)$, with $\delta$-interleaving morphisms $f: V \to W[\delta]$ and $g: W \to V[\delta]$.
	Consider the surjection $f: V \to \im f$
	and the induced matching $\mu_{sur}: \calB(V) \to \calB(\im f)$ (see \Cref{defn: induced_matching_sur}). Then
	\begin{enumerate}[(1)]
		\item
			$\coim \mu_{sur} \supseteq  \calB (V) _{2\delta}$,
		\item
			$\im \mu_{sur} = \calB (\im f)$, and 
		\item
			$\mu_{sur}$ takes $(b,d] \in \coim \mu_{sur}$ to $(b,d']$ with $d' \in [d-2\delta, d]$.
	\end{enumerate}
\end{lem}

\begin{lem} \label{lem: pm_inj}
	Assume that we have two $\delta$-interleaved persistence modules $(V, \pi^V)$ and $(W, \pi^W)$, with $\delta$-interleaving morphisms $f: V \to W[\delta]$ and $g: W \to V[\delta]$.
	Consider the injection $\im f \to W[\delta]$ and
	the induced matching $\mu_{inj} : \calB(\im f) \to \calB(W[\delta])$ (see \Cref{defn: induced_matching_inj}). Then
	\begin{enumerate}[(1)]
		\item
			$\coim \mu_{inj} = \calB (\im f)$, 
		\item
			$\im \mu_{inj} \supseteq  \calB (W[\delta]) _{2\delta}$, and
		\item
			$\mu_{inj}$ takes $(b,d'] \in \coim \mu_{inj}$ to $(b', d']$ with $b' \in [b-2\delta, b]$.
	\end{enumerate}
\end{lem}

\noindent
The proofs of these lemmas appear in \Cref{section: proofs_of_lems_for_isometry}.\\

\begin{proof}[Proof of \Cref{thm: thm_algebraic_stability}]
Let us consider the induced matching $\mu (f) = \mu_{inj} \circ \mu_{sur}$, (post)composing it with the map $\Psi_\delta : \calB (W[\delta]) \to \calB (W)$ defined for all bars by $(a,b] \mapsto (a+\delta, b+\delta]$ (this map shifts each bar by $\delta$ to the right).

We claim that $\Psi_{\delta} \circ \mu (f)$ is a $\delta$-matching between $\calB (V)$ and $\calB (W)$.
using \Cref{lem: pm_surj} and \Cref{lem: pm_inj}, we get the following diagram:

\noindent
\begin{center} %
	\begin{tikzpicture}
		\matrix (m) [matrix of math nodes,row sep=2em,column sep=4em,minimum width=2em]
		{
			 \calB(V) & \  &
			 \calB(W[\delta]) _{2\delta} &
			 \calB(W) _{2\delta} \\
			 \calB(V) _{2\delta} & \calB (\im f) & \im \mu_{inj} & \calB (W) \\
			(b,d] & (b,d'] & (b',d'] & (b'+\delta, d'+\delta]\\
		};
		\path[-stealth, decoration={snake,segment length=4,amplitude=.9,
			post=lineto,post length=2pt}]
		(m-2-1) edge node [above] {$\mu_{sur}$} (m-2-2)
		(m-2-2) edge node [above] {$\mu_{inj}$} (m-2-3)
		(m-2-3) edge node [above] {$\Psi_{\delta}$} (m-2-4)
		%
		(m-3-1) edge[decorate] (m-3-2)
		(m-3-2) edge[decorate] (m-3-3)
		(m-3-3) edge[decorate] (m-3-4)
		%
		(m-3-1) edge[draw=none] node [sloped] {$\in$} (m-2-1)
		(m-3-2) edge[draw=none] node [sloped] {$\in$} (m-2-2)
		(m-3-3) edge[draw=none] node [sloped] {$\in$} (m-2-3)
		(m-3-4) edge[draw=none] node [sloped] {$\in$} (m-2-4)
		(m-2-1) edge[draw=none] node [sloped] {$\subseteq$} (m-1-1)
		(m-2-3) edge[draw=none] node [sloped] {$\supseteq$} (m-1-3)
		(m-2-4) edge[draw=none] node [sloped] {$\supseteq$} (m-1-4);
	\end{tikzpicture}
\end{center}

\noindent
Here by \Cref{lem: pm_inj},
a bar $(b,d]\in \calB (V)_{2\delta}$ is mapped to $\mu_{sur}(b,d] = (b,d'] \in \calB (\im f)$, where $d-2\delta \leq d' \leq d$.
This bar is then mapped to $\mu_{inj}(b,d']= (b',d'] \in \big(  \calB (W[\delta]) \big) _{2\delta}$, where $b - 2\delta \leq b' \leq b$, by \Cref{lem: pm_surj}.
Finally, $(b',d']$ is shifted by $\delta$ to the right, i.e. it is mapped to $(b
+\delta, d'+\delta]$ by $\Psi_{\delta}$.

Note that it follows in particular that every bar in $\calB (V)_{2\delta}$ (i.e. a ''long enough" bar) is indeed matched by $\mu (f)$, and similarly one can check that any bar in $\calB (W)_{2\delta}$ is matched.
Moreover, from the information about $b', d'$ we obtain
$$
\left\{
\begin{array}{ll}
	d - 2\delta \leq d' \leq d \\
	b - 2\delta \leq b' \leq b
\end{array}
\right.
\Rightarrow		
\left\{
\begin{array}{ll}
	d-\delta \leq d' + \delta \leq d+ \delta \\
	b - \delta \leq b'+\delta \leq b+ \delta
\end{array}
\right. \;,
$$

\noindent
so $\Psi_\delta \circ \mu (f)$ is a $\delta$-matching between $\calB(V)$ and $\calB(W)$.
Hence
	$$d_{bot} \big( \calB(V), \calB(W) \big) \leq d_{int} (V,W) \;.$$
\end{proof}

\section{Proofs of Lemma \protect\ref{lem: pm_surj} and Lemma \protect\ref{lem: pm_inj} } \label{section: proofs_of_lems_for_isometry}

Recall our setting: for two persistence modules $(V, \pi^V)$ and $(W, \pi^W)$ and $\delta>0$, we assume that $f: V \to W[\delta]$ and $g: W\to V[\delta]$ are $\delta$-interleaving morphisms, i.e. $g[\delta]\circ f = \Phi_V^{2\delta}$ and $f[\delta] \circ g = \Phi_W^{2\delta}$, where
$\Phi^{2\delta}_V = \pi^V_{t,t+2\delta}$ and similarly for $\Phi^{2\delta}_W$.

\subsection{Proof of Lemma \protect\ref{lem: pm_surj}.}
	In order to avoid confusion, along the proof, we will denote by $\mu_{sur}(f)$ the matching that corresponds to the mapping $f: V \to \im f$ (and similarly for the other maps), although the matching itself  \emph{does not depend} on the map $f$.


\begin{proof}[Proof of \Cref{lem: pm_surj}]
	\begin{enumerate} [(1)]
		\item
			By assumption, we have the following commutative diagram
			\begin{center}
				\begin{tikzpicture}
				\matrix (m) [matrix of math nodes,row sep=2em,column sep=2em,minimum width=2em]
				{
					V & \im f & \im \Phi_V^{2\delta} \\
				};
				\path[-stealth, decoration={snake,segment length=4,amplitude=3,
					post=lineto,post length=10pt}]
				(m-1-1) edge node [above] {$f$} (m-1-2)
				(m-1-2) edge node [above] {$g[\delta]$} (m-1-3)
				(m-1-1) edge[bend left=-20] node [below] {$\Phi^{2\delta}_V$} (m-1-3);
				\end{tikzpicture} \;,
			\end{center}
			
			\noindent
			where all the three maps are surjective: $f$ and $\Phi_V^{2\delta}$ are by definition onto their images,
			and since the diagram commutes, $g[\delta]$ restricted to $\im f$ is onto $\im \Phi_V^{2\delta}$.

			By \Cref{claim: diagram_commutes_on_level_of_barcodes_surj_inj}, the following diagram commutes:
						\noindent
						\begin{center} %
							\begin{tikzpicture}
							\matrix (m) [matrix of math nodes,row sep=2em,column sep=4em,minimum width=2em]
							{
								\calB(V) & \calB (\im f) & \calB (\im \Phi_V^{2\delta}) \\
							};
							\path[-stealth, decoration={snake,segment length=4,amplitude=3,
								post=lineto,post length=10pt}]
							(m-1-1) edge node [above] {$\mu_{sur} (f)$} (m-1-2)
							(m-1-2) edge node [above] {$\mu_{sur} (g[\delta])$} (m-1-3)
							(m-1-1) edge[bend left=-20] node [below] {$\mu_{sur} ( \Phi^{2\delta}_V )$} (m-1-3);
							\end{tikzpicture} \;.
						\end{center}
			
			Note that $\coim \mu_{sur} ( \Phi^{2\delta}_V ) = \calB(V) _{2\delta}$
			by construction of the matching $\mu_{sur}$, i.e. all ``too short" intervals are forgotten during this process. In detail, for each starting (left) point, we list all bars of $\calB(V)$ and $\calB (\im \Phi_V^{2\delta}) = \{ (b,d-2\delta]\ :\ (b,d]\in \calB(V), d-b>2\delta \}$ in length-non-increasing order and then match them according to ``longest-first" order. In this case, each bar $(b,d] \in\calB(V)$ is matched with the bar $(b, d-2\delta] \in \calB (\im \Phi_V^{2\delta})$, as long as $d-b > 2\delta$, and intervals of smaller length are not matched.
			Thus, we obtain that
			$
			\coim \mu_{sur} (f) \supseteq \coim \mu_{sur} (\Phi^{2\delta}_V)
			= \big(  \calB(V) \big) _{2\delta} \;.
			$
		\item
			Follows from \Cref{prop: im_mu_sur}.
		\item
			Let $(b,d] \in \calB(V)$. We need to examine $\mu_{sur} \big( (b,d] \big)$.
			There are two cases:
			\begin{enumerate}[(i)]
				\item
					If $d-b > 2\delta$, then $(b,d]$ is mapped to $(b,d']$ by $\mu_{sur}(f)$, where by \Cref{prop: im_mu_sur} $d' \leq d$. The interval $(b,d']$ is in turn mapped to some interval $(b,d'']$ by $\mu_{sur}(g[\delta])$, with $d'' \leq d'$. On the other hand, by commutativity of the above diagram, we know that $(b,d''] = (b,d-2\delta]$, i.e.\ $d'' = d - 2\delta$, so in particular $d - 2\delta \leq d' \leq d$.
					
					\begin{center} %
						\begin{tikzpicture}
						\matrix (m) [matrix of math nodes,row sep=2em,column sep=4em,minimum width=2em]
						{
							\calB (V)_{2\delta} & \calB (\im f) & \calB (\im \Phi_V^{2\delta}) \\
							(b,d] & (b,d'] & (b,d''] \\
							 &  &  (b, d- 2\delta]\\
						};
						\path[-stealth, decoration={snake,segment length=4,amplitude=.9,
							post=lineto,post length=2pt}];
						%
						\draw[|->] (m-2-1) edge node [above] {$\mu_{sur} (f)$} (m-2-2);
						\draw[|->] (m-2-2) edge node [above] {$\mu_{sur} (g[\delta])$} (m-2-3);
						\draw[|->] (m-2-1) to[out=-80, in=180, looseness=1] node[above] {$\mu_{sur} ( \Phi^{2\delta}_V )$} (m-3-3) ;
						%
						\draw (m-2-1) edge[draw=none] node [sloped] {$\in$} (m-1-1);
						\draw (m-2-2) edge[draw=none] node [sloped] {$\in$} (m-1-2);
						\draw (m-2-3) edge[draw=none] node [sloped] {$\in$} (m-1-3);
						\draw (m-2-3) edge[draw=none] node [sloped] {$=$} (m-3-3);
						\end{tikzpicture}
					\end{center}
				\item
					If $d-b \leq 2\delta$, then $(b,d]$ (if it is in the coimage of $\mu_{sur}(f)$) is matched to $(b,d']$ with $d\geq d'$.
					But $d' > b \geq d - 2\delta$, so again $d'\in [d-2\delta,d]$.
			\end{enumerate}
	\end{enumerate}
\end{proof}

\subsection{Proof of Lemma \protect\ref{lem: pm_inj}.} \label{ssec-proof-3.2.2}
We shall give a proof all ''shifted" be $\delta$, to have more pleasant notations.

\begin{proof}[Proof of \Cref{lem: pm_inj}.]

\begin{enumerate} [(1)]
	
	\item
	Follows from \Cref{prop: coim_mu_inj}.
	
	\item
	By assumption, $f[\delta] \circ g = \Phi_W^{2\delta}$, i.e. the following diagram commutes:
	\begin{center}
		\begin{tikzpicture}
		\matrix (m) [matrix of math nodes,row sep=2em,column sep=2em,minimum width=2em]
		{
			W & \im g & W[2\delta] \\
		};
		\path[-stealth, decoration={snake,segment length=4,amplitude=3,
			post=lineto,post length=10pt}]
		(m-1-1) edge node [above] {$g$} (m-1-2)
		(m-1-2) edge node [above] {$f[\delta]$} (m-1-3)
		(m-1-1) edge[bend left=-20] node [below] {$\Phi_W^{2\delta}$} (m-1-3);
		\end{tikzpicture} \;.
	\end{center}
	
	 Thus, $\im \Phi_W^{2\delta} \subseteq \im f[\delta] \subseteq W[2\delta]$, i.e. there exist natural injections $j$ and $i$ respectively, so that the following diagram commutes:
	\begin{center}
		\begin{tikzpicture}
		\matrix (m) [matrix of math nodes,row sep=2em,column sep=2em,minimum width=2em]
		{
			\im \Phi_W^{2\delta} & \im f[\delta] & W[2\delta] \\
		};
		\path[-stealth, decoration={snake,segment length=4,amplitude=3,
			post=lineto,post length=10pt}]
		(m-1-1) edge node [above] {$j$} (m-1-2)
		(m-1-2) edge node [above] {$i$} (m-1-3)
		(m-1-1) edge[bend left=-20] node [below] {$\Phi_W^{2\delta}$} (m-1-3);
		\end{tikzpicture} \;.
	\end{center}

	Since all the morphisms here are injections, by \Cref{claim: diagram_commutes_on_level_of_barcodes_surj_inj}, we have a commutative diagram on the level of barcodes:
	
	\noindent
	\begin{center} \label{commut_inj_barcodes}
		\begin{tikzpicture}
		\matrix (m) [matrix of math nodes,row sep=2em,column sep=4em,minimum width=2em]
		{
			\calB( \im \Phi_W^{2\delta} ) & \calB ( \im f[\delta] ) & \calB (W[2 \delta]) \\
		};
		\path[-stealth, decoration={snake,segment length=4,amplitude=3,
			post=lineto,post length=10pt}]
		(m-1-1) edge node [above] {$\mu_{inj} (j)$} (m-1-2)
		(m-1-2) edge node [above] {$\mu_{inj} (i)$} (m-1-3)
		(m-1-1) edge[bend left=-20] node [below] {$\mu_{inj} (\Phi_W^{2\delta})$} (m-1-3);
		\end{tikzpicture} \;.
	\end{center}
	
	Note that
	\begin{align*}
		\calB (\im \Phi_W^{2\delta}) &=
			\{ (b, d-2\delta] \ :\ (b,d)\in \calB(W),\ d-b>2\delta \} \;,\\
		\calB (W[2\delta]) &=
			\{ (b-2\delta, d-2\delta] \ :\ (b,d] \in \calB (W)\} \;,	
	\end{align*}
	and $\mu_{inj} (\Phi_W^{2\delta}) (b,d-2\delta] = (b-2\delta, d - 2\delta]$.
	
	Hence $\im \mu_{inj} (i) \supseteq \im \mu_{inj} (\Phi_W^{2\delta}) = \calB (W[2\delta])_{2\delta}$, which is what we wanted to prove, written in shifted by $\delta$ notations.
	\item
	For an interval $(b,d] \in \calB (\im f[\delta])$, denote $\mu_{inj} (b, d] = (a ,d]$ for some $a$, such that $(a,d] \in \calB (W)$. By \Cref{prop: coim_mu_inj} we know that $a \leq b$.
	If $d-b \leq 2\delta$, then $a\geq d - 2\delta > b-2\delta$.
	If otherwise, $d-b > 2\delta$, and there exists an interval $(a+2\delta,d] \in \calB (\im \Phi_W^{2\delta})$ such that $\mu_{inj} (\Phi_W^{2\delta}) (a+2\delta, d] = \mu_{inj} (i) (b,d] = (a, d]$, so $b \leq a+2\delta$, hence again $b-2\delta \leq a\leq b$.
\end{enumerate}
\end{proof} 

\chapter{What can we read from a barcode?} \label{chp3_read_from_a_barcode}

\section{Infinite bars and characteristic exponents}
As a prelude, let us consider the case of barcodes that consist of infinite bars only.
Suppose $\calB$ is a barcode that consists of $N$ infinite bars:
$$
	(b_1, +\infty),\ (b_2, +\infty),\ \ldots,\ (b_N, +\infty),\ b_1\leq b_2 \leq \ldots \leq b_N \;.
$$

Let us note that infinite bars cannot be discarded when computing $d_{bot} (\calB, \calC)$ between $\calB$ and another barcode $\calC$, in the sense that while searching for a $\delta$-matching we can only ignore bars of length less than $2\delta$, that is, all infinite bars should appear in the coimage and in the image of the matchings we take into account.
In particular, for a barcode $\calC$, we would have $d_{bot} (\calB , \calC) < +\infty$ if and only if $\calB$ and $\calC$ contain exactly the same amount of infinite bars (cf. \Cref{exr: basic_properties_of_interleaved_pm}.1).

Take indeed another such barcode $\calC$ consisting of the intervals:
$$
	(c_1, +\infty),\ (c_2, +\infty),\ \ldots,\ (c_N, +\infty),\ c_1\leq c_2 \leq \ldots \leq c_N \;.
$$

\noindent
Thus, $d_{bot} (\calB, \calC) \geq \min_{\sigma\in S_N} \max_{i} |b_i - c_{\sigma(i)}|$, where $\sigma$ goes over all possible permutations of $L$ elements. (\Cref{fig: infinite_bars_matching} demonstrates the case $N=1$.)\\
We shall see right away that the right-hand-side in the above inequality can be simplified.

\begin{figure}[!ht]
	\centering
	\includegraphics[scale=1]{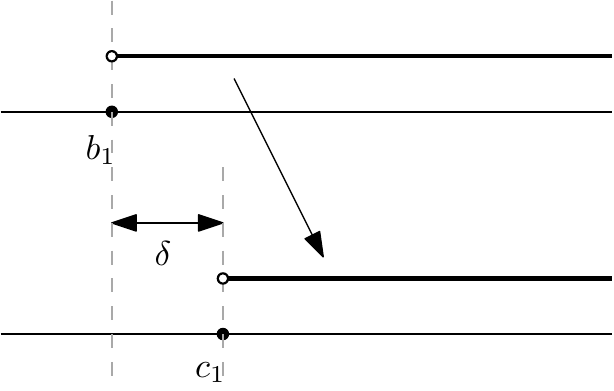}
	\caption{A $\delta$-matching between two infinite bars, $\delta = |b_1 - c_1|$.}
	\label{fig: infinite_bars_matching}
\end{figure}

\begin{lem}[The Matching Lemma] \label{lem: matching_lemma}
	\footnote{
	This claim is well known in the theory of optimal transportation, see e.g. formula (2.3) in \cite{bobkov_ledoux_2014_transport}.
	}
	For any two sets of points in $\R$, $b_1 \leq b_2 \leq \ldots \leq b_N$ and $c_1 \leq c_2 \leq \ldots \leq~c_N$, we have
	$$
		\min_{\sigma\in S_N} \max_i |b_i - c_{\sigma(i)}| = \max_i |b_i - c_i| \;.
	$$
\end{lem}

\begin{cor}\label{cor: cor_matching_lower_bound_on_botneck_dist}
	Let $V$ and $W$ be two persistence modules with barcodes $\calB = \calB (V), \calC = \calB (W)$, each containing $N$ infinite bars. Denote by $b_i \text{ and } c_i $ the corresponding end-points of the \emph{infinite} bars in $\calB$ and $\calC$ (ordered as in \Cref{lem: matching_lemma}). Then we have the following lower bound on the bottleneck distance between the two barcodes:
	\begin{equation} \label{eq: matching_lemma_statement}
		d_{bot} ( \calB, \calC ) \geq \max_i | b_i - c_i | \;.
	\end{equation}
		
\end{cor}


\begin{proof}[Proof of the Matching Lemma]
	Without loss of generality, assume that $b_1 < \ldots < b_N$ and $c_1 < \ldots < c_N$.
	(Otherwise, the claim follows from this case of pairwise distinct points by continuity of both sides of \Cref{eq: matching_lemma_statement}.)\\
	For $\sigma \in S_N$, put $T(\sigma) = \max_i |b_i - c_{\sigma(i)}|$.
	Let
	
		\[
			\sigma =
			\begin{pmatrix}
			1 & \ldots & N \\
			\sigma(1) & \ldots & \sigma(N)
			\end{pmatrix} \;
		\]
	be a permutation and assume that $\sigma (i+1) < \sigma (i)$.
	We can modify $\sigma$ to be	
		\[
		\begin{pmatrix}
		1		&\ldots&	 i & 			i+1 & 	\ldots &	N \\
		\sigma(1) & \ldots &\sigma(i+1)& \sigma(i)& \ldots & \sigma(N)
		\end{pmatrix} \;.
		\]
	by making a transposition of $\sigma(i)$ and $\sigma(i+1)$. This would be called an \emph{elementary modification}.
	\begin{exr}
		Prove that by a sequence of such elementary modifications, any $\sigma \in S_N$ can be transformed into the identity permutation, $\mathds{1}$.
		(\emph{Hint:} Consider any $\sigma \in S_N$ presented as above in two rows. One can first make the element $1$ of the second row be back at its place by subsequently permuting it with elements adjacent to it from the left. Then similarly "track" $2$ back to its place, and so on. This process terminates.)
	\end{exr}
	
	
	We claim that the identity permutation gives the minimal $T(\sigma)$ among all $\sigma \in S_N$. This general conclusion will follow from the case $N=2$.
	
	\begin{exr} \label{exr: matching_case_Nis2}
		Let $N=2$. Then $S_N$ contains two elements: $\mathds{1}$ and the transposition $(12)$. Note that one can get from $(12)$ to $\mathds{1}$ by an elementary modification.
		Let $b_1 < b_2$ and $c_1 < c_2$ be two pairs of pairwise distinct points. There are three different arrangements of these points on the line with respect to each other.
		Prove that in each case $T(\mathds{1})$ gives the minimum among $T (\sigma)$ (for $T$ as defined above).
		(\emph{Hint:} See \Cref{fig: matching_lemma_for_2}.)
	\end{exr}
	
	\begin{figure}[!ht]
		\centering
		\includegraphics[scale=1]{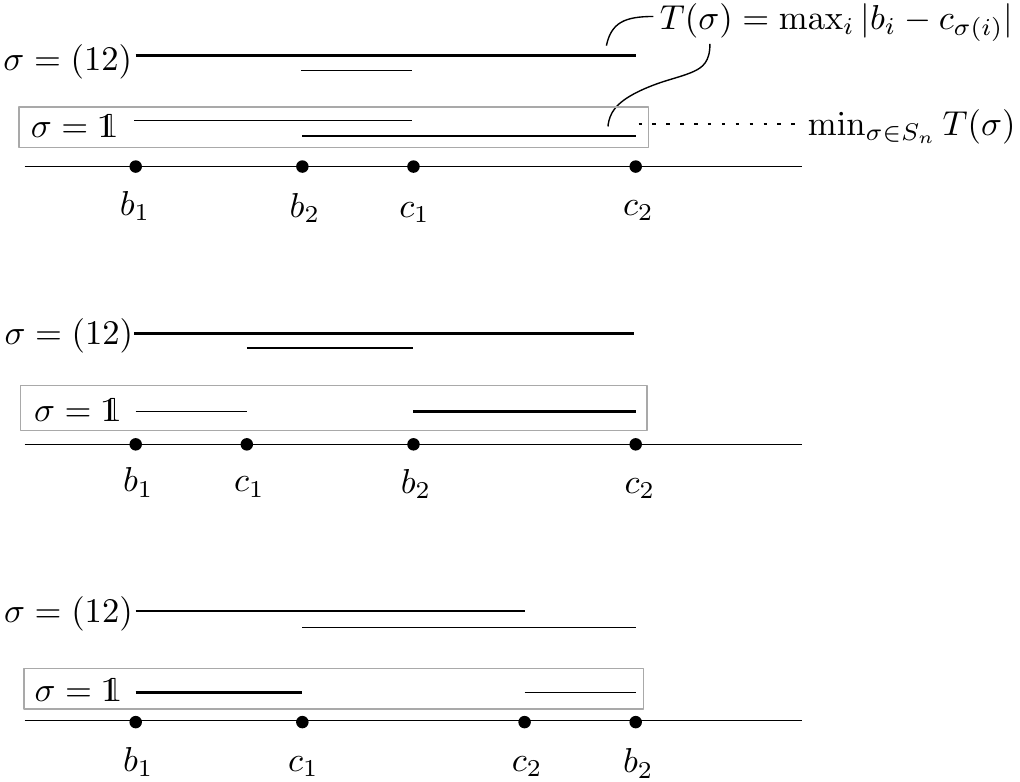}
		\caption{Three possible configurations of two intervals, $T(\sigma)$ is minimal for $\sigma = \mathds{1}$.}
		\label{fig: matching_lemma_for_2}
	\end{figure}
	

	Let us generalize the claim of \Cref{exr: matching_case_Nis2}. Let $\sigma \in S_N$ be a permutation and let $\sigma'$ be an elementary modification of $\sigma$, which switches $\sigma(i)$ with $\sigma(i+1)$. We claim that $T(\sigma') \leq T(\sigma)$. Let us show it. By definition,
	$$
		T(\sigma) = \max_j \big( |b_j - c_{\sigma(j)}| \big) =
		\max \Big( \max_{j\neq i, i+1} \big( |b_j - c_{\sigma(j)} \big),\ |b_i - c_{\sigma(i)}|,\ |b_{i+1} - c_{\sigma(i+1)}| \Big) \;,
	$$
	and similarly
	$$
	T(\sigma') = \max_j \big( |b_j - c_{\sigma'(j)}| \big) =
	\max \Big( \max_{j\neq i, i+1} \big( |b_j - c_{\sigma(j)} \big),\ |b_i - c_{\sigma(i+1)}|,\ |b_{i+1} - c_{\sigma(i)}| \Big) \;.
	$$
	Let us denote
		$A= \max_{j\neq i, i+1} \big( |b_j - c_{\sigma(j)} \big)$ and
		$B(\sigma) = \max \big( |b_i - c_{\sigma(i)}|,\ |b_{i+1} - c_{\sigma(i+1)}| \big)$
	and similarly
		$B(\sigma') = \max \big( |b_i - c_{\sigma(i+1)}|,\ |b_{i+1} - c_{\sigma(i)}| \big)$.
	Note that by \Cref{exr: matching_case_Nis2}, we have $B(\sigma') \leq B(\sigma)$.
	There are two possible cases:
	\begin{itemize}
		\item
			If $T(\sigma) = A$, then $B(\sigma) \leq A$ and since $B(\sigma') \leq B(\sigma)$ we get that $T(\sigma') = A$.
		\item
			If $T(\sigma) = B(\sigma)$, then $A \leq B(\sigma)$ and since $B(\sigma') \leq B(\sigma)$ we get that $T(\sigma') = \max \big( A, B(\sigma') \big) \leq B(\sigma) = T(\sigma)$.
	\end{itemize}
	
	Back to the general statement of the lemma, let $\sigma$ be an optimal permutation, i.e. a permutation with the minimal $T(\sigma)$. We saw that every elementary modification yields a new permutation $\sigma'$ with $T(\sigma')\leq T(\sigma)$.
	But since any $\sigma$ can be transformed into $\mathds{1}$ by a finite number of elementary modifications, we obtain that $T(\mathds{1}) \leq T(\sigma)$, hence by the optimality of $\sigma$, in fact $\mathds{1}$ gives the minimal $T(\sigma) = T(\mathds{1}) = \max_i |b_i - c_i|$.
\end{proof}

\subsection{Characteristic exponents}

	This section is based on Section 2.6.4 of \cite{entov_polterovich_calabi_2003}. The notion of Characteristic exponents was taken from the theory of dynamical systems, see e.g. \cite{sinai_dynamical_1989}.

Let $E$ be a finite dimensional vector space over $\F$ with $\dim E = L$.
\begin{defn}\label{def-charexp}
	A function $c: E \to \R \cup \{-\infty\}$ is called a \emph{characteristic exponent} if
	\begin{enumerate}
		\item
			$c(0) = -\infty$, $c(v)\in \R$ for all $v\neq 0$,
		\item
			$c(\lambda v) = c (v)$ for all $\lambda \in \F \setm \{0\}$,
		\item
			$c(v_1 + v_2) \leq \max \{ c(v_1), c(v_2) \}$ for all $v_1, v_2 \in E$.
	\end{enumerate}
\end{defn}

\begin{exr}
	Let $c:E \to \R \cup \{-\infty\}$ be a characteristic exponent. Check that for any $\alpha \in \R$, the set $\{ v : c(v) < \alpha \}$ is a subspace of $E$.
	Deduce that $c$ admits at most $\dim E$ distinct real values.
\end{exr}

Thus, every characteristic exponent corresponds to a \emph{flag} of vector spaces
$$
	\{0\} = E_0 \subsetneq E_1 \subsetneq E_2 \subsetneq \ldots \subsetneq E_k = E \;,
$$
where $\dim E_i = p_i$, and $0= p_0 <p_1 < p_2 <\ldots < p_k = L$,
such that there exist constants $\alpha_1 < \alpha_2 < \ldots < \alpha_k$, such that
$c\restr_{E_i\setm E_{i-1}} = \alpha_i$.

A multi-set that consists of each $\alpha_i$ taken with multiplicity $p_i - p_{i-1}$ will be called the \emph{spectrum of $c$}, denoted by $\spec (c)$.

The relation of this notion to our story is the following construction:
given a persistence module $(V, \pi)$, we can define a map $c: V_{\infty} \to \R$ by
	$c(v) = \inf \{ s:\ v\in {\im}(\pi_{s, \infty})\}$
(where, as usual, $V_\infty := V_t$ for $t \gg 0$).

\begin{exr}
	\begin{enumerate}
		\item
			The function $c$ defined above is a characteristic exponent.
		\item
			The spectrum of $c$ consists of the end-points of the infinite bars in $\calB (V)$ (taken with multiplicity).
	\end{enumerate}
\end{exr}

Finally let us consider the example of the  Morse persistence module $V=V(f)$ associated with a Morse function $f : X \to \R$ on a closed manifold $X$. In this case the terminal vector space $V_\infty$ is just $H_*(M)$. The induced characteristic exponent $c_f \colon H_*(M) \to \R$ is sometimes called a spectral invariant (see \cite{Vit92}, \cite{Sch00} and \cite{Oh05}).
The value $c_f(A)$ for $A \in H_*(M)$ is, intuitively,
the minimal critical value  such that the corresponding sub-level subset contains a (complete) representative of $A$.
The spectrum of $c_f$ consists of the so-called homologically essential critical values of $f$ (see \cite{Pol01}), which are special cases of min-max critical values (see e.g. \cite{Nicolaescu}.)



\section{Boundary depth and approximation}

\begin{defn} \label{defn: boundary_depth}
	Let $\calB$ be a barcode. The length of the longest finite bar in $\calB$ is called the \emph{boundary depth} of $\calB$ and is denoted by $\beta(\calB)$.
	If a barcode consists only of infinite bars, we set $\beta$ to be zero.
\end{defn}

%
%
%

\begin{thm}
	For a barcode $\calB$ write lengths of finite bars in the decreasing order:
	\begin{equation}\label{eq-listbeta}
	\beta_1 \geq \beta_2 \geq \dots\;
	\end{equation} We claim, following Usher and Zhang,
	that the function $\beta_k$ is Lipschitz on the space of barcodes with
	the Lipschitz constant being $2$. Our convention is that if $\calB$ has less than $k$ finite bars, $\beta_k(\calB)=0$.
\end{thm}

\begin{proof}
	Assume that two barcodes
	$\calB$ and $\calC$ are $\delta$-matched.
	It suffices to prove the inequality
	\begin{equation}\label{eq-betak}
	\beta_k(\calB) -\beta_k(\calC) \leq 2\delta\;.
	\end{equation}
	Fix a $\delta$-matching. If $\beta_k(\calB) \leq 2\delta$, inequality \eqref{eq-betak}
	holds trivially. Thus we assume
	\begin{equation}
	\label{eq-betadelta}
	\beta_k(\calB) > 2\delta\;.
	\end{equation}
	Any $\delta$-matching $\mu$ yields, in particular, the following:  after removing from
	both barcodes some bars of length $< 2\delta$, we match the rest so that in particular
	the length difference in each couple is less than $2\delta$. Denote the lengths
	of the matched intervals, in the decreasing order, as
	$$b_1 \geq b_2 \geq \dots \geq b_N \;,$$
	and
	$$c_1 \geq c_2 \geq \dots\geq c_N\;.$$
	By the Matching Lemma \ref{lem: matching_lemma}, thinking of matching the \emph{lengths} rather than the bars themselves, the optimal ``matching" is the monotone one. In particular:
	\begin{equation}\label{eq-bc}
	|b_k-c_k| < 2\delta\;,
	\end{equation}
	since this bound on the difference of lengths is true also for $\mu$, which might not be the optimal ``matching" terms of lengths.
	By \eqref{eq-betadelta}, no bar longer than the $k$-th one in list \eqref{eq-listbeta}
	is removed and hence $b_k = \beta_k(\calB)$. On the other hand, $c_k \leq \beta_k(\calC)$
	(since some bar longer than $c_k$ might have been erased). By \eqref{eq-bc},
	$$\beta_k(\calB)-\beta_k(\calC) \leq b_k - c_k \leq 2\delta\;,$$
	which yields \eqref{eq-betak}.
\end{proof}

\begin{rmk}
The notion of boundary depth was introduced by M. Usher in the context of filtered complexes (see \cite[section 3]{Ush13} for a detailed exposition).
Let us present this framework shortly.

%
%
%
%

\begin{defn}
	An \emph{$\R$-filtered complex} $(C, \de)$ over $\F$ consists of the following data:
	\begin{itemize}
		\item
			A finite dimensional $\F$-vector space $C$ with a linear map $\de : C \to C$, such that $\de ^2 = 0$
		\item
			For all $\lambda \in \R$, a subspace $C^\lambda \subseteq C$, such that
			\begin{enumerate}[1.]
				\item
					$C^\lambda \subseteq C^\mu$ for any $\lambda < \mu$ in $\R$,
				\item
					$\cap_{\lambda \in \R} C^\lambda = \{ 0 \}$, $\cup_{\lambda \in \R} C^\lambda = C$,
				\item
					For any $\lambda \in \R$, $\de C^\lambda \subseteq \cup_{\mu < \lambda} C^\mu $.
					
			\end{enumerate}
	\end{itemize}
\end{defn}

\noindent
Note that since $C$ is finite dimensional, there exist $\lambda_{-} < \lambda_{+}$ in $\R$, such that $C^{\lambda}= 0 $ for any $\lambda \leq \lambda_{-}$ and $C^\lambda = C$ for any $\lambda \geq \lambda_{+}$.


\begin{defn} \label{defn: boundary_depth_Usher}
	The \emph{boundary depth} of a filtered complex $(C, \de)$ is defined to be
	\begin{equation}
	b(C,\de) = \inf \{ \alpha \geq 0 \ |\  \forall \lambda \in \R, \ (\im \de) \cap C^\lambda \subseteq \de(C^{\lambda+\alpha}) \} \;.
	\end{equation}
	%
	In other words, $b (C, \de)$ is the smallest $\alpha \geq 0$ with the property that, whenever we have a boundary $x\in C$, we can find an element whose boundary is $x$ by 'looking up' the filtration no more than $\alpha$.
	Note that trivially $b (C, \de) \leq \lambda_{+} - \lambda_{-}$.
	We can connect this notion to our story by noting that $\{ H_{*} (C^\lambda) \}_\lambda$ is a persistence module.
\end{defn}

\begin{exr} \label{exr: alt_def_boundary_depth} Recalling our definition of boundary depth $\beta$ of a barcode (\Cref{defn: boundary_depth}), show that for a $\Z$-graded $\R$-filtered complex $(C, \de)$,
	$$\beta \Big( \calB \big( \{ H_{*} (C^\lambda) \}_\lambda \big) \Big) = b (C, \de)\;.$$
\end{exr}

\end{rmk}

\begin{exm}[\textbf{Approximating functions on $S^2$}] \label{exm: approx_function_heart_sphere}
Let us consider a Morse function $f: S^2 \to \R$. We want to know how well it can be approximated by a Morse function $g$ on the sphere, which has exactly two critical points, and such that the two functions have the same minimum and maximum. We look for a quantitative comparison.
Here $f$ can be thought of as a height function on the heart-shaped sphere, and $g$ is any Morse function on the sphere that has exactly two critical points (with the same maximum and minimum as $f$).
\Cref{fig: heart_sphere_morse_rsphere} illustrates this setting. (Also, cf. \Cref{sec: morse_approx_question}.)

\begin{figure}[!ht]
	\centering
	\includegraphics[scale=1]{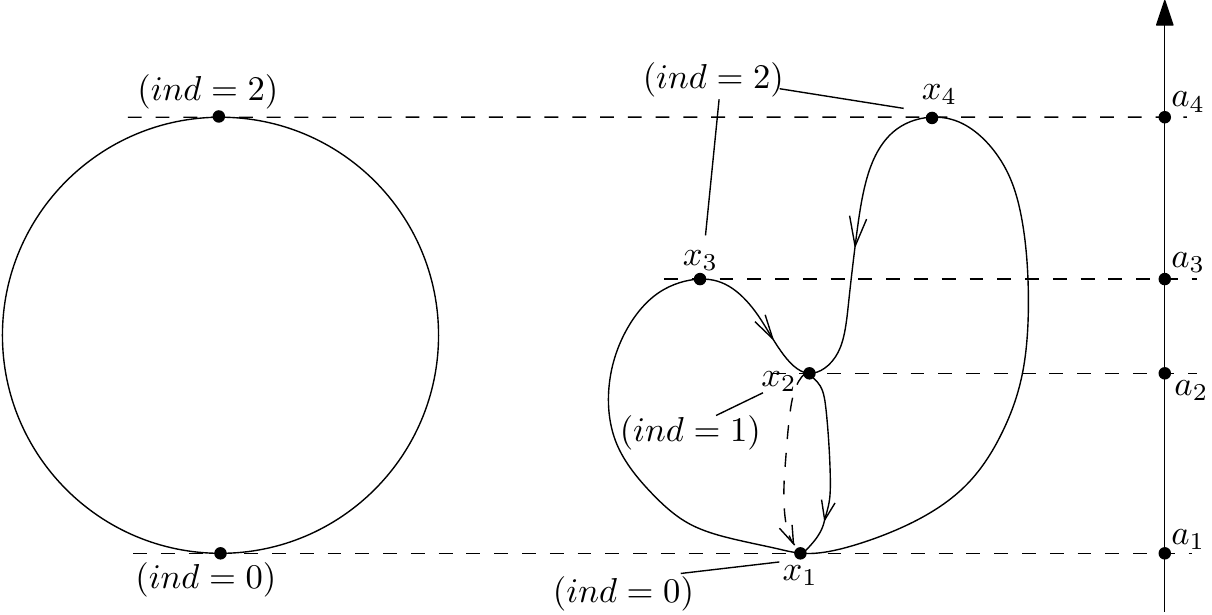}
	\caption{Heart-shaped sphere versus the round sphere - computing Morse homology.}
	\label{fig: heart_sphere_morse_rsphere}
\end{figure}

We consider the persistence modules of the Morse homology with respect to these functions.
In order to quantify how well $g$ can approximate $f$, we examine the corresponding barcodes.
We take the Morse homology with coefficients in $\Z_2$.

\noindent
Let $x_1 \in S^2$ be the minimum point, $x_2$ be a saddle point, $x_3$ be a local maximum, and $x_4$ be a global maximum of $f$.
The Morse indices of the critical points of the heart-shaped sphere are:
$$
	\ind (x_1) = 0,\ \ind (x_2) = 1,\ \ind(x_3)= \ind (x_4) = 2 \;.
$$
Also, we have (modulo 2): $\de x_1 = 0$, $\de x_2 = 2\cdot x_1 = 0$, and $\de x_3 = \de x_4 = x_2$.

Let us compute the Morse homology $H (t)$ of the sublevels $\{ f < t\}$:
\begin{itemize}
	\item
		\underline{For $t>a_4$}:
		$H_2 (t) = \Z_2 \langle x_3 + x_4 \rangle$ (as $x_3 - x_4 \in \ker \de$),
		$H_1 (t) = 0$ (as $x_2$ is a boundary point), and
		$H_0 (t) = \Z_2 \langle x_1 \rangle$.
	\item
		\underline{For $t\in (a_3, a_4)$}:
		$H_2 (t) = 0$ (as $\de x_3 = x_2$ is non-zero),
		$H_1 (t) = 0$ (as $\de x_2 = 0$ and $\de x_3 = x_2$),
		and $H_0 (t) = \Z_2 \langle x_1 \rangle$.
	\item
		\underline{For $t\in (a_2, a_3)$}:
		$H_2 (t) = 0$,
		$H_1 (t) = \Z_2 \langle x_2 \rangle$, and
		$H_0 (t) = \Z_2 \langle x_1 \rangle$.
	\item
		\underline{For $t\in (a_1, a_2)$}:
		$H_2 (t) = H_1 (t) = 0$, $H_0 (t) = \Z_2 \langle x_1 \rangle$.
	\item
		\underline{For $t < a_1$}:
		$H (t) = 0$.
\end{itemize}

\noindent
\Cref{fig: heart_sphere_barcode} presents the corresponding barcode, denoted by $\calB(f)$.
\begin{figure}[!ht]
	\centering
	\includegraphics[scale=1]{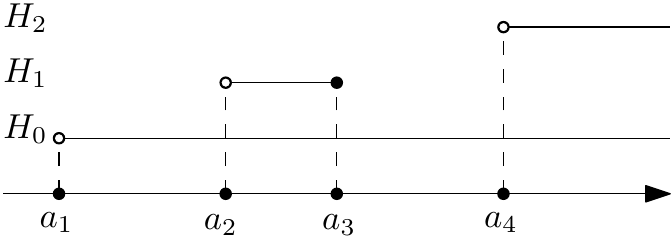}
	\caption{Barcode of the heart-shaped sphere.}
	\label{fig: heart_sphere_barcode}
\end{figure}

Let us remark that, as this example illustrates, the infinite bars correspond to the spectral invariants
$a_1 = c_f ([\text{point}])$ and $a_4 = c_f ([S^2])$ (the minimum and the maximum).
Also, the finite bar has length $a_3 - a_2$. This leads to a solution of our approximation question.


Indeed, let $g: S^2 \to \R$ be a Morse function of $S^2$ that has the same minimum and maximum as $f$. The corresponding barcode $\calB (g)$ has the same two infinite bars but no finite ones, as shown in \Cref{fig: sphere_height_fnc_barcode}. (Here $a_1 = \min g,\ a_4 = \max g$.)

\begin{figure}[!ht]
	\centering
	\includegraphics[scale=1]{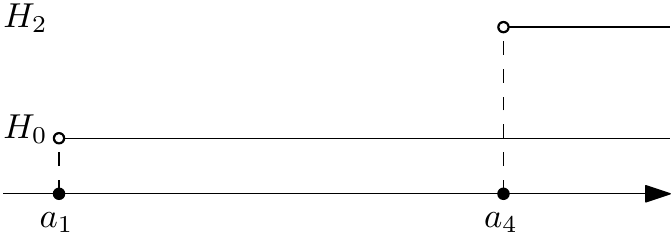}
	\caption{The barcode that corresponds to the round sphere.}
	\label{fig: sphere_height_fnc_barcode}
\end{figure}

By definition, the boundary depth of the heart-shaped sphere is $\beta (\calB (f)) = a_3 - a_2$, while $\beta (\calB (g)) = 0$.
(Note that these values can be obtained also from the alternative description of boundary depth given in \Cref{defn: boundary_depth_Usher} and \Cref{exr: alt_def_boundary_depth}.)

Going back to our approximation question at the end of \Cref{sec: morse_approx_question}, by the Isometry Theorem (\Cref{thm: isometry_thm}), \Cref{cor: cor_matching_lower_bound_on_botneck_dist} and (\ref{eq: interleaving_dist_smaller_than_functions_norm}), we get
\begin{equation}
	a_3 - a_2 \leq 2 d_{bot} \big( \calB(f), \calB(g) \big) = 2 d_{int} \big( V(f), V(g) \big) \leq 2 \| f-g \| \;,
\end{equation}

\noindent
hence $\| f-g \| \geq \frac{1}{2} (a_3 - a_2)$.
This enables us to quantify the obstruction to approximating $f: S^2 \to \R$ by a Morse function with exactly two critical points.

\end{exm}

\begin{exr}
	Find the barcode for the height function on the heart-shaped circle $S^1$.
\end{exr}

\section{The multiplicity function}
\label{sec: the_multiplicity_function_and_Z_k_persistence_modules}


Let $\calB$ be a barcode and $I\subset \R$ be a finite interval. Denote by $m(\calB, I)$ the number of bars in $\calB$ that contain $I$.
For $I = (a,b]$ and $c\leq \frac {b-a}{2}$, denote $I^c = (a+c, b-c]$.

\begin{figure}[!ht]
	\centering
	\includegraphics[scale=1]{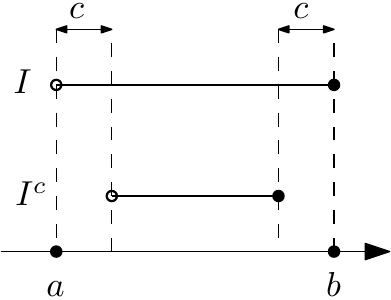}
	\caption{$I^c$: cutting-off $c$ on both sides of $I$.}
	\label{fig: I^c_exm}
\end{figure}

\begin{exr} \label{exr: small_d_bot_then_close_m}
	Assume that two barcodes $\calB$ and $\calC$ satisfy $d_{bot} (\calB, \calC) < 4c$. Assume also for an interval $I$ of length $>4c$ that $m(\calB, I) = m (\calB, I^{2c}) = m_0$.
	Then $m (\calC, I^c) = m_0$.
\end{exr}

\begin{defn}
Define the \emph{multiplicity function} to be
\begin{equation*}
	\mu_k (\calB) = \sup \{ c\ |\ \exists \text{ a finite interval $I$ of length } >4c \text{, s.t. }
								m(\calB, I) = m(\calB, I^{2c}) = k \}
\end{equation*}

In case there is no suitable $c$, we set $\mu_k (\calB) =0$.
\end{defn}

\noindent
In words, given $k \in \N$, the multiplicity function searches for the maximal ``window", an interval $I$ of length $>4c$ in $\R$, such that above it and above the shortened interval $I^{2c}$ there are exactly $k$ bars.
See \Cref{fig: multiplicity_func} for an example.
\begin{figure}[!ht]
	\centering
	\includegraphics[scale=1]{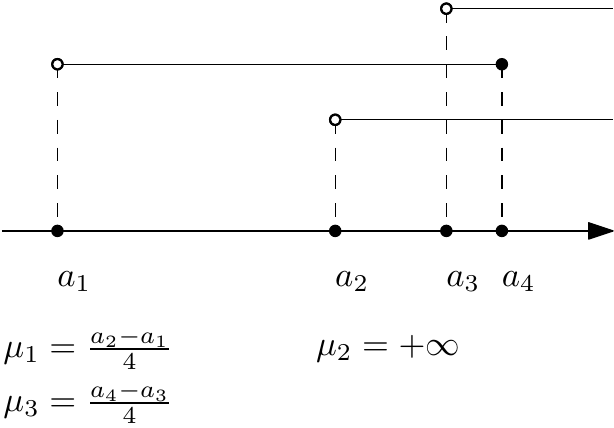}
	\caption{Example of computing the Multiplicity function.}
	\label{fig: multiplicity_func}
\end{figure}

\noindent
By \Cref{exr: small_d_bot_then_close_m} we can deduce that for any two barcodes $\calB$, $\calC$ and any $k\in \N$,
\begin{equation} \label{eq: upper_bound_difference_mu_d_bot}
	| \mu_k (\calB) - \mu_k (\calC) | \leq d_{bot} (\calB, \calC) \;.
\end{equation}

Here we give an application of (\ref{eq: upper_bound_difference_mu_d_bot}), namely, approximation by complex modules.

\begin{defn} \label{defn: complex_pm}
	We say that a persistence module $(V,\pi)$ over $\R$ \emph{admits a complex structure} $J$ if there is a morphism $J : V \to V$ that satisfies $J^2 = -\mathds{1}$.
\end{defn}

\noindent
We will call such a persistence module \emph{complex}.
In such a case, it follows that $\dim V_t$ is even for all $t\in \R$.

\begin{claim} \label{claim: complex_pm_even_mult_func}
	If a persistence module $(V,\pi)$ admits a complex structure, then $m(\calB,I)$ is even for every interval $I$, where $\calB$ is the barcode associated with $V$.
	In particular, it follows that for a complex persistence module $(V,\pi)$, we get that $\mu_k (\calB (V)) = 0$ for all odd $k\in \N$.
\end{claim}

\begin{proof}
Let $I = (a,b]$ and take $a<\tilde{a}<b$ sufficiently close to $a$. (See also \Cref{fig: m(BI)_even_J}.)
Every bar containing $I$ contributes $+1$ to $\dim \im \pi_{\til {a} b}$, i.e.\ $m(\calB, I) = \dim \im \pi_{\tilde{a}, b}$.
But $\pi_{\til{a} b} J_{\til{a}} = J_{b} \pi_{\til{a} b}$, hence $J_b (\im \pi_{\tilde{a}b}) \subseteq \im \pi_{\tilde{a}b}$, i.e.\ $J':=J_b\restr_{\im \pi_{\til{a} b}}$ also satisfies $J'^2 = -\mathds{1}$. As before, the existence of such $J'$ implies that $\dim \im \pi_{\tilde{a}b}$ is even.

\begin{figure}[!ht]
	\centering
	\includegraphics[scale=1]{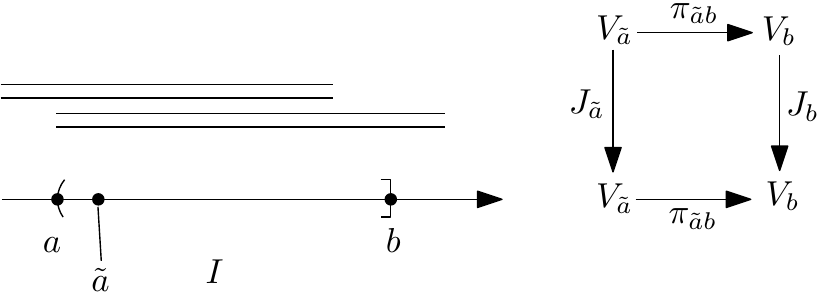}
	\caption{V admitting a complex structure $J$.}
	\label{fig: m(BI)_even_J}
\end{figure}
\end{proof}

\noindent
Denote by $\mu_{odd} (\calB) = \displaystyle{\max_{j \text{ odd }} {\mu_j (\calB)}}$.
See \Cref{fig: multiplicity_func}, where $\mu_{odd} = \max\{\mu_1, \mu_3\}$.

\begin{claim} \label{claim: d_int_to_complex_bounded_below}
	Let $(V, \pi)$ be a persistence module. Then for every persistence module $(W, \theta)$ that admits a complex structure,
	the interleaving distance between them is bounded from below:
	$$
	d_{int} \big( (V, \pi), (W, \theta) \big) \geq \mu_{odd} (\calB (V, \pi) ) \;.
	$$
\end{claim}

\noindent
Thus, the interleaving distance between any persistence module $(V, \pi)$ and the collection of complex modules is bounded from below by $\mu_{odd} (\calB (V, \pi))$.

\begin{proof}

By the Isometry Theorem, (\ref{eq: upper_bound_difference_mu_d_bot}) and \Cref{claim: complex_pm_even_mult_func} , if $(W, \theta)$ is a complex persistence module, we have
\begin{equation}\label{d_int_from_complex_lower_bound}
	\begin{split}
		d_{int} \big( (V, \pi), (W, \theta) \big) &= d_{bot} \big( \calB (V, \pi), \calB (W, \theta) \big) \geq \\
				&\geq
		| \mu_{odd} (\calB (V, \pi)) - \underbrace{\mu_{odd} (\calB(W, \theta))}_{=0} |
		= \mu_{odd} (\calB (V, \pi) )\;.
	\end{split}
\end{equation}
\end{proof}

\noindent
In case $\mu_{odd} (\calB (V, \pi) ) > 0$, we get an constraint to approximating a given persistence module $(V, \pi)$ by a complex persistence module.

\section{Representations on persistence modules}
	\label{sec: persistence_modules_with_involution}
\subsection{Theoretical development}
Recall that a representation of a group $G$ is a pair $(V, \rho)$ where $V$ is a finite-dimensional vector space and $\rho$ is a homomorphism from $G$ to ${\rm GL}(V)$. Here we want to adopt this concept to persistence modules.

\begin{dfn} A \emph{persistence representation of a group $G$} is a pair $((V, \pi), \rho)$ where $(V, \pi)$ is a persistence module and $\rho$ a homomorphism from $G$ to the group of persistence automorphisms of $(V, \pi)$.
A {\it persistence subrepresentation} $((W, \pi), \rho)$ of $((V, \pi), \rho)$ is a persistence submodule $(W, \pi)$ of $(V, \pi)$ such that for any $t \in \R$, $W_t$ is invariant under $\rho(g)_t$ for any $g \in G$.
\end{dfn}

\begin{ex}  \label{defn:pmi} A \emph{persistence module with involution} (abbreviated as \emph{pmi}), denoted by $((V, \pi), A)$,
	is a persistence representation of group $G = \Z_2$, where $A$ is a homomorphism from $G$ to the group of persistence automorphism of $(V, \pi)$ such that for any $t \in \R$, $A_t^2 = \mathds{1}$.
\end{ex}

\begin{dfn} \label{G_mor} Let $((V, \pi), \rho^{V})$ and $((W, \theta), \rho^{W})$ be two persistence representations of a group $G$. A $G$-\emph{persistence morphism} $\f: ((V, \pi), \rho^{V}) \to ((W, \theta), \rho^{W})$ is an $\R$-family of $G$-equivariant persistence morphisms $f_t: V_t \to W_t$, $t \in \R$. \end{dfn}

Given a $G$-persistence morphism $\f: ((V, \pi), \rho^{V}) \to ((W, \theta), \rho^{W})$, one can consider
\begin{equation} \label{ker}
\ker \f = \{ v \in V_t \,| \, f_t(v) = 0 \}_{t \in \R} \,\,\,\,\mbox{and}\,\,\,\,{\im}\f = \{ f_t(v) \in W_t \,| \, v \in V_t \}_{t \in \R}
\end{equation}
\begin{exercise}
Prove $(\ker\f, \rho^{V})$ is a persistence subrepresentation of $((V, \pi), \rho^{V})$ and similarly $({\im}\f, \rho^{W})$ is a persistence subrepresentation of $((W, \theta), \rho^{W})$.
\end{exercise}

\begin{ex} \label{ex-power} Consider a persistence representation $((V, \pi), \rho)$ of $\Z_p$. Let $\xi$ denote a $p$-th root of unity. Consider $(L_{\xi})_t = \ker(\rho(1)_t - \xi \cdot \mathds{1}_{V_t})$ for every $t \in \R$. Then $((\{(L_{\xi})_t\}_{t \in \R}, \pi), \rho)$ is a persistence subrepresentation of $((V, \pi), \rho)$.
\end{ex}

Recall that a $\delta$-shift of a persistence module $(V, \pi)$, denoted by $(V[\delta], \pi[\delta])$, is defined as $V[\delta]_t = V_{t + \delta}$ and $\pi[\delta]_{s,t} = \pi_{s+\delta,t+\delta}$. Also for any persistence morphism $\mathfrak f: (V, \pi) \to (W, \theta)$, its $\delta$-shift $\f[\delta]: (V[\delta], \pi[\delta]) \to (W[\delta], \theta[\delta])$ is defined as $(\f[\delta])_t = f_{t + \delta}$. Observe that if $((V, \pi), \rho)$ is a persistence representation of $G$, then $((V[\delta], \pi[\delta]), \rho[\delta])$ is also a persistence representation of $G$.

\begin{exercise} Consider a persistence representation $((V, \pi), \rho)$ of $G$. Define $\chi_{\delta}: (V, \pi) \to (V[\delta], \pi[\delta])$ by $(\chi_{\delta})_t = \pi_{t, t + \delta}$. Prove $\chi_{\delta}$ is a $G$-persistence morphism. \end{exercise}

\begin{dfn} \label{G-inter} Let $((V, \pi), \rho^{V})$ and $((W, \theta), \rho^{W})$ be persistence representations of group $G$. We call $(V, \pi)$ and $(W, \theta)$ are $(\delta, G)$-\emph{interleaved} if there exist $G$-persistence morphisms $\f: (V, \pi) \to (W[\delta], \theta[\delta])$ and $\g: (W, \theta) \to (V[\delta], \pi[\delta])$ such that the following diagrams commute,
\[ \xymatrixcolsep{5pc} \xymatrix{
(V, \pi) \ar[r]^-{\f} \ar@/_1.5pc/[rr]_{\chi^{V}_{2\delta}} & (W[\delta], \theta[\delta]) \ar[r]^-{\g[\delta]} & (V[2\delta], \pi[2\delta])}
\]
and
\[ \xymatrixcolsep{5pc} \xymatrix{
(W, \theta) \ar[r]^-{\g} \ar@/_1.5pc/[rr]_{\chi^{W}_{2\delta}} & (V[\delta], \pi[\delta]) \ar[r]^-{\f[\delta]} & (W[2\delta], \theta[2\delta])}.
\]
Accordingly, we can define $G$\emph{-interleaving distance}
$$ d_{G-{int}}((V, \pi), (W, \theta)) = \inf \{ \delta>0 \,| \, (V, \pi) \text{ and } (W, \theta) \text{ are } (\delta, G)\text{-interleaved} \}. $$
\end{dfn}

The following proposition is obvious from Definition \ref{G-inter}.

\begin{prop} \label{prop-G-int}
Let $((V, \pi), \rho^{V})$ and $((W, \theta), \rho^{W})$ be persistence representations of group $G$. Then
\[ d_{G-{int}}((V, \pi), (W, \theta)) \geq d_{int}((V, \pi), (W, \theta)).\]
\end{prop}

\begin{ex} \label{ex-4.4.9} Suppose two persistence representations of $G$, $((V, \pi), \rho^{V})$ and $((W, \theta), \rho^{W})$, are $(\delta, G)$-interleaved. Let us consider the persistence subrepresentation $((V', \pi), \rho^{V})$ of $((V, \pi),\rho^{V})$ and persistence subrepresentation $((W', \theta), \rho^{W})$ of $((W, \theta),\rho^{W})$. It is easy check $((V', \pi), \rho^{V})$ and $((W', \theta), \rho^{W})$ are also $(\delta, G)$-interleaved. Then one gets
\begin{align*}
d_{G-{int}}((V, \pi), (W, \theta)) & \geq d_{G-{int}}((V', \pi), (W', \theta)) \\
& \geq d_{int}((V', \pi),(W', \theta)) = d_{bot}((V', \pi), (W', \theta)).
\end{align*}
\end{ex}

In this section, we will not deal with the general representations of $G = \Z_p$, but only with $p=2$ and $p=4$. Note that if $\Z_4$ acts on a set, this action induces a $\Z_2$-action on the same set, by the correspondence $\Z_2 \to \Z_4$, $1 \mapsto 2$.
\noindent
We say that a pmi $((W, \theta), B)$ is a \emph{$\Z_4$-pmi} if its $\Z_2$-action $B$ comes from a $\Z_4$-action, i.e.\ if there exists a persistence morphism $C: (W, \theta) \to (W, \theta)$, such that $B=C^2$ and $C^4 = \mathds{1}$.

Let $((V, \pi), A)$ be a pmi. In Example \ref{ex-power} where $\xi = -1$, denote by $L^V$ the resulting persistence module constructed from $(-1)$-eigenspaces.

\begin{quest}
	How well can an arbitrary pmi be approximated by a $\Z_4$-pmi with respect to the $\Z_2$-interleaving distance?
\end{quest}

\begin{thm} \label{thm: d_int_pmi_Z4_pmi_lower_bound}
	Let $((V, \pi), A)$ be a pmi. The $\Z_2$-interleaving distance between $V$ and the collection of persistence modules with involution, whose $\Z_2$-action comes from a $\Z_4$-action, is bounded from below in terms of the multiplicity function: for any $\Z_4$-pmi $((W, \theta), B)$,
	$$
	d_{\Z_2-int} \big( V, W \big) \geq \mu_{odd} (L^V) \;.
	$$
\end{thm}

\begin{proof}
Our approach is as follows.

\begin{exr}
\begin{enumerate}[1.]
	\item
		Prove that if $(V, \pi)$ and $(W, \theta)$ are $(\delta, \Z_2)$-interleaved, then $L^V$ and $L^W$ are $\delta$-interleaved.
	\item
		Prove that $C (L^W) = L^W$, and deduce that $C^2 \restr_{L^W} = -\mathds{1}$.
\end{enumerate}
\end{exr}

\noindent
It follows that $L^W$ is a complex persistence module (see \Cref{defn: complex_pm}).
Hence, by \Cref{claim: d_int_to_complex_bounded_below},
\begin{equation}\label{d_int_from_Z4_lower_bound}
 d_{\Z_2-int} \big( (V,\pi), (W, \theta) \big) \geq
	d_{int} (L^V, L^W) \geq \mu_{odd} (L^V).
\end{equation}
\end{proof}
\subsection{Applications in geometry}

\begin{exm}
	Let $(X,\rho)$ be a finite metric space equipped with an isometry $A:X\to X$ which is an involution, i.e. $A^2 = \id$. Thus, $\Z_2$ acts on $X$ via $A$. Let $(Y,r)$ be another finite metric space endowed with a $\Z_4$-action, which induces a $\Z_2$-action on $Y$, we denote by $B$ this induced $\Z_2$-action. We wish to consider the Gromov-Hausdorff distance (defined below) between $(X,\rho,A)$ and $(Y,r,B)$.
	Let $C:X \rightrightarrows Y$ be a surjective correspondence, and $\dist (C)$ denote its distortion (see \Cref{defn: distortion_of_surj_correspondence}).
	A correspondence $C \subset X \times Y$ is said to be $\Z_2$-equivariant if
	$(x,y) \in C$ implies $(A(x), B(y)) \in C$. This definition is an analogue of the requirement of the following diagram to commute in case $C$ is a surjective function $f: X\to Y$:
	\begin{center}
		\begin{tikzpicture}[scale=1.5]
		\node (X) at (0,0) {$X$};
		\node (X') at (0,-1) {$X$};
		\node (Y) at (1,0) {$Y$};
		\node (Y') at (1,-1) {$Y$};
		\path[->,font=\scriptsize,>=angle 90]
		%
		(X) edge[->] node[left]{$A$} (X')
		(Y) edge[->] node[right]{$B$} (Y')
		%
		%
		%
		(X) edge[->] node[above]{$f$} (Y)
		(X') edge[->] node[below] {$f$} (Y');
		\end{tikzpicture}
	\end{center}

\end{exm}	
\noindent
	The \emph{Gromov-Hausdorff distance} between $X$ and $Y$ is defined to be the infimum $\inf_C \dist (C)$ over all surjective $\Z_2$-equivariant correspondences $C: X \rightrightarrows Y$. One would like to find some lower bound for this distance in the case described above of metric spaces with involution, where one of the involutions comes from a $\Z_4$-action.
	One can check that \Cref{thm: GH_dist_bounded_below_by_interleaving_distance} holds also when replacing the notions involved by $\Z_2$-equivariant ones. Hence considering the persistence modules associated to the Rips complexes of $X$ and $Y$, denote them by $V(X)$ and $V(Y)$, we would get
	$$
		d_{\Z_2 - GH} (X,Y) \geq \frac{1}{2} {d}_{\Z_2-int} (V(X), V(Y)) \;.
	$$
	See also \Cref{exms: lower_bound_distance_Z_2_Z_4} below.\\

We illustrate the given lower bound from \Cref{thm: d_int_pmi_Z4_pmi_lower_bound} using two examples:
\begin{exms}\label{exms: lower_bound_distance_Z_2_Z_4}
\begin{enumerate}[I.]

\item
	Let $X$ be the set of vertexes $\{ x_1, y_2, x_2, y_1 \}$ of a rectangle in the plane listed in cyclic order, of sides length $\ol{x_1 y_1} = 1$ and $\ol{x_1 y_2} = a \geq 1$ (with respect to the Euclidean metric), see \Cref{fig: rectangle_1a_Rips}.
	\begin{figure}[!ht]
		\centering
		\includegraphics[scale=1]{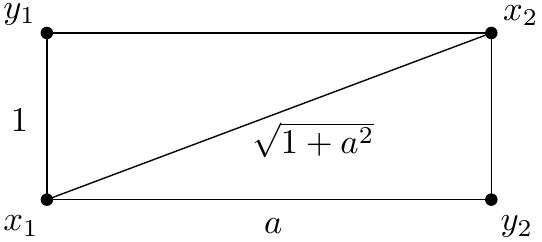}
		\caption{Example I. -- A rectangle with sides $1$ and $a$.}
		\label{fig: rectangle_1a_Rips}
	\end{figure}

	\begin{figure}[!ht]
		\centering
		\includegraphics[scale=1]{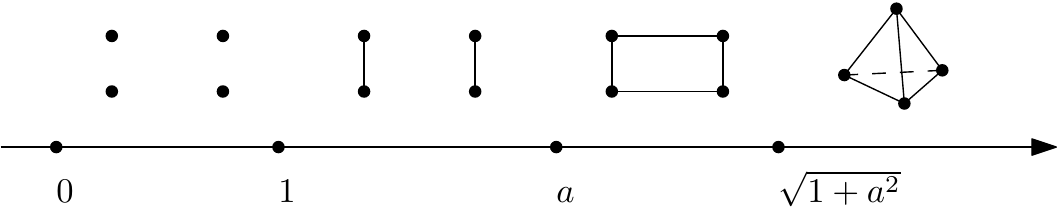}
		\caption{Example I. -- The Rips complex (drawn for the case $a > 1$).}
		\label{fig: rectangle_Rips_complex}
	\end{figure}
	
	Let us look at the $0$-homology of its Rips complex, $V = \{ H_0 (R_t(X), \R) \}_t$, which is a persistence module. See \Cref{fig: rectangle_Rips_complex} for an illustration of the Rips complex that corresponds to different $t\geq 0$, and \Cref{fig: rectangle_H0_Rips_barcode}
	the corresponding barcode.
	\begin{figure}[!ht]
		\centering
		\includegraphics[scale=1]{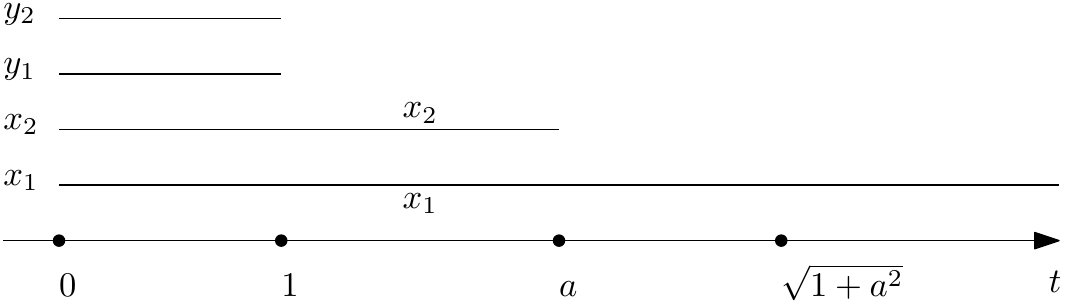}
		\caption{Example I. -- The barcode of $H_0(R_t(X))$ (for $a>1$).}
		\label{fig: rectangle_H0_Rips_barcode}
	\end{figure}
	
	Equip $V$ with an involution $A$, acting by exchanging vertexes that share a diagonal: $A(x_1) = x_2$ and $A(y_1) = y_2$.
	This defines a $\Z_2$-action on $V$.
	
	We want to estimate the $\Z_2$-interleaving distance between $V$ and the space of persistence modules with involution that comes from a $\Z_4$-action ($\Z_4$-pmi).
	
	As described in the previous section, consider the $(-1)$-eigenspace $L_t^V$ with respect to the action $A$:
	\begin{equation*}
		L^V_t =
		\begin{cases}
			\langle x_1 -x_2, y_1 - y_2 \rangle & \text{ if } t\in (0,1], \\
			\langle x_1 - x_2 \rangle & \text{ if } t\in (1,a], \\
			0 \ & \text{otherwise.}
		\end{cases}
	\end{equation*}
	
	\noindent
	The barcode of $L_t^V$ is illustrated in \Cref{fig: rectangle_LtV}.
	\begin{figure}[!ht]
		\centering
		\includegraphics[scale=1]{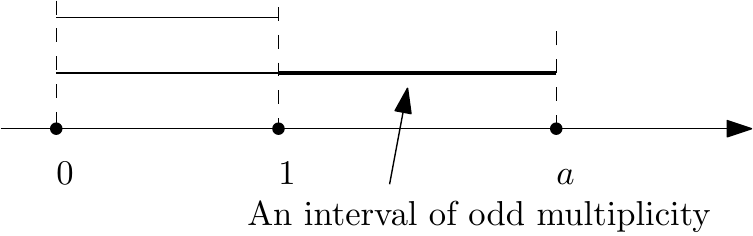}
		\caption{Example I. -- Barcode of the $(-1)$-eigenspace $L^V$.}
		\label{fig: rectangle_LtV}
	\end{figure}
	
	\noindent
	Thus, by \Cref{thm: d_int_pmi_Z4_pmi_lower_bound}, we get that
	$
		d_{\Z_2-int} \big( V, \Z_4 \text{-pmi} \big) \geq \mu_{odd} (L^V)= \frac{a-1}{4} \;.
	$
	
	In particular, for $a > 1$, the $\Z_2$-action taken in this example is not coming from a $Z_4$-action, as follows from our bound $d_{\Z_2-int} \big( V, \Z_4 \text{-pmi} \big) \geq \frac{a-1}{4} > 0$.
	On the other hand, for $a=1$, the action $A$ does come from a $\Z_4$-action (of rotating by $90^0$), and indeed we do not obtain any positive lower bound on ${d}_{\Z_2-int} (V, \Z_4 \text{-pmi}) = 0$.
	
\item \textbf{Morse-theoretical counterpart.}
	Consider a $\Z_4$-action on a smooth manifold $M$, with generator $\tau$, and denote its square by $\theta = \tau^2$.
	Let $F$ be a $\theta$-invariant Morse function on $M$.
	We want to minimize $\| F - \vphi^* G \|$,
	where $G:M \to \R$ is a Morse function invariant under the $\Z_4$ action of $\tau$ ($\tau^* G = G$), and $\vphi \in \Diff (M)$ is a diffeomorphism that satisfies $\vphi \theta = \theta \vphi$.

	Our approach is similar.
	Let us consider the $\Z_2$-persistence module $V = \left( \{H_0 (F < \alpha)\}_t, \theta_* \right)$ (taking the homology with respect to the level sets of $F$, recall \Cref{exm: morse_theory_pm}).
	Look again at the eigenspace $L^V$ that corresponds to $-1$.
	By \Cref{thm: d_int_pmi_Z4_pmi_lower_bound} and the lower bound given in (\ref{eq: interleaving_dist_smaller_than_functions_norm}),
		$$
		\| F - \vphi^* G \| \geq \mu_{odd} (L^V) \;.
		$$
	
	Let us take a concrete example. Let $M = S^2$ be a sphere and $F: S^2 \to \R$ be some Morse function (we consider the unit sphere around the origin in $\R^3$ with coordinates $(x,y,z)$).
	Suppose that $F$ has three critical values: the maximum, achieved at the north pole $(0,0,1)$, the minimum, achieved at two antipodal points on the equator, and a saddle point in the south pole. (See \Cref{fig: teeth_sphere_image} and \Cref{fig: teeth_sphere_barcode}.)
	Let $\tau$ be the rotation by $\frac{\pi}{2}$ around the $z$-axis.
	Then $\theta = \tau^2$ is the rotation by $\pi$ around the $z$-axis. Let us assume that $f$ is $\theta$-invariant.
	
	\begin{figure}[!ht]
		\centering
		\includegraphics[scale=1]{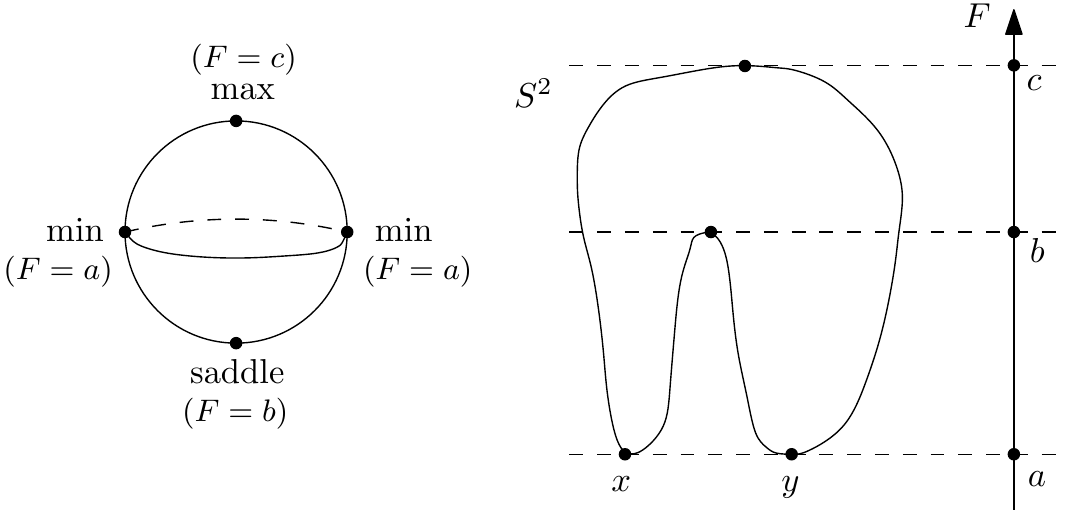}
		\caption{Example II. A Morse function on $S^2$ invariant under rotation by $\pi$.}
		\label{fig: teeth_sphere_image}
	\end{figure}
	\begin{figure}[!ht]
		\centering
		\includegraphics[scale=1]{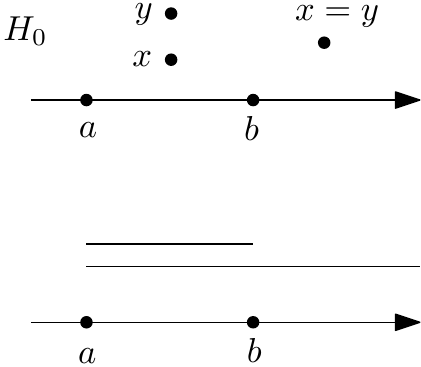}
		\caption{Example II. The associated barcode.}
		\label{fig: teeth_sphere_barcode}
	\end{figure}

	In this case, $L^V$ is $\spn (x-y)$ on $(a,b]$ and $0$ otherwise. So
	$\mu_{odd} = b-a$, and the above quantity $\|F - \vphi^* G \|$ is bounded here by $\frac{b-a}{4}$.
	
	Notice that similarly to the first example, when $b > a$ the action of $\theta$ does not come from a $\Z_4$-action, and we are able to distinguish it from the set of $\Z_4$-pmi.
	
\end{enumerate}
\end{exms}

\part{Applications to  metric geometry and function theory}\label{part-II}

\chapter{Applications of Rips complexes} \label{chap5-rips}

\section{\texorpdfstring{$\delta$}{delta} - hyperbolic spaces} \label{section: hyperbolic_spaces}
In this section, we follow the book \cite{Bridson_Haefliger_2011} by Bridson and Haefliger.

\begin{defn}
	Let $(Y,d)$ be a \emph{geodesic metric space} (i.e.\, any two points can be joined by a geodesic, which might not be unique).
	$Y$ is called \emph{$\delta$-hyperbolic} (with $\delta > 0$) if for any geodesic triangle, each of its sides lies in the $\delta$-neighborhood of the union of the other two sides (such a triangle is called \emph{$\delta$-slim}). See \Cref{fig: hyperbolic_triangle}.
	We say that $Y$ is \emph{hyperbolic} if it is $\delta$-hyperbolic for some $\delta > 0$.
\end{defn}

	\begin{figure}[!ht]
		\centering
		\includegraphics[scale=0.8]{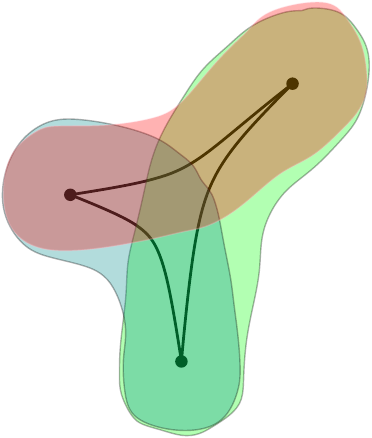}
		\caption{Hyperbolicity condition}
		\label{fig: hyperbolic_triangle}
	\end{figure}

\begin{exms}
	\begin{enumerate}
		\item
		Any space of bounded diameter is trivially hyperbolic.
		\item
		The Euclidean plane is \emph{not} hyperbolic, as for any $\delta>0$ a big enough equilateral triangle will not be $\delta$-slim.
		\item
		Consider the hyperbolic plane $\mathbb{H}$, with constant negative curvature $-1$. We claim that it is $\delta$-hyperbolic for some $\delta$.
		
		Take a triangle in $\mathbb{H}$. Suppose that it is not $r$-hyperbolic. Then there exists a point $p$ on one of its sides, such that a ball of radius $r$ around $p$, denoted by $B_r (p)$, does not intersect the other two sides. Hence half of the area of $B_r (p)$ is bounded from above by the area of the triangle (as half of this ball is contained in the triangle). Therefore, we get the estimate
		$
			2 \pi \sinh ^2 \frac{r}{2} \leq (\pi - \alpha - \beta - \gamma) \leq \pi \;,
		$
		where $\alpha, \beta, \gamma$ are the angles of the triangle, with maximal area on the right-hand-side attained on \emph{ideal} triangles - those that have all angles equal $0$.
		This leads to the following estimate: $r \leq 2\text{arsinh} \frac{1}{\sqrt{2}} = \ln (2+\sqrt{3}) \approx 1.31696$.
		
	\end{enumerate}
\end{exms}

Recall that for a metric space $(X,d)$, a subset $A \subset X$ is called \emph{$r$-dense} if for any point $x\in X$ there is a point $a\in A$ such that $d(x,a)<r$.

\begin{thm}[Following \cite{Bridson_Haefliger_2011}, chapter III.H] \label{thm: hypergeo_BH_Rt_contraction}
	Let $Y$ be a $\delta$-hyperbolic metric space and let $X \subseteq Y$ be a finite $r$-dense subset.
	Then for any $t> 4\delta + 6r$, every subcomplex of $R_t (X)$ contracts to a point in $R_t (X)$.
\end{thm}

Aside from getting information on when $R_t(X)$ is already contractible (given $\delta$ and $r$), one can adopt the following viewpoint.
Given a $\delta$-hyperbolic metric space $Y$, in case $\delta$ is unknown to us, \Cref{thm: hypergeo_BH_Rt_contraction} suggests a way to get a lower bound on $\delta$.
Namely, if we have an $r$-dense subset $X \subseteq Y$ for which $R_t (X)$ is not contractible, then $\delta \geq \frac{1}{4} (t - 6r)$.
Contractibility could be easier to check rather than finding the value of $\delta$ (or a lower bound) directly.

Let us comment on the connection to our story and give a few examples.
\begin{rmk}
		Let $(Y,d)$ be a locally compact uniquely geodesic $\delta$-hyperbolic manifold.
		Let $\hat{Y} \subset Y$ be a geodesically convex compact subset and let $X \subset \hat{Y}$ be a finite and $r$-dense in $\hat{Y}$.
		The proof \Cref{thm: hypergeo_BH_Rt_contraction} goes through in this case and we get that the boundary depth of the corresponding Rips complex satisfies
		$\beta \big( R (X) \big) \leq 6r + 4\delta$ (the longest finite bar is necessarily contained in $(0, 6r+4\delta]$).
		In particular, taking the density of $X$ to be $r = \delta$ we get $\beta \big( R(X) \big) \leq 10 \delta$, which provides a link between boundary depth and $\delta$-hyperbolic geometry.

\end{rmk}

We would like to introduce the notion of hyperbolic groups (see \cite{Bridson_Haefliger_2011} and \cite{Gromov1987hyperbolic}).
Let $\Gamma$ be a finitely generated group and let $S$ be some (finite) generating set of $\Gamma$, which we always assume to be symmetric, i.e. if an element belongs to $S$, so does its inverse.
The \emph{Cayley graph} $G = G(\Gamma, S) = (V,E)$ of $\Gamma$ (in the generality we will use) is given by the following data:
		\begin{itemize}
			\item
			The set of vertices is $V = \Gamma$.
			\item
			For each $x\in \Gamma$, $s\in S$, the vertices $x$ and $xs$ are joined by an edge, i.e.\ the set of edges is $E = \{ (x,xs):\ x\in \Gamma, s\in S \}$.
		\end{itemize}
		Given the pair $(\Gamma, S)$, we take the \emph{word metric} on the group $\Gamma$, defined as follows:
		the distance between two elements $g,h\in \Gamma$ is the least number of elements of $S$ required to write a word whose evaluation is $g^{-1} h$ (being a product in the order written).
		This metric corresponds to a metric $d$ on the Cayley graph of $\Gamma$: the distance between two vertices in $V$ is the length of the shortest path joining them in $G$.
		Let us illustrate this notion.
		\begin{exms}
			\begin{itemize}
				\item
				Take $\Gamma = F _k$ ($k\geq 2$) to be the free group on a set of $k$ elements. Let us take a generating set $S=\{s_1, s_1^{-1}, \ldots, s_k, s_k^{-1}\}$. Then the corresponding Cayley graph is a tree whose root is $\mathds{1}$, with one branch going out of the root per each element of $S$ and any other vertex also has degree $2k$.
				
				\item
				Another example is $\Gamma = \pi_1 (\Sigma_g)$, where $\Sigma_g$ is a surface of genus $g \geq 2$. Then we can take the set of generators to be a set of $g$ elements and their inverses. But this time the group is not free:
				$$
					\Gamma = \{ a_1, b_1, \ldots, a_g, b_g \ : \
					[a_1, b_1]\cdots [a_g, b_g] = \mathds{1} \} \;.
				$$
				For the very special case of a torus ($g=1$), the Cayley graph of $\Gamma$ is the $\Z^2$ grid.
				For both example, see e.g. \cite{hatcher_algebraic_top}, section 1.3.
			\end{itemize}
		\end{exms}
		
		Further, we can consider a ``topological realization" $Y$ of the Cayley graph $G$ of $\Gamma$, which for us means endowing the edges of the graph with a metric that extends $d$, requiring each edge to be of length $1$ (see e.g. \cite{drutu_kapovich_13}, section 1.3.4).
		
		
		Having the metric space $(Y,d)$, we can ask if it is $\delta$-hyperbolic for some $\delta>0$.
		If that is the case, we say that $\Gamma$ is a \emph{hyperbolic group}.
		
		\begin{exm}
			Let $\Gamma = F_k$ be the free group on $k$ elements, and let $S = \{ s_1, s_1^{-1}, \ldots,   s_k, s_k^{-1}\}$ be a generating set. The corresponding Cayley graph is a tree, and having no triangles, it is $\delta$-hyperbolic for any $\delta >0$, thus the free group is hyperbolic.
		\end{exm}
		
		
		Assume the group $\Gamma$ is torsion free, i.e. for any element $1 \neq g\in G$, for all $n\in \N$, $g^n \neq 1$. Fix some finite generating set $S$.
		Consider the topological realization $(Y,d)$ of the Cayley graph $G = G(\Gamma, S)$, assuming that it is $\delta$-hyperbolic for some $\delta > 0$, i.e. the group $\Gamma$ is hyperbolic.
		Set $X = \Gamma$ to be the collection of vertexes of the (combinatorial) graph $G$.
		Note that $X$ is $1$-dense in its geometric realization $(Y,d)$, so by the same arguments as in the proof of \Cref{thm: hypergeo_BH_Rt_contraction} below, for any $t > 6+4\delta$, the complex $R_t(X)$ is contractible.
	\begin{rmks}
		\begin{enumerate}
		\item
			By definition, $\Gamma$ acts transitively on $X$: for any $x,y\in X$ there exists $g\in \Gamma$, such that $gx = y$, where $g = yx^{-1}$.
			Also, $\Gamma$ acts freely on $X$: given $g,h\in \Gamma$, if $gx=hx$ for some $x$, then indeed $g=h$, just multiplying by $x^{-1}$.
		\item
			Moreover, since $\Gamma$ has no torsion, it also acts freely on simplices in $R_t (X)$. Indeed, otherwise there would exist a simplex $\sigma$ and $g\neq 1$ in $\Gamma$ with $g\sigma = \sigma$. But then after a finite amount of iterations, also some vertex would be fixed by $g^n \neq 1$ for some $n \in \N$.
		\item
			Thus, we can consider $K = R_t (X) / \Gamma$. For $t> 6+4\delta$, as mentioned above, the space $R_t (X)$ is contractible, hence simply connected, so $R_t (X)$ is the universal cover of $K$, with $\pi_1 (K) = \Gamma$ and $\pi_n (K) = 1$ for all $n\neq 1$. Such a space is called an Eilenberg-MacLane space $K (\Gamma, 1)$.
		\item
			Note that $K$ is a finite complex, since a ball of radius $t$ near $\mathds{1} \in X$ consists only of a finite number of points, $\Gamma$ being finitely generated (by transitivity, it is enough to examine $\mathds{1}$).
			Hence the $2$-skeleton on $K$ is finite, so $\Gamma$ is finitely presented (it can be defined via a finite number of relations between the generators, the number of $2$-simplexes being an upper bound for the number of relations).	
		\end{enumerate}
	\end{rmks}

Let us complete this section by proving the announced theorem.

\begin{proof}[Proof of \Cref{thm: hypergeo_BH_Rt_contraction}]
	Fix a base point $x_0 \in X$, some $t > 4\delta + 6r$, and a subcomplex $L \subseteq R_t (X)$. Consider the following two cases:
	\begin{itemize}
		\item
			Assume $d (x_0, v) < \frac{t}{2}$ for all $v\in L$.
			Then $d (v_1, v_2) < t$ for any pair $v_1, v_2 \in L$, so $L$ is contained in a full simplex in $R_t (X)$, hence is contractible.
		\item
			Suppose that there exists $v\in L$ such that $d (x_0, v) \geq \frac{t}{2}$, and fix $v$ to be a vertex for which $d(x_0, v)$ is maximal. Our idea is to gradually homotope $L$ inside $R_t (X)$ to arrive at the first case.
		
			Draw a geodesic $[x_0, v]$ between $x_0$ and $v$. Take $y\in [x_0, v]$ to be the point for which $d (y,v) = \frac{t}{2}$, and choose $v' \in X$ such that $d (v',y) \leq r$.
			Put $\rho = d (v,v')$. (See \Cref{fig: hyperbolic_retract_trick}.)
			
			\begin{figure}[!ht]
				\centering
				\includegraphics[scale=1]{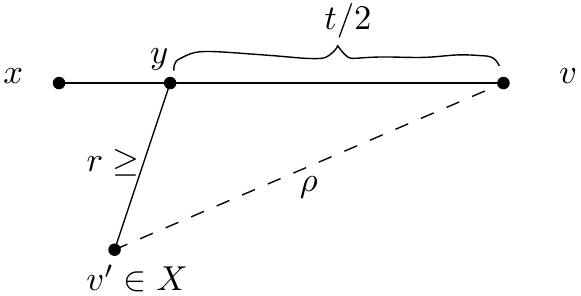}
				\caption{Taking an alternative point $v'$.}
				\label{fig: hyperbolic_retract_trick}
			\end{figure}
			
			Note that by the triangle inequality (for $\Delta yvv'$),
			$$
				\rho \leq \frac{t}{2} + r \text{, and } \rho \geq \frac{t}{2} - r > 2\delta + 3r-r = 2\delta + 2r \;,
			$$
			so we get in particular that
			\begin{equation} \label{ineq: bounds_on_rho}
				\rho > \frac{t}{2} - r > 2\delta + 2r, \text{ and } \rho < t \;.
			\end{equation}
			
			\begin{lem}\label{lem: hyper_geo_Rt_homotop}
				For any $u\in L$, if $d (u,v) < t$ then $d(u,v') < t$.
				(Under the assumptions of \Cref{thm: hypergeo_BH_Rt_contraction}.)
			\end{lem}
			
			Using the lemma (see a proof later), note that if
			$\sigma = [v, u_1, \ldots, u_k]$ is a simplex in $L \subset R_t (X)$, and
			$\sigma' = [v', u_1, \ldots, u_k]$ is a simplex in $R_t (X)$, then since $d (v,v') = \rho < t$,
			we get that
			$\Sigma = [v, v', u_1, \ldots, u_k ]$ is also a simplex in $R_t (X)$.
			
			\begin{figure}[!ht]
				\centering
				\includegraphics[scale=1]{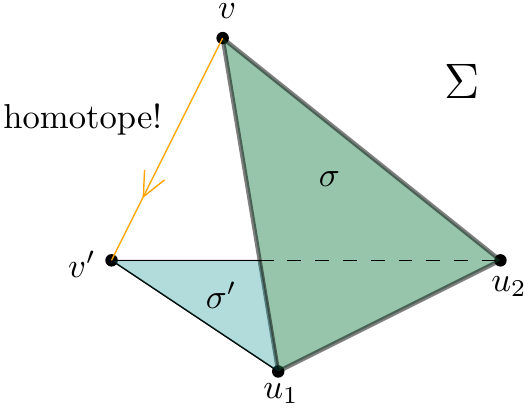}
				\caption{Homotope $\sigma$ to $\sigma '$.}
				\label{fig: hyperbolic_simplex_homotopy}
			\end{figure}
			
			Denote by $L' \subset R_t (X)$ the subcomplex that is obtained from $L$ by replacing the vertex $v$ with $v'$.
			We can homotope $L$ to $L'$ inside $\Sigma$ (bringing $v$ to $v'$, thus taking each $\sigma$ to $\sigma'$ as above), and keeping fixed all faces in $L$ that do not contain $v$.

			Note that by the triangle inequality, and the definition of $t$,
				$$
					d(x_0, v') \leq d(x_0, y) + d(y,v') \leq
					 d (x_0, v) - \frac{t}{2} + r < d(x_0, v) - (2\delta + 2r) < d (x_0,v) \;.
				$$
			Thus, in a finite number of steps, replacing $v$ that gives maximal $d(x_0,v)$ by $v'$ as described (which reduces $d(x_0, v)$ by at least $(2\delta + 2r)$), will lead us to the situation of Case 1.
	\end{itemize}
\end{proof}

\begin{proof}[Proof of \Cref{lem: hyper_geo_Rt_homotop}]
	Assume that $u \in L$ satisfies $d (u,v) < t$. We have to prove show that $d (u,v') < t$.
	
	Consider the geodesic triangle with vertices $x_0, u, v$ and recall our construction:
	$y\in [x_0, v]$ is such a point that $d (y,v) = \frac{t}{2}$ and $v'\in X$ is chosen to satisfy $d (y, v') \leq r$.
	By $\delta$-hyperbolicity, $y$ is included either in a $\delta$-neighborhood of $[x_0, u]$ or in that of $[v,u]$. That is, either there exists $w_1 \in [x_0, u]$ for which $d(y,w_1)<\delta$ or there exists $w_2 \in [v,u]$ for which $d (y,w_2) < \delta$. (See \Cref{fig: hyperbolic_lemma}.)

	\begin{figure}[!ht]
		\centering
		\includegraphics[scale=1]{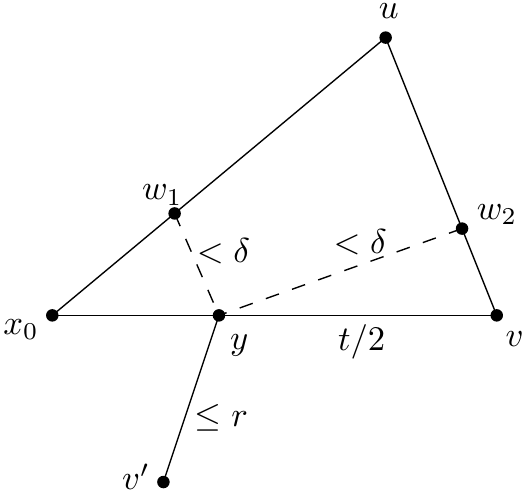}
		\caption{If $d (u, v) < t$, then $d(u, v') < t$.}
		\label{fig: hyperbolic_lemma}
	\end{figure}

	\begin{enumerate}[1.]
		\item
			Assume there exists $w_1 \in [x_0, u]$ satisfying $d(y, w_1)< \delta$.
			We have
			$d (u, v') \leq  d(u, w_1) + d(w_1,v')$. Let us estimate each summand separately.
			By maximality of $d(x_0,v)$ we get
			\begin{align*}
			d(u,w_1) = d (x_0, u) - d(x_0, w_1) &\leq d(x_0, v) - d(x_0, w_1) \\
			&\leq d(v,w_1) \leq d(v,y) + d(y,w_1) < \frac{t}{2} + \delta \;.
			\end{align*}
			And by definition of $w_1$ and $v'$,
			$$
				d(w_1,v') \leq d(w_1, y) + d(y, v') < r + \delta \;.
			$$
			So overall,
			$d(u,v') < \frac{t}{2} + \delta + r + \delta = \frac{t}{2} + \big( r+2\delta \big) < t$.
			
		\item
			Assume now we have a point $w_2 \in [v,u]$ for which $d (y,w_2) < \delta$.
			By the triangle inequality,
			$d (u,v') \leq d(u,w_2) + d(w_2,y) + d(y,v')$.
			We need to estimate $d(u,w_2)$. Note that
			$$
			\rho : = d(v,v') \leq d(v,w_2) + d(w_2,v') \leq d(v,w_2) + d(w_2,y) + d(y,v') <
				d(v,w_2) + \delta + r \;.
			$$
			Hence $d(v,w_2) > \rho - (\delta + r)$.
			So we have
			$d (u,w_2) = d(u,v) - d(v,w_2) < t - \rho + (\delta + r)$.
			Thus,
			$
				d(u,v') < t - \rho + 2(\delta + r) < t \;,
			$
			since $\rho > 2 (\delta + r)$ by (\ref{ineq: bounds_on_rho}).
	\end{enumerate}
\end{proof}


\section{\texorpdfstring{\v{C}ech}{Cech} complex, Rips complex and data analysis} \label{subsec: Cech_VS_Rips}
Let $M$ be a Riemannian manifold and let $X \subseteq M$ be a finite set of points ``approximating" $M$ (i.e. a sample of points from $M$). Having only $X$, can we reconstruct $M$? Or, rather, \emph{how well} can we reconstruct $M$?
%

Denote by $d$ the Riemannian distance on $M$ (and the distance induced on $X$). In applications, the metric space $(X,d)$ models a data cloud. One of the principles of the topological data analysis is
``don't trust large distances",
cf. \cite{Mahoney_Lim_Carlsson_08}.
Therefore, the objective is to reconstruct the topology of $M$ using ``local" geometry of the
set $X$. \\

For $t>0$, we write $B_t(x)$ for an open ball around $x$ of radius $t$ with respect to $d$.
One complex that would be of use here is the Rips complex associated to $(X,d)$.
Let us introduce also the \v{C}ech complex.

%


\begin{defn}
			The $\calU = \{U_i\}$ be a finite collection of subsets of a set $A$. We define the \emph{\v{C}ech complex} $\check{C} (\calU)$ associated to this collection as follows.
			The vertices are the sets $U_i$.
			An ordered collection $\sigma = [U_0, \ldots, U_k]$ is a $k$-simplex if $\cap_{j} U_j \neq \emptyset$.
			The boundary operators are defined in a standard manner, enabling one to consider the corresponding homology $H_* \left( \check{C} (\calU) \right)$.	
\end{defn}

\begin{rmk}
	A special case, that will be of particular interest to us, is the following setting.
	Let $(X,d)$ be as above and fix $t>0$. Consider the collection of open balls around each point $x_i \in X$ of radius $t/2$, denoted $U_i = B_{\frac{t}{2}} (x_i)$ (note the scaling!). We will examine the homology associated to the \v{C}ech complex of this collection $\calU_t = \{ U_i \}$, and denote it by $\check{C}_t (X)$ (for the notation to be similar to those of the Rips complex).
	Varying $t>0$, similarly to the case of Rips complex, we can consider the persistence module $H_* (\check{C}_t(X))$ with the persistence maps induced by the inclusion maps $i_{s,t}: \check{C}_s (X) \to \check{C}_t (X)$.
\end{rmk}

Let us point out that the persistence modules coming from the \v{C}ech and the Rips complexes corresponding to a finite metric space $(X,d)$ are $1$-interleaved after passing to a ``logarithmic scale".

\begin{lem} \label{lem: comparing_Rips_Cech_interleaving_logscale}
	Let $(X,d)$ be a finite metric space.
	Take $V_a = H_* \big( R_{2^a} (X) \big) $ and $W_a = H_* \big( \check{C}_{2^a} (X) \big) $ with morphisms induced from the Rips and the \v{C}ech complexes respectively. Then $V$ and $W$ are $1$-interleaved.
\end{lem}

\begin{proof}
We compare the two complexes, which are both subcomplexes of the full simplex generated by the points of $X$.
	\begin{enumerate}[(1)] 
		\item
			If $[y_0, \ldots, y_k]$ is a simplex in $R_t (X)$, then $d (y_i, y_j) < t$, so $y_i \in B_t (y_j)$ for all $i,j$. In particular, $y_0$ is a common point for all $B_t(y_j)$, so $[y_0, \ldots, y_k]$ determines a $k$-simplex in $\check{C}_{2t}$.
			Thus, $R_t \subset C_{2t}$.
		\item
			If $[y_0, \ldots, y_k]$ is a simplex in $\check{C}_t (X)$, i.e.\ $\cap_{j} B_{\frac{t}{2}} (y_j) \neq \emptyset$, then in particular for each pair $y_i, y_j$ the balls $B_{\frac{t}{2}} (y_i)$ and $B_{\frac{t}{2}}(y_j) $ intersect, hence $d (y_i, y_j) < t$.
			So $[y_0, \ldots, y_k]$ is a simplex in $R_{t} (X) \subseteq R_{2t}(X)$.
			We get then that $C_t \subset R_{2t}$.

	\end{enumerate}
	Passing to ``logarithmic scale" and consider the persistence modules as in the statement, we see that $V$ and $W$ are $1$-interleaved, with interleaving maps induced by the identity.
\end{proof}

\begin{rmk}
Let us mention that the logarithmically rescaled ``persistence modules" $V$ and $W$ discussed here, do not meet fully our initial \Cref{defn: pm} of persistence modules. Namely, property (4), which makes persistence modules to vanish on the left, starting at some point is not satisfied. However $V$ and $W$ are proper persistence modules as defined in Section \ref{sec-prop-pm}.
\end{rmk}

\begin{exm}
	\footnote{We thank Shira Tanny for communicating to us this example. See also \cite{deSilva_Ghrist_coverage_sensornet_ph_2007}.}
	Take a regular hexagon with side-length $1$ (see \Cref{fig: hexagon}).
	Let us compute the barcodes corresponding to its Rips and \v{C}ech complexes.
	
	\begin{figure}[!ht]
		\centering
		\includegraphics[scale=1]{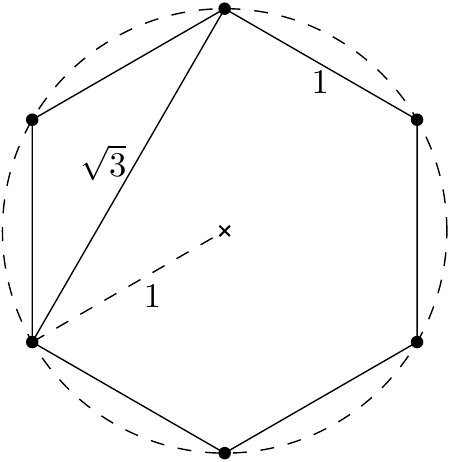}
		\caption{Computing the \v{C}ech complex $\check C_t$ for the hexagon.}
		\label{fig: hexagon}
	\end{figure}
	
	\noindent
	\textbf{The Rips complex.}
	We have the following homology groups as $t$ varies:
	\begin{itemize}
		\item
		For $0<t\leq 1$, we have six distinct points, so $H_0 = \R ^6$.
		\item
		For $1 < t \leq \sqrt 3$, we have an $S^1$, so $H_0 = \R$ and $H_1 = \R$.
		\item
		For $\sqrt 3 < t \leq 2$, we get a sphere $S^2$, which is obtained by glueing two discs along their boundary. These discs are created by the triangles shown in \Cref{fig: hexagon_sphere}.
		Hence $H_0 = \R$, $H_2 = \R$ and $H_1 = 0$.
		\begin{figure}[!ht]
			\centering
			\includegraphics[scale=0.7]{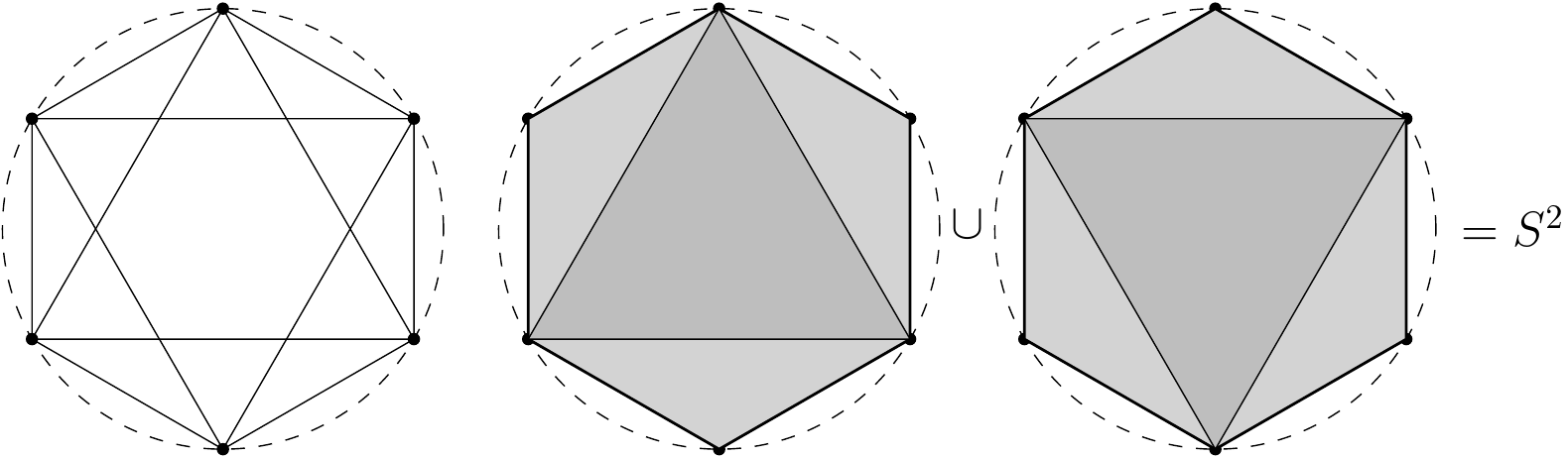}
			\caption{The Rips complex for $\sqrt 3 < t \leq 2$: we get have 8 triangles (2-simplices) which form two discs that are glued to make a $2$-sphere.} 	
			\label{fig: hexagon_sphere}
		\end{figure}

		\item
		For $t > 2$, we get a full $5$-simplex created by the vertices of the hexagon, so we only have $H_0 = \R$ left.
	\end{itemize}

	Overall we get the barcode shown in \Cref{fig: hexagon_barcode}.
	
	\begin{figure}[!ht]
		\centering
		\includegraphics[scale=1]{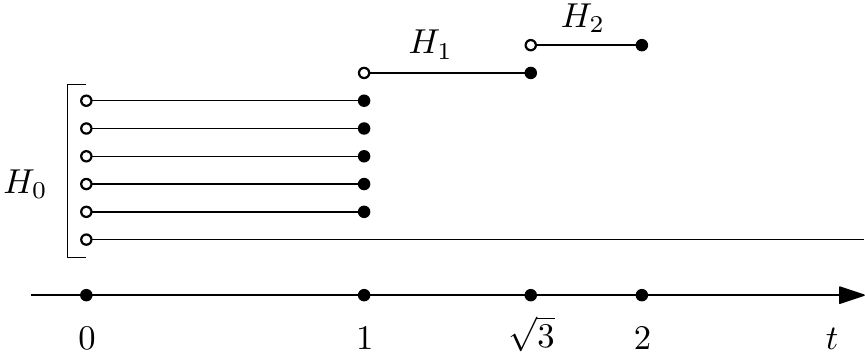}
		\caption{Barcode of the hexagon example: Rips.}
		\label{fig: hexagon_barcode}
	\end{figure}

	\noindent
	\textbf{The \v{C}ech complex.}
	This time, we have the following dependence on $t$:
	\begin{itemize}
		\item
			For $0 < t \leq 1$, we have $H_0 = \R ^6$.
		\item
			For $1 < t \leq \sqrt{3}$, we have an $S^1$, so $H_0 = \R$ and $H_1 = \R$.
			(See \Cref{fig: hexagon_cech_S1_1}.)
			\begin{figure}[!ht]
				\centering
				\includegraphics[scale=1]{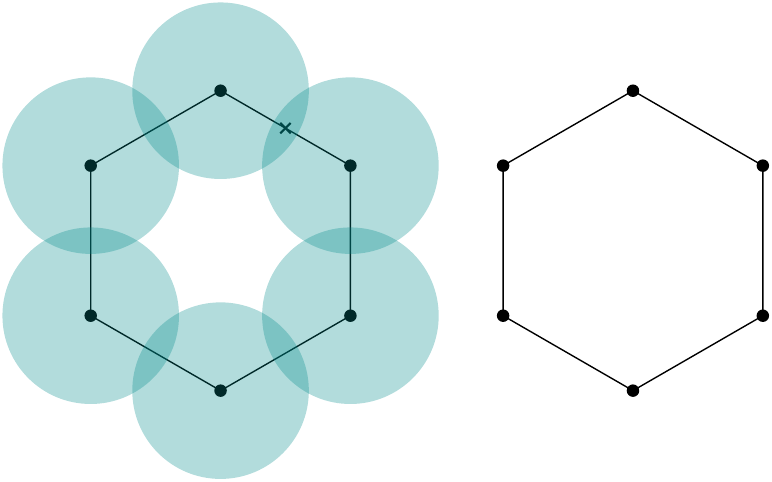}
				\caption{The \v{C}ech complex for $1<t \leq \sqrt{3}$.}
				\label{fig: hexagon_cech_S1_1}
			\end{figure}
		\item
			For $\sqrt{3} < t \leq 2$, we get the following (\Cref{fig: hexagon_cech_S1_2}) union of triangles which is homotopic to $S^1$, that is, we still have $H_0 = \R$ and $H_1 = \R$. (Note that the big equilateral ones are still not present)
			\begin{figure}[!ht]
				\centering
				\includegraphics[scale=1]{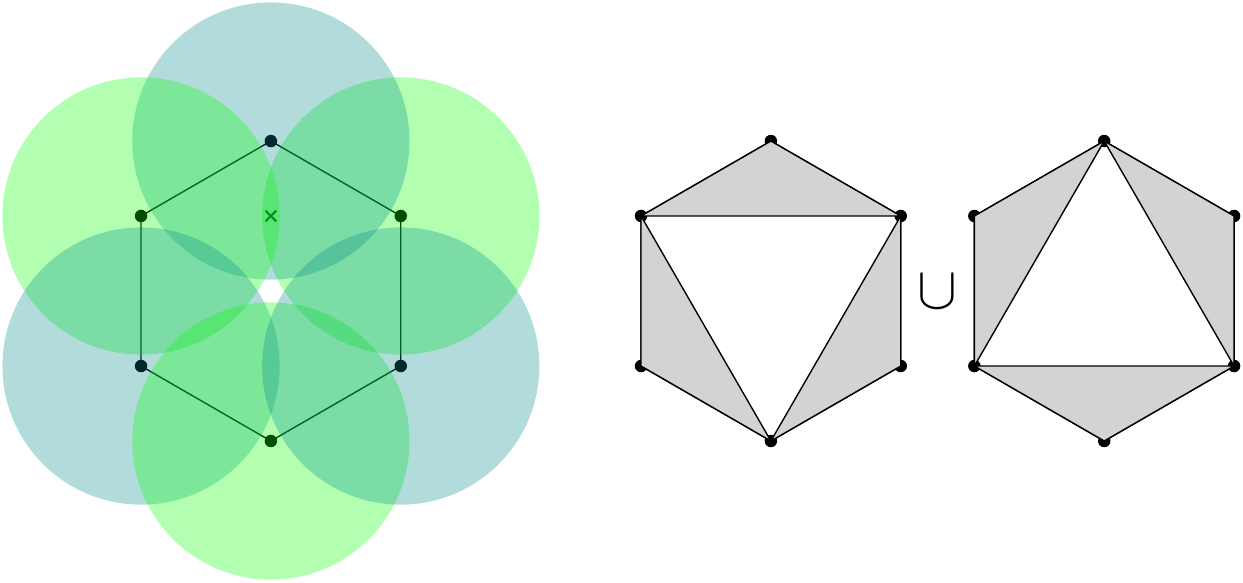}
				\caption{The \v{C}ech complex for $\sqrt{3} <t \leq 2$.}
				\label{fig: hexagon_cech_S1_2}
			\end{figure}
		\item
			For $t > 2$, we get the full $5$-simplex, so we only are left with $H_0 = \R$.
	\end{itemize}

\noindent
	See \Cref{fig: hexagon_barcode_cech} for the corresponding barcode.
	Comparing the two barcodes, we notice that the Rips complex captures a ``redundant" (in the sense of the topology of the hexagon) $2$-dimensional cell, while the \v{C}ech complex does not (cf. also \Cref{lem: Nerve_lemma}).
	Nonetheless, let us note that the \v{C}ech complex has the disadvantage of being harder to compute and handle (see \cite{oudot_persistence_book_15}, chapter 5), as we need to know (and store) the information about all possible simplices (i.e. intersections of any amount of balls around the sampled points). At the same time, for the Rips complex, as we mentioned in \Cref{exm: finite_metric_sp_Rips_pm}, the information required is only about $1$-simplices, i.e. about the distances between each pair of points from the sample set.

	\begin{figure}[!ht]
		\centering
		\includegraphics[scale=1]{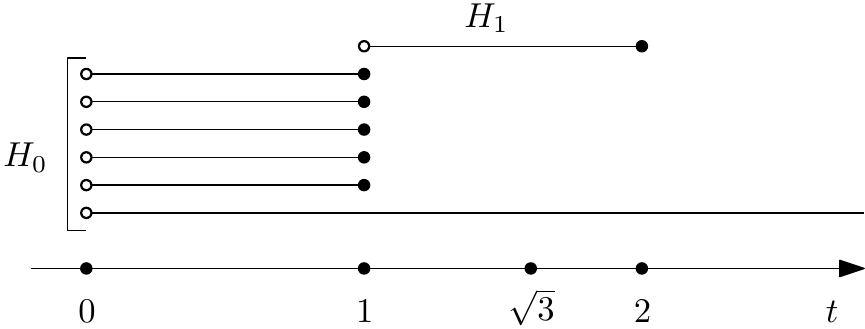}
		\caption{Barcode of the hexagon example: \v{C}ech.}
		\label{fig: hexagon_barcode_cech}
	\end{figure}

\end{exm}

	\begin{exr}
		Check that the two barcodes we found are $1$-interleaved after passing to the ``logarithmic scale" described in \Cref{lem: comparing_Rips_Cech_interleaving_logscale}.
	\end{exr}

\section{Manifold Learning} 
As presented in the beginning of the previous section, we want to study a Riemannian manifold $M$ extracting information about it from a finite sample of points $X = \{x_1, \ldots, x_N \} \subset M$. We present here certain approaches using the \v{C}ech and the Rips complexes of a cover of $M$ by balls around the sampled points (with respect to the distance $d$ on $M$).

A \emph{good cover} $\calU = \{U_i\}$ of a topological space is an open cover for which any intersection of finitely many elements of $\calU$ is either empty or contractible.
We will use the following result on \v{C}ech homology:

\begin{lem}[The Nerve Lemma, see e.g. \cite{hatcher_algebraic_top}] \label{lem: Nerve_lemma}
%
	Let $\calU = \{ U_i \}$ be a good cover of a manifold $M$.
	Then the homology of the corresponding \v{C}ech complex of $\calU$ equals to that of the manifold:
	$
		H_*(\check{C} (\calU)) = H_* (M) \;.
	$
\end{lem}

Let $X \subset M$ be a finite set of points.
As before, consider the Rips complex $R_{t}(X)$ with vertex set $X$ and simplices $\sigma$ formed by a subsets of $X$ that have diameter smaller than $t$. For $X$ dense enough (or $t$ big enough), the collection $\calU _t = \calU_t (X)= \{ B_{2^{t/2}} (x) \}_{x\in X}$ is a cover of $M$. In such a case, we consider also the \v{C}ech complex associated to this cover, denoting it by $\check{C}_{t} (X)$.
We know that the persistence modules $V_a = \check{C}_{2^a}$  and $W_a = R_{2^a}$ are $1$-interleaved.

\begin{thm} \label{thm: from_Rips_to_homology_of_manifold}
	Let $M$ be a Riemannian manifold and $X\subset M$ a finite sample of points.
	Suppose that there exists $\epsi_- < \epsi_+$ with $\epsi_+-\epsi_- > 4$, such that for any $t\in (\epsi_-,\epsi_+]$, the collection $\calU _t$ is a good cover of $M$.
	Then for each $k\geq 0$ the $k$-th homology of $M$ can be recovered from the corresponding Rips persistence module $(W, \pi^W)$ associated to $X$, i.e. we can reconstruct the homology of $M$ using persistence homology:
	$$
		\im \left( \pi^W_{\epsi_- +1, \epsi_+ -1} \right) \simeq H_k (M) \ \  \forall k\geq 0 \;.
	$$
\end{thm}


\begin{exms}
	\begin{enumerate}
		\item
			We can take $\epsi_+ = \log_2 (\text{convexity radius of }M)$, as then $\forall t \leq \epsi_+$, we have a good cover (in case $t$ is big enough so that it is a cover to begin with).
			Recall that the \emph{convexity radius} is the maximum over all $r>0$ for which at every point $x\in M$, the ball $B_r(x)$ of radius $r$ around $x$ is strictly convex. Here being a \emph{strictly convex} subset means that for any two points belonging to it, there exists a unique minimal geodesic joining them, that is contained in the subset.
		\item
			Let us give an example in which the conditions of \Cref{thm: from_Rips_to_homology_of_manifold} are satisfied. Take $\epsi_+ = \log_2 (\text{convexity radius of }M)$ and pick $\epsi_-$ so that $\epsi_+ - \epsi_- > 4$ (this $\epsi_-$ could be negative, as we work in multiplicative scale, taking balls of radius $2^{t/2}$).
			Take now a finite sample set $X \subset M$ to be a maximal collection of points such that $d(x, y) > \epsi_-$ for all $x, y\in X$ (i.e. a collection to which it is impossible to add more points preserving this condition).
			Then $\cup_{x\in X} B_t(x)$ is a  good cover  for every $t\in (\epsi_-, \epsi_+]$, and hence $H_* (M)$ can be recovered as described in \Cref{thm: from_Rips_to_homology_of_manifold}.
	\end{enumerate}
\end{exms}

In the proof of \Cref{thm: from_Rips_to_homology_of_manifold} we will use the following construction. Let $(V,\pi)$ be a persistence module and let $I \subset \R$ be an interval of the form $(a,b]$, where $b \leq \infty$. Consider a \emph{truncated} persistence module $(\ol V, \ol \pi)$, i.e. take $\ol V_t$ to be $V_t$ for $t\in I$ and zero otherwise, and truncate $\pi$ accordingly. See e.g. \cite{Chazal_DeSila_Glisee_Oudot} for the idea of such truncation, and \cite[ Theorem 3.3]{polterovichL_inferring_17} for a similar argument.

\begin{exr} \label{exr: truncated_interleaving}
	Let $(V,\pi)$ and $(W,\sigma)$ be two persistence modules which are $\delta$-interleaved, and fix some interval $I= (a,b]$ with $b\leq \infty$. Show that the truncated persistence modules with respect to $I$, $\ol V$ and $\ol W$, are again $\delta$-interleaved.
\end{exr}

\begin{proof}[Proof of \Cref{thm: from_Rips_to_homology_of_manifold}]
	Denote $J = (\epsi_-, \epsi_+]$. Fix an integer $k\geq 0$, we will write $V$ and $W$ meaning only homology of degree $k$.
	Since $\calU _t$ is a good cover for any $t\in (\epsi_-,\epsi_+]$, by the Nerve Lemma, we have
	$V_t = H_k (\check{C}(U_t)) = H_k (M)$ (so $\dim V_t$ is constant on $J$). Hence the number of intervals in $\calB (V)$ containing $J$ is exactly $\dim H_k (M)$ (for each such $t$, $\dim V_t = \dim H_k (M)$, but intervals could be longer than $J$).
	Consider ``truncated" persistence modules $\ol V$ and $\ol W$ with respect to $J$.
	By previous comparison between the Rips and the \v{C}ech complexes and \Cref{exr: truncated_interleaving}, we know that $\ol V$ and $\ol W$ are $1$-interleaved. Hence by the Isometry theorem (\Cref{thm: isometry_thm}), their barcodes satisfy
	$d_{bot} (\calB (\ol V), \calB (\ol W)) \leq 1$, i.e. there exists a $1$-matching $\mu : \calB (\ol V) \to \calB (\ol W)$ (see \Cref{defn: delta_matching}).
	
	Note first that $\calB (\ol{V})$ contains exactly $\dim (H_k (M))$ copies of $J$ and no other bars (shorter bars are not possible, since $\dim V_t$ and hence $\dim \ol V_t$ are constant on $J$). Each such copy of $J$ is of length greater than $4$, so it is matched by $\mu$ to a bar from $\calB (\ol W)$ which contains $J^1 = (\epsi_-+1, \epsi_+-1]$.
	On the other hand, each bar $\calB (\ol W)$ that contains $J^1$ is still of length greater than $2$, so it is matched by $\mu$ to a bar from $\calB (\ol V)$ that contains $J^2 = (\epsi_-+2, \epsi_+-2]$. Such a bar can only be of the form $J$ (these are the only bars in $\calB (\ol V)$), thus overall the number of intervals in $\calB (\ol W)$ containing $J^1$ is exactly $\dim (H_k (M))$.
	In other words, $\dim \im \left( \pi^W_{\epsi_-+1, \epsi_+-1} \right) = \dim H_k (M)$.
\end{proof}

\begin{rmks}
	A few comments are in order.
	\begin{itemize}
		\item In practice, long bars in the barcode of the Rips complex carry more reliable
information about homology of $M$ than short bars which can be interpreted as a ``topological" noise
(see \cite{ghrist_barcodes_2008}). Therefore, larger is the difference $(\epsi_-, \epsi_+]$,
more trustworthy is the calculation of $H_*(M)$ proposed in Theorem \ref{thm: from_Rips_to_homology_of_manifold}.
			
		\item
			In \cite[Propositions 3.1.]{Niyogi_Smale_Weinberger_08} P.~Niyogi, S.~Smale and S.~Weinberger consider the case where $X$ is an $\frac{\epsi}{2}$-dense collection of points sampled from a submanifold $M \subset \R^n$. Take the union of Euclidean balls $U = \{ \cup_{x_i \in X} B_t (x_i) \}$ centered at the points of $X$. It turns out that 
when $t$ varies in a certain interval depending on the geometry of $M$, the set $U$ 
deformation retracts to $M$, and in particular their homologies are equal. Furthermore, if $X$ consists
of a sufficiently large amount of independent identically distributed points sampled with respect to the uniform probability measure on $M$, the homology of $U$ equals to the homology of $M$.
 
		\item
			Let us also mention a paper \cite{Latschev2001} by Latschev, in which he obtains the following result, answering a question raised in \cite{Hausmann95}:
			For a closed Riemannian manifold $M$, there exists $\epsi_0 > 0$ small enough, so that for any $0 < \epsi \leq \epsi_0$, there is $\delta_\epsi > 0$, for which if $Y$ is a metric space that has Gromov-Hausdorff distance less than $\delta_\epsi$ to $M$, then its Rips complex $R_\epsi (Y)$ is homotopy equivalent to $M$. (Here $Y$ could be an infinite set.)
			Note that in particular it follows that, if $Y\subseteq M$ is finite and $\delta_\epsi$-dense in $M$, then $R_\epsi (Y)$ and $M$ have the same homotopy type.
	\end{itemize}
\end{rmks}

\chapter{Topological function theory}\label{chap5a-functheory}

\section{Prologue}
Topological function theory studies features of smooth functions on a manifold that are invariant under
the action of the diffeomorphism group. The simplest invariant of this kind is the uniform norm, as
opposed to, say, $L_p$-norms or $C^k$-norms, which depend on additional choices, a volume form or a metric,
respectively. The theory of persistence modules provides more sophisticated invariants coming
from the homology of the sublevel sets of a function. We have encountered some of them earlier in this book,
including spectral invariants and the boundary depth. In the present chapter we focus, roughly speaking, on the ``size" of the barcode which can be considered as a useful measure of oscillation of a function.
We provide bounds on this size in terms of norms of a function and its derivatives, and at the end
discuss some links to approximation theory.

\medskip\noindent{\bf Convention:} Throughout this chapter, we write $||\cdot||_0$ for the uniform norm; the lower index $0$ is meant to emphasize its distinction from the $L_2$-norm $||\cdot||_2$ which will be also widely used below.

For a Morse function $f$, write $\nu(f)$ for the number of finite bars in the barcode
of $f$. Recall that there are $\zeta(M)$ of infinite rays, where $\zeta$ stands for the total
Betti number of $M$. Here and below the bars are counted with the multiplicities.

Denote by $\nu(f,c)$ the number of finite bars of length $>c$, and
define an invariant
\medskip\noindent \begin{equation}\label{eq-ell}
\ell(f):= \text{length}\left(\calB(f) \cap [\min f,\max f]\right)
\end{equation} which measures the total length of all finite bars of $f$ and of the segments of the infinite
rays in the interval $[\min f,\max f]$. In this chapter we discuss these invariants, following the works \cite{CSEHM, Polterovich-Sodin,Polterovich2-Stojisavljevic}.

Let us start with a couple of observations. Obviously,
the function $\nu(f,c)$ is decreasing in $c$ and
\begin{equation}\label{eq-nu-ell}
c\nu(f,c) \leq \ell(f)\;.
\end{equation}
The functional $\ell$ is, generally speaking,
discontinuous under perturbations in the uniform norm: one can create an arbitrarily large
number of short bars by a small perturbation. However, for every two Morse functions $f$ and $h$ we have
\begin{equation}\label{eq-ell-ell}
\ell(f)-\ell(h) \leq (2\nu(f) + \zeta(M))||f-h||_0\;,
\end{equation}
and
\begin{equation}\label{eq-nu-perturb}
\nu(f, c) \geq \nu(h,c+ 2||f-h||_0)\;.
\end{equation}
Inequalities \eqref{eq-ell-ell} and  \eqref{eq-nu-perturb} immediately follow from the fact that the barcodes of $f$ and $g$ admit a $\delta$-matching with $\delta=||f-h||_0$.

A number of results presented in this chapter have counterparts in the calculus of functions of one variable.
In the case of a Morse function $f$ on the circle $S^1 =\R/(2\pi)\Z$, the notions coming from the barcode,  such as the number or the total length of finite bars, have a transparent meaning (cf. \cite[p. 137]{CSEHM}).
 All critical points of $f$ are either local minima or local maxima and they are located on $\mathbb{S}^1$ in an alternating fashion. More precisely, if there are $N$ local minima $x_1,\ldots, x_N$,  there are also $N$ local maxima $y_1,\ldots, y_N$,  and we may label them so that they are cyclically ordered as follows:
$$x_1, y_1, x_2, y_2, \ldots, x_N,  y_N, x_1.$$
The barcode of $f$ contains $N-1$ finite bars in degree $0$ whose left endpoints are local minima and whose right endpoints are local maxima, as well as  two infinite bars in degrees $1$ and $0$ starting at the global maximum
and the global minimum, respectively. From here it follows that
$$\ell(f)=\sum_{i=1}^{N}(f(y_i)-f(x_i)).$$
On the other hand,  the total variation of $f$ equals
$2 \sum_{i=1}^{N}(f(y_i)-f(x_i)) = 2\ell(f)$.
Therefore,
\begin{equation}
\label{eq-1Dcase}
\ell(f) = \frac{1}{2} \int_{0}^{2\pi}|f'(t)|dt\;.
\end{equation}
In particular, we conclude from \eqref{eq-nu-ell} and \eqref{eq-1Dcase} that
\begin{equation}
\label{eq-nu-uniform}
\nu(f,c) \leq \pi ||f'||_0/c\;.
\end{equation}
As we shall see in the next section, this inequality manifests a very general phenomenon.

Sometimes, it is useful to estimate $\ell(f)$ via $L_2$-norms of $f$ and its derivatives:
$$
\ell(f) \leq \sqrt{\frac{\pi}{2}}\bigg( \int_{0}^{2\pi}(f'(t))^2 dt \bigg)^{\frac{1}{2}}=\sqrt{\frac{\pi}{2}} \bigg| \int_{0}^{2\pi}f''(t)f(t)dt \bigg|^{\frac{1}{2}}\leq\sqrt{\frac{\pi}{2}}  \| f \|_2^{\frac{1}{2}} \| f'' \|_2^{\frac{1}{2}}\;.$$
This yields
\begin{equation}
\label{eq-ell-1D-norms}
\ell(f) \leq \sqrt{\frac{\pi}{8}} ( \| f \|_2 + \| f'' \|_2) \;.
\end{equation}
In Section \ref{sec-length-barcode} we discuss a two-dimensional generalization of this
inequality.

\medskip\noindent{\bf An apology:} Multiplicative numerical constants appearing in this chapter
are not sharp. Apparently, the problem of finding sharp constants is difficult even in the one-dimensional
case.

\medskip

Another piece of motivation comes from a beautiful observation by Shmuel Weinberger \cite{Weinberger-alternance}
relating barcodes of functions of one variable with Chebyshev's famous
 alternance (a.k.a. equioscillation) theorem. One of the versions of this theorem deals
with approximation of continuous functions on the circle.
Denote by $\calT_n$ the set of trigonometric polynomials on $S^1= \R/(2\pi \Z)$ of degree $\leq n$.

\begin{thm} [Chebyshev's theorem, \cite{Stepanets}]
\label{thm-chevyshev} A trigonometric polynomial $p \in \calT_{n-1}$ on
$S^1$ provides the best uniform approximation in $\calT_{n-1}$ to a continuous function $f$ if and only if there exist $2n$ points $0 \leq x_1 < \dots < x_{2n} < 2\pi$
so that the differences $f(x_i)-p(x_i)$ reach the maximal value $||f-p||_0$ with alternating signs.
\end{thm}

\medskip
\noindent
Existence of such a collection of extremal points of $f-p$ is called {\it alternance}. For instance,
the polynomial $p=0 \in \calT_{n-1}$ provides the best approximation to $f(x)= \cos (nx)$.
The alternance is given by points $x_k=\pi k/n$, $k = 0, \dots, 2n-1$.

Let us sketch a barcode-assisted proof of the fact that the alternance property yields the best approximation
under an extra assumption that $f-p$ is Morse.
We start with a general result.
\medskip

\begin{prop}\label{pro-alawein}
Let $h,q$ be two Morse functions on a smooth closed manifold $M$ such that
for some $c>0$, $q$ has strictly less than $2\nu(h,c)+\zeta(M)$ critical points. Then
$||h-q||_0 \geq c/2$.
\end{prop}
\begin{proof} Assume on the contrary that
$||h-q||_0 < (c-\epsilon)/2$, for some $\epsilon >0$.
Denote by $N$ the number of critical points of $q$. Exactly $\zeta(M)$ of them
contribute to infinite rays of the barcode. Thus the number of finite bars
in the barcode of $q$ cannot exceed $(N-\zeta(M))/2$. Therefore, by the assumption
of the proposition, $\nu(q,\epsilon) < \nu(h,c)$.
At the same time,  by \eqref{eq-nu-perturb},
$$\nu(q, \epsilon ) \geq \nu(h, \epsilon + 2||h-q||_0) \geq \nu(h,c),$$
and we get a contradiction.
\end{proof}

\medskip
\noindent
{\bf Proof of ``alternance $\;\Rightarrow\;$ best approximation for Morse $f-p$":}
Put $h=f-p$, $c= ||h||_0$. By the alternance property, the barcode of
$h$ consists of $n-1$ finite bars of length $2c$ and two infinite rays.
Thus $\nu(h,t) = n-1$ for every $t < 2c$. On the other hand, every non-constant
trigonometric polynomial $q$ of degree $\leq n-1$ has at most $2n-2$ critical points.
This count shows that the assumption of Proposition \ref{pro-alawein} reads
$$2n-2 < 2(n-1)+2\;,$$
and hence by this proposition
$||h - q||_0 \geq c$.
But $h-q = f-(p+q)$, and hence $||f-r||_0 \geq c$ for every trigonometric polynomial
$r \in \calT_{n-1}$. Since $||f-p||_0=c$, we conclude that $p$ is the polynomial of the best
approximation of degree $\leq n-1$.
\qed

\section{Invariants of upper triangular matrices}\label{sec-Jordan} 
We start with a problem of linear algebra. 
Let $C$ be a finite dimensional vector space equipped with a nilpotent operator $d: C \to C$
with $d^2=0$. Let $e_1,...,e_N$ be a basis in $C$ such that $d$ in this basis
is given by an upper-triangular matrix. A triangular change of basis is the one of the form
$$f_i = \sum_{j \leq i} a_{ij}e_j, \;\; a_{ii} \neq 0\;$$ Put $\Omega_N:= \{1,\dots, N\}$.
A basis $f_i$ is called {\it the Jordan basis} for $d$ if there exists a subset $I \subset \Omega_N$ 
and an injective map $\phi: I \to \Omega_N  \setminus I  $ such that $\phi(i) < i$ for all $i$ and 
\begin{equation}\label{eq-baran}
df_i= 0 \;\;\text{if}\;\; i \notin I,\;\;\text{and}\;\;df_i = f_{\phi(i)}\;\;\text{if}\;\; i \in I\;.
\end{equation} 

\medskip\noindent\begin{thm} \label{thm-baran} There exists a triangular change of basis $\{e_i\}$ taking
it to a Jordan basis. 
\end{thm} 

\medskip\noindent An equivalent formulation is that {\it for every nilpotent $N \times N$ upper-triangular matrix $d$ 
over a field with $d^2=0$ there exists a permutation matrix $p$ and an invertible $N \times N$ upper-triangular matrix $v$ such that
$vdv^{-1} =pjp^{-1}$, where $j$ is the Jordan canonical form of $d$.} While the formulation (and the proof!) could
have been given in the XIX-th century, the first published proof, to the best of our knowledge was given 
by Barannikov \cite[Lemma 2]{barannikov1994framed} in 1994. Other proofs (where the authors were unaware of Barannikov's work) are due to Thijsse \cite[Theorem 1.5]{thijsse1997upper} in 1997 and Melnikov \cite{melnikov2000b} in 2000. We present a proof due to Barannikov which provides an explicit algorithm for finding the desired
triangular change. 

\medskip\noindent
\begin{proof}  We construct the change of the basis recursively starting with $f_1:=e_1$. Since
$d$ is upper triangular and nilpotent, $de_1=0$. Assume that we constructed, by a triangular change of the first $i-1$ vectors of the basis, new vectors $f_1,...,f_{i-1}$, a set  $I \subset \Omega_{i-1}$          and a map  $\phi: I \to \Omega_{i-1} \setminus I $ which satisfy \eqref{eq-baran}.

Write 
$$de_i = \sum_{j \in I} p_j f_j + \sum_{m \in \Omega_{i-1} \setminus I} q_mf_m\;.$$ 
Taking $d$ again and using \eqref{eq-baran} we get that
$\sum_{j \in I} p_j f_{\phi(j)}=0$, which by linear independence yields that all $p_j$'s vanish.
Decompose 
$$\Omega_{i-1} \setminus I = J \sqcup K\;,\;\text{where}\;\; J :=  {\im}\phi\;.   $$
We have
$$de_i = \sum_{j \in J} q_jdf_{\phi^{-1}(j)} + \sum_{k \in K} q_kf_k\;.$$
Set 
$$f_i = e_i - \sum_{j \in J} q_jf_{\phi^{-1}(j)}\;.$$
If  $q_k = 0$ for all $k \in K$, we have $df_i=0$. The set $I$ and the map $\phi$
remain without changes. Otherwise there exists maximal $n \in K$ with $q_n \neq 0$.
We replace $f_n$ by $f_n= \sum_{k \in K} q_kf_k$, so that $df_i = f_n$,
add $i$ to $I$ and put $\phi(i)=n$. This completes the description of the recursion step. 
\end{proof} 

\medskip

We wish to apply the above result to the following situation which appears in several
meaningful applications. Consider a complex $(C_*,d)$ where $C= \bigoplus_{k=0}^L C_k$.
Suppose that we are given a non-ordered basis $E_i$ in $C_i$, and a function 
$u: E\to \R$, where $E:= \bigcup_i E_i$. Write $m_i$ for the cardinality of $E_i$.  
Assume that the differential $d$ decreases the filtration: $u(de) \leq u(e)$ for every $e \in E_i$.
The reader is familiar with such a situation in the context of Morse homology where
$E_i$ is the set of critical points of index $i$ on a closed manifold, and $u(e)$ is the critical value
of the function at $e$. We call a complex with the above structure as {\it a filtered complex with a preferred
basis}. 

Extend the filtration to the whole $C$ by setting
$$u\left(\sum_{e \in E} a_ee\right) = \max_{a_e \neq 0} u(e)\;.$$
Consider the family of subspaces $C^t \subset C$ consisting of $x \in C$ with $u(x) < t$.
Since $d$ preserves $C^t$,  we have a family of homologies $H_*(C^t,d)$ together with morphisms
induced by inclusions $C^s \subset C^t$ for $s < t$. This yields a persistence module whose barcode
we denote by $\calB$.   

Order now the elements of the basis $E$ as follows. For $x,y \in E_i$ with the same $i$
declare $x \prec y$ if $u(x) < u(y)$. In case $u(x)=u(y)$, order them arbitrarily.
For $x \in E_i$ and $y \in E_j$ with $i \neq j$ put $x \prec y$ whenever $i <j$. We denote the ordered
collection of vectors by $e^k_i$  emphasising the degree $k$ of each vector. The order is lexicographic 
with respect to $(k,i)$. Theorem \ref{thm-baran} guarantees the existence of a triangular change yielding a Jordan basis. It is straightforward (and is left as an exercise
to the reader) to perform such a change within each $C_i$ separately, thus keeping vectors of the basis homogeneous in terms of the degree. We denote the vectors of the Jordan basis by $\{f^k_i\}$.
The graded version of condition \eqref{eq-baran} looks as follows: for every $k=0,\dots, L$, 
there exists a subset $I_k \subset \Omega_{m_k}$
and an injective map $\phi_k: I_k \to \Omega_{m_{k-1}}$ such that 
\begin{equation}\label{eq-baran-graded}
df^k_i= 0 \;\;\text{if}\;\; i \notin I_k,\;\;\text{and}\;\;df^k_i = f^{k-1}_{\phi_k(i)}\;\;\text{if}\;\; i \in I_k\;.
\end{equation} 

We say that $i \in I_{k}$ is {\it essential} if $a^{k-1}_i:= u(f^{k-1}_{\phi_k(i)}) < b^k_i:= u(f^k_i)$.
In this case we denote by $F^{k-1}_i$ the interval $(a^{k-1}_i, b^k_i]$. Denote by $G^k_j$ the ray
$(c^k_j,+\infty)$ where $j \in \Omega_{m_k} \setminus I_k$ and $c^k_j = u(f^k_j)$. Denote by $\calC$ the barcode
consisting of intervals $F^{k-1}_i$ and rays $G^k_j$ taken with multiplicities. The next result is the highlight
of our discussion. 

\medskip\noindent\begin{thm}\label{thm-barcodes-coincide} The barcodes $\calB$ and $\calC$ coincide. 
\end{thm}

\medskip\noindent While Theorem \ref{thm-barcodes-coincide} uses the language of barcodes
and persistence modules which did not exist in 1994,  Barannikov informed us that he was aware of this
result. 

\begin{proof} Fix a degree $k$, and take any $t \in \R$.  Since the basis $f^k_i$ is obtained from $e^k_i$ by a degree-homogeneous triangular change, the subcomplex $C^t_k$ is generated
by vectors $f^k_i$ with $u(f^k_i) < t$. 
The homology of $H_k(C^t, d)$ can be readily calculated since we know the matrix of $d$ in this
basis.  First,
if for some $j \notin I_k$ we have $c^k_j < t$, the vector $f^k_j$ contributes a generator
to $H_k(C^t,d)$. Second, take $i \in {\im}(\phi_{k+1})$. Look at the cycle $f^k_i$.
It contributes a generator to $H_k(C^t, d)$ if an only if $\phi_{k+1}^{-1}(i) \in I_{k+1}$
is essential and $t \in F^k_i$. Indeed, if $t \leq a^k_i$,
this cycle does not lie in $C^t$, and if $t > b^{k+1}_i$, it is killed by $f^{k+1}_{\phi_k^{-1}(i)}$.
Look now at all $f^k_i$ and all $f^k_j$ selected in this way for the given $t$. Since they are linearly independent,
the interval modules generated by these elements are direct summands in the persistence module
$H_k(C^t,d)$, and hence the degree $k$ part of the barcode $\calB$ is formed by the rays $G^k_j$ and 
the intervals $F^k_i$. We complete the proof by varying $k$ from $0$ to $L$.
\end{proof} 

\medskip
The following corollary will be used later on in this Chapter. 

\medskip\noindent\begin{cor}\label{cor-number-bars} Given a filtered complex with a preferred basis,
the number of finite bars in the barcode $\calB$ of the homology persistence module does not exceed half of the dimension of the complex. 
\end{cor}

\medskip\noindent
As this is obvious for the barcode of the complex $\calC$, the statement follows from Theorem \ref{thm-barcodes-coincide}.

\section{Simplex counting method}
Our goal here is to extend inequality \eqref{eq-nu-uniform} to arbitrary manifolds.
In this section we follow \cite{CSEHM} and discussions with Lev Buhovsky.

\subsection{A combinatorial lemma}
Let $\Sigma$ be a finite simplicial complex with the vertex set $K$.
By definition, a simplex $\sigma$ is a non-empty subset of $K$. The dimension
of a simplex $\sigma$ is its cardinality minus one. The complex $\Sigma$ is a collection
of simplices satisfying the following assumption: if $\sigma \in \Sigma$, then
every subset of $\sigma$ belongs to $\Sigma$ as well. Write $|\Sigma|$ for the total
number of simplices in $\Sigma$.

A {\it filtration} on $\Sigma$ is a function $u: K \to \R$. We extend it to
all simplices $\sigma$ in $\Sigma$ by setting $u(\sigma)= \max_{x \in \sigma}u(x)$.
Denote by $C_k$ the vector space over $\calF$ spanned by all $k$-dimensional simplices in $\Sigma$. 
The boundary operator $d$, sending a simplex to its oriented boundary, extends to a differential 
$d:C_i \to C_{i-1}$. Thus we get a chain complex $(C_*,d)$ with a preferred basis
consisting of all simplices in $\Sigma$ and the filtration $u$. Denote by $V(\Sigma,u)$
the corresponding homology persistence module, see  Section \ref{sec-Jordan}. Applying
Corollary \ref{cor-number-bars}, we immediately get the following combinatorial statement, which is
the heart of the simplex counting method discussed in this section.

\medskip\noindent
\begin{thm} \label{thm-simcount} The barcode of $V(\Sigma, u)$ has at most $|\Sigma|/2$ finite bars.
\end{thm}

\subsection{Bars and oscillation}
Given a finite simplicial complex $\Sigma$, identify each $n$-dimensional simplex of $\Sigma$ with the
standard simplex $\{z_i \geq 0, \sum z_i =1\}$ in $\R^{n+1}$. In this way $\Sigma$ becomes a topological
space.

By a triangulation of a smooth closed manifold $M$ we mean a pair $T=(\Sigma,h)$ consisting of a
finite simplicial complex $\Sigma$ and a homeomorphism $h: \Sigma \to M$ between $\Sigma$ (considered as a topological space) and $M$. Now we are ready to introduce the central notion of this section.

\medskip\noindent
\begin{dfn}\label{def-oscillation} The {\it oscillation} $\text{Osc}(f,T)$ of a continuous function $f$
with respect to the triangulation $T$ is given by
$$\text{Osc}(f,T) = \max_\sigma \max_{x,y \in h(\sigma)} |f(x)-f(y)|\;,$$
where the first maximum is taken over all simplices in $\Sigma$.
\end{dfn}

\medskip

Any function $f: M \to \R$ induces a filtration $u$ on the vertex set $K$ of $\Sigma$ by
$u(v):= f(h(v))$, and hence gives rise to a persistence module $V(\Sigma,u)$. Let us
compare this module with the Morse persistence module $V(f):= H(\{f < t\})$.

\medskip\noindent
\begin{thm}\label{thm-osc-int} The modules $V(\Sigma,u)$ and $V(f)$ are $\delta$-interleaved
with $\delta= \text{Osc}(f,T)$.
\end{thm}
\begin{proof} Put $M^t= \{f < t\}$. Observe that
\begin{equation}
\label{eq-int-simpl-1}
h(\Sigma^t) \subset M^{t+\delta}\;.
\end{equation}
On the other hand consider the union $U$ of all images $h(\sigma)$, where $\sigma$ is a simplex in $\Sigma$,
which have non-empty intersection with $M^t$. Note that $u(x) \leq t+\delta$ for every vertex $x$ of $h^{-1}(U)$.
Thus
\begin{equation}
\label{eq-int-simpl-2}
M^t \subset h(\Sigma^{t+\delta})\;.
\end{equation}
It remains to mention that the simplicial homology of
$\Sigma^t$ is canonically isomorphic to the singular homology of $\Sigma^t$ considered as a topological space,
and hence to the one of $h(\Sigma^t)$. Thus inclusions \eqref{eq-int-simpl-1} and \eqref{eq-int-simpl-2}
provide the desired interleaving.
\end{proof}

\medskip\noindent

\begin{thm}\label{bars-Lipschitz} Let $T=(\Sigma,h)$ be a triangulation of a closed manifold $M$,
and let $f: M \to \R$ be a Morse function on $M$. Then
\begin{equation}
\label{eq-main-=ineq}
\nu(f, 2\Osc(f,T)) \leq |\Sigma|/2\;.
\end{equation}
\end{thm}

\begin{proof} By Theorem \ref{thm-osc-int} and the isometry theorem, the barcodes of $V(f)$ and $V(\Sigma,u)$
are $\delta$-matched with $\delta= \text{\Osc}(f,T)$. It follows that every finite bar of length exceeding $2\delta$ in
$V(f)$ is necessarily matched with a bar in $V(\Sigma,u)$. But by Theorem \ref{thm-simcount},
there are at most $|\Sigma|/2$ such bars.
\end{proof}

\medskip\noindent
Theorem \ref{bars-Lipschitz}  naturally brings us to the following topological invariant $S(f,c) \in \N$, $c >0$
of a continuous function $f$ on $M$. Consider all possible triangulations $T=(\Sigma, h)$ of $M$ with
$\Osc(f,T)< c$. By definition, $S(f,c)$ is the minimal possible number of simplices in a complex $\Sigma$ corresponding
to such a triangulation. With this language, Theorem \ref{bars-Lipschitz} can be restated as
\begin{equation}\label{eq-S-beta}
\nu(f,2c) \leq S(f,c)/2
\end{equation}
for every Morse function $f$ and $c > 0$.

\medskip\noindent
Assume now that a closed $d$-dimensional manifold $M$ is equipped with a Riemannian metric.
If $f$ is $C^1$-smooth, the invariant $S(f,c)$ can be easily estimated from above by the $C^1$-norm
$||\nabla f||_0= \max_x |\nabla f|$ with respect to  the metric.
Indeed,  for every $r>0$ small enough $M$ admits a triangulation
into $\leq  k \cdot r^{-d}$ simplices of the diameter $\leq r$, where $k$ depends only on the metric.
The oscillation of $f$ on each such simplex does not exceed
$r||\nabla f||_0$. The next result is an immediate consequence of  \eqref{eq-S-beta}.

\medskip\noindent\begin{cor} For every Morse function $f$ on $M$,
\label{cor-lip}
\begin{equation}\label{eq-L-beta}
\nu(f,c) \leq k' \cdot \frac{||\nabla f||_0^d}{c^d} \;,
\end{equation}
where $k'$ depends only on the metric.
\end{cor}

\medskip\noindent
This is the desired extension of  inequality \eqref{eq-nu-uniform} for functions of one variable.

\medskip
\noindent
\begin{exm}\label{exm-sin-Lip}{\rm Consider the $d$-dimensional torus $\R^d/\Z^d$
with the function
$$f(x) =2c \cdot \sum_{i=1}^d \sin (2\pi n x_i),\;\; c>0\;.$$
Then  $\nu(f,c) \approx n^d$ and $||\nabla f||_0 \approx cn$, so inequality \eqref{eq-L-beta}
is sharp up to multiplicative constants.
}
\end{exm}

\section{The length of the barcode}\label{sec-length-barcode}

Inequality \eqref{eq-ell-1D-norms} relating the length of the barcode of a Morse function $f$ and
the $L_2$-norms of $f$ and its second derivative extends to surfaces. This is done in the paper \cite{Polterovich2-Stojisavljevic}, by using differential-geometric methods developed in \cite{Polterovich-Sodin}.
In this section we present this generalization  for functions
on a flat two-dimensional torus,
where the proofs are slightly more direct and transparent. Let us mention that a generalization of these results
to dimensions $\geq 3$ is currently out of reach.

\subsection{The fundamental inequality}
Consider the torus $\T^2 =\R^2/(2\pi \cdot \Z^2)$ equipped with the Euclidean metric
$dx_1^2+dx_2^2$. Denote by
$\Delta f =\frac{\partial ^2 f}{\partial x_1^2} + \frac{\partial ^2 f}{\partial x_2^2}$
the Laplace-Beltrami operator and by $d\mu$ the Lebesgue measure $dx_1dx_2$.

\medskip
\noindent
\begin{thm}\label{thm- length-bars-torus}
For every Morse function $f : \T^2 \to \R$
\begin{equation} \label{eq-length-bars-torus}
\ell(f) \leq 3(||f||_2+||\Delta f||_2)\;.
\end{equation}
\end{thm}

\medskip
\noindent The proof is given in Sections \ref{subsec-BI} and \ref{subsec-NL} below.
As a consequence, applying \eqref{eq-nu-ell}, we get that
\begin{equation}\label{eq-nu-lap-torus}
\nu(f,c) \leq \frac{3(||f||_2+||\Delta f||_2)}{c}\;.
\end{equation}

\begin{exm}\label{exm-trig-1}{\rm Denote by $\calT_\lambda$ the space of trigonometric polynomials on $\T^2$
spanned by  $\sin (n_1 x_1 + n_2 x_2)$ and $\cos (n_1 x_1 + n_2 x_2)$
with $(n_1^2+n_2^2) \leq \lambda$. We leave it as an exercise to check that for every polynomial
$p \in \calT_\lambda$ one has
\begin{equation}\label{eq-p-lam}
||\Delta p||_2 \leq \lambda ||p||_2\;,
\end{equation}
and hence inequality
\eqref{eq-nu-lap-torus} reads as
\begin{equation}\label{eq-nu-lap-p}
\nu(p,c) \leq   \frac{3||p||_2(1+\lambda)}{c}\;.
\end{equation}
Take, for instance, $p(x)= \sin (nx_1) + \sin (nx_2)$. One readily checks that $p$ has
$n^2$ points of maximum (resp., minimum) with critical values $2$ (resp., $-2$), and
$2n^2$ saddles with critical value $0$. The barcode of $p$ consists of infinite rays
$(2,+\infty)$, $(-2,+\infty)$ and, with multiplicity $2$, $(0,+\infty)$, as well
as of the bars $(-2,0]$ and $(0,2]$ of multiplicity $n^2-1$ each.
It follows that
$$\ell(p) = 4n^2+4$$
and $\nu(p,c) = 2n^2-2$ for $c\in (0,2)$ and $0$ for $c >2$.
On the other hand $\lambda = n^2$ and $||p||_2 = 2\pi$,
so, taking $c < 2$ arbitrary close to $2$, we get that the right hand side of inequalities
\eqref{eq-length-bars-torus} and  \eqref{eq-nu-lap-p} in this case is $6\pi(n^2+1)$ and $3\pi(n^2+1)$,
respectively.  It follows that both the left and the right hand side of both inequalities
in this example are of the order $\sim n^2$.
}
\end{exm}

Recall that the Sobolev $W^{2,q}$-norm of a function $f$ is
the sum of the $L_q$-norms of $f$ and its first and second derivatives.
In dimension two, the expression $||f||_2+||\Delta f||_2$  in the right hand side of \eqref{eq-nu-lap-torus}
is equivalent to the Sobolev $W^{2,2}$ norm.   It is
instructive to compare it with the $C^1$-norm $||\nabla f||_0$ in inequality \eqref{eq-L-beta}.
These two norms are known to be incomparable in dimension $2$. Recall that according to the Sobolev inequality,
in dimension $2$ we have $||f||_{C^1} \leq \text{const} \cdot ||f||_{W^{2,q}}$ with $q >2$.
We see that our case $q=2$ lies just beyond the borderline of applicability of the Sobolev inequality.

\medskip
\noindent\begin{exm}\label{ex-harmonics}{\rm It is known
that there exists a sequence of eigenfunctions $f_\lambda$ of the Laplace-Beltrami operator such that
$\Delta f_\lambda+\lambda f_\lambda=0$, $ \lambda \to \infty$, $||f_\lambda||_2=1$,
 and $m_\lambda:= ||f_\lambda||_0 \to \infty$ (See Exercise \ref{exer-mult} below).
At the same time, the zeroes of $f_\lambda$ are necessarily $\approx \lambda^{-1/2}$-dense in the torus
(see e.g. \cite{Mangoubi} and references therein). Therefore, there exist a pair of points on the torus
at the distance at most $\approx \lambda^{-1/2}$ from one another such that the values of $f_\lambda$ at these
points differ by $m_\lambda$. This implies that  $||\nabla f_\lambda||_0/\lambda^{1/2}$ is unbounded as $\lambda \to \infty$.
We conclude that for $c$ varying in a bounded region and $\lambda$ large enough, inequality
$$\nu(f_\lambda, c) \leq \text{const}\cdot \lambda/c$$
which follows from \eqref{eq-nu-lap-torus}
is strictly better than
$$\nu(f_\lambda, c) \leq \text{const} \cdot \frac{||\nabla f_\lambda||_0^2}{c^2}$$
provided by \eqref{eq-L-beta}.
}
\end{exm}

\medskip\noindent\begin{exr}\label{exer-mult}{\rm Prove existence of an unbounded $L_2$-normalized
sequence of the Laplace eigenfunctions on the flat torus by combining the following facts.
First, eigenvalues of the Laplacian on the standard flat torus are the integers
represented as the sum of two squares. Write $r(n)$ for the number of different
ways in which an integer $n$ can be represented as the sum of two squares. It is a classical fact of
number theory that $r(n)$ is an unbounded function \cite{numbers}.
In other words, the multiplicity $m$ of
an eigenvalues of the Laplacian on $\T^2$ may be arbitrary large. It is not hard to show
(see \cite[Chapter 4.4]{Chavel}) that the corresponding space of eigenfunctions
contains a function $f$ with $||f||_2=1$ and $||f||_0 \gtrsim \sqrt{m}$.
}
\end{exr}

\subsection{The Banach indicatrix}\label{subsec-BI} We start the proof of Theorem \ref{thm- length-bars-torus} with
the following topological consideration. For a surface with boundary $A$,
we write $\zeta(A)$ for its total Betti number and $|\partial A|$ for the number of boundary components.

\medskip\noindent
\begin{prop}\label{prop-top} For every two-dimensional submanifold $A \subset \T^2$
with non-empty boundary
\begin{equation} \label{eq-top}
\zeta(A) \leq |\partial A|+2\;.
\end{equation}
\end{prop}

\begin{proof}
Let $A \subset \T^2$ be a two-dimensional submanifold with non-empty boundary.
Each connected component $A_k$ of $A$ is
\begin{itemize}
\item[{i)}] either  diffeomorphic to the  sphere with a number of discs removed,
\item[{ii)}] or  diffeomorphic to the torus with a number of discs removed,
\end{itemize}
and there is at most one component of type (ii). Note that $\zeta(A_k) = |\partial A_k|$ for every
component of type (i) and $\zeta(A_k) = |\partial A_k|+2$ for a component of type (ii).
This yields \eqref{eq-top}.
\end{proof}

\medskip

For a Morse function $f: \T^2 \to \R$ denote by $u(t)$ the number of connected components
of the level set $f^{-1}(t)$. Define {\it the Banach indicatrix}
$$I:= \int_{\min f}^{\max f} u(t) dt\;.$$
The Banach indicatrix or similar quantities were considered as a measure of oscillation of a function
in \cite{Kronrod, Yomdin, Polterovich-Sodin}. Since for regular $t$ we have $u(t) = |\partial M^t|$, where $M^t= \{f \leq t\}$,
Proposition \ref{prop-top} yields
\begin{equation} \label{eq-length-bars-1}
\ell(f) = \text{length}\left(\calB(f) \cap [\min f,\max f]\right) = \int_{\min f}^{\max f} \zeta(M^t)dt \leq 3I\;.
\end{equation}

\subsection{Normal lifts of the level lines}\label{subsec-NL}
Identify the unit tangent bundle
$$U\T^2 := \{(x,\xi) \in T\T^2\;:\; \xi \in T_x\T^2, |\xi|=1\}$$
with the 3-torus $\T^2 \times S^1$, and equip it with the {\it Sasaki metric}
$\rho$ given by $dx_1^2 + dx_2^2 + d\phi^2$, where $(x_1,x_2)$ are the Euclidean coordinates
on $\T^2$ and $\phi$ is the polar angle of the tangent vector $\xi$. Denote by $\gamma$ a connected
component of a regular level of $f$. Choose the length
parameter $s$ along $\gamma$ and consider the normal lift $\tilde{\gamma}(s) =(\gamma(s), n(s))$ to
$U\T^2$ together with the field of positive normals $n= \nabla f/|\nabla f|$.

\medskip
We start with the following calculation.  Denote by $H_f = \partial^2 f/\partial x^2$ the Hessian of $f$ and write
for $||H_f||_{op}$ for its operator norm. In what follows, the dot stands for the derivative
with respect to the natural parameter $s$, and $\nabla$ for the covariant derivative with respect to the  Euclidean Levi-Civita connection.

\medskip
\noindent
\begin{lem} \label{lem-geom} The length of the tangent vector to
$\tilde{\gamma}$ with respect to the Sasaki metric is given by
$$|\dot{\tilde{\gamma}}|^2_{\rho} =
1+\frac{(H_f\dot{\gamma},\dot{\gamma})^2}{|\nabla f|^2}\;.$$
\end{lem}

\begin{proof} Denote $w = \dot{\gamma}$.
Differentiating the identity $(n(s),n(s))=1$ we get that
$(\nabla_w n,n)=0$ and hence $\nabla_w n= (\nabla_w n,w)w$.
Furthermore,
$$\nabla_w n =\nabla_w \frac{\nabla f}{|\nabla f|} =
\frac{1}{|\nabla f|} \nabla_w\nabla f + \Big{(} \nabla (|\nabla
f|^{-1}),w\Big{)}\nabla f\;,$$ and  $H_fw = \nabla_w\nabla
f$. Therefore,
$$\nabla_w n = \frac{(H_f w,w)}{|\nabla f|} w\;.$$
Using now that $|\dot{\widetilde\gamma}|^2_{\rho}= |w|^2 +
|\nabla_w n|^2$, we get the statement of the lemma.
\end{proof}

\medskip

Now comes a crucial observation: the normal lift of any simple closed curve on $\T^2$ is
{\it non-contractible} in $U\T^2$, and hence its $\rho$-length is $\geq 2\pi$.
Therefore, writing $L_t$ for the $\rho$-length of the normal lift of $f^{-1}(t)$,
we have
\begin{equation}
\label{eq-normal-lift}
L_t \geq 2\pi u(t)\;,
\end{equation}
where $u(t)$ stands as above for the number of connected components
of $f^{-1}(t)$.

Applying \eqref{eq-normal-lift} and Lemma \ref{lem-geom} we get that
\begin{align*}
I=\int_{\min f}^{\max f} u(t) dt & \leq (2\pi)^{-1}\int_{\min f}^{\max f} L_t dt \\
& \leq J:= (2\pi)^{-1}\int_{-\infty}^{+\infty} dt\;
\int_{f^{-1}(t)} (1+||H_f||_{op}^2/|\nabla f|^2)^{1/2}\;
ds\;.
\end{align*}
By the co-area formula and Cauchy-Schwarz,
$$J = (2\pi)^{-1} \int_{\T^2} (|\nabla f|^2 + ||H_f||_{op}^2)^{1/2} d\mu \leq  K:= \left( \int_{\T^2} |\nabla f|^2 + ||H_f||_{op}^2 d\mu\right)^{1/2}
 \;,$$
where $d\mu$ stands for the Euclidean area $dx_1 dx_2$.

Observe now that
$$\tr(H_f^2) = (\tr H_f)^2 - 2\det H_f\;,$$
and
$$\det H_f = d(\partial f/\partial x_1) \wedge d(\partial f/\partial x_2)\;.$$
Recalling that $\tr H_f = \Delta f$ and using Stokes formula, we get
$$\int_{\T^2} \tr(H_f^2) d\mu = ||\Delta f||_2^2\;.$$
Since $||H_f||_{op}^2 \leq \tr(H_f^2)$,
it follows that
\begin{equation}\label{eq-BL-1}
\int_{\T^2} ||H_f||_{op}^2 d\mu  \leq ||\Delta f||_2^2  \;.
\end{equation}
Additionally,
$$\left| \int_{\T^2} |\nabla f|^2 d\mu\right|  = \left|\int_{\T^2} f\Delta fd\mu\right|  \leq ||f||_2 ||\Delta f||_2\;.$$
It follows that
$$I \leq K \leq (||f||_2 ||\Delta f||_2 + ||\Delta f||_2^2)^{1/2} \leq ||f||_2 +||\Delta f||_2\;.$$
Combining this with \eqref{eq-length-bars-1}, we get inequality \eqref{eq-length-bars-torus},
and hence complete the proof of Theorem \ref{thm- length-bars-torus}.

\section{Approximation by trigonometric polynomials}
Consider the space of trigonometric polynomials $\calT_\lambda$ on the torus $\T^2$
introduced in Example \ref{exm-trig-1}. For a Morse function $f$ on $\T^2$
address the following question: what is the optimal uniform approximation of $f$
by a trigonometric polynomial from $\calT_\lambda$, maybe after a change of variables by
an area-preserving diffeomorphism of $\T^2$. In other words, writing $\calD $ for the group of area-preserving diffeomorphisms of $\T^2$, we introduce the quantity
$$\delta_\lambda(f) := \inf_{\phi \in  \calD,\; p \in \calT_\lambda} ||f \circ \phi-p||_0\;.$$
The next result formalizes an intuitively clear principle that in order to achieve a good uniform approximation of a highly oscillating function by a polynomial from $\calT_\lambda$, the frequency $\lambda$ should be chosen quite high.

\medskip
\noindent
\begin{thm}\label{thm-approx} For every Morse function $f$ on $\T^2$ and $c > 0$
\begin{equation}\label{eq-approx}
\lambda+1 \geq  \frac{c- 2\delta_\lambda(f)}{3(||f||_2 + 2\pi\delta_\lambda(f))} \cdot \nu(f,c)\;.
\end{equation}
\end{thm}
\begin{proof}
Take a Morse trigonometric polynomial $p \in \calT_\lambda$ with
\begin{equation}
\label{eq-fp-vsp}
||f \circ \phi -p||_0 = \delta\;,
\end{equation}
for some area-preserving diffeomorphism $\phi \in \calD$ of the torus. By \eqref{eq-nu-perturb},
$$\nu(f, c) \leq \nu(p,c-2\delta)\;.$$
Furthermore, since  $||f \circ \phi||_2 = ||f||_2$ (here we use that $\phi$ preserves area),
\eqref{eq-fp-vsp} yields
$$||p||_2 \leq ||f||_2 + 2\pi\delta\;.$$ Therefore, by
\eqref{eq-nu-lap-p}
$$ \nu(f,c) \leq   \frac{3( ||f||_2 + 2\pi\delta)(1+\lambda)}{c-2\delta}\;.$$
This yields \eqref{eq-approx}\;.
\end{proof}

\medskip\noindent
\begin{exm}\label{exm-approx-test}{\rm Let us test this result for the function $f(x) = \sin (nx_1) + \sin (nx_2)$
considered in Example \ref{exm-trig-1} (where this function was called $p$).
Here we address the question what is the minimal possible $\lambda$
such that $\delta_\lambda(f) \leq \delta$. Take $c <2$ and arbitrarily close to $2$. By using calculations from Example \ref{exm-trig-1}, we see that $\nu(f,c) = 2n^2-2$ and $||f||_2 =2\pi$. Substituting $\delta_\lambda(f)=\delta$
into \eqref{eq-approx} one readily shows that if $\delta < 1$,
$\lambda+1 \geq k(\delta)(n^2-1)$, where $k(\delta)$ is a numerical constant depending on $\delta$.
We also see that inequality \eqref{eq-approx} does not yield any non-trivial constraint on $\lambda$ if we take $\delta=1$.
}
\end{exm}

\medskip

We conclude this chapter with an application of approximation theory to barcodes
observed in \cite{Polterovich2-Stojisavljevic}. In \cite{Yudin}
Yudin proved a lower bound for the $C^0$-distance $\text{dist}_{C^0}(f, \calT_\lambda)$ from $f$ to $\calT_\lambda$ in terms of the modulus of continuity of $f$.
Recall that the latter depends on a choice of the scale $r>0$ and is defined as
$$\omega_1(f,r)=\sup_{|t|\leq r} \max_x |f(x+t)-f(x)|\;.$$
Yudin's theorem states that
\begin{equation}
\label{eq-Yudin-1}
\text{dist}_{C^0}(f, \calT_\lambda) \leq 4 \omega_1 \bigg(f,\frac{C_0}{\sqrt{\lambda}} \bigg)\;,
\end{equation}
where the constant $C_0$ is given by $C_0=\sqrt{\Lambda_1(D^2(\frac{1}{2}))}$, where $\Lambda_1(D^2(\frac{1}{2}))$ is the first Dirichlet eigenvalue of $\Delta$ inside the 2-dimensional disk $D^2(\frac{1}{2})$ of radius $\frac{1}{2}$.
Moreover, the trigonometric polynomial $p \in \calT_\lambda$ with
$$||f-p||_0 \leq 4 \omega_1 \bigg(f,\frac{C_0}{\sqrt{\lambda}} \bigg)$$
can be chosen with
\begin{equation}
\label{eq-Yudin-2}
||p||_2 \leq ||f||_2\;,
\end{equation}
see \cite{Polterovich2-Stojisavljevic}.

Recall that in the range of $f$, i.e., in the interval $[\min f,\max f]$, the barcode
consists of $\nu(f)$ finite bars and $3$ infinite bars corresponding to the $0$- and the $1$-dimensional
homology of $\T^2$. Introduce also the average length of  bars in the range of $f$,
$$\ell_{av}(f):= \frac{\ell(f)}{\nu(f)+3}\;.$$

Methods of approximation theory yield that, interestingly enough,
the average bar length of a Morse function $f$ on a flat torus can controlled by the  $L_2$-norm of $f$ and the  modulus of continuity of $f$ on the scale $1/\sqrt{\nu(f)}$:

\begin{thm}
\label{modulcont}
There exist constants $C_0, C_1,C_2>0$ such that for any Morse function $f$ on $\T^2$
with $\nu(f) >0$
\begin{equation}
\label{averagebar}
\ell_{av}(f) \le C_1\|f\|_2 + C_2\omega_1\left(f,\frac{C_0}{\sqrt{\nu(f)}}\right).
\end{equation}
\end{thm}

\begin{proof}
Fix $\lambda >0$ which will be chosen later.
By \eqref{eq-Yudin-1} and \eqref{eq-Yudin-2},
there exists a Morse trigonometric polynomial $p \in \calT_\lambda$ with $||f -p||_0 \leq \delta$
and $||p||_2 \leq ||f||_2$, where $\delta = 4\omega_1(f, C_0/\sqrt{\lambda})$.
Applying \eqref{eq-ell-ell} we get that
$$
\ell(f) \leq \ell(p) + (2\nu(f) + 4)\delta\;.
$$
By \eqref{eq-length-bars-torus} and \eqref{eq-p-lam}
$$\ell(p) \leq \sqrt{\frac{2}{\pi}} \cdot (1+\lambda)||p||_2\;.$$
Combining all these inequalities together, we get that
$$
\ell(f) \leq \sqrt{\frac{2}{\pi}} \cdot (1+\lambda)||f||_2 + (2\nu(f) + 4)\delta\;.
$$
Dividing both sides of this inequality by $\nu(f)+3$ and setting $\lambda = \nu(f)$,
we get \eqref{averagebar} for suitable $C_1,C_2 >0$.
\end{proof}

\part{Persistent homology in symplectic geometry} \label{part: applications_symp}

\chapter{A concise introduction to symplectic geometry} \label{chp-intro-sg}

\section{Hamiltonian dynamics} \label{HD}
The origins of symplectic geometry go back to classical mechanics. Consider the motion
of a mass $m$ particle in the linear space $\R^n$ equipped with the coordinate $q = (q_1, ..., q_n)$
in the field of a potential force.  According to Newton's second law, the equation of motion is given by
\begin{equation} \label{newton}
m \cdot \ddot{q}(t) = -\partial V/\partial q\;,
\end{equation}
where $V(q,t)$ is a time-dependent potential function.
This equation very rarely admits an explicit analytic solution. Symplectic geometry provides a mathematical language
which enables one to develop a qualitative theory of dynamical systems of classical mechanics. Such a theory starts
with the following ingenious construction. Introduce a new {\it momentum} variable $p(t) = m \cdot \dot q(t)$. For a function $H(q,p) = \frac{1}{2m} |p|^2 + V(q)$, usually called total energy function,
\[ \frac{\partial H}{\partial p_i} = \frac{p_i}{m} \,\,\,\,\mbox{and}\,\,\,\, \frac{\partial H}{\partial q_i} = \frac{\partial V}{\partial q_i}. \]
Therefore the second order differential equation (\ref{newton}) can be transferred into a system of first order differential equations
\begin{equation} \label{ham-equ}
{\dot q_i (t) = {\displaystyle\frac{\partial H}{\partial p_i}}, \,\,\dot p_i(t) = {\displaystyle{- \frac{\partial H}{\partial q_i}}}}.
\end{equation}
This system is called {\it a Hamiltonian system}.   The coordinates $(q, p)$ form a vector space $\R^{2n}$, which is called {\it phase space}. By the Hamiltonian equations, in the phase space $\R^{2n}$, the orbit $(q(t), p(t))$ of a moving particle is the integral trajectory of the vector field,
\begin{equation} \label{s-hv}
X_H(q(t), p(t)) = \sum_{i=1}^n \left( \frac{\partial H}{\partial p_i} \frac{\partial}{\partial q_i} - \frac{\partial H}{\partial q_i} \frac{\partial}{\partial p_i} \right),
\end{equation}
called the {\it Hamiltonian vector field}. Consider a (standard) 2-form $\omega_{std} = \sum_{i=1}^n dp_i \wedge dq_i$. Observe the relation $\iota_{X_H} \omega_{std} = -dH$. Therefore, we obtain the following remarkable geometric property of the flow of $X_H$, denoted as $\phi_H^t$ and called the {\it Hamiltonian flow}.

\begin{theorem} \label{LT2} The Hamiltonian flow $\phi_H^t$ preserves $\omega_{std}$. \end{theorem}
\begin{proof} From Cartan's formula
\begin{align*}
\mathcal L_{X_H} \omega_{std}  = d \iota_{X_H} \omega_{std} + \iota_{X_H} d\omega_{std} = d (-dH) + 0  = 0.
\end{align*}
Thus we get the conclusion.
\end{proof}

Notice that two properties of $\omega_{std}$ are necessary here - one is that $\omega_{std}$ is closed, i.e. $d\omega_{std} =0$ and the other is that $\omega_{std}$ is non-degenerate, i.e. the top wedge power $\omega_{std}^n$ is a volume form. This two-form $\omega_{std}$ is called the standard {\it symplectic form} on $\R^{2n}$. Generalizing this local model, we can also consider symplectic forms, i.e. closed non-degenerate differential $2$-forms on manifolds. This geometric structure is studied within symplectic geometry.

As an immediate consequence of Theorem \ref{LT2} we deduce that Hamiltonian flows are {\it conservative}.
\begin{theorem} \label{LT} (Liouville Theorem) The Hamiltonian flow $\phi_H^t$ preserves
the standard volume form ${\rm Vol}  = dp_1 \wedge dq_1 \wedge ... \wedge dp_n \wedge dq_n$ on the
phase space.
\end{theorem}

\medskip
\noindent
Indeed,  ${\rm Vol} = \omega_{std}^n/n!$.

\section{Symplectic structures on manifolds} \label{sec-2}

\begin{dfn} Let $M^{2n}$ be an even-dimensional manifold. A {\it symplectic structure} on $M^{2n}$ is a non-degenerate and closed two-form $\omega$, i.e. $\omega^n$ is a volume form of $M^{2n}$ and $d\omega =0$. The pair $(M^{2n}, \omega)$ is called a symplectic manifold.
\end{dfn}

\begin{ex} \label{ex-sym-mfd} Here are some examples of symplectic manifolds.
\begin{itemize}
\item[(0)] Any area form provides a symplectic structure on an oriented surface.
\item[(1)] As we have seen in the previous section, $(\R^{2n}, \omega_{std})$ is a symplectic manifold. In fact, Darboux Theorem (see Section 3.2 in \cite{MS98}) says any symplectic manifold $(M, \omega)$ is locally modeled by $(\R^{2n}, \omega_{std})$. Explicitly, for any $x \in M$, there exists a neighborhood $U$ of $x$ and map $\phi: U \to \R^{2n}$ such that $\phi^*\omega_{std} = \omega$.
\item[(2)] There is a canonical symplectic structure $\omega$ on the cotangent bundle $T^*M$ of any manifold $M$. Namely, consider the following $1$-form $\lambda$ (called the {\it Liouville form}). For any point $(q,p) \in T^*M$ and $v \in T_{(q,p)}T^*M$, set $\lambda_{(q,p)}(v) = p (\pi_* v)$, where $\pi: T^*M \to M$ is the canonical projection. Then define $\omega = d \lambda$. One can readily check that in $(q,p)$-coordinates $\lambda = pdq$ and $\omega=dp \wedge dq$.
\item[(3)] On the complex projective space $\C P^n$ (any $n \geq 1$), there is a famous {\it Fubini-Study symplectic structure}. First, for $\C^n$ with coordinates $z = (z_1, ..., z_n)$, consider the $2$-form $\omega_{FS} = \frac{\sqrt{-1}}{2} \partial \overline{\partial} \ln (|z|^2 +1)$. Take local charts of $\C P^n$ where each chart $U_i = \{[z_0, ..., z_n] \in \C P^n \,| \, z_i \neq 0\}$. Because $U_i$ can be identified with $\C^n$, Fubini-Study structure on $\C P^n$ is given by gluing $\omega_{FS}$ over each $U_i$.
 \item[(4)] Every complex submanifold of $\C P^n$ is symplectic with respect to the induced Fubini-Study form.
 Non-degeneracy of the restriction of $\omega_{FS}$ to a complex submanifold follows from the fact that bilinear form
 $\omega_{FS}(\xi,J\eta)$ is a Riemannian metric on $\C P^n$ .  Here $J$ stands for the complex structure on $\C P^n$.
   \item[(5)] If $(M_1, \omega_1)$ and $(M_2, \omega_2)$ are symplectic manifolds, then $(M_1 \times M_2, \pi_1^*\omega_1 \oplus (- \pi_2^* \omega_2))$ is also a symplectic manifold, where $\pi_i: M_1 \times M_2 \to M_i$ is a projection.
\end{itemize}
\end{ex}

\begin{dfn} \label{dfn-sympo} A symplectomorphism $\phi: (M_1, \omega_1) \to (M_2, \omega_2)$ is a diffeomorphism such that $\phi^*\omega_2 = \omega_1$. Given a symplectic manifold $(M, \omega)$, denote the group of symplectomorphisms on $M$ as $\Symp(M, \omega)$. \end{dfn}

\begin{remark} Every symplectomorphism $\phi: (M_1, \omega_1) \to (M_2, \omega_2)$ between $2n$-dimensional
symplectic manifolds is volume preserving
with respect to the canonical volume forms $\omega_i^n/n!$ on $M_i$, $i=1,2$.
In general, however, volume preserving diffeomorphisms exhibit more flexible behaviour than symplectomorphisms.
This is highlighted, for instance, by {\it Gromov's non-squeezing theorem}  \cite{Gro85}. We discuss this result
in Section \ref{sec-6}, and prove its version in Section \ref{sec-app} below.   \end{remark}

\section{Group of Hamiltonian diffeomorphisms}
The discussion in Section \ref{HD} can be generalized to the following definition.

\begin{dfn} Let $(M, \omega)$ be a symplectic manifold. Given a compactly supported smooth function $H: M \times [0,1] \to \R$, define the {\it Hamiltonian vector field} $X_H$ as the solution of the equation $\iota_{X_H} \omega = -dH$. The flow $\phi_H^t$ of $X_H$ is called a {\it Hamiltonian flow}. The time-1 map of this flow, $\phi = \phi^1_H$, is called a {\it Hamiltonian diffeomorphism}. Collection of all the Hamiltonian diffeomorphisms on $(M, \omega)$ is denoted as $\Ham(M, \omega)$. \end{dfn}

\begin{ex} Take $S^2 = \{(x,y,z) \in \R^3 \,| \, x^2 + y^2 + z^2 = 1\}$ and let $H: (S^2, \omega_{area}) \to \R$ be $H(x,y,z) = z$. It's easy to check that the Hamiltonian flow generated by $H$ is just rotation along $z$-axis. \end{ex}

\begin{remark} Note that the Hamiltonian vector field $X_H$ does not change when one adds to a Hamiltonian function $H_t$ a time-dependent constant. In order to get rid of this ambiguity, we {\it normalize} $H_t$ as follows.
When $M$ is open, we suppose that $H_t$ is compactly supported and  there exists a compact subset of $M$ containing the support of $H_t$ simultaneously for all $t$. When $M$ is closed, we assume that $H_t$ has zero mean with respect to the canonical volume form, i.e. $\int_M H_t \omega^n/n! = 0$ for all $t$.  In this way, each compactly supported Hamiltonian flow is generated by the unique normalized function.\end{remark}

Let $(M, \omega)$ be a closed symplectic manifold. By the same argument as in the proof of Theorem \ref{LT}, any Hamiltonian diffeomorphism $\phi$ is an element in $\Symp(M, \omega)$. In fact, $\phi \in \Symp_0(M, \omega)$, the identity component of $\Symp(M, \omega)$. The following simple example shows that in general $\Symp(M, \omega)$ is strictly larger than $\Ham(M, \omega)$.

\begin{ex} Consider $\mathbb T^2 = \R^2/\Z^2$ with symplectic structure induced from $\R^2$ and its coordinate is $(q,p)$ mod $1$. Diffeomorphism $\psi^t (q, p) = (q + t, p)$ lies in $\psi^t \in \Symp_0(\mathbb T^2, \omega_{std})$ for every $t$. However, one can prove that $\psi^t \notin \Ham(\mathbb T^2, \omega_{std})$ when $t \notin \Z$, see Section 14.1 in \cite{Pol01}. \end{ex}

The following proposition shows that $\Ham(M, \omega)$ has a group structure under compositions.

\begin{prop} (Proposition 1.4.D in \cite{Pol01}) \label{comp} Let $\phi, \psi \in \Ham(M, \omega)$ be Hamiltonian diffeomorphisms generated by normalized time-dependent Hamiltonian functions $F_t$ and $G_t$, and
let $\phi^t$ be the Hamiltonian flow of $F_t$. Then
\begin{itemize}
\item[(1)] $\phi \circ \psi$ is a Hamiltonian diffeomorphism generated by $F_t(x) + G_t((\phi^t)^{-1}(x))$;
\item[(2)] $\phi^{-1}$ is a Hamiltonian diffeomorphism generated by $-F_t((\phi^t)^{-1}(x))$.
\end{itemize}
\end{prop}

\begin{exercise}
\begin{itemize}
\item[{(i)}] Prove Proposition \ref{comp}.
\item[{(ii)}] Suppose $\phi \in \Ham(M, \omega)$ is generated by function $H_t$. Prove that for any $\theta \in \Symp(M, \omega)$, $\theta^{-1} \circ \phi \circ \theta$ is a Hamiltonian diffeomorphism generated by $H_t \circ \theta$. Deduce that $\Ham(M, \omega)$ is a normal subgroup of $\Symp(M, \omega)$.
    \end{itemize}
     \end{exercise}

The following  fundamental properties on $\Ham(M, \omega)$ were obtained by A.~Banyaga \cite{Ban78}.

\begin{theorem} \label{path} Let $\{\gamma_t\}_{t \in [0,1]}$ be a smooth path in $\Ham(M, \omega)$. Then there exists a (time-dependent) function $F: M \times [0,1] \to \R$ such that for any $x \in M$ and $t \in [0,1]$,
\[ \frac{d}{dt} (\gamma_t(x)) = X_{F_t}(\gamma_t(x)). \]
\end{theorem}

\begin{theorem} \label{simple} Let $(M, \omega)$ be a closed symplectic manifold. Then the group $\Ham(M, \omega)$ is a simple group, i.e., its only normal subgroups are the trivial group and the group itself. \end{theorem}

\section{Hofer's bi-invariant geometry} \label{sec-Hofer}
Let $(M,\omega)$ be a closed symplectic manifold. It is useful to think of $\Ham(M, \omega)$ as a Lie subgroup of the group of diffeomorphisms of $M$. The Lie algebra of $\Ham(M, \omega)$ consists of vector fields $X$ on $M$ such that $X(x) = \frac{d}{dt} |_{t=0} (\gamma_t(x))$ where $\{\gamma_t\}_{t \in [0,1]}$ is a smooth path in $\Ham(M, \omega)$ with $\gamma_0=\mathds{1}_M$, the identity map on $M$. Thanks to Theorem \ref{path}, $X(x) = X_{F(0,x)}$ where $F(t,x)$ is the unique normalized function generating the path $\{\gamma_t\}_{t \in [0,1]}$. Therefore, the Lie algebra of $\Ham(M, \omega)$ is identified with function space $\mathfrak g : = C^{\infty}(M)/\R$. For a quantitative study, we will choose a $L_{\infty}$-norm on $\mathfrak g$: $||F|| = \max_M F - \min_M F$. Let us emphasize that this norm is invariant under the adjoint action of $\Ham(M, \omega)$ on $\mathfrak g$ given by the standard action of diffeomorphisms on functions. It gives rise to a Finsler structure on $\Ham(M, \omega)$. Thus we can define the length of a path in $\Ham(M, \omega)$. Assume the path $\{\gamma_t\}_{t \in [0,1]}$ in $\Ham(M, \omega)$ is generated by a function $F_t$. Define the {\it Hofer length} as
\begin{equation} \label{infty}
{\rm length}(\{\gamma_t\}_{t \in [0,1]}) = \int_0^1||F_t||dt.
\end{equation}
Then the distance between two Hamiltonian diffeomorphisms is defined as follows.
\begin{dfn}\label{dfn-hofer} {\it Hofer's metric} on $\Ham(M, \omega)$ is defined as
\[ \d(\phi, \psi) : = \inf\{{\rm length}(\{\gamma_t\}_{t \in [0,1]}) \,| \, \mbox{$\gamma_t$ connects $\phi$ and $\psi$}\} \]
for any $\phi, \psi \in \Ham(M, \omega)$. Accordingly, {\it the Hofer norm} on $\Ham(M,\omega)$ is defined as $||\phi||_{\rm Hofer} = \d(\phi, \mathds{1}_M)$.
\end{dfn}

\begin{exercise} \label{hofer-dis} Prove $\d$ satisfies the following properties;
\begin{itemize}
\item[(1)] for any $\phi, \psi \in \Ham(M, \omega)$, $\d(\phi, \psi) \geq 0$;
\item[(2)] for any $\phi, \psi \in \Ham(M, \omega)$, $\d(\phi, \psi) = \d(\psi, \phi)$;
\item[(3)] for any $\phi, \psi, \theta \in \Ham(M, \omega)$, $\d(\phi, \theta) \leq \d(\phi, \psi) + \d(\psi, \theta)$;
\item[(4)] for any $\phi, \psi, \theta \in \Ham(M, \omega)$, $\d(\theta \circ \phi, \theta \circ \psi) = \d(\phi \circ \theta, \psi \circ \theta) = \d(\phi, \psi)$, that is, $\d(\cdot, \cdot)$ is bi-invariant under the action of $\Ham(M, \omega)$.
\end{itemize}
\end{exercise}


Recall that, for a space $X$, a function $d: X \times X \to \R$ satisfies the properties as (1) - (3) above is called a {\it pseudo-metric}. A metric $d$ is a pseudo-metric which satisfies a non-degeneracy condition, that is, $d(x,y) >0$ for any $x \neq y$ in $X$.

\medskip
\noindent
\begin{thm}\label{thm-nondeg} (\cite{Hof90,polterovich1993symplectic,lalonde1995geometry})  Hofer's metric $\d$ is non-degenerate.
\end{thm}

\noindent
We outline a proof of this result for closed symplectic manifolds $M$ with $\pi_2(M)=0$ in Section \ref{sec-FHF} below.

\begin{remark} One can define a bi-invariant metric on $\Ham(M,\omega)$ by using $L_p$-norm with $p<\infty$
instead of $L_\infty$ norm. Surprisingly $p=\infty$ is the only choice such that the corresponding $\d(\cdot, \cdot)$ is non-degenerate, see \cite{eliashberg1993bi} and Theorem 2.3.A in \cite{Pol01}. A general result due to L.~Buhovsky and Y.~Ostrover \cite{BO11} states that any bi-invariant Finsler metric on $\Ham(M, \omega)$ with non-degenerate distance is necessarily equivalent to Hofer's metric.
\end{remark}

\begin{exercise} \label{alternative} Show the following dichotomy: given a closed symplectic manifold $(M, \omega)$, any bi-invariant metric on $\Ham(M, \omega)$  is either non-degenerate or vanishes identically. (Hint: use Theorem \ref{simple}.) This exercise brings together algebra and geometry of $\Ham(M, \omega)$.\end{exercise}

\medskip

H.~Hofer \cite{Hof90} used his metric in order to define an interesting invariant of subsets in symplectic
manifolds called {\it the displacement energy}. Call a subset $A \subset M$ displaceable if there exists a
Hamiltonian diffeomorphism $\phi \in \Ham(M,\omega)$ such that $\phi(A) \cap A=\emptyset$. Roughly speaking,
displaceable subsets define a natural small scale on symplectic manifold. Hofer's metric enables one
to quantify the notion of displaceability.

\begin{dfn}\label{dfn-disp} (Displacement energy) For a displaceable subset $A \subset M$, define
\[ e(A) = \inf\{\d(\phi, \mathds{1}_M) \,| \, \phi \in \Ham(M, \omega), \, \phi(A) \cap A = \emptyset\}. \]
\end{dfn}

By (4) in Exercise \ref{hofer-dis}, $e(A) = e(\psi(A))$ for any $\psi \in \Ham(M, \omega)$.

\begin{ex} Let $(M, \omega)$ be any symplectic manifold and subset $A = \{{\rm pt}\}$. We claim that $e(A) =0$.
Indeed, by using Darboux Theorem (see (1) in Example \ref{ex-sym-mfd}), introduce local coordinates $(q_1, ..., q_n,p_1,..., p_n)$ near $A$. The function $H(q_1, ..., q_n, p_1, ..., p_n) = \ep \cdot p_1$ for some $\ep>0$ generates a Hamiltonian diffeomorphism $\phi (q_1, ..., q_n, p_1, ..., p_n) = (q_1 + \ep, q_2, ..., q_n, p_1, ..., p_n)$. Of course it displaces $A$, so by definition $e(A) \leq \ep$. Let $\ep$ go to zero and we get $e(A) = 0$. An appropriate cut-off of $H$ makes this argument rigorous. \end{ex}

In contrast to this, we have the following corollary of Theorem \ref{thm-nondeg} on non-degeneracy of Hofer's
metric.

\begin{cor} (\cite{eliashberg1993bi}) \label{cor-disp}  The displacement energy of any non-empty open displaceable subset is strictly
positive.
\end{cor}

\begin{exr}\label{exr-disp-1} Write $[\phi,\psi]$ for the commutator $\phi\psi\phi^{-1}\psi^{-1}$
of Hamiltonian diffeomorphisms $\phi$ and $\psi$. Show, by using bi-invariance of Hofer's norm, that
$||[\phi,\psi]||_{\rm Hofer} \leq 2||\psi||_{\rm Hofer}$.
\end{exr}

\medskip\noindent
{\bf Proof of Corollary \ref{cor-disp}:}
Take any two non-commuting Hamiltonian diffeomorphisms $f,g$ supported in $A$.
Assume that a Hamiltonian diffeomorphism
$\theta$ displaces an open subset $A \subset M$.
Note that $\theta g^{-1} \theta^{-1}$ is supported in $\theta(A)$ and hence it commutes
with $f$. We leave it is an exercise to check that $[f,g] = [f,[g,\theta]]$. Applying Exercise \ref{exr-disp-1}
twice, we get that
$$4||\theta||_{\rm Hofer} \geq ||[f,g]||_{\rm Hofer}\;.$$
Since this is true for every $\theta$ displacing $A$, we get that
$$e(A) \geq \frac{1}{4}||[f,g]||_{\rm Hofer} > 0\;,$$
where the last inequality follows from non-degeneracy of Hofer's metric.
\qed

\medskip

In fact, another way around, positivity of displacement energy on open subsets
immediately yields non-degeneracy of Hofer's metric. Indeed, if $\phi \neq \mathds{1}_M$, then it must displace some non-empty open subset $A \subset M$, and therefore $\d(\phi, \mathds{1}_M) \geq e(A) >0$.

Let us mention also that definitions of Hofer's metric and of the displacement energy extend in a
straightforward way to open symplectic manifolds. Non-degeneracy of Hofer's metric and positivity
of displacement energy on open displaceable subsets hold in this case as well.

\section{A short tour in coarse geometry} \label{sec-cg}

Coarse geometry is the study of metric spaces from a ``large scale'' point of view. For example, given a metric space $(X, d_X)$, one can construct another metric space called its {\it asymptotic cone} (see Section 2 in \cite{Gro93}) representing the space ``viewed from infinity''. For a group equipped with bi-invariant metric, the asymptotic cone possesses
a natural group structure \cite{calegari2011stable}.  Furthermore, given two metric spaces $(Y, d_Y)$ and $(X, d_X)$, instead of looking for isometric embeddings from $(Y, d_Y)$ to $(X, d_X)$, one can consider {\it quasi-isometric embeddings} (see Definition \ref{dfn-quasi-iso}) to ignore details at small scales. In this section, we focus on $(X, d_X) = (\Ham(M, \omega), \d)$ and $(Y, d_Y) = (\R^n, d_{\infty})$, $n \in \N \cup \{\infty\}$, where $d_{\infty}(v, w) = \max_{i} |v_i - w_i|$ for $v, w \in \R^{n}$. With this language the questions that we are interested can be formulated as follows.

\begin{question} \label{que-2} What are properties of the asymptotic cone of $(\Ham(M, \omega), \d)$?
\end{question}

\begin{question} \label{que-1} Does there exist a quasi-isometric embedding from the metric space $(\R^n, d_{\infty})$ to the $(\Ham(M, \omega), \d)$ for some $n \in \N \cup \{\infty\}$?
\end{question}

The asymptotic cone of $(\Ham(M, \omega), \d)$ was investigated in \cite{egg-team}. By using Floer theory and a chaotic model called {\it egg-beater map} that was constructed in \cite{PS16}, \cite{egg-team} proves that if $M = \Sigma_g$ a symplectic surface with genus $g \geq 4$, then there exists a monomorphism from $\mathbb F_2$, the free group with two generators, into the asymptotic cone of $(\Ham(M, \omega), \d)$. This is the first time that a non-abelian embedding involving Hamiltonian diffeomorphism groups has been discovered, and interested readers are referred to \cite{egg-team} for more details. For a related appearance of egg-beater maps, see the end of Section \ref{sec-full-powers} below.

In what follows, we will address Question \ref{que-1}. Let us start from the following definition.

\begin{dfn} \label{dfn-quasi-iso} Consider two metric spaces $(Y, d_Y)$ and $(X, d_X)$. A map $f: Y \to X$ is called a quasi-isometric embedding if there exist some constants $C  \geq C'>0$ and $A \geq 0$ such that for any $y, y' \in Y$,
\[ C'\cdot d_Y(y,y')  - A \leq d_X(f(y), f(y')) \leq C\cdot d_Y(y,y') + A. \]
\end{dfn}

\begin{ex} There exists a quasi-embedding from $(\R^n, d_{\infty})$ to $(\Z^n, d_{\infty})$ simply sending every $n$-tuple of real numbers $(a_1, ..., a_n)$ to $(\floor{a_1}, ..., \floor{a_n})$. In terms of Definition \ref{dfn-quasi-iso}, $C = C' = 1$ and $A = 1$. On the other hand, obviously the map from $(\Z^n,d)$ to $(\R^n,d)$ sending every $n$-tuple of integers to itself is a quasi-isometric embedding. Therefore, $(\Z^n, d_{\infty})$ and $(\R^n, d_{\infty})$ ``look the same'' on the large scale. \end{ex}

The following result from M.~Usher \cite{Ush13} gives a positive answer to Question \ref{que-1} for certain symplectic manifolds.

\begin{theorem} \label{usher} (Theorem 1.1 in \cite{Ush13}) Let $(M, \omega)$ be a symplectic manifold admitting a nonconstant autonomous function such that all the contractible orbits of its Hamiltonian flow are constant. Then there exists a quasi-isometric embedding from $(\R^{\infty}, d_{\infty})$ to $(\Ham(M, \omega), \d)$. \end{theorem}

\begin{ex} Symplectic surface $(\Sigma_{g\geq 1}, \omega_{area})$ is an easy example satisfying the assumption in Theorem \ref{usher}. Fix a non-contractible loop $\gamma$ in $\Sigma_{g \geq 1}$. Its sufficiently small neighborhood $U$ has a local coordinate $(s, \theta)$ where $s \in (-\ep, \ep)$ and $\theta \in S^1$. Take a smooth function $H(s,\theta) = f(s)$ for some compactly support function $f:(-\ep, \ep) \to \R$ and $H(s,\theta) = 0$ outside $U$. Then its Hamiltonian orbits are either constant or closed curves wrapped around $\gamma$. In particular, they are non-contractible. See page 2-3 in \cite{Ush13} for more complicated examples. Finally we emphasize that Theorem \ref{usher} does {\it not} apply to $(S^2, \omega_{area})$.
\end{ex}

Furthermore, under the hypothesis of Theorem \ref{usher}, the diameter of metric space $(\Ham(M, \omega), \d)$ is infinite. It is a famous conjecture in symplectic geometry that for {\it any} symplectic manifold $(M, \omega)$ the Hofer diameter is infinity. At the moment, it is confirmed for a wide class of symplectic manifolds including, for instance, all symplectic manifolds $(M, \omega)$ with $\pi_2(M) = 0$ and complex projective spaces, see \cite{Sch00}, \cite{Ost03} and \cite{entov_polterovich_calabi_2003}.

\section{Zoo of symplectic embeddings} \label{sec-6}

Consider a ball $B^{2n}(r) = \{(x_1, y_1,..., x_n, y_n) \in \R^{2n} \,| \, \pi \sum_{i=1}^n (x_i^2 + y_i^2) < r\}$ and a cylinder $Z^{2n}(R) = \{(x_1, y_1,..., x_n, y_n) \in \R^{2n} \,| \, \pi(x_1^2 + y_1^2) < R\}$. Celebrated Gromov's non-squeezing theorem is stated as follows.

\begin{theorem} \label{non-squ} Suppose that there exists a symplectic embedding $$\phi: (B^{2n}(r), \omega_{std}) \hookrightarrow (Z^{2n}(R), \omega_{std}).$$ Then $r \leq R$. \end{theorem}
Observe even if $r >R$ there always exists a volume preserving diffeomorphism ``squeezing'' $B^{2n}(r)$ into $Z^{2n}(R)$. Therefore, Theorem \ref{non-squ} shows a certain rigidity phenomenon in symplectic geometry. The original proof of Theorem \ref{non-squ} in \cite{Gro85} is based on the theory of pseudo-holomorphic curves. This theory is regarded as one of the most important tools in symplectic geometry. See \cite{MS12} for a detailed exposition on pseudo-holomorphic curves. We will outline a proof of a weaker version of this theorem in Section \ref{sec-app} below.

\begin{exercise} Define cylinder $Y^{2n}(R) = \{(x_1, y_1,..., x_n, y_n) \in \R^{2n} \,| \, \pi(x_1^2 + x_2^2) < R\}$. Show that for any $R<r$ there exists a symplectic embedding from $(B^{2n}(r), \omega_{std})$ to $(Y^{2n}(R), \omega_{std})$, and hence $Z(R)$ and $Y(R)$ are not symplectomorphic. \end{exercise}

Gromov's non-squeezing theorem is closely related with a class of symplectic invariants called {\it symplectic capacities}. Denote by $\mathcal M(2n)$ the class of all symplectic manifolds possibly with boundaries and of dimension $2n$. In what follows in this section, we use the symbol $(M, \omega_M) \hookrightarrow (N, \omega_N)$ to denote the existence of a symplectic embedding $\phi: (M, \omega_M) \to (N, \omega_N)$.

\begin{dfn} \label{dfn-symp-cap} A symplectic capacity is a map $c: \mathcal M(2n) \to [0, \infty]$
(note that the value $\infty$ is allowed) satisfying the following axioms.
\begin{itemize}
\item[(1)] ({\it monotonicity}) If $(M, \omega_M) \hookrightarrow (N, \omega_N)$, then $c(M, \omega_M) \leq c(N, \omega_N)$.
\item[(2)] ({\it conformality}) For any $\lambda >0$, $c(M, \lambda \cdot \omega) = \lambda \cdot c(M, \omega)$.
\item[(3)] ({\it normalization}) $c(B^{2n}(1), \omega_{std}) = c(Z^{2n}(1), \omega_{std}) = 1$.
\end{itemize}
\end{dfn}

Note that the existence of a symplectic capacity is equivalent to Theorem \ref{non-squ}. Indeed, if there exists a symplectic capacity $c$, then (1)-(3) in Definition \ref{dfn-symp-cap} together tell us that
\[ r = c(B^{2n}(r), \omega_{std}) \leq c(Z^{2n}(R), \omega_{std}) = R, \]
which led to a proof of Theorem \ref{non-squ}. Conversely, one can consider {\it Gromov radius} which is defined as follows,
\[ c_G(M, \omega) := \sup\{r >0 \,| \, \mbox{$\exists$ a symplectic embedding} \,\, (B^{2n}(r), \omega_{std}) \hookrightarrow (M, \omega)\}.\]
It is readily to check that $c_G$ satisfies axioms (1)-(3) in Definition \ref{dfn-symp-cap} above. We leave the details as an exercise to readers.

Sometimes, one considers capacities defined on smaller collections of symplectic manifolds, for instance,
on all open subsets of $\R^{2n}$. For an open subset $U$, put $c_H(U) = \sup e(V)$, where the supremum  is
taken over all {\it bounded} domains $V \subset U$, and $e$ is the displacement energy introduced
in Definition \ref{dfn-disp}. H.~Hofer \cite{Hof90} showed that $c_H$ satisfies
axioms (1)-(3) in Definition \ref{dfn-symp-cap} above, which led to yet another proof of Theorem \ref{non-squ}.

We refer to \cite{CHLS07} for a nice summary of different capacities and their relations.
\medskip

Symplectic embedding problems are usually divided into two classes. The first one is the obstructions to the existence of symplectic embeddings, which often come from certain symplectic capacities. The other one is the constructions of symplectic embeddings, see Schlenk's book \cite{Sch05}. Both problems can be difficult in general. Let us give some examples. Consider an ellipsoid
\[ E(a_1, ..., a_n) = \left\{(x_1, y_1, ..., x_n, y_n) \in \R^{2n} \,\bigg| \, \pi \sum_{i=1}^n \frac{x_i^2 + y_i^2}{a_i} < 1\right\}\]
and a polydisk
\[ P(a_1, ..., a_n) = \left\{(x_1, y_1, ..., x_n, y_n) \in \R^{2n} \,\bigg| \, \pi \cdot\frac{x_i^2 + y_i^2}{a_i} < 1, \,\forall i=1, ..., n\right\}.\]
Let us think of both $E(a_1, ..., a_n)$ and $P(a_1, ..., a_n)$ as elements in $\mathcal M(2n)$ with symplectic structure $\omega_{std}$ induced from $(\R^{2n}, \omega_{std})$.

\begin{ex} \label{emb}
\medskip
\begin{itemize}
\item[(1)] (McDuff \cite{McD11}) $E(a_1,a_2) \hookrightarrow E(b_1,b_2)$ if and only if $N(a_1,a_2) \leq N(b_1,b_2)$ where $N(m,n)$ is the sequence of all nonnegative integer linear combination of $m,n$, written in an increasing order with repetitions. For instance, $N(1,2) = (0,1,2,2,3,3...)$.
\medskip
\item[(2)] (Hutchings \cite{Hut11}) When $1 \leq a \leq 2$, $P(a,1) \hookrightarrow P(b,b)$ if and only if $a \leq b$.
\medskip
\item[(3)] (Hind and Lisi \cite{HL15}) $P(1,2) \hookrightarrow B^4(a)$ if and only if $a \geq 3$.
\end{itemize}
\end{ex}

In Section \ref{sec-app}, combining persistent homology theory with machinery from Floer theory, we are able to associate a barcode to each domain (under some condition of non-degeneracy). The upshot is that some obstructions to the existence of symplectic embeddings can be easily read from these data, which provides a new method to study symplectic embedding problems.

\chapter{Hamiltonian persistence modules} \label{chp7_hamiltonian_persistence_modules}

\section{Conley-Zehnder index}  \label{sec-CZ}

Denote by ${\rm Sp}(2n)$ the group of $2n \times 2n$ symplectic matrices with entries in $\R$, that is, $M \in {\rm Sp}(2n)$ satisfies
\[ M^T \Omega M = \Omega, \,\,\,\,\mbox{where}\,\,\,\,\Omega =  \begin{pmatrix} 0 & -\mathds{1}_n \\  \mathds{1}_n & 0 \end{pmatrix}, \]
and $\mathds{1}_n$ is the $n\times n$ identity matrix. Conley-Zehnder index assigns an integer to a path of symplectic matrices $\Phi: [0,1] \to {\rm Sp}(2n)$ where $\Phi(0) = \mathds{1}$ and $\Phi(1)$ does not have $1$ in its eigenvalues. It is denoted by $\mu_{CZ}(\Phi)$, and it is an important ingredient in the definition of Floer theory. Roughly speaking, the Conley-Zehnder index of the path $\Phi$ is an intersection number between $\Phi$ and the ``cycle'' $\Sigma \subset \text{Sp}(2n)$ consisting of all matrices $A$ possessing $1$ as their eigenvalue. The fact that $\Sigma$
can be considered as a cycle goes back to Arnold's seminal paper \cite{Arn67} on the Maslov index. V.I.~Arnold
showed that $\Sigma$ is a stratified manifold whose top stratum has codimension one in ${\rm Sp}(2n)$ and other strata have codimension $\geq 3$. Furthermore, $\Sigma$ admits a natural co-orientation, and hence can be considered as a cycle representing an element
in cohomology $H^{1}(\text{Sp}(2n),\Z)$ called the Maslov class. An extra difficulty in defining the intersection number
$\Phi \circ \Sigma$ is due to the fact that $\Phi$ starts at $\Sigma$ and may intersect the lower strata. The next definition is from \cite{RS93} and it takes care of these nuances.

In preparation for this definition, for any smooth path of symplectic matrices $\Phi = \{\Phi(t)\}_{t \in [0,1]}$, one considers
\[ S(t) : = \Omega \,\dot{\Phi}(t) \Phi(t)^{-1}. \]
It is easy to check that $\{S(t)\}_{t \in [0,1]}$ is a smooth path of symmetric matrices.

\begin{dfn} \label{dfn-CZ} For a smooth path of symplectic matrices $\Phi: [0,1] \to {\rm Sp}(2n)$, a number $t \in [0,1]$ is called {\it a crossing} if $\det(\Phi(t) - \mathds{1}) = 0$. For any crossing $t \in [0,1]$, the restriction of $S(t)$ to $\ker(\Phi(t) - \mathds{1})$ defines a quadratic form $\Gamma(\Phi, t)$, called {\it crossing form}. We call a crossing $t \in [0,1]$ {\it regular} if $\Gamma(\Phi,t)$ is non-degenerate. Suppose $\Phi$ has only regular crossings, then the Conley-Zehnder index of $\Phi$ is defined by
\[ \mu_{CZ}(\Phi) : = \frac{1}{2} {\rm sign}(\Gamma(\Phi, 0)) +
\sum_{{\tiny \begin{array}{cc} t \in (0,1) \\ \mbox{crossing} \end{array}}}
{\rm sign}(\Gamma(\Phi, t)), \]
where ``{\rm sign}'' denotes the signature of a quadratic form which is equal to the number of positive squares minus the number of negative squares in the canonical form of this quadratic form. \end{dfn}

Observe that $\Phi(0)= \mathds{1}$ implies that $\ker(\Phi(0) - \mathds{1}) = \R^{2n}$, and then $t=0$ is always a crossing. The crossing form at $t=0$ is simply $S(0)$. Moreover, under the assumption that $t=0$ is a regular crossing, $S(0)$ is a non-degenerate quadratic form. Then, for this quadratic form, the number of its positive squares plus the number of its negative squares is equal to $2n$, in particular, an even number. Then ${\rm sign}(\Gamma(\Phi, 0))$ is also even, which implies $\mu_{CZ}(\Phi)$ is always an integer.

Definition \ref{dfn-CZ} assigns the Conley-Zehnder index to a path which has only regular crossings. Meanwhile, it is a standard fact that $\mu_{CZ}(\Phi) = \mu_{CZ}(\Psi)$ if two paths $\Phi$ and $\Psi$ are homotopic with fixed endpoints. Then for any smooth path $\Psi: [0,1] \to {\rm Sp}(2n)$ with the conditions that $\Psi(0) = \mathds{1}$ and $\det(\Psi(1)- \mathds{1}) \neq 0$, define $\mu_{CZ}(\Psi)$ to be $\mu_{CZ}(\Phi)$ where $\Phi$ is any path which is homotopic to $\Psi$ with endpoints fixed and in addition has only regular crossings.

The following example computes the Conley-Zehnder index of a path of symplectic matrices generated by a quadratic Hamiltonian.

\begin{ex} (Harmonic oscillation) \label{basic-CZ} On $\C (\simeq \R^2)$ with the coordinate $z = q + i p$, consider the  Hamiltonian function $H(z) = \pi \alpha |z|^2$ (or $\pi \alpha (q^2 + p^2)$) for some $\alpha \in \R \backslash \Z$. Its Hamiltonian vector field is
\[ X_H(q,p) = \begin{pmatrix} 0 & 2 \pi \alpha \\ - 2 \pi \alpha & 0 \end{pmatrix} \begin{pmatrix} q \\ p \end{pmatrix}, \]
and its flow is the rotation $\phi_H^t(z) = e^{(-2 \pi \alpha t)i} z$. The linearization of this flow defines a smooth path of symplectic matrices $\Phi: [0,1] \to {\rm Sp}(2)$ given by
\[ \Phi(t) = \begin{pmatrix} \cos(2 \pi \alpha t) & \sin(2 \pi \alpha t) \\ - \sin(2 \pi \alpha t) & \cos(2 \pi \alpha t) \end{pmatrix}. \]
Since $\alpha \notin \Z$, $\Phi(1)$ does not have $1$ among its eigenvalues. Observe that $t \in (0,1)$ is a crossing if and only if $t = \frac{k}{ \alpha}$ for some $k \in \Z \backslash \{0\}$. More precisely, when $\alpha<0$, $k \in \{\ceil*{\alpha}, ..., -1\}$, and when $\alpha>0$, $k \in \{1, ..., \floor*{\alpha}\}$. At each crossing $t= \frac{k}{\alpha}$, $\ker(\Phi(t) - \mathds{1}) = \C$, and the associated crossing form is
\begin{align} \label{reg-cro}
\Gamma\left(\Phi, t\right)  =  \begin{pmatrix} 0 & -1 \\ 1 & 0 \end{pmatrix}  \begin{pmatrix} 0 & 2 \pi \alpha \\ -2 \pi \alpha & 0 \end{pmatrix} =  \begin{pmatrix} 2 \pi \alpha & 0 \\0 & 2 \pi \alpha \end{pmatrix}.
\end{align}
So each crossing is regular and ${\rm sign}(\Gamma(\Phi, k/\alpha)) = \pm2$, and $+$ or $-$ depends on the sign of $\alpha$. Moreover, computation in (\ref{reg-cro}) also holds for $k=0$, that is, $t=0$. Hence, by Definition \ref{dfn-CZ},
\begin{equation} \label{ex-CZ}
\mu_{CZ}(\Phi) = \left\{\begin{array}{rcl} (-2)|\ceil*{\alpha}| - 1 & \mbox{if} & \alpha<0 \\ 2\floor*{\alpha} +1 & \mbox{if} & \alpha>0 \end{array} \right. .
\end{equation}
\end{ex}

In what follows we will deal with the following index which is normalized from $\mu_{CZ}$,
\begin{equation} \label{normalize-index}
{\rm Ind}(\Phi) := n - \mu_{CZ}(\Phi).
\end{equation}
\begin{exercise} Suppose a path $\Phi$ is generated by a sufficiently small quadratic Hamiltonian $H$ on $\C^n(\simeq \R^{2n})$, then ${\rm Ind}(\Phi)$ is equal to the Morse index of $H$, that is, the number of negative squares of $H$. (cf. Example \ref{basic-CZ}) \end{exercise}

\begin{exercise} Assume that $\Phi: [0,1] \to \text{Sp}(2n)$ is a smooth {\it loop}, i.e. $\Phi(0) = \Phi(1) =\mathds{1}$ and
$\dot{\Phi}(0)=\dot{\Phi}(1)$. Define the {\it Maslov index}
$$\mu(\Phi) = \sum_{t \in [0,1)} \text{sign}(\Gamma(\Phi,t))\;,$$
where the sum is taking over {\it all} crossings including the point $\Phi(0)$. Prove that
the concatenation $\Psi \sharp \Phi$ of any path
$\Psi$ and a loop $\Phi$ has the Conley-Zehnder index
$${\rm Ind}(\Psi \sharp \Phi) = {\rm Ind}(\Psi) - \mu (\Phi)\;.$$
\end{exercise}

These exercises are very useful for calculation of the Conley-Zehnder index in practice.


\section{Filtered Hamiltonian Floer theory} \label{sec-FHF}
Hamiltonian Floer theory was introduced in Floer's proof of the famous Arnold conjecture on the minimal number of fixed points of a Hamiltonian diffeomorphism on a symplectic manifold, see \cite{Flo89}. It can be regarded as a generalization of the classical Morse theory.

Let us recall its construction on symplectic manifolds $(M, \omega)$ with $\pi_2(M)=0$. For simplicity, we will only consider homology with coefficients in $\Z_2$. Consider the space $\mathcal LM$ of all smooth contractible loops $x: S^1 \to M$. For any $x \in \mathcal LM$, one can take a disc $D \subset M$ spanning $x$ and consider an area functional $\mathcal A(x) = - \int_D \omega$. Due to condition $\pi_2(M) = 0$, $\mathcal A$ is a well-defined function on $\mathcal LM$. Following the basic idea from Morse theory, one would investigate the critical points of $\mathcal A$.  It turns out the critical points of $\mathcal A$ are just constant loops. To overcome this degeneracy, we will perturb $\mathcal A$ in the following way. Fix a time-dependent Hamiltonian $H: \R/\Z \times M \to \R$, and define the {\it symplectic action functional} $\mathcal A_H: \mathcal LM \to \R$ by
\begin{equation} \label{act-fcn}
\mathcal A_H (x) = - \int_D \omega + \int_0^1 H(x)dt,
\end{equation}
where $D$ is any disc spanning $x$. This perturbation will be the major object of our interest in the sense that we will study Morse theory for $\mathcal A_H$ on $\mathcal LM$.

First of all, the tangent space $T_x\mathcal LM$ at $x \in \mathcal LM$ can be identified with the space of tangent vector fields $\xi(t) \in T_{x(t)} M$.

\begin{exercise}\label{ex-diff-A} Prove that
\begin{equation} \label{diff-A}
d\mathcal A_H(\xi) = \int_0^1 dH(\xi) - \omega(\xi, \dot{x}(t))dt.
\end{equation}
\end{exercise}
By the relation $dH = - \omega(X_H, \cdot)$ where $X_H$ is the Hamiltonian vector field of $H$, (\ref{diff-A}) can be rewritten as
\[ d\mathcal A_H(\xi) = \int_0^1 \omega(\xi, X_H - \dot{x}(t)) dt.\]
Then one gets the following famous proposition.
\begin{prop} (Least action principle) \label{prop-lap} An element $x \in \mathcal LM$ is a critical point of $\mathcal A_H$ if and only if $x$ is a contractible 1-periodic orbit of the Hamiltonian flow of $H$. \end{prop}

Denote $P:= \{ \mbox{critical points of $\mathcal A_H$} \}$. Notice that these are the objects which appeared in the Arnold conjecture because fixed points of a Hamiltonian diffeomorphism correspond to 1-periodic orbits of its Hamiltonian flow. To invoke Morse theory, one also needs a metric on $\mathcal LM$. Recall that an {\it almost complex structure $J$} on a manifold $M$ is a smooth field of automorphisms $J_p: T_p M \to T_p M$ such that $J_p^2 = -\mathds{1}$ for any $p \in M$. For a symplectic manifold $(M, \omega)$, an almost complex structure $J$ is called {\it $\omega$-compatible} if $\omega(\cdot,J\cdot)$ defines a Riemannian metric on $M$. Denote by $\mathcal J(M, \omega)$ the collection of all $\omega$-compatible almost complex structure. A standard fact, due to M.~Gromov \cite{MS98}, is that $\mathcal J(M, \omega)$ is non-empty and contractible. Now choose a loop $J(t)$ of $\omega$-compatible almost complex structure on $(M, \omega)$. For any $x \in \mathcal LM$ and vector fields $\xi, \eta \in T_x \mathcal LM$, define a metric on $\mathcal LM$ by
\begin{equation} \label{dfn-metric}
\left< \xi(t), \eta(t) \right> : = \int_0^1 \omega(\xi(t), J(t) \eta(t)) dt.
\end{equation}

A closed orbit $x \in P$ is called {\it non-degenerate} if the differential
$\phi_* : T_{x(0)}M \to T_{x(0)}M $  of the time-one map $\phi=\phi_H^1$ of the Hamiltonian flow of $H$ at the fixed point $x(0)$ does not contain $1$ in its eigenvalues. Geometrically, this means that the graph of $\phi$ is transversal to the diagonal at $(x,x)$. We say that $H$ and $\phi$ are non-degenerate if this property is satisfied for all orbits from $P$.  Note also that the non-degeneracy of $x \in P$ in terms of linearization of the Hamiltonian flow is in fact equivalent to the non-degeneracy of $x \in P$ as a critical point of symplectic action functional $\mathcal A_H$.

Let $x \in P$ be a closed orbit of a non-degenerate Hamiltonian diffeomorphism $\phi=\phi_H^1$.
Choose any spanning disc $w: D^2 \to M$ with $w|_{S^1}=x$, where we identify $S^1 = \partial D^2$.
Since $w^*TM$ is a symplectic vector bundle over a contractible base space, there exists a trivialization $w^*TM \simeq D^2 \times (\R^{2n}, \omega_0)$. Under this trivialization, the linearization of flow $\phi_H^t$ at $x(0)$ gives rise to a smooth path $\Phi: [0,1] \to {\rm Sp}(2n)$ such that $\Phi(0) = \mathds{1}$ and $\Phi(1)$ does not contain $1$ in its eigenvalues. Definition \ref{dfn-CZ} assigns the Conley-Zehnder index to the orbit $x$. Under our normalization (\ref{normalize-index}), we denote the index of a 1-periodic Hamiltonian orbit $x \in P$ by ${\rm Ind}(x) := {\rm Ind}(\Phi)$. It can be shown that ${\rm Ind}(x)$ is independent of the choice of trivializations. Moreover, under our assumption $\pi_2(M) = 0$, it is also independent of the spanning disc of $x$.

Next, for any $x,y \in P$, with respect to the metric defined in (\ref{dfn-metric}), one can consider the space of gradient trajectories of $\mathcal A_H$ from $x$ to $y$, denoted by $\widetilde{\mathcal M}(x,y)$. Notice that any such gradient trajectory is actually a cylinder $u(s,t): \R \times \R/\Z \to M$ satisfying the equation
\begin{equation}\label{CR-eqn}
\frac{\partial u}{\partial s} + J_t(u) \frac{\partial u}{\partial t} - \nabla H_t(u) = 0
\end{equation}
with asymptotic conditions $\lim_{s \to \infty} u(s,t) = y(t)$ and $\lim_{s \to -\infty} u(s,t) = x(t)$. This is a perturbed version of Cauchy-Riemann equation, more precisely, $u$ is a {\it pseudo-holomorphic curve}, see \cite{MS12}. It is a great insight by M.~Gromov in his famous paper \cite{Gro85} that methods from algebraic geometry can be generalized if one replaces complex structures with $\omega$-compatible almost complex structures as in (\ref{CR-eqn}). Then the classical theory of holomorphic curves extends to this non-integrable situation, which remarkably revolutionized symplectic geometry in the past few decades.

Observe that there exists an $\R$-action on $\widetilde{\mathcal M}(x,y)$ simply by $T \cdot u(s,t) = u(s+T,t)$ for any $T \in \R$. Then one can consider the moduli space $\mathcal M(x,y) : = \widetilde{\mathcal M}(x,y)/\R$. A crucial and highly non-trivial fact is that generically $\mathcal M(x,y)$ is a compact finite dimensional manifold of dimension ${\rm Ind}(x) - {\rm Ind}(y)-1$. In particular, if ${\rm Ind}(x) - {\rm Ind}(y) = 1$, then $\mathcal M(x,y)$ is a collection of finite many points.  Put $n(x,y) = \# \mathcal M(x,y) \,\mbox{mod $\Z_2$}$.

Finally, we assemble all the ingredients above to formulate the following version of Morse theory, which we call {\it Hamiltonian Floer theory}. Fix a degree $k \in \Z$, and denote
\[ \CF_k(M, H) = {\rm Span}_{\Z_2} \left<x \in P \,| \, {\rm Ind}(x) = k \right>. \]
Consider a $\Z_2$-linear map $\partial_k: \CF_k(M, H) \to \CF_{k-1}(M, H)$ defined by
\begin{equation} \label{floer-diff}
\partial_k x = \sum_{y \in P, \,\,{\rm Ind}(y) = k-1} n(x,y) y.
\end{equation}
It turns out that $\partial$ is a differential, i.e., $\partial^2 =0$. Moreover, any generator $y$ which appears on the right-hand side of (\ref{floer-diff}) has symplectic action $\mathcal A_H(y) < \mathcal A_H(x)$. Denote the Hamiltonian Floer homology by $\HF_k(H) = \frac{\ker(\partial_k)}{{\im}(\partial_{k+1})}$ for any $k \in \Z$.

Similarly to the classical Morse theory, it is easy to add an extra ingredient, filtration, into this new homology theory. For any $\lambda \in \R$ and degree $k \in \Z$, denote
\[ \CF^{\lambda}_k(M, H) = {\rm Span}_{\Z_2} \left< x \in P \, | \, {\rm Ind}(x) = k, \,\,\mbox{and}\,\, \mathcal A_H(x) < \lambda \right>. \]
Since $\partial_k$ strictly decreases the symplectic action, the differential $\partial_k: \CF^{\lambda}_k(M, H) \to \CF_{k-1}^{\lambda}(M,H)$ is a well-defined $\Z_2$-linear map. Denote the {\it filtered} Hamiltonian Floer homology by
\[ \HF_k^{\lambda}(H) : = \frac{\ker(\partial_k: \CF^{\lambda}_k(M, H) \to \CF_{k-1}^{\lambda}(M,H))}{{\im}(\partial_{k+1}: \CF^{\lambda}_{k+1}(M, H) \to \CF_{k}^{\lambda}(M,H))}. \]
For any $\lambda \leq \eta$, there is a well-defined map $\iota_{\lambda, \eta}: \HF_k^{\lambda}(H) \to \HF_k^{\eta}(H)$ induced by the inclusion $\CF_{k}^{\lambda}(M,H) \to \CF_{k}^{\eta}(M,H)$. It is easy to see that for any $\lambda \leq \eta \leq \theta$, $\iota_{\lambda, \theta} = \iota_{\eta, \theta} \circ \iota_{\lambda, \eta}$.

Recall that a Hamiltonian $H: M \times S^1 \to \R$ is called normalized if
$$\int_M H(\cdot,t)\omega^n = 0 \;\;\forall t \in S^1\;.$$
A  remarkable fact due to M.~Schwarz \cite{Sch00} is that for normalized Hamiltonians,  filtered Hamiltonian Floer homology $\HF_k^{\lambda}(H)$ only depends on $\phi = \phi^1_H$, the time-1 map of Hamiltonian flow $\phi^t_H$ generated by $H$. We shall denote this homology by $\HF_k^{\lambda}(\phi)$.

This discussion leads to the following definition.

\begin{dfn} \label{dfn-HFM} Given a symplectic manifold $(M, \omega)$ with $\pi_2(M)=0$, a Hamiltonian diffeomorphism $\phi = \phi_H^1$ generated by some Hamiltonian function $H: \R/\Z \times M \to \R$ and a degree $\ast \in \Z$, the collection of data $\{\{\HF_*^{\lambda}(\phi)\}_{\lambda \in \R}; \{\iota_{\lambda, \eta}\}_{\lambda \leq \eta} \}$ is called a {\it Hamiltonian persistence module in degree $\ast$}, denoted by $\mathbb {HF}_*(\phi)$. The barcode of $\mathbb{HF}_*(\phi)$ is denoted by $\mathcal B_*(\phi)$, and $\mathcal B(\phi) = \cup_{\ast \in \Z} \mathcal B_*(\phi)$. \end{dfn}

\begin{ex} \label{ex-special} (i) Let $(M, \omega)$ be a compact symplectic manifold with $\pi_2(M) =0$ and $H$ be a $C^{\infty}$-small autonomous Morse function with the zero mean. In this case, 1-periodic Hamiltonian orbits are constant loops and they are in bijection with critical points of $H$. Moreover, one can show the Hamiltonian Floer complex reduces  to the standard Morse complex. Then barcode $\mathcal B(\phi)$ where $\phi = \phi_H^1$ is simply the barcode of the corresponding filtered Morse homology. Since $M$ is a compact, any such Morse function $H$ has a global maximum $A = \max_{M} H$ and a global minimum $B = \min_{M} H$. In particular, $\mathcal B(\phi)$ contains two infinite length bars $[A, \infty)$ and $[B, \infty)$.

(ii) A more special case than (1) above is $H \equiv 0$ which generates Hamiltonian diffeomorphism $\phi_H^1 = \mathds{1}_M$, the identity map on $M$. Since $H$ is degenerate, we cannot apply the theory developed above
directly and regard it as the limit of arbitrarily small Morse functions $H_i$. We define its barcode $\mathcal B(\mathds{1}_M)$ as the limit of $\mathcal B(\phi_{H_i}^1)$ in the bottleneck
distance. Then it is easy to see that $\mathcal B(\mathds{1}_M)$ contains only bar $[0, \infty)$ with multiplicity
$\sum_{i} b_i(M)$, the total Betti number of $M$.
\end{ex}

Recall Hofer's metric $\d$ defined in Definition \ref{dfn-hofer}. The following theorem brings algebra and dynamics together.

\begin{theorem}\label{DST}(Dynamical Stability Theorem \cite{PS16}) \label{thm-dst} Let $(M, \omega)$ be a symplectic manifold with $\pi_2(M) =0$. For any pair of non-degenerate Hamiltonian diffeomorphisms $\phi, \psi \in \Ham(M, \omega)$, $d_{bot}(\mathcal B(\phi), \mathcal B(\psi)) \leq d_{\rm Hofer}(\phi, \psi)$. \end{theorem}

\medskip\noindent
An immediate consequence of Theorem \ref{DST} is the non-degeneracy of Hofer's metric $\d$
for symplectic manifolds with $\pi_2=0$, see Theorem \ref{thm-nondeg}.

\begin{cor} Let $(M, \omega)$ be a compact symplectic manifold with $\pi_2(M) =0$. If a Hamiltonian diffeomorphism $\phi  \in \Ham(M, \omega)$ is not identity $\mathds{1}_M$, then $\d(\phi, \mathds{1}_M) >0$. \end{cor}
\begin{proof} Exercise \ref{alternative} says we have a dichotomy that either $d_{\rm Hofer}$ is non-degenerate or vanishes identically. Take a $C^{\infty}$-small autonomous Morse function $H$ with the zero mean,
and denote $\phi = \phi_H^1$. Part (i) in Example \ref{ex-special} says $[A, \infty) \in \mathcal B(\phi)$ where $A = \max_M H >0$. Then (ii) in Example \ref{ex-special} together with Theorem \ref{DST} implies the following inequality,
\[ 0 < A \leq d_{bot}(\mathcal B(\phi), \mathcal B(\mathds{1}_M))
	\leq d_{\rm Hofer}(\phi, \mathds{1}_M). \]
This rules out the case of being vanished identically, therefore, $d_{\rm Hofer}$ is non-degenerate.
\end{proof}

The key to the proof of Theorem \ref{DST} is the following well-known result in Hamiltonian Floer theory, see Section 6 in \cite{SZ92}. For two Hamiltonian functions $H, G: \R/\Z \times M \to \R$, denote $E_{H, G} := \int_0^1 \max_M(G-H)(t, \cdot) - \min_M (G -H)(t,  \cdot) dt$.

\medskip\noindent
\begin{theorem} \label{SZ92-result} Consider a symplectic manifold $(M, \omega)$ with $\pi_2(M) =0$ and two Hamiltonian functions $H, G$. For any $\lambda \in \R$ and degree $\ast \in \Z$, there exist chain maps $\phi_{\lambda}: \CF_*^{\lambda}(H) \to \CF_*^{\lambda + E_{H, G}}(G)$ and $\psi_{\lambda}: \CF_*^{\lambda}(G) \to \CF_*^{\lambda+ E_{H, G}}(H)$ such that $\psi_{\lambda + E_{H, G}} \circ \phi_{\lambda}$ is homotopic to the inclusion $\CF_*^{\lambda}(H) \hookrightarrow \CF_*^{\lambda + 2 E_{H,G}}(H)$ and $\phi_{\lambda + E_{H, G}} \circ \psi_{\lambda}$ is homotopic to the inclusion $\CF_*^{\lambda}(G) \hookrightarrow \CF_*^{\lambda + 2 E_{H,G}}(G)$. \end{theorem}

\begin{proof} (Proof of Theorem \ref{DST}) Suppose $\phi$ is generated by $H$ and $\psi$ is generated by $G$. By Theorem \ref{SZ92-result}, for $\lambda \in \R$ and degree $* \in \Z$, there exist maps
\[ \Phi_\lambda: \HF_*^\lambda(H) \to \HF_*^{\lambda + E_{H,G}}(G)\,\,\,\,\mbox{and} \,\,\,\, \Psi_\lambda: \HF_*^{\lambda}(G) \to \HF_*^{\lambda + E_{H,G}}(H) \]
such that the following diagrams commute,
\[ \xymatrixcolsep{5pc} \xymatrix{
\HF_*^\lambda (H) \ar[r]^-{\Phi_\lambda} \ar@/_1.5pc/[rr]_{\iota_{\lambda, \lambda + 2E_{H,G}}} & \HF_*^{\lambda + E_{H,G}}(G) \ar[r]^-{\Psi_{\lambda+E_{H,G}}} & \HF_*^{\lambda+ 2E_{H,G}}(H)}
\]
and
\[\xymatrixcolsep{5pc} \xymatrix{ \HF_*^\lambda (G) \ar[r]^-{\Psi_\lambda} \ar@/_1.5pc/[rr]_{\iota_{\lambda, \lambda + 2E_{H,G}}} & \HF_*^{\lambda + E_{H,G}}(H) \ar[r]^-{\Phi_{\lambda+E_{H,G}}} & \HF_*^{\lambda+2E_{H,G}}(G)}.
\]
In terms of Definition 1.3.1 in Chapter 1, $\mathbb{HF}_*(\phi)$ and $\mathbb{HF}_*(\psi)$ are $E_{H, G}$-interleaved. Then by Isometry Theorem,
\[ d_{bot}(\mathcal B_*(\phi), \mathcal B_*(\psi))  = d_{int}(\mathbb{HF}_*(\phi), \mathbb{HF}_*(\psi)) \leq E_{H,G}. \]
Finally, inequality $d_{bot}(\mathcal B(\phi), \mathcal B(\psi)) \leq \max_{*\in \Z} d_{bot}(\mathcal B_*(\phi), \mathcal B_*(\psi))$ implies the desired conclusion.
\end{proof}

\medskip

In general, the information extracted from the Floer homological barcode $\mathcal B(\phi)$
of a Hamiltonian diffeomorphism $\phi$ is closely related to some invariants which have been intensively studied in symplectic topology. For example, spectral invariants of $\phi$ introduced by C.~Viterbo \cite{Vit92}, M.~Schwarz \cite{Sch00}  and Y.-G.~Oh \cite{Oh05} can be read directly from the left endpoints of infinite bars.
Let $\phi$ be any non-degenerate Hamiltonian diffeomorphism of a closed symplectic manifold
with $\pi_2(M)=0$. Observe that for large values of $\lambda$ the filtered Floer homology ${\rm HF}^\lambda_*(\phi)$
of $\phi$ coincides with the homology $H_*(M,\calF)$ of the manifold. Let $c(\phi,\cdot): H_*(M,\calF) \to \R$
be the associated characteristic exponent introduced in Chapter \ref{chp3_read_from_a_barcode}.  Its spectrum
coincides with the set of  spectral invariants . It is a non-trivial fact that the maximal spectral invariant equals $c(\phi,[M])$ and the minimal one equals $c(\phi, [pt])$, where $[M]$ is the fundamental class of $[M]$ and $[pt]$ is the class of the point. The difference
\begin{equation}\label{eq-spectral-norm}
\gamma(\phi)= c(\phi,[M])- c(\phi, [pt])\;,
\end{equation}
which is called {\it the spectral norm} of $\phi$, defines an interesting geometry on ${\rm Ham}(M, \omega)$.
In particular,  S.~Seyfaddini \cite{Sey13} proved that the spectral norm is continuous in $C^0$-topology
on Hamiltonian diffeomorphisms.

It is an immediate consequence of Theorem \ref{DST} and Corollary \ref{cor: cor_matching_lower_bound_on_botneck_dist}
that $c(\phi,[M])$, $c(\phi, [pt])$ and the spectral norm are Lipschitz in Hofer's metric.

Let us mention also that a recent paper by A.~Kislev and E.~Shelukhin \cite{KS18} extends Theorem \ref{DST}
to the spectral norm:
\begin{equation} \label{c0-sn}
d_{bot}(\mathcal B(\phi), \mathcal B(\psi)) \leq \frac{1}{2} \gamma(\psi^{-1} \circ \phi)\;.
\end{equation}

Another example of an invariant of a Hamiltonian diffeomorphism contained in its barcode is the  \emph{boundary depth},
that is the length of the longest finite bar, see Chapter \ref{chp3_read_from_a_barcode} above.
It was first introduced by M.~Usher in \cite{Ush11}, \cite{Ush13}.

\medskip

\begin{remark} Observe when $\pi_2(M) \neq 0$, the value of symplectic action functional (\ref{act-fcn}) on a contractible Hamiltonian 1-periodic orbit $x$ may depend on its spanning disk. To overcome this difficulty, \cite{HS95} studied an extended version of Hamiltonian Floer theory. Later on, \cite{UZ16} constructed barcodes in this Hamiltonian Floer theory and proved a stability result as an analog of the classical Isometry Theorem in persistent homology theory.  \end{remark}

\section{Constraints on full powers}\label{sec-full-powers}
Let us start this section with the following example of a persistence representation (see Section \ref{sec: persistence_modules_with_involution}) coming from Hamiltonian Floer theory.

\begin{ex} \label{zp-ham} Let $(M, \omega)$ be a symplectic manifold with $\pi_2(M)=0$ and $\phi \in \Ham(M, \omega)$ where $\phi = \phi_H^1$ generated by some Hamiltonian function $H: \R/\Z \times M \to \R$.
Every Hamiltonian diffeomorphism $\theta \in \Ham(M, \omega)$ induces a ``push-forward'' morphism
 there exists a well-defined morphism between filtered Hamiltonian Floer homologies,
\begin{equation} \label{conj-1}
(P_{\theta})_*:  \HF_*^{\lambda}(\phi) \to  \HF_*^{\lambda}( \theta  \circ  \phi  \circ \theta^{-1}).
\end{equation}
To see this, observe that the Hamiltonian diffeomorphism $\theta \circ \phi \circ \theta^{-1}$ is generated by the Hamiltonian function $H':= H \circ \theta^{-1}$. The diffeomorphism $\theta$ extends to the loop space
$\calL M$ by $x(t) \mapsto \theta(x(t))$. One easily checks that $\theta^* \calA_{H'} = \calA_H$.
Furthermore, $\theta$ sends the loop of compatible almost-complex structures
$J(t)$ on $M$ to another such loop, $J'(t)$. In this way $\theta$ identifies the Floer complex of $H$ associated $J(t)$ with the one of $H'$ associated to $J'(t)$.

Observe that 1-periodic orbits of the Hamiltonian flow $\phi_H^t$ are in one to one correspondence with 1-periodic orbits of the Hamiltonian flow $\theta \circ \phi_H^t  \circ \theta^{-1}$. Explicitly, $x(t)$ corresponds to $\theta(x(t))$. Furthermore, $\theta$ acts in a natural way on loops of almost-complex
structures defining the metric on the loop space, and hence

Let us focus now on a particular case when $\theta = \psi$ and $\phi = \psi^p$. The equality $\psi \circ \psi^p \circ \psi^{-1} = \psi^p$ and (\ref{conj-1}) imply a morphism
\begin{equation} \label{push}
(P_{\psi})_*: \HF_*^{\lambda}({\psi}^p) \to \HF_*^{\lambda}({\psi}^p).
\end{equation}
One can check that $(\mathbb{HF}_*(\psi^p), (P_{\psi})_*)$ is a persistence representation of group $G = \Z_p$. What needs to be emphasized is that $\psi$ acting on itself, i.e., $\psi \circ  \psi \circ \psi^{-1} = \psi$, only induces the identity map on the Hamiltonian Floer homology, but $\psi$ acting on higher power $\psi^p$ where $p \geq 2$ sometimes generates non-trivial morphisms. In fact, under a homotopy argument as in Lemma 3.1 in \cite{PS16}, $(P_{\psi})_*$ is the same as the morphism induced by the loop rotation $x(t) \to x(t + 1/p)$.
\end{ex}

Note that Example \ref{zp-ham} demonstrates that higher powers of Hamiltonian diffeomorphisms can admit non-trivial automorphisms on Hamiltonian persistence modules. Therefore, it will be interesting to investigate which Hamiltonian diffeomorphisms can be written as a full $p$-th powers (of another Hamiltonian diffeomorphism) for $p \geq 2$.
For the sake of simplicity, we shall focus on the case $p=2$, thus addressing the question about obstructions
to existence of square roots in the context of Hamiltonian diffeomorphisms.


In the set-up of diffeomorphisms, J.~Milnor \cite{Mil83} found an obstruction for a diffeomorphism $\phi$ of a manifold $M$ to be a full square. Given a diffeomorphism $\phi: M \to M$, consider the space $X(\phi)$ of its {\it primitive $2$-periodic orbits}. By definition, an element of $X(\phi)$
is given by a  non-ordered pair of distinct points $(x,y) \in M$ with $\phi x=y$ and $\phi y=x$.
J.~Milnor observed that if $\phi=\psi^2$ and the set $X(\phi)$ is finite, {\it  it necessarily contains
an even number of elements}. Indeed, $\psi$ acts on $X(\phi)$ by sending $(x,y)$ to $(\psi(x),\psi(y))$,
and it is an elementary exercise to show that this action is free.

P.~Albers and U.~Frauenfelder \cite{AF14} extended Milnor's approach in the context Hamiltonian diffeomorphisms. In what follows, we will use the barcodes of Hamiltonian persistence modules developed in Section \ref{sec-FHF} to provide an obstruction of a Hamiltonian diffeomorphism to be a full square. For general $p$, we leave it as an exercise to interested readers.

Let $(M, \omega)$ be a closed symplectic manifold with $\pi_2(M)=0$. By Example \ref{zp-ham}, $(\mathbb{HF}_*(\phi^2), (P_{\phi})_*)$ is a persistence representation of $\Z_2$. Consider eigenvalue $\xi = -1$ in Example \ref{ex-power}. The eigenspaces $(L_{-1})_t$, $t \in \R$ forms a persistence subrepresentation of $(\mathbb{HF}_*(\phi^2), (P_{\phi})_*)$, denoted by $\mathbb L_{-1}(\phi^2)$. If, in addition, $\phi$ is a full power $\phi = \psi^2$, then Example \ref{zp-ham} says $(P_{\psi})_*$ is a $\Z_4$-action on persistence module $\mathbb{HF}_*(\phi^2)$. Meanwhile, it is easy to see that
\[ (P_{\psi})^2_* = (P_{\phi})_* = -\mathds{1} \,\,\,\,\mbox{on $\mathbb L_{-1}(\phi^2)$}. \]
In other words, $(P_{\psi})_*$ restricts to a complex structure on $\mathbb L_{-1}(\phi^2)$. By using the multiplicity function defined in Section \ref{sec: the_multiplicity_function_and_Z_k_persistence_modules}, Claim \ref{claim: d_int_to_complex_bounded_below} implies the following obstruction.

\begin{prop}\label{prop-obst} For any bar ${\rm I} \in \mathcal B(\mathbb L_{-1}(\phi^2))$, the multiplicity of ${\rm I}$ is even. \end{prop}

This can be viewed as a Hamiltonian dynamics analog to Milnor's obstruction. \\

\medskip\noindent
{\bf Outlook.} It has been conjectured in \cite{PS16} that for every $p \geq 2$, the complement
of the set
\[ {\rm Power}_p (M) = \{ \phi = \psi^p\,| \, \psi \in \Ham(M) \} \]
contains an arbitrarily large Hofer ball. This conjecture was confirmed in \cite{PS16,Zha16,PSS17}
by using the obstruction described in Proposition \ref{prop-obst} for certain symplectic manifolds, including
closed surfaces of genus $\geq 4$. Recently A.~Chor handled the case of surfaces of genus $2$ and $3$
(unpublished). The problem is still open for most manifolds, including the two-dimensional sphere and torus.

Let us mention also that the set ${\rm Power}_p (M)$ contains all autonomous Hamiltonian diffeomorphisms,
i.e. the ones generated by time-independent Hamiltonian functions. In dimension $2$, the energy conservation
law guarantees that the level curves of the autonomous Hamiltonian are invariant under the Hamiltonian flow.
Therefore, autonomous Hamiltonian diffeomorphisms exhibit deterministic dynamical behavior and  provide the simplest example of integrable systems of classical mechanics. This suggests that one should look for  Hamiltonian diffeomorphisms lying far from the autonomous ones among chaotic dynamical systems. And indeed, the centers of large Hofer balls lying in the complement of ${\rm Power}_p (M)$ can be chosen as chaotic maps known as egg-beaters (see \cite{PS16}).

\section{Non-contractible class version} \label{sec-ham-non}
The Hamiltonian Floer homology in Section \ref{sec-FHF} is constructed from contractible loops. There is a different version of Hamiltonian Floer theory that is constructed from non-contractible loops. This will be used in \Cref{chp8_symplectic_persistence_modules}. Here we give a brief description of its construction, and interested readers can check Section 2 in \cite{PS16} for more details.

Given a symplectic manifold $(M, \omega)$ with $\pi_2(M) = 0$, fix a non-zero homotopy class of the free loop space $\alpha \in \pi_0(\mathcal LM)$, and denote $\mathcal L_{\alpha}(M) = p^{-1}(\alpha)$ where $p: \mathcal LM \to \pi_0(\mathcal LM)$ is the natural projection. Assume that $\alpha$ satisfies the following {\it symplectically atoroidal} condition, for any loop in $\mathcal L_{\alpha} M$ which is a topological torus $\rho: \mathbb T^2 \to M$,
\begin{equation} \label{atoroidal}
\int_{\mathbb T^2} \rho^*\omega = \int_{\mathbb T^2} \rho^* c_1 = 0,
\end{equation}
where $c_1 = c_1(TM, \omega)$ is the first Chern class of $(M, \omega)$. Under this assumption, we will use elements from $\mathcal L_{\alpha} M$ to construct another version of Hamiltonian Floer homology. The different aspect in this version is that first a reference point $x_{\alpha} \in \mathcal L_{\alpha}M$ will be fixed, and all the ingredients in the construction of this Hamiltonian Floer homology will be defined in a relative sense, i.e., relative to the data from $x_{\alpha}$.

For any time-dependent Hamiltonian function $H: \R/\Z \times M \to \R$, define the {\it symplectic action functional} $\mathcal A_H: \mathcal L_{\alpha} M \to \R$ by
\[ \mathcal A_H(x) = - \int_{\bar{x}} \omega + \int_0^1 H(t, x(t))dt,\]
where $\bar{x}$ is any cylinder connecting $x$ and the reference point $x_{\alpha}$. This is similar to the symplectic action functional defined in (\ref{act-fcn}) in Section \ref{sec-FHF}. Observe that the first condition in (\ref{atoroidal}) implies that $\mathcal A_H(x)$ is independent of the choice of cylinder $\bar{x}$. Moreover, similarly to Proposition \ref{prop-lap} (Least action principle), one can check that $x \in \mathcal L_{\alpha}M$ is a critical point of $\mathcal A_H$ if and only if $x$ is a 1-periodic orbit of the Hamiltonian flow of $H$ such that $[x] = \alpha$. Denote by $P_{\alpha}(H)$ the collection of all 1-periodic orbits of the Hamiltonian flow of $H$ in the class $\alpha$.

Furthermore, the grading of $x \in P_{\alpha}(H)$ is also well-defined. Explicitly, we first choose a non-canonical trivialization of the symplectic vector bundle $x^*_{\alpha}TM$ over $S^1$. Then any cylinder $\bar{x}$ connecting $x\in P_{\alpha}(H)$ with $x_{\alpha}$ defines a trivialization of $x^*TM$. Based on the machinery developed in Section \ref{sec-CZ}, one can compute the Conley-Zehnder index of $x$. Under the normalization in (\ref{normalize-index}), this defines the index of $x \in P_{\alpha}(H)$. Observe that the second condition in (\ref{atoroidal}) implies that this index is independent of the choice of $\bar{x}$.

For a fixed homotopy class $\alpha \in \pi_0(\mathcal LM)$, a filtration $\lambda \in \R$ and degree $k \in \Z$, denote
\[ \CF_k^{\lambda}(M, H)_{\alpha} = {\rm Span}_{\Z_2} \left\{x \in P_{\alpha}(H)\,| \, {\rm Ind}(x) = k, \,\,\mbox{and}\,\, \mathcal A_H(x) < \lambda \right\}. \]
The general construction of Hamiltonian Floer theory as demonstrated in Section \ref{sec-FHF} generates a $\Z_2$-linear map $\partial_k: \CF_k^{\lambda}(M, H)_{\alpha} \to \CF_{k-1}^{\lambda}(M, H)_{\alpha}$ which can be proved to be a differential. Denote the {\it filtered} Hamiltonian Floer homology in the class $\alpha$ by $\HF_k^{\lambda}(H)_{\alpha}$, the $k$-th homology of the chain complex $( \CF_*^{\lambda}(M, H)_{\alpha}, \partial_*)$.

Now we can form a persistence module, called a {\it Hamiltonian persistence module in the class $\alpha$ in degree $*$}, denoted by
$$\mathbb{HF}_*(H)_{\alpha} = \left\{\{\HF_*^{\lambda}(H)_{\alpha}\}_{\lambda \in \R}; \{\iota_{\lambda, \eta}: \HF_*^{\lambda}(H)_{\alpha} \to \HF_*^{\eta}(H)_{\alpha}\}_{\lambda \leq \eta} \right\}.$$
Moreover, denote by $\mathcal B(H)_{\alpha}$ the total barcode of $\mathbb{HF}_*(H)_{\alpha}$. It has the following special property.

\begin{exercise} Suppose $\alpha$ is a non-zero homotopy class of free loop space of $(M, \omega)$. Then $\mathcal B(H)_{\alpha}$ consists of only finite length bars. \end{exercise}

\medskip

The Dynamical Stability Theorem (Theorem \ref{DST}) extends to Floer homological barcodes associated to non-contractible loops. This was used in \cite{PS16} in order to construct Hamiltonian diffeomorphisms lying
arbitrarily far from the set  ${\rm Power}_p (M)$ described at the end of the previous section.

\section{Barcodes for Hamiltonian homeomorphisms}\label{sec-barhomeo}
The celebrated Eliashberg-Gromov Theorem \cite{MS98, PR14} states that a $C^0$-limit of symplectomorphisms, whenever it is smooth, is also symplectic. This result serves as the starting point of a rapidly developing area called
$C^0$-symplectic topology. Its objective is the study of a delicate interplay between rigidity and flexibility
for the group of  {\it Hamiltonian homeomorphisms} $\overline{\rm Ham}(M, \omega)$ of a symplectic manifold $(M,\omega)$.
By definition, this group consists of all $C^0$-limits of Hamiltonian diffeomorphisms.

On the flexible side, a recent striking result due to L.~Buhovsky, V.~Humiliere and S.~Seyfaddini \cite{BHS18A} provides  a $C^0$-counterexample to the Arnold conjecture. It turns out that every closed and connected symplectic manifold of dimension at  least four admits Hamiltonian homeomorphisms with just a single fixed point. Nevertheless such a point still carries footprints of symplectic rigidity!  In fact, as it was shown by
F.~Le Roux, S.~Seyfaddini and C.~Viterbo \cite{LSV18} for surfaces of genus $\geq 1$ and by L.~Buhovsky, V.~Humiliere and S.~Seyfaddini
\cite{BHS18} for all closed aspherical manifolds, one can associate with every Hamiltonian homeomorphism a Floer homological barcode {\it up to a shift}. Here is the precise statement. Consider the completion
of the space of barcodes with respect to the bottleneck distance. This space consists of infinite barcodes with the following property: for every $\epsilon >0$ such a barcode possesses a finite number of bars (with multiplicities)
of length greater than $\epsilon$. Denote by $\overline{\rm Barcodes}$ the quotient of this space by the group
of translations. Note that the bottleneck distance descends to this space. It turns out that
the mapping sending a Hamiltonian diffeomorphism to its Floer homological barcode
extends to a continuous map
\begin{equation} \label{c0-b}
\mathcal B: ({\rm Ham}(M, \omega), d_{C^0}) \to (\overline{\rm Barcodes}, d_{bot})\;.
\end{equation}
The proof involves Kislev-Shelukhin inequality \eqref{c0-sn} discussed above.

Furthermore, F.~Le Roux, S.~Seyfaddini and C.~Viterbo \cite{LSV18} explored, by using barcodes,
the group of Hamiltonian homeomorphisms of a surface of genus $\geq 1$.
In order to formulate their result, call two Hamiltonian homeomorphisms $f$ and $g$ {\it weakly conjugate} if there
exists a finite chain $h_1=f$, $h_2,\dots, h_{N-1}$, $h_N = g$ such that the closures of the conjugacy classes of $h_i$ and $h_{i+1}$ intersect for all $i=1,\dots, N-1$. Roughly speaking, weakly conjugate elements cannot be distinguished by
any continuous functional on the group. It turns out that the Floer homological barcode is a weak conjugacy invariant.
Furthermore, the paper exhibits an example of a Hamiltonian homeomorphism whose barcode has the following ``exotic" property: the set of endpoints of its barcode is unbounded. As a corollary, such a homeomorphism is not weakly conjugate to
any smooth Hamiltonian diffeomorphism! Existence of a dense conjugacy class in a groups is known as {\it the Rokhlin property}, see e.g. a paper by E.~Glasner and B.~Weiss \cite{glasner2008topological} for a historical account. Thus, the
group of Hamiltonian homeomorphisms is not Rokhlin in quite a strong sense.

\chapter{Symplectic persistence modules}\label{chp8_symplectic_persistence_modules}
\section{Liouville manifolds} \label{sec-Liouville}

\begin{dfn} \label{dfn-lm} A {\it Liouville manifold} $(M, \omega, X)$ is a connected symplectic manifold with a fixed complete vector field $X$ of $M$ generating a flow $X^t$ such that
\begin{itemize}
\item[{(i)}] $\omega = d\lambda$ where $\lambda = \theta_X\omega$;
\item[{(ii)}] there exists a closed connected hypersurface $P \subset M$ such that $P$ is transversal to $X$, bounds an open domain $U$ of $M$ with compact closure and $M = U \sqcup \bigcup_{t \geq 0} X^t(P)$. This vector field $X$ is called a {\it Liouville vector field} and its flow $X^t$ is called a {\it Liouville flow}. Any such hypersurface $P$ and any such domain $U$ are called {\it star-shaped}.
    \end{itemize} \end{dfn}

\begin{exercise} Given a Liouville manifold $(M, \omega, X)$, the Liouville flow $X^t$ acts on $M$ by conformal symplectomorphisms: $(X^t)^* \omega = e^t \omega$. \end{exercise}

Let $(M, \omega, X)$ be a Liouville manifold. A star-shaped hypersurface $P$ from the defining properties of $(M, \omega, X)$ can be used to decompose $M$ into the following two pieces,
\begin{equation} \label{decomp-lm}
M = M_{*,P} \sqcup {\rm Core}_P(M)
\end{equation}
where $M_{*,P} = \bigcup_{t \in \R} X^t(P)$ and ${\rm Core}_P(M) = \bigcap_{t <0} X^t(U)$ where $U$ is the open domain bounded by $P$. One can show that this decomposition is independent of the choice of the start-shaped hypersurface $P$, see Section 1.5 in \cite{EKP06}. Here are two standard examples of Liouville manifolds.

\begin{ex} \label{ex-lm-1} The symplectic linear space equipped with the radial vector field as follows is a Liouville manifold:
\[ (M, \omega_{std}, X_{rad}) = \left(\R^{2n}, \,\,\sum_{i=1}^n dp_i \wedge dq_i, \,\,\frac{1}{2} \sum_{i=1}^n \left(q_i \frac{\partial}{\partial q_i} + p_i \frac{\partial}{\partial p_i} \right) \right). \]
The decomposition (\ref{decomp-lm}) is
\[ \R^{2n} = \left(\R^{2n} \backslash \{{0}\} \right) \sqcup \{{0}\}. \]
An important observation is that a domain $U \subset \R^{2n}$ which contains ${0}$ is star-shaped in the sense of Definition \ref{dfn-lm} if and only if it is {\it strictly} star-shaped with respect to ${0}$ in the standard sense. Here ``strictly'' means $\partial {\overline{U}}$ is transversal to the radial vector field $X_{rad}$.
\end{ex}

\begin{ex}\label{ex-lm-2}
Fix a closed Riemannian manifold $N$. Its cotangent bundle is a Liouville manifold with respect to a canonical vector field $X_{can}$ as follows,
\[ (M, \omega_{can}, X_{can}) = \left(T^*N, \,\,\sum_{i=1}^n dp_i \wedge dq_i, \,\,\sum_{i=1}^n p_i \frac{\partial}{\partial p_i} \right). \]
Here $q_i$ is the position coordinate and $p_i$ is the momentum coordinate. In this case, decomposition (\ref{decomp-lm}) is
\[ T^*N = \left(T^*N \backslash 0_N \right) \sqcup 0_N \]
where $0_N$ is the zero-section of $T^*N$.  A standard example of a star-shaped domain is the open unit codisc bundle $U_g^*N$ associated to any Riemannian metric $g$ on $N$, that is, $U_g^*N : = \{({q}, {p}) \in T^*N \,| \, |{p}|_{g^*_{{q}}} <1 \}$.
\end{ex}

\begin{dfn}  Given a Liouville manifold $(M, \omega, X)$, denote $\lambda =\omega(X,\cdot)$. A symplectomorphism $\phi$ of a Liouville manifold is called {\it exact} if $\phi^*\lambda - \lambda = dF$ for some function $F$ on $M$.
Compactly supported exact symplectomorphisms form a group which we denote by $\Symp_{ex}(M,\omega,X)$.
The identity component of this group is denoted by $\Symp^0_{ex}(M,\omega,X)$. We shall often abbreviate
$\Symp_{ex}(M)$ and $\Symp^0_{ex}(M)$.
\end{dfn}

\begin{ex} \label{above-coor} Consider the Liouville manifold $(\R^{2n}, \omega_{std}, X_{rad})$, and its star-shaped domains which contains ${0}$. Prove that any symplectomorphism of this manifold is exact. Find an example
of a non-exact symplectomorphism of $T^*\T^2$.
 \end{ex}

Given a Liouville manifold $(M, \omega, X)$ and a star-shaped hypersurface $P$ from the defining properties of $(M, \omega, X)$, every point $m \in M_{*,P}$ in the decomposition (\ref{decomp-lm}) can be identified with a point $(x,u) \in P \times \R_+$, explicitly, $m = X^{\ln u}(x)$. In particular, $P = \{u=1\}$ and the star-shaped domain $U \subset M$ that is enclosed by $P$ is $\{u<1\}$. Finally, we take the convention that ${\rm Core}_P(M) = \{u=0\}$. \\

For any $x \in P$, consider the $\omega$-orthogonal complement of $T_xP$ in $T_xM$ which is defined by
\[ (T_xP)^{\omega} : = \left\{v \in T_x M \,| \, \omega_x(v, w) = 0 \,\,\mbox{$\forall w \in T_x P$} \right\}. \]
\begin{exercise} Check that $\dim (T_x P)^{\omega} = 1$ and $(T_xP)^{\omega} \subset T_x P$. \end{exercise}

These $1$-dimensional subspaces of $TP$ integrate to give a $1$-dimensional foliation $\mathcal F(P)$ of $P$ called {\it the characteristic foliation of $P$}. Denote by $C(P)$ the set of all closed leafs of $\mathcal F(P)$. Consider the restriction of the $1$-form $\theta_X \omega|_P$ to $P$.

For the rest of this chapter, we will make a non-degeneracy assumption which  will be important for the construction of symplectic persistence modules below. We say that a star-shaped hypersurface $P$ of a Liouville manifold $(M, \omega, X)$ is {\it non-degenerate} if its {\it action spectrum}
\begin{equation} \label{dfn-non-deg}
{\rm Spec}: = \left\{ \int_{\gamma} \theta_X \omega|_P \,\bigg| \, \gamma \in C(P) \right\} \,\,\,\,\mbox{is a discrete subset of $\R$}.
\end{equation}
Any star-shaped domain $U$ where $\partial {\overline{U}}$ satisfies the non-degenerate condition (\ref{dfn-non-deg}) is called a {\it non-degenerate star-shaped domain}.

\section{Symplectic persistence module} \label{sec-spm}
In Chapter \ref{chp7_hamiltonian_persistence_modules} we have studied filtered Floer homology for functions
on symplectic manifolds. In a seminal work \cite{floer1994symplectic} A.~Floer and H.~Hofer defined invariants
of an open domain in a symplectic manifold by taking direct or inverse limits of Floer homologies of special
collections of functions on this domain. Below we discuss this approach in the context of Floer-homological
persistence modules. We start with a brief reminder on the inverse limits.

\begin{dfn} A partially ordered set $(I, \preceq)$ is {\it downward directed} if for every $i, j \in I$, there exists a $k \in I$ such that $k \preceq i$ and $k \preceq j$. One can view this $(I, \preceq)$ as a category where an object is an element in $I$, and the morphism set between $i$ and $j$ contains a single element if and only if $i \preceq j$ and empty otherwise.

An {\it inverse system of vector spaces over $\Z_2$} is a functor $(A, \sigma)$ from a downward directed partially ordered set $(I, \preceq)$ to the category of vector spaces. Explicitly, $A$ assigns to each $i \in I$ a vector space $A_i$ over $\Z_2$ and $\sigma$ assigns to each pair $i, j \in I$ with $i \preceq j$ a $\Z_2$-linear map $\sigma_{ij}: A_i \to A_j$, such that $\sigma_{ik}  =\sigma_{jk} \circ \sigma_{ij}$, and $\sigma_{ii} = \mathds{1}_{A_i}$, the identity map on $A_i$.\end{dfn}


\begin{dfn} \label{dfn-invlim}
Let $(A, \sigma)$ be an inverse system of vector spaces over $\Z_2$. The {\it inverse limit} of $(A, \sigma)$ is defined as
\[ \varprojlim_{i \in I} A: = \left\{ \{x_i\}_{i \in I}  \in \Pi_{i \in I} A_i \,\big| \, i \preceq j \,\,\Rightarrow\,\, \sigma_{ij}(x_i) = x_j \right\}.  \]
Note that for any $i \in I$ there is a canonical projection map $\pi_i: \varprojlim_{i \in I} A \to A_i$ such that for $i\preceq j$, $\sigma_{ij} \circ \pi_i = \pi_j$.
\end{dfn}

\begin{exercise} \label{invlim-unversal}
Let $(A, \sigma)$ be an inverse system of vector spaces over $\Z_2$, and $\varprojlim_{i \in I} A$ denotes the inverse limit of $(A, \sigma)$. Prove that $\varprojlim_{i \in I} A$ satisfies the following universal property: for any pair $(B, \{\tau_i\}_{i \in I})$ where $\tau_i: B \to A_i$ such that $\sigma_{ij} \circ \tau_i = \tau_j$, there exists a unique morphism $\Phi: B \to \varprojlim_{i \in I} A$ such that $\pi_i \circ \Phi = \sigma_i$ for any $i \in I$, where $\pi_i: \varprojlim_{i \in I} A \to A_i$ is the canonical projection (see Definition \ref{dfn-invlim}).
 \end{exercise}

In Hamiltonian Floer theory, inverse system appears in the following construction. Given a non-degenerate star-shaped domain $U$ of a Liouville manifold $(M, \omega, X)$, denote by $\mathcal H(U)$ the collection of all autonomous Hamiltonian functions on $M$ that are compactly supported in $U$. Define a partial order in $\mathcal H(U)$ by $H\preceq G$ if and only if $H(x,u) \geq G(x,u)$ for any $(x,u) \in M$. Recall that the existence of coordinate $(x,u)$ is elaborated after Example \ref{above-coor}.

Following the argument in subsection 4.4 and 4.5 in \cite{BPS03}, given $H, G \in \mathcal H(U)$ with $H \preceq G$, one can consider a {\it monotone homotopy} from $H$ to $G$, i.e., a smooth homotopy $\{H_s\}_{s \in [0,1]}$ such that $H_0 = H$, $H_1 = G$, and $\partial_{s}H_s \leq 0$. This homotopy induces a $\Z_2$-linear map
\begin{equation}\label{monotone-htp}
\sigma_{H,G}: \HF_*^{(a, \infty)}(H) \to \HF_*^{(a, \infty)}(G) \,\,\,\,\,\mbox{for any $a>0$.}
\end{equation}
Here, $\HF_*^{(a, \infty)}(H)$ stands for Hamiltonian Floer homology of the function  $H$ with coefficients in $\Z_2$ and within the action window $(a, \infty)$. The monotonicity of our homotopy guarantees that action window $(a, \infty)$ is preserved under the map $\sigma_{H,G}$. Let us mention that in order to define Hamiltonian Floer homology, one has to work with arbitrarily small generic perturbations of functions involved. For the sake of simplicity, we shall ignore this nuance.

Moreover, one can easily check that if $H_1 \preceq H_2 \preceq H_3$ in $\mathcal H(U)$, then $\sigma_{H_1, H_3}  = \sigma_{H_2, H_3} \circ \sigma_{H_1, H_2}$. In other words, over the partially ordered set $\mathcal H(U)$, we obtain an inverse system of vector spaces over $\Z_2$.

\begin{dfn} \label{dfn-fsh} Let $U$ be a non-degenerate star-shaped domain of a Liouville manifold $(M, \omega, X)$. For any $a>0$, the {\it filtered symplectic homology of $U$} is defined as
\[ \SH_*^{(a,\infty)}(U) : = \varprojlim_{H \in \mathcal H(U)} \HF_*^{(a, \infty)}(H). \]
\end{dfn}

From this definition, we can directly check the following two properties.

\begin{exercise} \label{exe-fsh-str}
$\;$

(1) For any $a>0$ and degree $* \in \Z$, $\SH_*^{(a,\infty)}(U)$ is finite dimensional over $\Z_2$.

(2) For any $a \leq b$, the canonical morphism $\HF_*^{(a,\infty)}(H) \to \HF_*^{(b,\infty)}(H)$
induces a $\Z_2$-linear map $\theta_{a,b}:  \SH_*^{(a,\infty)}(U) \to \SH_*^{(b,\infty)}(U)$.
\end{exercise}


Let $U$ be a non-degenerate star-shaped domain of a Liouville manifold $(M,\omega,X)$. For any $a>0$, set $\SH_*^{\ln a}(U):= \SH_*^{(a,\infty)}(U)$ ({\bf mind the logarithmic scale!}).  It follows that the collection of data
\[ \mathbb{SH}_*(U) = \left\{ \left\{ \SH_*^{\ln a}(U) \right\}_{a >0}, \left\{\theta_{s,t}: \SH_*^{\ln a}(U) \to \SH_*^{\ln b}(U)\right\}_{a \leq b} \right\} \]
forms a proper persistence module (see Section \ref{sec-prop-pm}).

\begin{dfn} \label{dfn-spm} The proper persistence module $\mathbb{SH}_*(U)$
is called {\it symplectic persistence module of $U$}.
\end{dfn}

\noindent
 By Normal Form Theorem (see Section \ref{sec-prop-pm}), this module possesses
 a proper barcode denoted by $\mathcal B_*(U)$. For the sake of brevity,
 we omit the adjective ``proper" throughout this chapter.

Let us finish this section with a useful construction. By using the Liouville vector field  $X$ on a Liouville manifold $(M, \omega, X)$, one can rescale a star-shaped domain $U$ as follows.  For any $C>0$, put $CU : = \phi_{X}^{\ln C} (U)$. It is easy to see that there exists an isomorphism
\begin{equation} \label{dfn-rescale}
r_C: \SH_*^{t+ \ln C}(CU) \simeq \SH_*^{t}(U) \,\,\,\,\mbox{for any $t \in \R$ and degree $* \in \Z$}.
\end{equation}
In fact, $r_C$ is induced by rescaling all the ingredients in the construction of filtered symplectic homology. Note that this rescaling results in a uniform shift of $\mathcal B_*(U)$ by $\ln C$.

\medskip\noindent
\begin{remark}\label{rem-noncontr}{\rm While defining Hamiltonian Floer homology, sometimes it is useful to consider closed orbits in a given free homotopy class $\alpha$ of loops on a symplectic manifold, see Section \ref{sec-ham-non} above. This, in a
straightforward way, gives rise to the symplectic persistence module of a domain $U$ in the class $\alpha$.
}
\end{remark}

\section{Examples of \texorpdfstring{$\mathbb{SH}_*(U)$}{SH(U)}}
In this section we present some examples of  symplectic persistence modules.

\medskip\noindent
\begin{ex} \label{ex-ellipsoid} Denote by  $E(1, N, ..., N)$ the ellipsoid in $\R^{2n} (= \C^n)$  defined by
\[ E(1,N, ..., N) = \left\{ (z_1, ..., z_n) \in \C^n \,\bigg| \, \pi \left( \frac{|z_1|^2}{1} + \frac{|z_2|^2}{N} + \ldots \frac{|z_n|^2}{N}\right) < 1  \right\}, \]
where $N \geq 1$ is an integer, see Section \ref{sec-6}. It is easy to check that its action spectrum equals $\Z$. In particular, $E(1, N, ..., N)$ is a non-degenerate star-shaped domain of the Liouville manifold $(\R^{2n}, \omega_{std}, X_{rad})$.
 We shall prove in  Section \ref{sec-comp} that
for any $a >0$,
\begin{equation} \label{sh-ell} \SH_*^{(a,\infty)}(E(1, N, ..., N)) = \Z_2 \,\,\,\,\mbox{when $* = - 2 \big|\ceil*{-a}\big| - 2(n-1)\bigg| \ceil*{\frac{-a}{N}} \bigg|$},
\end{equation}
and the homologies vanish in all other degrees. This readily yields

\begin{equation} \label{sh-ell-0}
\mathbb{SH}_0(E(1, N, ..., N)) = \Z_2(-\infty, 0),
\end{equation}
where $\Z_2(-\infty, 0)$ denotes the interval module $(-\infty, 0)$ over the field $\Z_2$.
In particular, for the ball
$B^{2n}(1) = E(1, ..., 1)$
\begin{equation} \label{sh-ball}
 \SH_*^{(a,\infty)}(B^{2n}(1)) = \Z_2 \,\,\,\,\mbox{only when $* = - 2n \big|\ceil*{-a}\big|$}, \;\; \forall a>0\;.
 \end{equation}
For instance, $\mathbb{SH}_0(B^{2n}(1)) = \Z_2(-\infty, 0).$
\end{ex}

\begin{ex} \label{ex-codisc} (I) Let $N$ be a closed manifold and $g$ be a Riemannian metric on $N$. Consider the unit codisc bundle $U_g^*N$ over $N$. For a generic choice of the metric $g$, $U_g^*N$ is a non-degenerate star-shaped domain of $(T^*N, \omega_{can}, X_{can})$. Fix a non-zero homotopy class $\alpha$ of the free loop space of $N$
(and hence of $T^*N$, since $T^*N$ retracts to the zero section). Consider the symplectic persistence module
$\mathbb{SH}_*(U_g^*N)_{\alpha}$ of $U$ in the class $\alpha$ (cf. Remark \ref{rem-noncontr};
this notation emphasizes the dependence on the class $\alpha$).

According to \cite[Theorem 3.1.(i)]{Web06},  for any $a>0$, we have an isomorphism between the following two vector spaces,
\begin{equation} \label{web-iso}
\SH_*^{(a, \infty)}(U_g^*N)_{\alpha} \simeq {\rm H}_*(\Lambda_{\alpha}^{a} N)
\end{equation}
where $\Lambda_{\alpha}^{a} N$ is the space of loops in $N$ in the class $\alpha$
of length $< a$.  Moreover, it can be shown that isomorphism (\ref{web-iso})  extends to an isomorphism of the persistence modules $\mathbb{SH}_*(U_g^*N)_{\alpha}$ and $V(N,g)_\alpha$, where the latter persistence module
is defined exactly as in Example \ref{exm-geodesics}) for loops in the class $\alpha$.
\medskip

(II) Represent a torus $N = \mathbb T^2$ as a surface a revolution with a profile function that has two local minima, with open ends identified, see Figure \ref{torus-g}. Equip $N$ with the Riemannian metric $g$ induced
from the Euclidean one in $\R^3$. The local minima of the profile function generate two simple closed geodesics denoted by $\gamma_1$ and $\gamma_2$, and its maximum generates a closed geodesic $\Gamma$. Assume that $N$
is pinched at $\gamma_1$ and $\gamma_2$ in the following sense: the length of $\Gamma$ is $> 2$, and that the lengths of $\gamma_1,\gamma_2$ are $< 1$. Put $s_i = -\ln \length_g(\gamma_i)$. In what follows we emphasize the dependence
of $g$ on the vector $s=(s_1,s_2)$ with $s_1 \geq s_2$ and write $g=g_s$.

\begin{figure}[H]
\centering
\includegraphics[width=11cm]{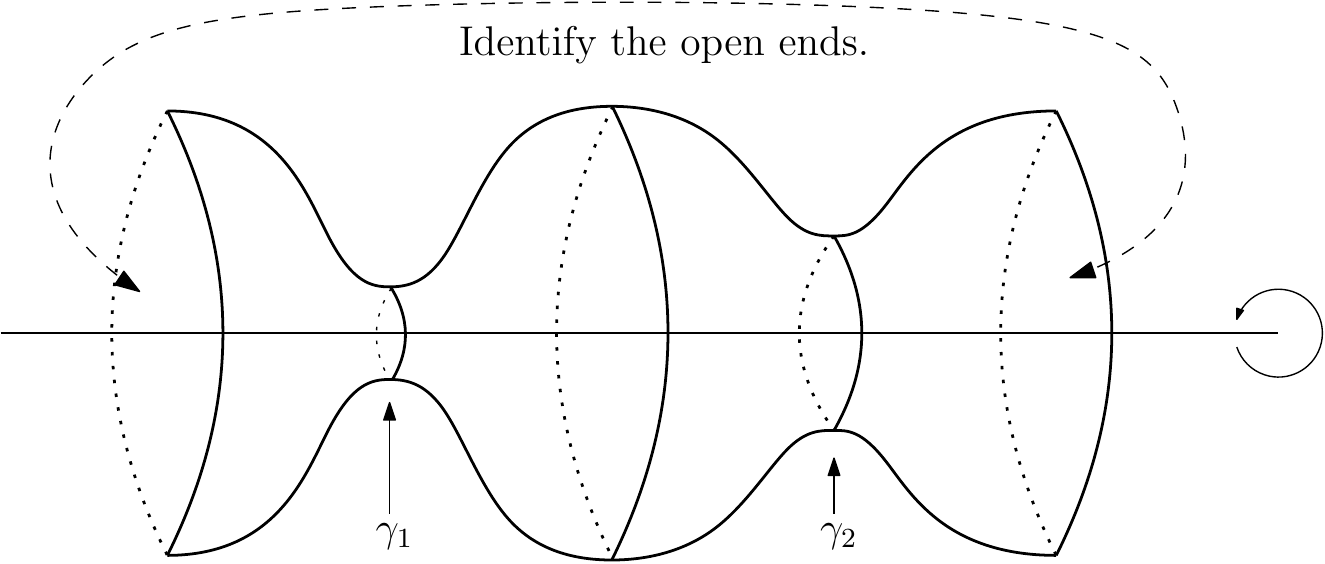}
\caption{A Riemannian metric $g$ on $\mathbb T^2$}
\label{torus-g}
\end{figure}

Note that $\gamma_1,\gamma_2$ and $\Gamma$ are the only geodesics in their homotopy class which we denote
by $\alpha$. Consider the persistence module $V(N,g_s)_\alpha$ truncated on the ray $(-\infty, \ln(3/2))$
(see definition before Exercise \ref{exr: truncated_interleaving}). An elementary argument from differential geometry (see Section 6 in \cite{SZ18}) implies that the barcode $\calB^{(s)}$ of this truncated module looks as in Figure \ref{par-barcode}.

\begin{figure}[h]
\centering
\includegraphics[width=8cm]{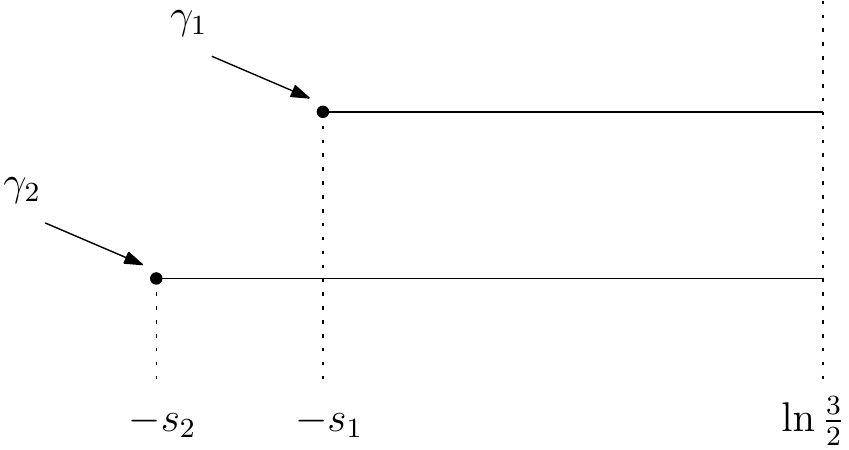}
\caption{Barcode $\calB^{(s)}$}
\label{par-barcode}
\end{figure}

Later we will see that this partial information on the symplectic persistence modules is already useful enough to quantitatively compare unit codisc bundles $U_g^*N$ corresponding to different Riemannian metrics $g$.
\end{ex}

\section{Symplectic Banach-Mazur distance}\label{sec-sbm}
In Section \ref{sec-FHF} we have seen that the barcodes of Hamiltonian persistence modules can be helpful to study Hofer's geometry on Hamiltonian diffeomorphisms. In particular, Dynamical Stability Theorem (see Theorem \ref{thm-dst}) states that the bottleneck distance provides a lower bound of Hofer's metric. In what follows, a pseudo-metric $d_{\SBM}$ between two star-shaped domains will be defined, and similarly we will see that barcodes of symplectic persistence modules can be used to study this distance.

Denote by $\mathcal S^{2n}$ the set of all the non-degenerate star-shaped domains of a Liouville manifold
$(M, \omega, X)$. For an exact symplectomorphism $\psi \in \Symp_{ex}(M)$ and $C >0$, define its rescaling
$$\phi(C)= X^{\ln C} \circ \phi \circ X^{-\ln C} \in \Symp_{ex}(M)\;,$$
where as above $X^t$ stands for the Liouville flow.

For $U,V \in \mathcal S^{2n}$, a {\it Liouville
morphism $\phi$ from $U$ to $V$} is a compactly supported exact symplectomorphism $\phi$ of $M$ such that
$\phi(\overline{U}) \subset V$. Sometimes we denote such a morphism by $U \xhookrightarrow{\phi} V$.

\begin{exercise}\label{ex-rescalings} Suppose $\overline{U} \subset V$, and $U \xhookrightarrow{\phi} V$ with $\phi \in \Symp^0_{ex}(M)$.
Let $\phi_t$, $t \in [0,1]$ be an isotopy joining the identity with $\phi$. Show that by a suitable choice of the rescaling
factor $C(t)$ one can modify this isotopy to $\psi_t = \phi_t(C(t))$ so that $\psi_t=\phi_t$ for $t=0,1$ and
$\psi_t(\overline{U}) \subset V$ for all $t \in [0,1]$.
\end{exercise}

\begin{dfn} \label{dfn-SBM} Let $U, V \in \mathcal S^{2n}$.
A real number $C > 1$ is called {\it $(U,V)$-admissible} if there exists a pair of symplectomorphisms $\phi,\psi \in \Symp_{ex}(M)$ such that $\frac{1}{C}U \xhookrightarrow{\phi} V \xhookrightarrow{\psi} CU$ and
$\psi \circ \phi \in \Symp^0_{ex}(M)$.
\end{dfn}

\begin{exercise} \label{ex-BM-sym} Show that for a $C$-admissible pair $\phi,\psi$ one has
$\frac{1}{C}V \xhookrightarrow{\psi(C^{-1})} U \xhookrightarrow{\phi(C)} CV$ and furthermore
$\phi(C) \circ \psi(C^{-1}) \in \Symp^0_{ex}(M)$. {\it Hint:} For the latter statement, use that $\Symp^0_{ex}(M)$ is a normal subgroup of $\Symp_{ex}(M)$.
\end{exercise}

\begin{dfn}  (Ostrover, Polterovich, Usher \cite{Ush18}) \label{dfn-SBM-2} Define the {\it symplectic Banach-Mazur distance between $U$ and $V$} by
\[ d_{\SBM}(U,V) = \inf \;\{\ln C\,| \,C \;\text{is}\;\mbox{$(U,V)$-admissible}\}\;.\]
\end{dfn}

This distance can be considered as a non-linear analogue of the Banach-Mazur classical distance on convex bodies,
see e.g. \cite{rudelson2000distances}. The importance of the assumption $\psi \circ \phi \in \Symp^0_{ex}(M)$
in Definition \ref{dfn-SBM} was understood in \cite{GU17} in a more general context of unknotted Liouville embeddings.

\begin{remark} \label{rmk-fix-class} One can modify the notion of the symplectic Banach-Mazur distance $d_{\rm SBM}$ by considering Liouville morphisms coming from exact symplectomorphisms acting trivially on the fundamental group, or preserving a fixed free homotopy class $\alpha$. \end{remark}

\begin{exercise} Check that $d_{\SBM}$ is a pseudo-metric on $\mathcal S^{2n}$. In particular, use Exercise \ref{ex-BM-sym}
to show that $d_{\SBM}$ is symmetric.
\end{exercise}

\begin{exercise} Show that if $U, V \in {\mathcal S}^{2n}$ are exactly symplectomorphic, then
$d_{\SBM}(U, V) = 0$. An interesting open question is whether $d_{\SBM}$ is a genuine metric on the quotient space ${\mathcal S}^{2n}/\Symp_{ex}(M)$. \end{exercise}

\begin{exercise} Show that $d_{\SBM}(U, CU) =|\ln C|$ for any $U \in \mathcal S^{2n}$ and $C > 0$. This implies that, as a pseudo-metric space, $(\mathcal S^{2n}, d_{\SBM})$ has infinite diameter.
\end{exercise}

The following theorem is a stability result that involves star-shaped domains. We will prove it in Section \ref{sec-app}.

\begin{theorem} \label{tst} (Topological Stability Theorem) Let $U,V \in \mathcal S^{2n}$. Denote the barcodes of persistence modules $\mathbb{SH}_*(U)$ and $\mathbb{SH}_*(V)$ by $\mathcal B_*(U)$ and $\mathcal B_*(V)$ respectively. Then \[d_{bot}(\mathcal B_*(U), \mathcal B_*(V)) \leq d_{\SBM}(U,V).\] \end{theorem}

\begin{ex} Consider the $4$-dimensional ellipsoids $E(1,8)$ and $E(2, 4)$. Observe that they have the same volume. By Example \ref{ex-ellipsoid}, at degree $* = 0$,
\[ \mathcal B_0(E(1,8)) = (-\infty, 0) \,\,\,\,\mbox{and}\,\,\,\, \mathcal B_0(E(2,4)) = (-\infty, \ln 2). \]
Therefore, Theorem \ref{tst} implies that
\[ d_{\rm SBM}(E(1,8), E(2,4)) \geq \ln 2. \]
In the same spirit, one can check that $d_{\rm SBM}(E(r, rN, ..., rN), B^{2n}(R)) \geq |\ln r - \ln R|$.
\end{ex}

\begin{ex} Let $g_s$ and $g_t$ be two metrics of revolution on the two-dimensional torus as in
Example \ref{ex-codisc}. Note that by Exercise \ref{exr: truncated_interleaving} and Figure \ref{par-barcode}
\[ d_{bot}\left(\mathcal B_*(U^*_{g_{s}}\mathbb T^2)_{\alpha}, \mathcal B_*(U^*_{g_{t}}\mathbb T^2)_{\alpha}\right) \geq d_{bot}(\calB^{(s)},\calB^{(t)}) \geq  \frac{1}{2} |s - t|_{\infty}. \]
Therefore, Theorem \ref{tst} implies that
\[ d_{\rm SBM}\left(U^*_{g_{s}} \mathbb T^2, U^*_{g_{t}} \mathbb T^2\right) \geq \frac{1}{2} |s - t|_{\infty}. \]
Interested readers can refer to a recent work \cite{SZ18} for a generalization of this result.
\end{ex}

\section{Functorial properties}

As sample applications of symplectic persistence modules, in Section \ref{sec-app} below we shall deduce a version of Gromov's famous  non-squeezing theorem, as well as establish Topological Stability Theorem, Theorem \ref{tst}. To this end, we discuss some useful functorial properties of filtered symplectic homology.

\begin{theorem} \label{thm-fc-sh} Let $(M, \omega, X)$ be a Liouville manifold, and $U, V$ are two non-degenerate star-shaped domains of $(M, \omega, X)$.
\begin{itemize}
\item[(1)]  Every Liouville morphism $\phi$ from $U$ to $V$ induces a $\Z_2$-linear map $f^a_{\phi}: \SH_*^{(a, \infty)}(V) \to \SH_*^{(a, \infty)}(U)$, for every $a>0$ and degree $* \in \Z$. Moreover, denote by $\theta^U$ and $\theta^V$ the structure maps of the symplectic persistence modules of $U$ and $V$, respectively. Then we have the following commutative diagram. For any $0<a\leq b$ and degree $* \in \Z$,
\[ \xymatrix{
\SH_*^{(a, \infty)}(V) \ar[r]^-{f^a_\phi} \ar[d]_-{\theta^{V}_{a,b}} & \SH_*^{(a, \infty)}(U) \ar[d]^-{\theta^{U}_{a,b}} \\
\SH_*^{(b, \infty)}(V) \ar[r]_-{f^b_\phi} & \SH_*^{(b, \infty)}(U).} \]
If $W$ is another non-degenerate star-shaped domain such that $U \xhookrightarrow{\phi} V \xhookrightarrow{\psi}W$, then for every $a>0$, $f^a_{\psi \circ \phi} = f^a_{\phi} \circ f^a_{\psi}$.

\item[(2)] Write $r_C$, $C >  1$, for the rescaling isomorphism from (\ref{dfn-rescale}) above.
Denote by $\theta$ the structure maps of the symplectic persistence modules. Set $i= f_{\mathds{1}}$ to be the morphism
induced by the identity map $\mathds{1}$ on $M$ (viewed as a Liouville morphism from $U$ to $CU$) as in item (1) above.
Then we have the following two commutative diagrams. For every $a>0$ and degree $* \in \Z$,
\[ \xymatrix{
\SH_{*}^{(a, \infty)}(U) \ar[rd]_{\theta_{a,Ca}} && \SH_{*}^{(Ca, \infty)}(C U) \ar[ll]_{r_C}^{\simeq} \ar[ld]^{i^{Ca}} \\
& \SH_{*}^{(Ca, \infty)}(U) & }\]
and
\[ \xymatrix{
\SH_{*}^{(a, \infty)}(U)  && \SH_{*}^{(Ca, \infty)}(C U) \ar[ll]_{r_C}^{\simeq}  \\
& \SH_{*}^{(a, \infty)}(CU) \ar[ru]_{\theta_{a,Ca}} \ar[lu]^{i^{a}} & }\;.\]
A similar conclusion can be drawn for $0< C <1$.

\item[(3)] Suppose $\overline{U} \subset V$, and $\phi$ is a Liouville morphism from $U$ to $V$. If $\phi \in \Symp^0_{ex}(M)$,
then $f_{\phi} = i$, where $f_{\phi}$ is the morphism induced by $\phi$ and $i= f_{\mathds{1}}$ is the morphism induced by the identity map $\mathds{1}$ on $M$ (viewed as a Liouville morphism from $U$ to $V$).

\end{itemize}
\end{theorem}

\begin{remark} Similarly to Remark \ref{rmk-fix-class}, if a non-zero homotopy class $\alpha$ of the free loop space is fixed, then all the morphisms in Theorem \ref{thm-fc-sh} are required to fix this class. \end{remark}

Instead of presenting the complete proof of Theorem \ref{thm-fc-sh}, we only give the outline. Suppose that
$\phi$ is a Liouville morphism from $U$ to $V$. Item (1) in Theorem \ref{thm-fc-sh} comes from the observation that any $\phi$ induces a morphism  $\phi_*: \calH(U) \to \calH(V)$ of the function spaces, which is given by the push-forward by $\phi$. Note also that for every $F \in \calH(U)$ there exists $G \in \calH(V)$ such that $G \geq \phi_*(F)$. Thus $\phi_*$ induces a morphism $\tau_F: \SH_*^{(a, \infty)}(V) \to \HF_*^{(a, \infty)}(F)$ which is the composition of the following morphisms,
\[ \SH_*^{(a, \infty)}(V) \xrightarrow{\pi_G} \HF_*^{(a, \infty)}(G) \xrightarrow{\sigma_{G, \phi_*(F)}} \HF_*^{(a, \infty)}(\phi_*F) \simeq \HF_*^{(a, \infty)}(F) \]
where $\pi_G$ is the canonical projection (see Definition \ref{dfn-invlim}), and $\sigma_{G, \phi_*(F)}$ is the morphism induced by a monotone homotopy from $G$ to $\phi_*(F)$ (see (\ref{monotone-htp})). It is readily to check that for any $H \geq F$ in $\mathcal H(U)$, $\sigma_{H, F} \circ \tau_H  = \tau_F$. Hence, by Exercise \ref{invlim-unversal}, there exists a well-defined morphism from $\SH_*^{(a, \infty)}(V)$ to $\SH_*^{(a, \infty)}(U)$. Ideas of the proofs of (2) and (3) are similar, and they can be checked using the proof of Lemma 4.15 in \cite{Gut15} by carefully studying the moduli space of connecting trajectories. Additionally, the proof of (3) involves
Exercise \ref{ex-rescalings}.

\begin{ex} \label{ex-res} Let $0 < 1< R$. For brevity, denote by $B_1$ and $B_2$ the balls $B^{2n}(1)$ and $B^{2n}(R)$ of $\R^{2n}$, respectively. Note that $B_2 = R B_1$. For every $a>0$ and degree $* \in \Z$, denote by $\theta_a$ the structure morphism $\theta_{a/R, a}: \SH_{*}^{(a/R, \infty)}(B_1) \to \SH_{*}^{(a, \infty)}(B_1)$, and by $i^a$ the morphism $f^a_{\mathds{1}}: \SH_{*}^{(a, \infty)}(B_2) \to \SH_{*}^{(a, \infty)}(B_1)$ induced by the identity map $\mathds{1}$ on $\R^{2n}$ (viewed as a Liouville morphism from $B_1$ to $B_2$). Then item (2) in Theorem \ref{thm-fc-sh} implies the following commutative diagram,
\[ \xymatrix{
\SH_{*}^{(a/R, \infty)}(B_1) \ar[rd]_-{\theta_a} && \SH_{*}^{(a, \infty)}(B_2) \ar[ll]_{r_{1/R}}^{\simeq} \ar[ld]^{\,\,\,\,\,\,i^a} \\
& \SH_{*}^{(a, \infty)}(B_1) & }\]
where $r_{1/R}$ is the rescaling isomorphism from (\ref{dfn-rescale}).
\end{ex}

\begin{ex} \label{top-inter} Let $(M, \omega, X)$ be a Liouville manifold, and $U, V$ are two non-degenerate star-shaped domains of $(M, \omega, X)$. Suppose there exist Liouville morphisms
\begin{equation} \label{comp-liouville}
U/C \xhookrightarrow{\phi} V\xhookrightarrow{\psi} C U \,\,\,\,\mbox{for some $C > 1$}
\end{equation}
such that the composition $\psi \circ \phi$ lies in $\Symp^0_{ex}(M)$.  First, by item (1) in Theorem \ref{thm-fc-sh}, for any $a >0$, we have the following commutative diagram,
\[ \xymatrix{
\SH^{(a, \infty)}_{*}(C  U) \ar[rr]^{f^{a}_{\psi}} \ar@/_1.5pc/[rrrr]_{f^{a}_{\psi \circ \phi}} && \SH^{(a, \infty)}_{*}(V) \ar[rr]^-{f^{a}_{\phi}} && \SH^{(a, \infty)}_{*}(U/C)}. \]
Then, by item (3) in Theorem \ref{thm-fc-sh}, $f^{a}_{\psi \circ \phi} = i^{a}_{CU, U/C}$ where $i_{CU, U/C} = f_{\mathds{1}}$ is the morphism
induced by the identity map $\mathds{1}$ on $M$ (viewed as a Liouville morphism from $U/C$ to $CU$). Denote by $i_{CU, U}$ and $i_{U, U/C}$ the induced morphisms in the same manner. Last but not least, item (2) in Theorem \ref{thm-fc-sh} implies the following commutative diagram,
\[ \xymatrix{
\SH^{(a, \infty)}_{*}(C U) \ar[rr]^{i^{a}_{CU, U/C}} \ar[dd]^{\simeq}_{r_{1/C}} \ar[rd]^{i^{a}_{CU, U}} & & \SH^{(a, \infty)}_{*}(U/C)\\
& \SH^{(a, \infty)}_{*}(U) \ar[ru]^{i^{a}_{U, U/C}} \ar[rd]_{\theta_{a, Ca}} & \\
\SH^{(a/C, \infty)}_{*}(U) \ar[ru]_{\theta_{a/C, a}} \ar[rr]_{\theta_{a/C, Ca}} && \SH^{(Ca, \infty)}_{*}(U) \ar[uu]^{\simeq}_{r_{1/C}} } \]
where $r_{1/C}$ is the rescaling isomorphism from (\ref{dfn-rescale}).

Looking at the lower horizontal arrow of this diagram and passing to the logarithmic scale as in Definition \ref{dfn-spm}, we see that  Liouville morphisms as in (\ref{comp-liouville}) whose composition lies in
$\Symp^0_{ex}(M)$ induce the structure morphism $\theta_{a - \log C, a + \log C}$ of the symplectic persistence module $\mathbb{SH}_{*}(U)$.
\end{ex}

\section{Applications} \label{sec-app}
The first application of symplectic persistence modules is a proof of the following version of Gromov's non-squeezing theorem.

\begin{theorem} Let $B^{2n}(r)$ be a ball and $E(R, R_{\dagger}, ..., R_{\dagger})$ be an ellipsoid of $\R^{2n}$ (see Example \ref{ex-ellipsoid}). Assume $R_{\dagger} \geq R$. Suppose there exists a Liouville morphism from $B^{2n}(r)$ to $E(R, R_{\dagger}, ..., R_{\dagger})$. Then $R \geq r$. \end{theorem}

\begin{proof} Without loss of generality, assume $r = 1$. Denote by $\phi$ the compactly supported exact symplectomorphism on $\R^{2n}$ such that $\phi(\overline{B^{2n}(1)}) \subset E(R, R_{\dagger}, ..., R_{\dagger})$. Suppose that the support of $\phi$ is contained in a balls $B^{2n}(R_{\bullet})$ with a sufficiently large $R_{\bullet}$. For brevity, denote by $B_1$ and $B_2$ the balls $B^{2n}(1)$ and $B^{2n}(R_{\bullet})$, respectively. Then one has the following relation,
\[ \phi(B_1) \subset E(R, R_{\dagger}, ..., R_{\dagger}) \subset B_2 = \phi(B_2). \]
By item (1), item (2) in Theorem \ref{thm-fc-sh} and Example \ref{ex-res}, we have the following commutative diagram. For any $a>0$ and degree $* \in \Z$,
\[ \xymatrix{
\SH_*^{(a, \infty)}(\phi(B_2)) \ar[r] \ar[d]_{\simeq} & \SH_*^{(a, \infty)}(E(R, R_{\dagger}, ..., R_{\dagger})) \ar[r]^-{i^a} & \SH_*^{(a, \infty)}(\phi(B_1)) \ar[d]^{\simeq} \\
\SH^{(a, \infty)}_*(B_2) \ar[rr]^{i^{a}_{B_2, B_1}} \ar[rd]_{\simeq} & & \SH_*^{(a, \infty)}(B_1) \\
& \SH_*^{(a/R_{\bullet}, \infty)} (B_1) \ar[ru]_{\,\,\,\,\,\,\theta_{a/R_{\bullet}, a}} & }\]
where both $i$ and $i_{B_2, B_1}$ are morphisms induced by the identity map $\mathds{1}$ on $\R^{2n}$, viewed as Liouville morphisms from $\phi(B_1)$ to $E(R, R_{\dagger}, ..., R_{\dagger})$ and from $B_1$ to $B_2$, respectively. At degree $* = 0$, by Example \ref{ex-ellipsoid} and rescaling,
\[ \mathbb{SH}_0(E(R, R_{\dagger}, ..., R_{\dagger})) = \Z_2(-\infty, \ln R) \,\,\,\,\mbox{and}\,\,\,\, \mathbb{SH}_0(B_1) = \Z_2(-\infty, 0). \]
For any $a < R_{\bullet}$, $\theta_{a/R_{\bullet}, a} \neq 0$, which implies that $i^{a}_{B_2, B_1}\neq 0$. Then at degree $\ast = 0$, $i: \Z_2(-\infty, \ln R) \to \Z_2(-\infty, 0)$ is nonzero. Hence, by Exercise \ref{exr: morphism_between_Q(segment)s}, $\ln R \geq 0$, that is, $R \geq 1$.
\end{proof}

The second application is the proof of Topological Stability Theorem, Theorem \ref{tst}.

\begin{proof} [Proof of Theorem \ref{tst}] By Definition \ref{dfn-SBM}, for any $\ep>0$, there exists some $C > 1$, and $\frac{1}{C}  U \xhookrightarrow{\phi} V \xhookrightarrow{\psi} C  U$ such that $\psi \circ \phi \in \Symp^0_{ex}(M)$ and $\ln C \leq d_{\SBM}(U,V) + \ep$. Applying functor $\SH_{*}^{(a, \infty)}(\cdot)$ and Example \ref{top-inter}, one can show that $\mathbb{SH}_*(U)$ and $\mathbb{SH}_*(V)$ are $\ln C$-interleaved. We leave details as an exercise to the readers.

By Isometry Theorem,
\[ d_{bot}(\mathcal B_*(U), \mathcal B_*(V)) = d_{int}(\mathbb{SH}_{*}(U), \mathbb{SH}_{*}(V)) \leq \ln C \leq d_{\SBM}(U, V) + \ep. \]
Letting $\ep \to 0$, we get the conclusion.
\end{proof}

\begin{remark} With the help of Theorem \ref{tst}, we are able to answer some coarse geometry questions that are elaborated in Section \ref{sec-cg}. For instance, one of the main results in Usher's recent work \cite{Ush18} shows that when $(M, \omega, X) = (\R^{2n}, \omega_{std}, X_{rad})$ with $n \geq 2$, the pseudo-metric space $(\mathcal S^{2n}, d_{\SBM})$ admits a quasi-isometric embedding from $(\R^N, d_{\infty})$ to $(\mathcal S^{2n}, d_{\SBM})$ for any $N \in \mathbb N$. See \cite{SZ18} for a similar result but in the set-up of cotangent bundles. It will be interesting and worthwhile to explore more applications of Theorem \ref{tst} to coarse geometry of the space of symplectic embeddings. \end{remark}

\section{Computations} \label{sec-comp}

We end this chapter a computation of the filtered symplectic homology of the ellipsoid $E(1, N, ..., N)$. The result was stated in Example \ref{ex-ellipsoid}, and the idea of its computation is quite enlightening. We want to emphasize that the filtered symplectic homology of an ellipsoid is one of the very few cases that can be computed explicitly. As a comparison, more advanced work is required to attain the isomorphism in Example \ref{ex-codisc} between the filtered symplectic homology of a unit codisc bundle and the filtered loop space homology.

\medskip
The computation of $\SH_*^{(a, \infty)}(E(1, N, ..., N))$ is based on the following two principles.
\begin{itemize}
\item{(Principle One)} It is difficult to analyze the inverse system directly from its definition, see Definition \ref{dfn-invlim}. The following proposition is useful from the computational perspective, which reduces 
    computation of the inverse limit of our system to the one of a special sequence.
\begin{prop} \label{prop-il} Let $(A, \sigma)$ be an inverse system of vector spaces over $\Z_2$. A sequence $\{i_{\nu}\}_{\nu \in \N}$ is {\it downward exhausting} for $(A, \sigma)$ if for every $i_{\nu+1} \preceq i_{\nu}$, $\sigma_{i_{\nu+1}i_{\nu}}: A_{i_{\nu+1}} \to A_{\nu}$ is an isomorphism, and for every $i \in I$, there exists $\nu \in \N$ such that $i_{\nu} \preceq i$. Then for any downward exhausting sequence $\{i_{\nu}\}_{\nu \in \N}$ for $(A, \sigma)$, the canonical projection $\pi_{i_{\nu}}: \varprojlim_{i \in I} A \to A_{i_{\nu}}$ is an isomorphism.
\end{prop}
\item{(Principle Two)} Recall that for $H, G \in \mathcal H(U)$ with $H \preceq G$, a monotone homotopy from $H$ to $G$ induces a $\Z_2$-linear map $\sigma_{H,G}: \HF_*^{(a, \infty)}(H) \to \HF_*^{(a, \infty)}(G)$ for any $a>0$. In general, $\sigma_{H,G}$ is neither injective nor surjective. However, the following proposition says that under a certain condition, $\sigma_{H,G}$ will be an isomorphism.

\begin{prop} \label{prop-htp-iso} Let $U$ be a non-degenerate star-shaped domain of Liouville manifold $(M, \omega, X)$, $H \preceq G$ in $\mathcal H(U)$ and $a>0$. Suppose there exists a monotone homotopy $\{H_s\}_{s \in [0,1]}$ from $H$ to $G$ such that for any $s \in [0,1]$, $H_s$ does not have 1-periodic orbit with action equal to $a$, then $\sigma_{H,G}: \HF_*^{(a, \infty)}(H) \to \HF_*^{(a, \infty)}(G)$ is an isomorphism. \end{prop}
\end{itemize}

Let $U$ be a non-degenerate star-shaped domain of $(\R^{2n}, \omega_{std}, X_{rad})$.  For any $a>0$, consider a sequence of functions in $\mathcal H(U)$ denoted by $h_a(U) = \{H_i\}_{i \in \N}$ and  shown in Figure \ref{downseq}.
\begin{figure}[H]
\centering
\includegraphics[width=6.5cm]{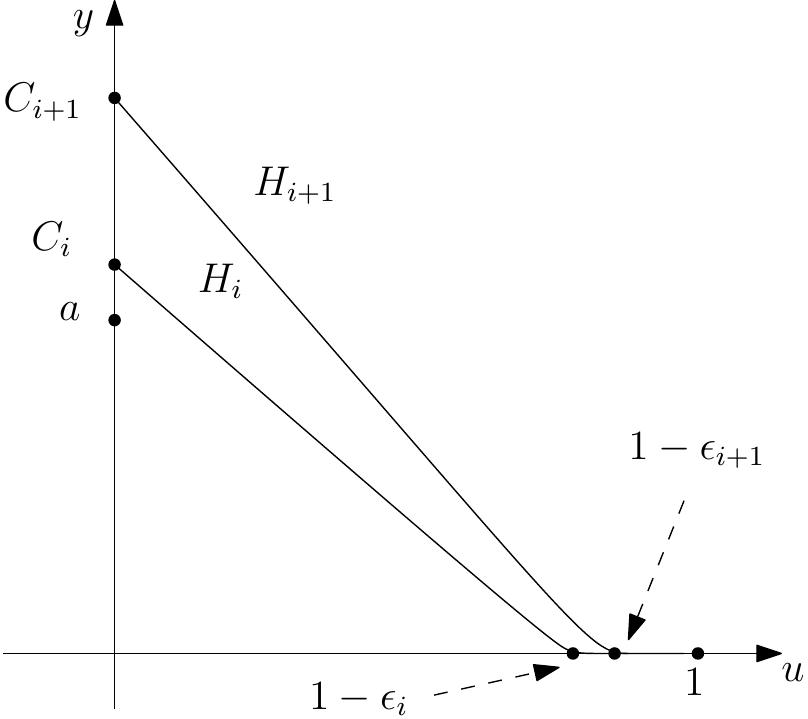}
\caption{An example of downward exhausting sequence}
 \label{downseq}
\end{figure}
Explicitly, $h_a(U) = \{H_i\}_{i \in \N}$ satisfies the following properties,
\begin{itemize}
\item{} $H_{i+1} \geq H_i$;
\item{} $H_1(0) >a$ and $H_i (0) = C_i$ where $C_i \to \infty$ as $i \to \infty$;
\item{} $H_i$ is identically zero for $u  \geq 1- \ep_i$ where $\ep_i \to 0$ as $i \to \infty$.
\end{itemize}

Since $C_i$ diverges, for any $H \in \mathcal H(U)$, there exists $i \in \N$ such that $H_i \preceq H$ for some $H_i \in h_a(U)$. Moreover, by Proposition \ref{prop-htp-iso}, there exists a monotone homotopy from $H_{i+1}$ to $H_i$ such that the induced map $\sigma_{H_{i+1}, H_{i}}: \HF_*^{(a, \infty)}(H_{i+1}) \to \HF_*^{(a, \infty)}(H_{i})$ is an isomorphism. In other words, for any $a>0$, $h_a(U)$ defines a downward exhausting sequence for inverse system $(\HF_*^{(a, \infty)}(H), \sigma_{H,G})$. Then Proposition \ref{prop-il} implies the following useful formula for the computation of filtered symplectic homology,
\begin{equation} \label{comp-fsh}
\SH^{(a, \infty)}_*(U)  =  \HF_*^{(a, \infty)}(H_{i})
\end{equation}
for any $H_i \in h_a(U)$.

\medskip

Let $U = E(1, N, ..., N)$. View each point $z = (z_1, ..., z_n) \in \C^n \setminus \{0\}$ as a pair $(x,u)$ where $u(z)= \pi \left( \frac{|z_1|^2}{1} + \frac{|z_2|^2}{N} + \ldots \frac{|z_n|^2}{N}\right)$ and $x(z) = \frac{z}{\sqrt{u(z)}}$. For any $a >0$, consider
\[ H_a (x,u) = \frac{-a- \delta}{1-\ep} u + (a + \delta) \]
for some small $\ep>0$, and smoothen $H_a(x,u)$ at  $u = 1-\ep$. Note that the function $u(z)$, and therefore $H_a(z)$, extend
smoothly to $z=0$.  Moreover, the value of $\delta$ is so small that the interval $(a, \frac{a+\delta}{1-\ep})$ contains no values of symplectic actions of 1-periodic Hamiltonian orbits of $H_a(x,u)$. Then, in the action window $(a, \infty)$, there exists only one 1-periodic orbit of $H_a(x,u)$ and it is the global maximum at $u=0$. Therefore, the filtered Floer homology is a 1-dimensional vector space over $\Z_2$ generated by this fixed point.

With a proper choice of $\ep$ and $\delta$, in the neighborhood of $u=0$,
\begin{align*}
H_a(z_1, ..., z_n) & = \frac{- a - \delta}{1-\ep}  \left( \pi |z_1|^2 + \pi \frac{|z_2|^2}{N} + \ldots \pi\frac{|z_n|^2}{N}\right) + (a + \delta)\\
& = \pi \ceil*{- a} |z_1|^2+ \sum_{i=2}^n \pi \ceil*{\frac{-a}{N}} |z_i|^2 + \sum_{i=1}^n \pi \alpha_i |z_i|^2 + (a + \delta)
\end{align*}
where each $\alpha_i \in (-1, 0)$. Then by the discussion in Section \ref{sec-CZ},
\begin{equation} \label{CZ}
{\rm Ind}(0) = - 2 \big|\ceil*{-a}\big| - 2(n-1)\bigg| \ceil*{\frac{-a}{N}} \bigg|.
\end{equation}
Therefore, we conclude
\begin{equation*}  \SH_*^{(a,\infty)}(E(1, N, ..., N)) = \Z_2 \,\,\,\,\mbox{when $* = - 2 \big|\ceil*{-a}\big| - 2(n-1)\bigg| \ceil*{\frac{-a}{N}} \bigg|$},
\end{equation*}
and the homologies vanish in all other degrees. Thus we proved (\ref{sh-ell}).

\backmatter

\bibliographystyle {plain}
\bibliography {biblio_listApr8-2019}

\begin{multicols}{2} 
\noindent\\ Leonid Polterovich \\ Faculty of Exact Sciences\\ School of Mathematical Sciences\\Tel Aviv University\\69978 Tel Aviv, Israel\\polterov@tauex.tau.ac.il

\columnbreak

\noindent \\
Daniel Rosen \\ Fakult\"{a}t f\"{u}r Mathematik\\ Universit\"{a}tstr. 150\\Ruhr-Universit\"{a}t Bochum\\44780 Bochum, Germany\\daniel.rosen@rub.de
\end{multicols}

\begin{multicols}{2} 
\noindent\\Karina Samvelyan \\ Faculty of Exact Sciences\\ School of Mathematical Sciences\\Tel Aviv University\\69978 Tel Aviv, Israel\\karina.samvelyan@gmail.com

\columnbreak

\noindent\\Jun Zhang \\ Faculty of Exact Sciences\\ School of Mathematical Sciences\\Tel Aviv University\\69978 Tel Aviv, Israel\\junzhang@mail.tau.ac.il
\end{multicols}

\end{document}